%% file: Alterfold.tex
\documentclass[12pt,reqno]{amsart}
\usepackage{amsmath, wasysym}
\usepackage{bbm, dsfont}
\usepackage{amssymb}
\usepackage{xypic,color}
\usepackage{graphicx}%
\usepackage{hyperref}
\usepackage{cleveref}
\usepackage{comment}
\usepackage{mathrsfs}
\usepackage{float}
\usepackage{tikz, tikz-cd}
\usetikzlibrary{arrows.meta, knots, calc}
\usepackage{bm}
\usepackage{mathtools} 
\usepackage{extarrows} 

\usepackage{ulem}
\usepackage{enumerate}
\usepackage{mathdots}
\usetikzlibrary{decorations.pathreplacing,calligraphy}

\theoremstyle{plain}
\newtheorem*{acknowledgement}{Acknowledgement}

\numberwithin{equation}{section}

\newtheorem{theorem}{Theorem}[section]
\newtheorem{corollary}[theorem]{Corollary}
\newtheorem{lemma}[theorem]{Lemma}
\newtheorem{proposition}[theorem]{Proposition}
\newtheorem{notation}[theorem]{Notation}
\newtheorem{remark}[theorem]{Remark}
\newtheorem{example}[theorem]{Example}

\newtheorem{definition}[theorem]{Definition}

\newcommand{\bZ}{\mathbb{Z}}

\newcommand{\bS}{\mathbb{S}}
\newcommand{\bB}{\mathbb{B}}
\newcommand{\bD}{\mathbb{D}}

\newcommand{\cA}{\mathcal{A}}  
\newcommand{\cB}{\mathcal{B}}

\renewcommand{\hom}{\mathrm{Hom}}

\newcommand{\End}{\mathrm{End}}

\newcommand{\Irr}{\mathrm{Irr}}

\newcommand{\id}{\mathrm{id}}
\newcommand{\ev}{\mathrm{ev}}

\newcommand{\coev}{\mathrm{coev}}
\newcommand{\Tr}{\mathrm{Tr}}

\newcommand{\ONB}{\mathrm{ONB}}

\newcommand{\cyl}{\mathbf{Cyl}}
\renewcommand{\k}{\mathds{k}}
\newcommand{\SL}{\mathrm{SL}}

\newcommand{\cC}{\mathcal{C}}
\newcommand{\cZ}{\mathcal{Z}}
\newcommand{\ot}{\otimes}
\newcommand{\sgn}{\operatorname{sgn}}
\newcommand{\tr}{\operatorname{tr}}

\tikzset{
    partial ellipse/.style args={#1:#2:#3}{
        insert path={+ (#1:#3) arc (#1:#2:#3)}
    }
}

\usepackage{dsfont}
\usepackage{ulem}
\renewcommand{\1}{\mathds{1}}

\newcommand{\tube}{T}
\input{tikz_command.tex}

\input{tikzlibraryipe.code.tex}

\tikzstyle{transition}=[rectangle,draw=black!50,fill=white,thick,
                        inner sep=5pt,minimum size=5mm]
\begin{document}

\title[3-Alterfold]{3-Alterfolds and Quantum Invariants}

\author{Zhengwei Liu}
\address{Z. Liu, Yau Mathematical Sciences Center and Department of Mathematics, Tsinghua University, Beijing, 100084, China}
\address{Yanqi Lake Beijing Institute of Mathematical Sciences and Applications, Huairou District, Beijing, 101408, China}
\email{liuzhengwei@mail.tsinghua.edu.cn}

\author{Shuang Ming}
\address{S. Ming, Yanqi Lake  Beijing Institute of Mathematical Sciences and Applications, Beijing, 101408, China}
\email{sming@bimsa.cn}

\author{Yilong Wang}
\address{Y. Wang, Yanqi Lake  Beijing Institute of Mathematical Sciences and Applications, Beijing, 101408, China}
\email{wyl@bimsa.cn}

\author{Jinsong Wu}
\address{J. Wu, Yanqi Lake  Beijing Institute of Mathematical Sciences and Applications, Beijing, 101408, China}
\email{wjs@bimsa.cn}

\maketitle

\begin{abstract}
In this paper, we introduce the concept of 3-alterfolds with embedded separating surfaces.
When the separating surface is decorated by a spherical fusion category,  we obtain quantum invariants of 3-alterfold, which is consistent with many topological moves. These moves provide evaluation algorithms for various presentations of $3$-alterfold, e.g. Heegaard splittings, triangulations, link surgeries. In particular, we obtain quantum invariants of 3-manifolds containing surfaces, generalizing those of 3-manifolds containing framed links. 
Moreover, in this framework, we topologize fundamental algebraic concepts. For instance, we implement the Drinfeld center by tube diagrams as a blow up of framed links.
The topologized center leads to a quick proof of the equality between the Reshetikhin-Turaev invariants and the Turaev-Viro invariants for spherical fusion categories.
In addition, we topologize half-braiding, $S$-matrix and the generalized Frobenius-Schur indicators, etc.
In particular, the latter leads to a quick topological proof of the equivariance of the generalized Frobenius-Schur indicators.
\end{abstract}

\tableofcontents

\section{Introduction}
We introduce the concept of 3-alterfolds, which are pairs of the form $(M, \Sigma)$, where $M$ is a closed compact oriented 3-manifold and $\Sigma$ is a disjoint union of embedded surfaces that separates $M$ into alternatively colored regions. We call the colors $A$ and $B$ to distinguish them.
When $\Sigma$ is decorated by tensor diagrams $\Gamma$ of a spherical fusion category $\cC$, we call the triple $(M, \Sigma, \Gamma)$ a $\cC$-decorated 3-alterfold, which plays a major role in this paper.

We extend the spherical graphical calculus to obtain a partition function $Z$ on $\cC$-decorated 3-alterfolds that is consistent with topological moves including attaching/eliminating $\cC$-decorated $A$-colored 0-, 1-, 2- and 3-handles to the alterfolds, and surgery moves. The partition function is unique up to a nonzero parameter $\zeta$ detecting the Euler characteristic of the $B$-colored region.
In particular, these topological moves allow us to evaluate the quantum invariant of decorated $3$-alterfold in different ways. 
Now we show how our invariants generalize some remarkable known invariants.

Firstly, through 0-, 1-, 2- and 3-moves, we can evaluate the partition function of a triangulated $3$-alterfolds as state sums, generalizing the Turaev-Viro invariants \cite{TurVir92} of $3$-manifolds. The invariant is independent of the $A$-colored region, so it is an invariant of $B$-colored $3$-manifolds with decoration on their boundary surfaces. It is significant that triangulations on the boundary surfaces are not required. In particular, we obtain quantum invariants of 3-manifolds with boundary, which has no decorations.

Secondly, by considering $3$-alterfolds of $A$-colored tubes in $B$-colored region, we construct a modular tensor category $\mathcal{A}$ with morphisms given by $\mathcal{C}$ decorated $A$-colored handlebodies. Later, we prove this category is equivalent to the Drinfeld center $\mathcal{Z(C)}$ of $\cC$. The tube category $\cA$ can be realized as a topological description of the Drinfeld center. 
It has been proved by Kirillov \cite{Kir11} that the tube category (only consisting cylinderical morphisms) is equivalent (as a category) to $\mathcal{Z(C)}$. When $\mathcal{C}$ is the bi-module category associated with a subfactor, it is equivalent (as a braided fusion category) to the \textit{tube system} described in \cite{KSW05}. Compared with M\"{u}ger's description \cite{Mug03a,Mug03b} of the Drinfeld center by categorical (or weak) Morita equivalence, the tube category produces topological interpretation of half-braiding, braiding, and the minimal central idempotents etc. (See Figure \ref{tab:top}). In particular, evaluatating the partition function of closed $B$-colored $3$-manifold using surgery moves computes the Reshetikhin-Turaev\cite{ResTur91} invariant of $\mathcal{Z(C)}$.

Thirdly, it is remarkable that by evaluating our partition function in two different ways (via triangulations and surgeries respectively), we obtain the equality $TV_{\mathcal{C}}(M)=RT_{\mathcal{Z(C)}}(M)$ between the Turaev-Viro invariant of $M$ associated to $\cC$ and the Reshetikhin-Turaev invariant associated to $\mathcal{Z(C)}$ for any closed oriented 3-manifold $M$. This result have been proved for modular tensor categories by Turaev \cite{Tur94}, Walker \cite{Wal91}, Roberts \cite{Rob95}, for unitary fusion categories comes from finite depth subfactors by Kawahigashi, Sato and Wakui \cite{KSW05}, and for spherical fusion categories by Turaev-Virelizier\cite{TurVir17} and Balsam-Kirillov \cite{BalKir10, Bal10}.

Fourthly, by evaluating our invariant on $3$-alterfolds whose $A$-colored regions are knotted handlebodies without decoration on the separating surfaces, we generalize the Turaev-Viro invariant for pseudo $3$-manifolds introduced by Benedetti and Petronio in \cite{BenPet96} which is originally defined for quantum groups associated to representation theory of $SU(2)$. As a corollary, we derive an equality of the Turaev-Viro invariant of pseudo $3$-manifolds and the evaluation of a link diagram colored by the canonical Lagrangian algebra of $\mathcal{Z(C)}$, which generalize an equality of Detcherry, Kalfagianni and Yang relating Turaev-Viro invariant and the colored Jones polynomials \cite{Jon85, DKY18}.

Last but not least, we show that the generalized Frobenius-Schur indicator of spherical fusion categories can be realized as the partition function of $3$-alterfolds consisting of $A$-colored linked tubes with decorations on the boundary tori, on which $\SL_2(\mathbb{Z})$ acts in a natural way. The generalize FS-indicators are important invariants of spherical fusion categories, and the equivariance of the $\SL_2(\bZ)$ action on their variables is crucial in the solution of the congruence subgroup problem of modular tensor categories \cite{NgSch10}. In Theorem \ref{thm:equivariance}, we give a simple and intuitive proof of the equivariance by the homeomorphism invariance of our partition function $Z$. Topological interpretation of the generalized FS-indicators and their equivariance was studied by Farnsteiner and Schweigert in \cite{FarSch22} under certain hypothesis on the existence of a topological field theory for decorated 3-manifolds with free boundary. By showing the existence of partition functions on decorated 3-manifolds with free boundary in Theorem \ref{thm:3mboundary}, we are able to establish the equivariance without any assumptions. In addition, we give an explanation of the necessity of the conjugation by $\operatorname{diag}(1,-1)$ in the $\SL_2(\bZ)$-action on one side of the equivariance equation, which, algebraically, looks merely like a mysterious coincidence. We argue that this conjugation is given by an orientation reversing homeomorphism on the torus. Moreover, we explain how our topological construction of the generalized FS-indicators can naturally be extended to higher genus, and the equivariance of such topological indicators under the mapping class groups of higher genus surfaces will follow easily.
We organize the paper as follows.
In Section 2, we give a short review on fusion categories and Drinfeld center.
In Section 3, we give the definition of $\mathcal{C}$-decorated 3-alterfold and a partition function on it. 
We prove the main theorems (Theorems \ref{thm:partition function} and \ref{thm:partition function2}) characterizing the invariance of the partition function under certain topological moves.
In Section 4, We defined a braided fusion category $\mathcal{A}$ using the partition function. 
By the main theorems, we show $\cA$ is actually a modular tensor category and braided equivalent to the Drinfeld center of $\mathcal{C}$.
In Section 5, We interpret the generalized Frobenius-Schur indicator as the quantum invariant of certain $\cC$-decorated 3-alterfold. 
We then give a simple topological proof of equivariance of the generalized Frobenius-Schur indicators. 
In Section 6, we discuss the future work and possible directions.

\begin{acknowledgement}
    The authors would like to thank A. Jaffe, Y. Kawahigashi, S.-H. Ng, N. Reshetikhin, Y. Ruan and C. Schweigert for discussions.
    The authors are supported by grants from Yanqi Lake Beijing Institute of Mathematical Sciences and Applications. Z.L. was supported by NKPs (Grant no. 2020YFA0713000), by BMSTC and ACZSP (Grant No. Z221100002722017), and by Templeton Religion Trust (TRT 159). Z.L., Y.W. and J.W. were supported by Beijing Natural Science Foundation Key Program (Grant No. Z220002).
\end{acknowledgement}

\section{Preliminaries}

In this section, we briefly review basic concepts and properties on tensor categories, modular tensor category, Drinfeld center etc.

\subsection{Fusion Categories and Graphical Calculus} 
Let $\mathcal{C}$ be a monoidal category with tensor product $\otimes^{\mathcal{C}}$ and tensor unit object $\mathbbm{1}_{\mathcal{C}}$.
A monoidal functor between two monoidal categories $\mathcal{C}$ and $\mathcal{B}$ is a pair $(F, J)$, where $F: \mathcal{C} \to \mathcal{B}$ is a functor, and 
\[J_{X, Y}: F(X)\otimes^{\cB} F(Y) \rightarrow F(X \otimes^{\mathcal{C}} Y)\]
is a natural isomorphism, satisfying coherence conditions for preserving monoidal structures.
A monoidal functor $(F, J)$ is an equivalence of monoidal categories if the underlying functor $F$ is an equivalence of categories.
A monoidal category $\mathcal{C}$ is strict if the associativity,  the left and the right unit isomorphisms are all identities.
Any monoidal category is monoidally equivalent to a strict one. 
For strict monoidal categories, we can present tensor products and compositions of morphisms using planar diagrams, see \cite{BakKir01, Kir11}.
We always assume monoidal categories appeared in this paper to be strict. 

An object $X^*$ in $\mathcal{C}$ is said to be a left dual of $X$ if there exist an evaluation $\ev_X:X^*\otimes X\to \mathbbm{1}$ and a coevaluation $\coev_X: \mathbbm{1} \to X\otimes X^*$ such that 
\begin{align}\label{eq:leftdual}
(\id_X\otimes \ev_X)\circ(\coev_X\otimes \id_X)=\id_X, \quad (\ev_X\otimes X^*)\circ(\id_{X^*}\otimes \coev_X)=\id_{X^*}.
\end{align}
Right duals are defined similarly.
In graphical calculus, we draw the evaluation and the coevaluation as cup and cap respectively. The identities \ref{eq:leftdual} are depicted as 
\begin{align*}
\raisebox{-1cm}{
\begin{tikzpicture}[xscale=0.5]
\draw [blue, -<-=0.5] (0, -1)--(0, 0) .. controls +(0, 1) and +(0, 1) .. (1.5, 0);
\draw [blue, -<-=0.5] (1.5, 0) .. controls +(0, -1) and +(0, -1) .. (3, 0)--(3, 1) node [black, above] {\tiny $X$};
\end{tikzpicture}}
=
\raisebox{-1cm}{
\begin{tikzpicture}
    \draw [blue, ->-=0.5] (0, 1) node [black, above] {\tiny $X$} --(0, -1);
\end{tikzpicture}
}\,,
\quad 
\raisebox{-1cm}{
\begin{tikzpicture}[xscale=0.5]
\draw [blue, -<-=0.5] (0, 1) node [black, above] {\tiny $X$} --(0, 0) .. controls +(0, -1) and +(0, -1) .. (1.5, 0);
\draw [blue, -<-=0.5] (1.5, 0) .. controls +(0, 1) and +(0, 1) .. (3, 0)--(3, -1);
\end{tikzpicture}}
=
\raisebox{-1cm}{
\begin{tikzpicture}
    \draw [blue, -<-=0.5] (0, 1) node [black, above] {\tiny $X$} --(0, -1);
\end{tikzpicture}
}\,, 
\end{align*}
where $\raisebox{-0.5cm}{
\begin{tikzpicture}[scale=0.5]
    \draw [blue, -<-=0.5] (0, 1) node [black, above] {\tiny $X$} --(0, -1);
\end{tikzpicture}
} = \raisebox{-0.5cm}{
\begin{tikzpicture}[scale=0.5]
    \draw [blue, ->-=0.5] (0, 1) node [black, above] {\tiny $X^*$} --(0, -1);
\end{tikzpicture}
}$.
A monoidal category $\mathcal{C}$ is left (resp.~right) rigid if every object admits a left (resp.~right) dual, and $\mathcal{C}$ is rigid if it is left rigid and right rigid. It is well-known that $X \mapsto X^{**}$ extends to a monoidal functor $**$ \cite{EGNO15}.

A pivotal category is a rigid monoidal category equipped with an isomorphism $j: \id\to **$ between monoidal functors, which is called a pivotal structure on $\mathcal{C}$. 
Given a pivotal structure $j$ on $\mathcal{C}$, the left pivotal trace $\Tr_L$ is defined as follows.
\begin{align*}
    \Tr_L(f)=\ev_{X^*} (j_X\otimes \id_{X^*})(f\otimes \id_{X^*})\coev_X\,, \quad \text{for all } f\in \hom(X, X)\,.
\end{align*}
The right pivotal trace $\Tr_R$ is defined similarly. If $\Tr_L$ is equal to $\Tr_R$ on all endomorphisms in $\mathcal{C}$, then $j$ is called spherical, and $\mathcal{C}$ is called a spherical category. 
In this case, we simply write $\Tr=\Tr_{L}=\Tr_{R}$, and call $\Tr$ the trace on $\mathcal{C}$. 
A pivotal strict rigid monoidal category is said to be strict pivotal if $j_X=\id_{X}$ for all object $X$ in $\mathcal{C}$, and $(X\otimes Y)^*= Y^*\otimes X^{*}$ for all $X, Y\in\mathcal{C}$. In \cite{NgSch07b}, it is shown that any pivotal category is equivalent to a strict pivotal one as pivotal categories. Hence, we assume all pivotal categories in this paper are strict pivotal. Under this assumption, we depict the left (right) pivotal trace in our graphical calculus as
\begin{align*}
    \Tr(f)=\Tr_L(f)=\Tr_R(f)=
    \raisebox{-0.6cm}{
    \begin{tikzpicture}[scale=0.5]
        \draw [blue] (0, 0) .. controls +(0, 1) and +(0, 1).. (1.5, 0);
        \draw [blue] (0, -1) .. controls +(0, -1) and +(0, -1).. (1.5, -1);
        \draw (-0.75, -1) rectangle (0.75, 0); \draw [blue, -<-=0.5] (1.5, 0)--(1.5, -1);
        \node at (0, -0.5) {\tiny $f$};
    \end{tikzpicture}
    }
    =    \raisebox{-0.6cm}{
    \begin{tikzpicture}[scale=0.5]
        \draw [blue] (0, 0) .. controls +(0, 1) and +(0, 1).. (-1.5, 0);
        \draw [blue] (0, -1) .. controls +(0, -1) and +(0, -1).. (-1.5, -1);
        \draw (-0.75, -1) rectangle (0.75, 0); \draw [blue, ->-=0.5] (-1.5, 0)--(-1.5, -1);
        \node at (0, -0.5) {\tiny $f$};
    \end{tikzpicture}
    }
\end{align*}

Let $\k$ be an algebraically closed field of characteristic zero. A monoidal category $\mathcal{C}$ is a fusion category if $\mathcal{C}$ is $\k$-linear, finite, semisimple, rigid and $\End(\mathbbm{1})=\k$. In particular, finiteness implies the underlying category consists of finitely many isomorphism classes of simple objects. A fusion category $\mathcal{C}$ is spherical if it admits a spherical pivotal structure. 

For a $\k$-linear category $\mathcal{C}$, the Karoubi envelope \cite{Del07} $\widetilde{\mathcal{C}}$ of $\mathcal{C}$ is define to be the $\k$-linear category with objects of the form $(X, p)$, where $X\in \text{Ob}(\mathcal{C})$ and $p$ is an idempotent in $\End(X)$ (i.e., $p^2=p$). Morphism spaces between objects are defined to be
\begin{align*}
\hom((X, p), (Y, q))=\{f\in \hom(X, Y)\mid f=q\circ f \circ p \}.
\end{align*}
A monoidal structure on $\mathcal{C}$ naturally induces one on $\widetilde{\cC}$ by
\begin{align*}
(X, p)\otimes  (Y, q)=(X Y, p\otimes q)\,.
\end{align*}
It is easy to see that if $\mathcal{C}$ is rigid (resp.~pivotal, spherical), then $\widetilde{\mathcal{C}}$ is rigid (reps.~pivotal, spherical).

Fix a spherical fusion category $\mathcal{C}$, we denote the set of isomorphism classes of simple objects as $\Irr(\mathcal{C})=\{\1=X_{0}, X_1,\ldots X_{r}\}$. The quantum dimension $d_X$ of object $X$ is defined to the $\Tr(\id_{X})$. In particular, we define $d_i$ to be the quantum dimension of the simple object $X_i$. The global dimension of $\mathcal{C}$ is equal to 
\[\mu := \displaystyle{\sum_{0\le i \le r}d_i^{2}}\,.\]

Let $X$ and $Y$ be objects in $\mathcal{C}$, we define a basis $\{\phi_{j}\}$ of $\hom(X, Y)$ and a basis $\{\phi_{j}'\}$ of $\hom(Y, X)$ be a pair of dual basis if $\Tr(\phi_{j}\phi_{k}')=\delta_{jk}$.

\begin{lemma}\cite[Lemma 5, Unit Decomposition]{FarSch22} \label{lem:unit1}
Suppose $\{\phi_{i, j}\}_j$ and $\{\phi_{i, j}'\}_{j}$ be a pair of dual base in $\hom(X, X_i)$ and $\hom(X_i, X)$ for $i=1, \ldots, r$. 
Then 
\begin{align*}
\sum_{i}\sum_{j} d_{i} \phi_{i, j}^{'} \circ \phi_{i,j}=\id_{X}\,.
\end{align*}
\end{lemma}

We depict the equality in the above lemma by
\begin{align*}
\sum_{i,j} d_i
    \vcenter{\hbox{
    \begin{tikzpicture}
        \draw [blue, -<-=0.5, -<-=0.1, -<-=0.9] (0, -1.5) --(0, 1.5) node [pos=0.5, right, black] {\tiny $X_i$} node [pos=0.1, right, black] {\tiny $X$} node [pos=0.9, right, black] {\tiny $X$};
        \node [draw, fill=white] at (0, 0.55) {\tiny $\phi_{i,j}$};
        \node [draw, fill=white] at (0, -0.55) {\tiny $\phi_{i,j}'$};
    \end{tikzpicture}}}
    = \vcenter{\hbox{
\begin{tikzpicture}
    \draw [blue, ->-=0.5] (0, 1)--(0, -1) node [pos=0.5, right, black] {\tiny $X$};
\end{tikzpicture}
}}\,.
\end{align*}

\begin{notation}
    We will denote by the dual base $\{\phi_j\}_j$ and $\{\phi_j'\}_j$ by $\phi$ and $\phi'$ and suppress the summation if there is no confusion.
\end{notation}

The following two lemmas are direct consequences of basic linear algebra. We include them here to familiarize the reader with our graphical calculus conventions, and for future reference.

\begin{lemma}\label{lem:unit2}
Suppose $\phi_{j}$ and $\phi_{j}^{'}$ be dual base in $\hom(X, \1)$ and $\hom(\1, X)$, and $\psi_k$ and $\psi_k^{'}$ be dual base in $\hom(Y, \1)$ and $\hom(\1, Y)$. Then for any $f\in\hom(X, Y)$ and $g\in \hom(Y, X)$
\begin{align*}
\sum_{j} \phi_{j}\circ g\circ f\circ\phi_{j}^{'}
=\sum_{k} \psi_{k}\circ f\circ g\circ\psi_{k}^{'}\,,
\end{align*}
depicted as
\begin{align*}
    \vcenter{\hbox{
    \begin{tikzpicture}
        \draw [blue, -<-=0.5, -<-=0.2, -<-=0.8] (-1.5, 0) --( 1.5, 0);
        \node [draw, fill=white] at (1.5, 0) {\tiny $\phi$};
        \node [draw, fill=white] at (-1.5, 0) {\tiny $\phi'$};
        \node [draw, fill=white] at ( 0.4, 0) {\tiny $g$};
        \node [draw, fill=white] at (-0.4, 0) {\tiny $f$};
    \end{tikzpicture}
    }}
    =
    \vcenter{\hbox{
    \begin{tikzpicture}
        \draw [blue, -<-=0.5, -<-=0.2, -<-=0.8] (-1.5, 0) --( 1.5, 0);
        \node [draw, fill=white] at (1.5, 0) {\tiny $\psi$};
        \node [draw, fill=white] at (-1.5, 0) {\tiny $\psi'$};
        \node [draw, fill=white] at ( 0.4, 0) {\tiny $f$};
        \node [draw, fill=white] at (-0.4, 0) {\tiny $g$};
    \end{tikzpicture}
    }}\,.
\end{align*}
\end{lemma}

\begin{lemma}\label{lem:unit3}
   Suppose $\phi_{j}$ and $\phi_{j}^{'}$ be dual base in $\hom(X, \1)$ and $\hom(\1, X)$.
   We have 
   \begin{align*}
    \vcenter{\hbox{
    \begin{tikzpicture}
        \draw [blue, -<-=0.2, -<-=0.8] (-1.5, 0) --( 1.5, 0);
        \draw [white, thick] (0.4, 0)--(-0.4, 0);
        \node [draw, fill=white] at (1.5, 0) {\tiny $f$};
        \node [draw, fill=white] at (-1.5, 0) {\tiny $g$};
        \node [draw, fill=white] at ( 0.4, 0) {\tiny $\phi$};
        \node [draw, fill=white] at (-0.4, 0) {\tiny $\phi'$};
    \end{tikzpicture}
    }}
    =
    \vcenter{\hbox{
    \begin{tikzpicture}
        \draw [blue, -<-=0.5, -<-=0.2, -<-=0.8] (-0.5, 0) --( 0.5, 0);
        \node [draw, fill=white] at ( 0.4, 0) {\tiny $f$};
        \node [draw, fill=white] at (-0.4, 0) {\tiny $g$};
    \end{tikzpicture}
    }}\,.
\end{align*}
\end{lemma}

A tensor diagram $\Gamma$ of $\mathcal{C}$ on a surface is a finite graph with vertices thicken to a coupon colored by morphisms in $\mathcal{C}$ and edges colored by objects in $\mathcal{C}$ coherent with the incoming and outgoing edges attached to it such that (1) if an edge is blue, then it is directed; (2) if a connected component of $\Gamma$ contains no coupons/vertices and is red, then it must be a closed loop on the surface, and it is understood as colored by the sum of all simple objects $X_i$ with weight $d_i$. The red closed curve is called the Kirby color or $\Omega$-color of $\cC$.

A crucial property of the Kirby color of a spherical fusion category $\mathcal{C}$ is the following handle slide property.

\begin{lemma}[Handle Slide]
    We have that
    \begin{align*}
    \vcenter{\hbox{
\begin{tikzpicture}
    \draw [blue, ->-=0.5] (0, 1)--(0, -1) node [pos=0.5, left, black] {\tiny $X$};
    \begin{scope}[shift={(0.3, 0)}]
     \draw [red] (0, 0) .. controls +(0, 0.6) and +(0, 0.6) .. (1, 0) (0, 0) .. controls +(0, -0.6) and +(0, -0.6) .. (1, 0);
    \end{scope}
\end{tikzpicture}
}}=
\vcenter{\hbox{
    \begin{tikzpicture}
  \draw [red] (0, 0) .. controls +(0, 0.6) and +(0, 0.6) .. (1, 0) (0, 0) .. controls +(0, -0.6) and +(0, -0.6) .. (1, 0);
  \draw [blue] (-0.2, 0.2) .. controls +(0, 0.7) and +(0, 0.7) .. (1.2, 0.2) (-0.2, -0.2) .. controls +(0, -0.7) and +(0, -0.7) .. (1.2, -0.2);
  \draw [blue, ->-=0.5] (1.2, 0.2)--(1.2, -0.2) node [pos=0.5, right, black] {\tiny $X$};
  \draw [blue] (-0.2, 0.2) .. controls +(0, -0.1) and +(0, -0.1) .. (-0.4, 0.2)--(-0.4, 1) (-0.2, -0.2) .. controls +(0, 0.1) and +(0, 0.1) .. (-0.4, -0.2)--(-0.4, -1);
    \end{tikzpicture}
    }}\,.
    \end{align*}
\end{lemma}
\begin{proof}
It follows from Lemma \ref{lem:unit1} directly.
\end{proof}

\subsection{Drinfeld Center}
A monoidal category $\mathcal{C}$ is braided if it admits a natural isomorphism $c_{X,Y}:X\otimes Y\rightarrow Y\otimes X$ satisfying the hexagon axioms (\cite[Chapter 8]{EGNO15}).
A twist $\theta$ is an autoequivalence of the identity functor $\text{id}_{\mathcal{C}}$ subject to
\begin{align*}
\theta_{X\otimes Y}=(\theta_X\otimes \theta_Y)c_{Y,X} c_{X,Y}
\end{align*}
for any $X$, $Y \in \cC$. A twist is called a ribbon structure if $(\theta_{X})^*=\theta_{X^*}$.
For a simple object $X_j$, $\theta_{X_j}$ is a multiple of $\id_{X_j}$, and we denote this scalar by $\theta_j$.
A half braiding $e_X$ assigns each object $Y$ in $\mathcal{C}$ a natural isomorphism
$$e_{X}(Y):X\otimes Y\rightarrow Y\otimes X$$
such that $e_X(\mathbbm{1})=\id_X$ and
\begin{align*}
    (Y_1\otimes e_X(Y_2))(e_X(Y_1)\otimes Y_2)=e_X(Y_1\otimes Y_2),
\end{align*}
for all $Y_1, Y_2\in\mathcal{C}$.
The Drinfeld center $\mathcal{Z}(\mathcal{C})$ of a monoidal category $\mathcal{C}$ is a monoidal category with objects $(X, e_X)$ and morphism spaces
\begin{align*}
\hom((X, e_X), (Y, e_Y))=\{f\in \hom(X, Y)\mid (f\otimes \id_V) e_X(V)=e_Y(V) (\id_V\otimes f), \forall V\in \mathcal{C}\}\,,
\end{align*}
where the tensor product of $(X, e_X), (Y, e_Y)$ is $(XY, e_{XY})$ and the tensor unit is $(\mathbbm{1}, e_{\mathbbm{1}})$.
The Drinfeld center $\mathcal{Z}(\mathcal{C})$ is a braided monoidal category.
If $\mathcal{C}$ is rigid, pivotal and spherical respectively, then $\mathcal{Z}(\mathcal{C})$ is rigid, pivotal and spherical respectively (see \cite{Mug03b} for detail).

For simplicity, in our graphical calculus, we draw half braidings as crossings, but they should be thought of as special coupons. By naturality, a red string through a half braiding can still be considered as being colored by the sum over simple objects, weighted by their dimensions. Moreover, the naturality also implies that we are able to perform handle slide through half braidings. Precisely, we have the following lemma. 

\begin{lemma}
Suppose $\mathcal{C}$ is a fusion category and $V\in\mathcal{C}$,  $(X, e_X), (Y, e_Y)\in \mathcal{Z}(\mathcal{C})$.
Then we have that
    \begin{align*}
    \raisebox{-1cm}{
\begin{tikzpicture}[scale=1.2]
    \draw [blue, ->-=0.5] (0, 1)--(0, -1) node [pos=0.5, left, black] {\tiny $V$};
    \begin{scope}[shift={(0.3, 0)}]
     \draw [red] (0, 0) .. controls +(0, 0.6) and +(0, 0.6) .. (1, 0) (0, 0) .. controls +(0, -0.6) and +(0, -0.6) .. (1, 0);
     \path [fill=white] (0.4, -1) rectangle (0.6, 1);
     \draw [blue, -<-=0.1, -<-=0.9] (0.5, -1)--(0.5, 1) node [pos=0.1, right] {\tiny $Y$} node [pos=0.9, right, black] {\tiny $X$};
     \node [draw, fill=white] at (0.5, 0) {\tiny $f$};
    \end{scope}
\end{tikzpicture}
}=
    \raisebox{-1cm}{
    \begin{tikzpicture}[scale=1.2]
  \draw [red] (0, 0) .. controls +(0, 0.6) and +(0, 0.6) .. (1, 0) (0, 0) .. controls +(0, -0.6) and +(0, -0.6) .. (1, 0);
  \draw [blue] (-0.2, 0.2) .. controls +(0, 0.7) and +(0, 0.7) .. (1.2, 0.2) (-0.2, -0.2) .. controls +(0, -0.7) and +(0, -0.7) .. (1.2, -0.2);
  \draw [blue, ->-=0.5] (1.2, 0.2)--(1.2, -0.2) node [pos=0.5, right, black] {\tiny $V$};
  \draw [blue] (-0.2, 0.2) .. controls +(0, -0.1) and +(0, -0.1) .. (-0.4, 0.2)--(-0.4, 1) (-0.2, -0.2) .. controls +(0, 0.1) and +(0, 0.1) .. (-0.4, -0.2)--(-0.4, -1);
    \path [fill=white] (0.4, -1) rectangle (0.6, 1);
     \draw [blue, -<-=0.1, -<-=0.9] (0.5, -1)--(0.5, 1) node [pos=0.1, right, black] {\tiny $Y$} node [pos=0.9, right, black] {\tiny $X$};
     \node [draw, fill=white] at (0.5, 0) {\tiny $f$};
    \end{tikzpicture}
    }
    \end{align*}.
\end{lemma}

\begin{proof}
It follows from Lemma \ref{lem:unit1} and the naturality of the half braidings.
\end{proof}
For all $f\in \hom_{\mathcal{C}}(X, Y)$ and $(X, e_X), (Y, e_Y)\in \mathcal{Z}(\mathcal{C})$, there exists a projector $P_{\mathcal{C}}: \hom_{\mathcal{C}}(X, Y)\to \hom_{\mathcal{Z}(\mathcal{C})}((X, e_X), (Y, e_Y))$ given by
\begin{align*}
    P_{\mathcal{C}}(f)=\frac{1}{\mu}
       \vcenter{\hbox{
\begin{tikzpicture}[scale=1.2]
    \begin{scope}[shift={(0.3, 0)}]
     \draw [red] (0, 0) .. controls +(0, 0.6) and +(0, 0.6) .. (1, 0) (0, 0) .. controls +(0, -0.6) and +(0, -0.6) .. (1, 0);
     \path [fill=white] (0.4, -1) rectangle (0.6, 1);
     \draw [blue, -<-=0.1, -<-=0.9] (0.5, -1)--(0.5, 1) node [pos=0.1, right, black] {\tiny $Y$} node [pos=0.9, right, black] {\tiny $X$};
     \node [draw, fill=white] at (0.5, 0) {\tiny $f$};
    \end{scope}
\end{tikzpicture}}}\,.
\end{align*}
We see that $\mathcal{Z}(\mathcal{C})$ is semisimple if $\mathcal{C}$ is semisimple by using $P_{\mathcal{C}}$ (see \cite{Mug03b} for more details).
Suppose $\mathcal{C}$ is a braided spherical fusion category, and let
$$S_{jk}=\Tr(c_{X_{k}, X_j^{*}} \circ c_{X_j^{*}, X_k})$$
then the matrix $(S_{jk})_{j,k=0}^r$ is called the $S$-matrix of $\mathcal{C}$. Let $T_{j,k}=\delta_{j,k}\theta_j$, we call $(T_{jk})_{j,k=0}^r$ the $T$-matrix of $\cC$.
A braided fusion category $\mathcal{C}$ is a modular tensor category if $\mathcal{C}$ is spherical and the $S$-matrix is invertible \cite{Mug03c}.
According to M\"uger \cite{Mug03b}, the Drinfeld center $\mathcal{Z}(\mathcal{C})$ of a spherical fusion category $\mathcal{C}$ is a modular tensor category.

\section{Partition Function over 3-Alterfold}
In this section, we give the definition of decorated 3-alterfolds and elementary moves of it. In addition, we define a partition function $Z$ that is invariant under the elementary moves. 

\subsection{3-Alterfold}
We begin with the definition of 3-alterfold.
\begin{definition}
By a 3-alterfold, we mean a pair $(M, \Sigma)$ where $M$ is an oriented, compact $3$-manifold without boundary and $\Sigma$ is an embedded oriented compact surface without boundary in $M$ separating $M\setminus \Sigma$ into connected components that are alternatively colored by colors $A$ and $B$. We denote the $B$-colored (resp.\ $A$-colored) region by
$R_{B}(M, \Sigma)$ (resp.\ $R_{A}(M, \Sigma)$). We further require the orientation of $\Sigma$ to be consistent with the boundary orientation of $R_{A}(M, \Sigma)$, which is the opposite of the boundary orientation of $R_{B}(M, \Sigma)$.
\end{definition}

\begin{remark}
Note that a $3$-alterfold actually is an alternatively shaded 3-manifold.
When the 3-alterfold $(M, \Sigma)$ is clear from the context, we simply denote $R_B(M, \Sigma)$ (resp.~$R_A(M, \Sigma)$) by $R_B$ (resp.~$R_A$). Also, the definition could be extended to a multicolored $3$-manifold with an consistently oriented embedded surface separating the manifold into multiple colored regions.
\end{remark}

\begin{definition}
Let $\mathcal{C}$ be a spherical fusion category. By a $\mathcal{C}$-decorated 3-alterfold, we mean a triple $(M, \Sigma, \Gamma)$ where $(M, \Sigma)$ is a 3-alterfold and $\Gamma$ is a $\mathcal{C}$-diagram embedded in the separating surface $\Sigma$.
\end{definition}

\begin{remark}
In our convention, the diagram $\Gamma$ is a $\mathcal{C}$-diagram when reading from $R_{B}$. When reading from the $R_{A}$, $\Gamma$ can be viewed as a diagram of $\mathcal{C}^{\otimes op}$ or $\mathcal{C}^{op}$.
\end{remark}

\begin{example}\label{exa:sphere}
Let $M=\mathbb{R}^{3}\cup \{\infty\}\cong \bS^{3}$ and $\Sigma=\{x^{2}+y^{2}+z^{2}=1\}\cong\bS^{2}$ be the separating sphere with the normal vector pointing to the $\infty$, and $\Gamma$ be a diagram over the separating sphere. Then $(\bS^{3}, \bS^{2}, \Gamma)$ is a $\mathcal{C}$-decorated 3-alterfold with $R_{A}=\{x^{2}+y^{2}+z^{2}<1\}$ and $R_{B}=\{x^{2}+y^{2}+z^{2}>1\}$. If we change the orientation of the separating surface $\Sigma$, then the $A-B$ coloring of the regions are interchanged, and $\Gamma\subset \Sigma$ now becomes diagram in $\mathcal{C}^{\otimes op}$.
\end{example}

We define the following elementary moves regarding the diagram evaluation and color interchange.
\begin{itemize}
\item \textbf{Planar Graphical Calculus:} Let $D\subset \Sigma$ be an embedded disk and $\Gamma_{D}=D\cap \Gamma$. Suppose $\Gamma_{D}'=\Gamma_{D}$ as morphisms in $\mathcal{C}$. We change the $\Gamma$ to $\Gamma'$ by replacing $\Gamma_{D}$ to $\Gamma_{D}'$.

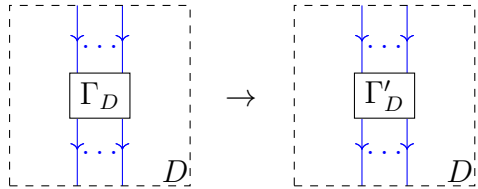
$$
\begin{array}{ccc}
   \vcenter{\hbox{
\begin{tikzpicture}
\draw[->-=0.8, ->-=0.2, blue] (-.3, 1.2)--(-.3, -1.2);
\draw[->-=0.8, ->-=0.2, blue] (.3, 1.2)--(.3, -1.2);
\draw[blue] (0, -.75) node{$\ldots$};
\draw[blue] (0, .65) node{$\ldots$};
\draw[fill=white] (-.4, -.3) rectangle (.4, .3);
\draw (0, 0) node{$\Gamma_{D}$};
\draw[dashed] (-1.2, -1.2) rectangle (1.2, 1.2);
\draw (1, -1) node{$D$};
\end{tikzpicture}}}
&
\rightarrow
&
\vcenter{\hbox{
\begin{tikzpicture}    
\draw[->-=0.8, ->-=0.2, blue] (-.3, 1.2)--(-.3, -1.2);
\draw[->-=0.8, ->-=0.2, blue] (.3, 1.2)--(.3, -1.2);
\draw[blue] (0, -.75) node{$\ldots$};
\draw[blue] (0, .65) node{$\ldots$};
\draw[fill=white] (-.4, -.3) rectangle (.4, .3);
\draw (0, 0) node{$\Gamma_{D}'$};
\draw[dashed] (-1.2, -1.2) rectangle (1.2, 1.2);
\draw (1, -1) node{$D$};
\end{tikzpicture}}}\\
\end{array}
$$

\item \textbf{Move 0:} Let $P$ be a point in the interior of $R_B$. We change the color of a tubular neighborhood $P_{\epsilon}$ of $P$ from $B$ to $A$.
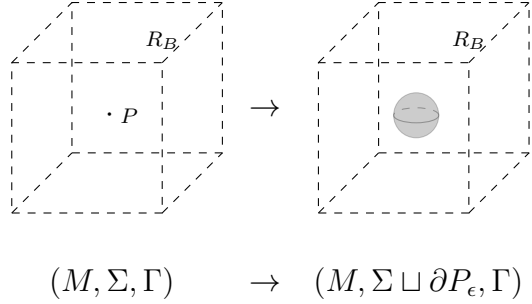
$$
\begin{array}{ccc}
   \vcenter{\hbox{
\begin{tikzpicture}
\draw[dashed] (0, 0) rectangle (2, 2);
\draw[dashed] (.8, .8) rectangle (2.8, 2.8);
\draw[dashed] (0, 2)--+(0.8, 0.8);
\draw[dashed] (2, 2)--+(0.8, 0.8);
\draw[dashed] (2, 0)--+(0.8, 0.8);
\draw[dashed] (0, 0)--+(0.8, 0.8);
\draw (2, 2.3) node{\tiny{$R_{B}$}};
\draw (1.3, 1.3) node[right]{\tiny{$P$}} node{$\cdot$};
\end{tikzpicture}}}  & \rightarrow & \vcenter{\hbox{
\begin{tikzpicture}
\draw[dashed] (0, 0) rectangle (2, 2);
\draw[dashed] (.8, .8) rectangle (2.8, 2.8);
\draw[dashed] (0, 2)--+(0.8, 0.8);
\draw[dashed] (2, 2)--+(0.8, 0.8);
\draw[dashed] (2, 0)--+(0.8, 0.8);
\draw[dashed] (0, 0)--+(0.8, 0.8);
\filldraw[black,opacity=0.2] (1.3, 1.3) circle (0.3);
\draw[dashed, opacity=0.3] (1.3, 1.3) [partial ellipse=0:180:0.3 and 0.1];
\draw[opacity=0.3] (1.3, 1.3) [partial ellipse=180:360:0.3 and 0.1];
\draw (2, 2.3) node{\tiny{$R_{B}$}};
\end{tikzpicture}}}\\
\vspace{2mm}\\
(M, \Sigma, \Gamma)&\rightarrow & (M, \Sigma \sqcup \partial P_{\epsilon}, \Gamma)\\
\end{array}
$$

\item \textbf{Move 1:} Let $S$ be an embedded arc in $R_{B}$ with $\partial S$ meeting $\Sigma$ transversely and not intersecting $\Gamma$. We change the color of the tubular neighborhood $S_{\epsilon}$ of $S$ from $B$ to $A$, and put an $\Omega$-colored circle $C$ on the belt of $S_{\epsilon}$.

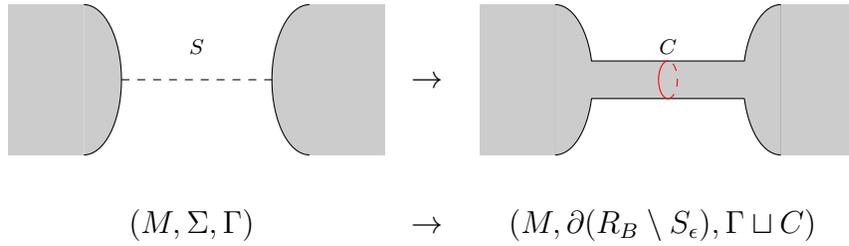
$$
\begin{array}{ccc}
   \vcenter{\hbox{
\begin{tikzpicture}
\path[fill=black,opacity=0.2]
(-1.5, 0) [partial ellipse=-90:90:0.5 and 1];
\draw (-1.5, 0) [partial ellipse=-90:90:0.5 and 1];
\path[fill=black,opacity=0.2]
(-2.5, -1) rectangle (-1.5, 1);
\begin{scope}[xscale=-1]
\path[fill=black,opacity=0.2]
(-1.5, 0) [partial ellipse=-90:90:0.5 and 1];
\draw (-1.5, 0) [partial ellipse=-90:90:0.5 and 1];
\path[fill=black,opacity=0.2]
(-2.5, -1) rectangle (-1.5, 1);
\end{scope}
\draw[dashed] (-1, 0)--(1, 0);
\draw (0, 0.2) node[above]{\tiny{$S$}};
\end{tikzpicture}}}  & \rightarrow & \vcenter{\hbox{
\begin{tikzpicture}
\path[fill=black,opacity=0.2]
(-1.5, 0) [partial ellipse=-90:90:0.5 and 1];
\draw (-1.5, 0) [partial ellipse=-90:90:0.5 and 1];
\path[fill=black,opacity=0.2]
(-2.5, -1) rectangle (-1.5, 1);
\begin{scope}[xscale=-1]
\path[fill=black,opacity=0.2]
(-1.5, 0) [partial ellipse=-90:90:0.5 and 1];
\draw (-1.5, 0) [partial ellipse=-90:90:0.5 and 1];
\path[fill=black,opacity=0.2]
(-2.5, -1) rectangle (-1.5, 1);
\end{scope}
\draw[dashed] (-1, 0)--(1, 0);
\draw (0, 0.2) node[above]{\tiny{$C$}};
\path[fill=white]
(-1.1, -0.25) rectangle (1.1, 0.25);
\path[fill=black, opacity=0.2]
(-1.1, -0.25) rectangle (1.1, 0.25);
\draw (-1.02, 0.25)--(1.02, 0.25);
\draw (-1.02, -0.25)--(1.02, -0.25);
\draw[red, dashed] (0, 0) [partial ellipse=-90:90:0.125 and 0.25];
\draw[red] (0, 0) [partial ellipse=90:270:0.125 and 0.25];
\end{tikzpicture}}}\\
\vspace{2mm}\\
(M, \Sigma, \Gamma)&\rightarrow & (M, \partial(R_{B}\setminus S_{\epsilon}), \Gamma\sqcup C)\\
\end{array}
$$

\item \textbf{Move 2:} Let $D$ be a disk in $R_{B}$ with $\partial D\subset \partial R_{B}$ intersect $\Gamma$ only at edges transversely. We change the color of a tubular neighborhood of $D$ from $B$ to $A$, then cut the diagram along $\partial D$ and put a pair of sum of dual basis $\phi$ and $\phi'$ on both side of $D_{\epsilon}$.

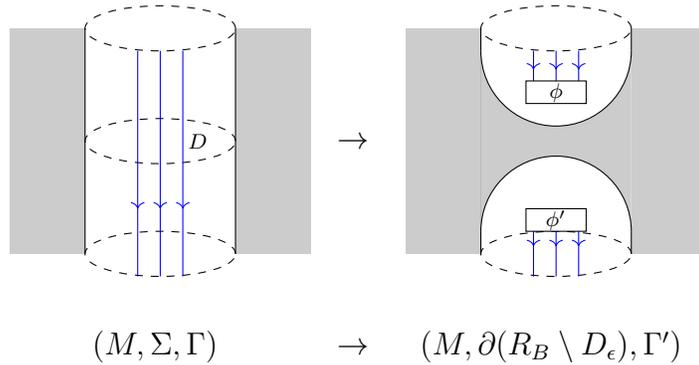
$$
\begin{array}{ccc}
   \vcenter{\hbox{
\begin{tikzpicture}
\draw[dashed] (0, 1.5) [partial ellipse=0:360:1 and 0.3];
\draw[dashed] (0, 0) [partial ellipse=0:360:1 and 0.3];
\draw[dashed] (0, -1.5) [partial ellipse=0:360:1 and 0.3];
\draw (0.5, 0) node{\tiny{$D$}};
\draw (-1, -1.5)--(-1, 1.5);
\draw (1, -1.5)--(1, 1.5);
\path[fill=black, opacity=0.2]
(-1, -1.5) rectangle (-2, 1.5);
\path[fill=black, opacity=0.2]
(1, -1.5) rectangle (2, 1.5);
\draw [blue, ->-=0.7] (0, 1.2)--(0, -1.8);
\draw [blue, ->-=0.7] (-0.3, 1.2)--(-0.3, -1.8);
\draw [blue, ->-=0.7] (0.3, 1.2)--(0.3, -1.8);
\end{tikzpicture}}}  & \rightarrow & \vcenter{\hbox{
\begin{tikzpicture}
\draw[dashed] (0, 1.5) [partial ellipse=0:360:1 and 0.3];
\draw[dashed] (0, -1.5) [partial ellipse=0:360:1 and 0.3];
\draw (-1, -1.5)--(-1, -1.2);
\draw (1, 1.5)--(1, 1.2);
\draw (-1, 1.5)--(-1, 1.2);
\draw (1, -1.5)--(1, -1.2);
\path[fill=black, opacity=0.2]
(-1, -1.5) rectangle (-2, 1.5);
\path[fill=black, opacity=0.2]
(1, -1.5) rectangle (2, 1.5);
\draw [blue, ->-=0.5] (0, 1.2)--(0, 0.7);
\draw [blue, ->-=0.5] (-0.3, 1.2)--(-0.3, 0.7);
\draw [blue, ->-=0.5] (0.3, 1.2)--(0.3, 0.7);
\draw [blue, ->-=0.5] (0, -1)--(0, -1.8);
\draw [blue, ->-=0.5] (-0.3, -1)--(-0.3, -1.8);
\draw [blue, ->-=0.5] (0.3, -1)--(0.3, -1.8);
\path[fill=black, opacity=0.2]
(-1, -1.2) arc (180:0:1)--(1, 1.2) arc (0:-180:1)--(-1, -1.2);
\draw (-1, -1.2) arc (180:0:1);
\draw (1, 1.2) arc (0:-180:1);
\draw [fill=white](-0.4, 0.5) rectangle (0.4, 0.8); \node at (0, 0.65) {\tiny{$\phi$}};
\draw [fill=white](-0.4, -1.2) rectangle (0.4, -0.9); \node at (0, -1.05) {\tiny{$\phi'$}};
\end{tikzpicture}}}\\
\vspace{2mm}\\
(M, \Sigma, \Gamma)&\rightarrow & (M, \partial(R_{B}\setminus D_{\epsilon}), \Gamma')\\
\end{array}
$$
\item \textbf{Move 3:} Let $T$ be a $B$-colored $3$-ball with $\partial T\subset \Sigma$ and $\partial T\cap \Gamma=\emptyset$. We change the color of $T$ from $B$ to $A$.
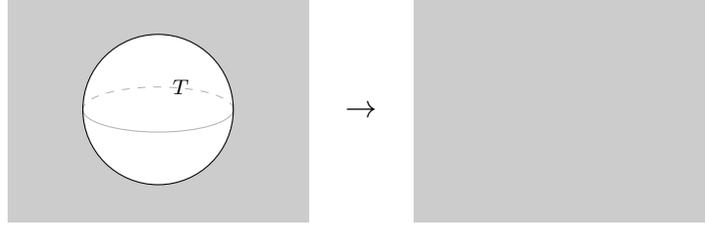
$$
\begin{array}{ccc}
\vcenter{\hbox{
\begin{tikzpicture}
\path[fill=black, opacity=0.2]
(-2, -1.5) rectangle (2, 1.5);
\draw [fill=white] (0, 0) [partial ellipse=0:360:1];
\draw[opacity=0.3] (0, 0)[partial ellipse=0:-180:1 and 0.3];
\draw[dashed, opacity=0.3] (0, 0)[partial ellipse=0:180:1 and 0.3];
\draw (0.3, 0.3) node{\tiny{$T$}};
\end{tikzpicture}
}}
&\rightarrow&
\vcenter{\hbox{
\begin{tikzpicture}
\path[fill=black, opacity=0.2]
(-2, -1.5) rectangle (2, 1.5);
\end{tikzpicture}
}}\\
\vspace{2mm}\\
(M, \Sigma, \Gamma)&\rightarrow& (M, \Sigma\setminus \partial T, \Gamma)\\
\end{array}
$$

\end{itemize}

\begin{remark}\label{rem:changecolor}
Applying Move $i$ changes the color of a $3$-dimensional $i$-handle from $B$-color to $A$-color. Given a $\mathcal{C}$-decorated 3-alterfold $(M, \Sigma, \Gamma)$, one can always make it consist of no $B$-colored region by applying a sequence of elementary moves as follows. 

We first choose a handle decomposition of $R_{B}$ (see \cite{Kir89}), and then apply Moves 0, 1, 2 to the 0-, 1- and 2-handles respectively. In this way, we get a $\mathcal{C}$-decorated 3-alterfold with $B$-colored region homeomorphic to a disjoint union of $3$-balls. Then we apply Planar Graphical Calculus to evaluate diagrams on the separating spheres to scalars. Finally, we apply Move 3 to fill the $B$-colored $3$-balls.
\end{remark}

\subsection{Partition Function over 3-Alterfold} The usual planar graphic calculus naturally assigns a scalar to $\mathcal{C}$-diagrams over a sphere, which is known as the partition function of the spherical fusion category $\mathcal{C}$. It has been widely studied in the framework of planar algebras \cite{Jon21}. It is further extended as surface algebras by the first author \cite{Liu19} and by Costantino-Geer-Patureau-Turaev \cite{CGPT20} to pairs $(H, \Gamma)$, where $H$ is a multi-handlebody and $\Gamma$ is a $\mathcal{C}$-diagram over the boundary of $H$.

The partition function defined in \cite{Liu19} for a $3$-ball with the empty diagram on the boundary is a nonzero free parameter $\zeta \ne 0 \in \k$. The partition function for handlebody of genus $g$ is $\zeta^{1-g}$. Note that the Euler characteristic of a handlebody of genus $g$ is $1-g$, this means the powers of $\zeta$ in the partition function in \cite{Liu19} corresponds to the contribution (via the tensor product of TQFTs) of the Euler theory in the sense of \cite{FreedNote}. In \cite{CGPT20}, the possible algebraic inputs include pivotal fusion categories and certain nonsemisimple pivotal categories. 

In the case when the input category $\mathcal{C}$ is spherical fusion, the extended partition functions in both \cite{Liu19} and \cite{CGPT20} fit into our framework as follows. We choose an embedding of $(H, \Gamma)$ into a connected $\mathcal{C}$-decorated 3-alterfold $(M, \Sigma, \Gamma)$ such that $\partial H$ is mapped to $\Sigma$ and $H$ is mapped to the $B$-colored region. Therefore, the partition function is well define on connected $\mathcal{C}$-decorated 3-alterfolds with $B$-colored regions homeomorphic to a multi-handlebody, and independent of the topology of the $A$-colored regions. More precisely, we reinterpret their theorem in terms of $\mathcal{C}$-decorated 3-alterfolds as follows.

\begin{theorem}[\cite{Liu19,CGPT20}]\label{thm:B-handlebody}
Let $\mathcal{C}$ be a spherical fusion category and $\zeta\ne 0$. 
Then there exists a unique partition function $Z$ maps connected(M is connected) $\mathcal{C}$-decorated 3-alterfold $(M, \Sigma, \Gamma)$ with $R_{B}$ homeomorphic to a multi-handlebody to ground field $\k$, subject to the following:
\begin{itemize}
    \item If $R_{B}=\emptyset$, $Z(M,\emptyset,\emptyset)=1$.
    \item \textbf{Homeomorphims:} The evaluation $Z(M, \Sigma, \Gamma)$ only depends on the orientation perserving homeomorphism class of $(R_{B}, \Gamma)$.
    \item \textbf{Planar Graphical Calculus:} If $\Gamma'$ and $\Gamma$ are diagrams on $\Sigma$ that are identical outside of a contractible region $D$, and $\Gamma_{D}:=\Gamma\cap D$ equals to $\Gamma_{D}':=\Gamma'\cap D$ as morphisms in $\mathcal{C}$. Then
    $$Z(M, \Sigma, \Gamma)=Z(M, \Sigma, \Gamma').$$
    \item \textbf{Move 2:} Suppose $(M, \Sigma', \Gamma')$ is derived from $(M, \Sigma, \Gamma)$ by applying Move $2$, then
    $$Z(M, \Sigma', \Gamma')=\zeta Z(M, \Sigma, \Gamma).$$
    \item \textbf{Move 3:} Suppose $(M, \Sigma', \Gamma')$ is derived from $(M, \Sigma, \Gamma)$ by applying Move $3$, then
    $$Z(M, \Sigma', \Gamma')=\frac{1}{\zeta} Z(M, \Sigma, \Gamma).$$
\end{itemize}
\end{theorem}

The proof of the theorem is straight forward, we sketch the proof here for completeness. For a more detailed proof, we refer readers to \cite{Liu19} and \cite{CGPT20}.

\begin{proof}[Sketch of the proof:]Given a 3-alterfold with $R_{B}$ homeomorphic to a multi-handlebody, one can apply a sequence of Move $2$ on a set of complete cutting disks to make the resulting $B$-colored region homeomorphic to a disjoint union of $3$-balls. Then one can apply the planar graphical calculus to evaluate the closed diagrams on the boundaries of the 3-balls above, and apply Move $3$ to make $R_{B}=\emptyset$ afterwards. 
In this way, one gets a scalar in $\k$. Now we argue that the scalar is well-defined, i.e., it is independent of the choice of the set of complete cutting disks. It is well known \cite{Kup96} that two such sets of disks can be related by a sequence of the following moves: (a) adding a contractible disk and (b) handle slide one disk on top of another. Thus one only need to show the evaluation is invariant on both moves, which can be shown by induction on the genus $g$ of the multi-handlebody. To determine the power of $\zeta$, one solves the equation coming from the cancellation of $2$-$3$-handles. The uniqueness of the partition function follows from the definition.
\end{proof}

It is natural to ask how Move $0$ and Move $1$ affect the evaluation. In \cite{CGPT20}, they gave a partial answer to this question.
\begin{proposition}\cite[Proposition~5.3]{CGPT20}\label{pro:stringgenus}
Let $(M, \Sigma, \Gamma)$ be a connected $\mathcal{C}$-decorated 3-alterfold with $R_{B}(M, \Sigma)$ homeomorphic to a multi-handlebody. Suppose $S$ is an embedded arc in $R_{B}(M, \Sigma)$ such that along $S$, Move 1 is applicable and results in another $\mathcal{C}$-decorated 3-alterfold $(M, \Sigma', \Gamma')$ with $R_{B}(M, \Sigma')$ a multi-handlebody.Then
\begin{equation}\label{eq:moveone}
Z(M, \Sigma', \Gamma')=\frac{1}{\zeta} Z(M, \Sigma, \Gamma).
\end{equation}
\end{proposition}

This can be seen as follows. Given the arc $S$ satisfying the condition in Proposition \ref{pro:stringgenus}, there exists a cutting disk $D$ in $R_{B}(M, \Sigma')$ that intersects the $\Omega$-colored circle on the belt of $S_{\epsilon}$ only once. One gets Equation \eqref{eq:moveone} immediately by the $1$-$2$-handle cancellation. In \cite{Liu19} and \cite{JLW17}, this property is called the string-genus relation in the sense that if one sees an $\Omega$-colored string bound a hole, then they can be removed simultaneously.

\begin{figure}
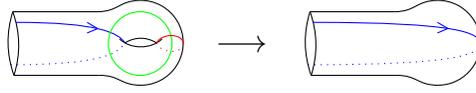

    \centering
    \StabAndAfter
    \caption{$1$-$2$-handle cancellation}
    \label{fig:my_label}
\end{figure}

Recall from Remark \ref{rem:changecolor} that via a sequence of Moves 0 and 1, one can always change the $B$-colored region of an arbitrary 3-alterfold into a multi-handlebody. Thus, it is natural to ask whether we can extend the partition function to arbitrary $\mathcal{C}$-decorated 3-alterfolds that is invariant under Moves 0 and 1. The following theorem gives an affirmative answer to this question.

\begin{theorem}\label{thm:partition function}
Let $\mathcal{C}$ be a spherical fusion category over $\k$, and $\zeta \in \k^\times$ be a nonzero scalar. There exists a unique partition function $Z$ from $\mathcal{C}$-decorated 3-alterfolds to the ground field $\k$ satisfying the following conditions.
\begin{itemize}
    \item \textbf{Normalization:} If $R_{B}=\emptyset$, $Z(M,\emptyset,\emptyset)=1$.
    \item \textbf{Disjoint Union:} Suppose $(M, \Sigma, \Gamma)$ and $(M', \Sigma', \Gamma')$ are two $\mathcal{C}$-decorated 3-alterfolds. Then
    $$Z(M\sqcup M', \Sigma \sqcup \Sigma', \Gamma\sqcup\Gamma')=Z(M, \Sigma, \Gamma)Z(M', \Sigma', \Gamma').$$
     \item \textbf{Homeomorphims:} The evaluation $Z(M, \Sigma, \Gamma)$ only depends on the orientation perserving homeomorphism class of $(R_{B}, \Gamma)$.
    \item \textbf{Planar Graphical Calculus:} If $\Gamma'$ and $\Gamma$ are diagrams on $\Sigma$ that are identical outside of a contractible region $D$, and $\Gamma_{D}:=\Gamma\cap D$ equals to $\Gamma_{D}':=\Gamma'\cap D$ as morphisms in $\mathcal{C}$. Then
    $$Z(M, \Sigma, \Gamma)=Z(M, \Sigma, \Gamma').$$
    \item \textbf{Move 0:} Suppose $(M, \Sigma', \Gamma')$ is derived from $(M, \Sigma, \Gamma)$ by applying Move $0$, then
    $$Z(M, \Sigma', \Gamma')=\zeta\mu Z(M, \Sigma, \Gamma).$$
    \item \textbf{Move 1:} Suppose $(M, \Sigma', \Gamma')$ is derived from $(M, \Sigma, \Gamma)$ by applying Move $1$, then
    $$Z(M, \Sigma', \Gamma')=\frac{1}{\zeta} Z(M, \Sigma, \Gamma).$$
    \item \textbf{Move 2:} Suppose $(M, \Sigma', \Gamma')$ is derived from $(M, \Sigma, \Gamma)$ by applying Move $2$, then
    $$Z(M, \Sigma', \Gamma')=\zeta Z(M, \Sigma, \Gamma).$$
    \item \textbf{Move 3:} Suppose $(M, \Sigma', \Gamma')$ is derived from $(M, \Sigma, \Gamma)$ by applying Move $3$, then
    $$Z(M, \Sigma', \Gamma')=\frac{1}{\zeta} Z(M, \Sigma, \Gamma).$$
\end{itemize}
\end{theorem}

In the spherical fusion category case, we remark that the quon language and the theory by Costantino-Geer-Patureau-Turaev(CGPT) can be embedded into our alterfold theory.
Later, we see that link surgery produces a surgery move for 3-alterfold (see Theorem \ref{thm:linksurgery} for surgery move and Theorem \ref{thm:partition function2} in Section 3.3 for an alternative description of Theorem \ref{thm:partition function}). 
We summarize the correspondence in the following table.

\begin{align}\label{tab:topmove}
\begin{tabular}{c|c|c|c|c}
    Move & Region $A$ & Region $B$ & Quon Terminology & CGPT\\
    \hline
    Move 2 & Add $2$-handle & Cancel $1$-handle & \small Unit Decomposition & Cutting\\
    \hline
    Move 1 & Add $1$-handle & Cancel $2$-handle & String-Genus & \small Red Capping/Digging\\
    \hline 
    \small Surgery Move & None & Link Surgery & None & None
\end{tabular}
\end{align}

We outline the proof of our main theorem of this section as follows. Given a $\mathcal{C}$-decorated 3-alterfold $(M, \Sigma, \Gamma)$, by finitely many steps of Move $0$ and Move $1$, one can make the $B$-colored region a multi-handlebody, which reduces to the case of Theorem \ref{thm:B-handlebody}. Then we show the evaluation of the partition function does not depend on the way that $A$-colored $0/1$-handles been added. We implement this idea using the language of triangulation.

\begin{definition}
Let $(M, \Sigma, \Gamma)$ be a $\mathcal{C}$-decorated 3-alterfold, and let $\Delta$ be a triangulation of the closure of $R_{B}$. We say $\Delta$ is a triangulation of 
$(M, \Sigma, \Gamma)$ if its restriction on the surface, $\Sigma \cap \Delta$, is in generic position with $\Gamma \subset \Sigma$. More precisely, we require that
\begin{itemize}
\item The vertices of $\Delta$ in $\partial R_{B}$ do not intersect $\Gamma$,
\item The edges of $\Delta$ do not meet the coupons of $\Gamma$ and meet the strings of $\Gamma$ transversely.
\end{itemize}  
A simplex in $\Delta$ is said to be a boundary simplex if it lies in $\Sigma$, otherwise, we call it inner simplex.
\end{definition}
\begin{remark}
Given a triangulation $\Delta$ of $(M, \Sigma, \Gamma)$, the coupons of $\Gamma$ lies in the interior of faces of $\Delta\cap\Sigma$.
\end{remark}

Let $\Delta$ be a triangulation of a $\mathcal{C}$-decorated 3-alterfold $(M, \Sigma, \Gamma)$. We say $\Delta'\subset\Delta$ is a \textit{partial triangulation} of $\Delta$ if $\Delta'$ is closed under the boundary maps of $\Delta$. 
For instance, if $\Delta^{\le 2}\subset \Delta$ is the subset consisting of all vertices, edges and faces of $\Delta$, then $\Delta^{\le 2}$ is a partial triangulation of $\Delta$.

\begin{definition}
Let $\Delta$ be a partial triangulation of some triangulation of $(M, \Sigma, \Gamma)$, we define a new 3-alterfold $(M, \Sigma^{\Delta}, \Gamma^{\Delta})$ as follows:
\begin{itemize}
\item 
Firstly, applying Move $0$ to $(M, \Sigma, \Gamma)$ with respect to each vertex in $\Delta$ lying in the interior of $R_{B}$, we get a $\mathcal{C}$-decorated 3-alterfold $(M, \Sigma^{(1)}, \Gamma)$.
\item 
Secondly, applying Move $1$ to $(M, \Sigma^{(1)}, \Gamma)$ with respect to each edges in $\Delta$ not contained in $\Sigma^{(1)}$, we get a $\mathcal{C}$-decorated 3-alterfold $(M, \Sigma^{(2)}, \Gamma^{(2)})$.
\item 
Finally, we apply Move $2$ to $(M, \Sigma^{(2)}, \Gamma^{(2)})$ with respect to each faces in $\Delta$ not contained in $\Sigma^{(2)}$. The result is the desired $\mathcal{C}$-decorated 3-alterfold $(M, \Sigma^{\Delta}, \Gamma^{\Delta})$.
\end{itemize}
\end{definition}

\begin{remark}
Let $\Delta$ be a triangulation of $(M, \Sigma, \Gamma)$. Then $(M, \Sigma^{\Delta}, \Gamma^{\Delta})=(M, \Sigma^{\Delta^{\le 2}}, \Gamma^{\Delta^{\le 2}})$, since the construction of $(M, \Sigma^{\Delta}, \Gamma^{\Delta})$ does not depend on the tetrahedrons in $\Delta$.
\end{remark}

Let $\Delta$ be a triangulation, then $R_{B}(M, \Sigma^{\Delta})$ is homeomorphic to a disjoint union of $3$-balls embedded in components $M_{1}, M_{2},\ldots M_{n}$ of $M$, which means we can evaluate the partition function to each of the connected components $(M_i, \Sigma_{i}^{\Delta}, \Gamma_{i}^{\Delta})$ using the rules defined in Theorem \ref{thm:B-handlebody}. We define
\begin{equation}\label{eq:partiontrig}
Z(M, \Sigma, \Gamma; \Delta)=\mu^{-|V|}\zeta^{-|F|+|E|-|V|}\prod_{1\le i \le n}Z(M_{i}, \Sigma^{\Delta}_{i}, \Gamma^{\Delta}_{i})
\end{equation}
where $|V|, |E|$ and $ |F|$ are the number of inner vertices, inner edges, and inner faces of $\Delta$. Later we will see the scalar $\mu^{-|V|}\zeta^{-|F|+|E|-|V|}$ is necessary to make our definition independent of the choice of the triangulation.

By the definition, one obtain
$$Z(M, \Sigma, \Gamma; \Delta)=\prod_{1\le i\le n} Z(M_i, \Sigma_{i}, \Gamma_{i};\Delta|_{M_{i}})$$
where $\Delta|_{M_i}$ is the triangulation $\Delta$ restricted to the $i$-th component, which is a triangulation of $(M_{i}, \Sigma_i, \Gamma_i)$. 

To be more precise, we introduce some notations and give an explicit formula for $Z(M, \Sigma, \Gamma; \Delta)$. Without loss of generality, we may assume $M$ is connected. At each inner faces $f\in F$, the Move $2$ produces a sum of dual basis, we denote this vector by $\star_{f}\in H_{f}\otimes H_{f}^{*}$. Notice each simplex is oriented and naturally one has $H_{f}^{*}=H_{\bar{f}}$, where $\bar{f}$ is $f$ with the opposite orientation. On the other hand, the connected components $B_{t}$ (which are $3$-balls) of $R_{B}(M, \Sigma^{\Delta})$ are in one-to-one correspondence with the tetrahedra $t \in T$. Let $I(t)$ be the set of inner faces of $\partial t$. For each $t$, the sequences of Move 2 described above results in a graph on $\partial B_t$ with a single uncolored coupon on each inner face $f \in I(t)$, and the edges of this graph are segments of the loops corresponding to Move 1 (colored by $\sum_{X \in \Irr(\mathcal{C})} X$) or segments of the original graph $\Gamma$ on $\Sigma$. We denote this partially colored graph in $\partial B_t$ by $\Gamma_t$, and identify it with a vector in $\bigotimes_{f \in I(t)} H_f^*$ (in the sense that coupons of $\Gamma_t$ can be filled with corresponding morphisms in $H_f$ to result in a number). With the notations above, $Z(M, \Sigma^{\Delta}, \Gamma^{\Delta})$ can be computed by "filling the coupons", i.e.
\begin{equation}\label{eq:contraction}
Z(M, \Sigma^{\Delta}, \Gamma^{\Delta})=\zeta^{|T|} \cdot \left\langle\bigotimes_{t\in T}\Gamma_{t}\,, \ \bigotimes_{f\in F}\star_{f}\right\rangle\,,  
\end{equation}
where $\langle\cdot, \cdot \rangle$ is the evaluation (or contraction) of tensors.

\begin{proposition}\label{prop:turaev-viro}
Let $\Delta$ be a triangulation of $(M, \emptyset, \emptyset)$ and $R_{A}(M, \emptyset)=\emptyset$. Then
$$Z(M, \emptyset, \emptyset; \Delta)= TV_{\mathcal{C}}(M)$$
where $TV_{\mathcal{C}}(M)$ is the Turaev-Viro invariant of $M$ associated to $\mathcal{C}$. In particular, the evaluation does not depend on the triangulation in this case.
\end{proposition}
\begin{proof}
By equation \eqref{eq:partiontrig}, $Z(M, \emptyset, \emptyset; \Delta)=\mu^{-|V|}\zeta^{-|F|+|E|-|V|}Z(M, \emptyset^{\Delta}, \emptyset^{\Delta})$. Decomposing the $\Omega$-color created by Move $1$ as $\sum_{X\in \Irr(\mathcal{C})}d_{X}\text{id}_{X}$. The equation \eqref{eq:contraction} can be written as a summation over the set of coloring of inner edges $c: E\rightarrow \Irr(\mathcal{C})$.
\begin{align*}
Z(M, \emptyset^{\Delta}, \emptyset^{\Delta})=\sum_{c}\dim(c)\left\langle\bigotimes_{t\in T}\Gamma_{t}^{c}\,, \ \bigotimes_{f\in F}\star_{f}^{c}\right\rangle
\end{align*}
where $\dim(c)=\prod_{e}d(c(e))$, and $\Gamma^{c}_{t}$ and $\star^{c}_{f}$ be the corresponding summands. Suppose $\partial f=e_{1}^{f}+e_{2}^{f}+e_{3}^{f}$, and choose a basis on $H_{f}=\hom(\1, c(e_1)\otimes c(e_2)\otimes c(e_3))$. Then
$$\left\langle\bigotimes_{t\in T}\Gamma_{t}^{c}\,, \ \bigotimes_{f\in F}\star_{f}^{c}\right\rangle$$
equals to the sum of the product of $6j$-symbols (see \cite[page 348]{Tur94}). Comparing with the definition of Turaev-Viro invariant via $6j$-symbols in \cite[page 348]{Tur94}. We have
\begin{align*}
Z(M, \emptyset, \emptyset; \Delta)=&\mu^{-|V|}\zeta^{-|F|+|E|-|V|}Z(M, \emptyset^{\Delta}, \emptyset^{\Delta})\\
=&\mu^{-|V|}\zeta^{-|F|+|E|-|V|}\mu^{|V|}\zeta^{|T|}TV_{\mathcal{C}}(M)\\
=&\zeta^{-\chi(M)}TV_{\mathcal{C}}(M)\\
=&TV_{\mathcal{C}}(M)
\end{align*}

\noindent as any odd-dimensional compact manifold without boundary has 0 Euler characteristic due to Poincar\'e duality. Finally, the invariance of $Z(M, \emptyset, \emptyset; \Delta)$, i.e., the independence on the choice of triangulations, follows from the invariance of the Turaev-Viro invariant.
\end{proof}
\begin{theorem}\label{thm:trigindep}
The evaluation $Z(M, \Sigma, \Gamma; \Delta)$ does not depend on the triangulation $\Delta$.
\end{theorem}
To prove the theorem, we recall a results by Newman \cite{New26}, Pachner \cite{Pac91} and Lickorish \cite{Lic99} on relating two triangulations of a $3$-manifold with non-empty boundary using local moves called shellings and inverse shellings. We briefly recall the definition of these local moves following Lickorish \cite{Lic99}.

\begin{definition}
Suppose $\sigma_{0}$ and $\sigma_{1}$ are simplices of a combinatorial $n$-manifold $M$ with boundary $\partial M$, such that the join $T=\sigma_{0}\star \sigma_{1}$ is an $n$-simplex of $M$ satisfying $\sigma_{0}\cup \partial M=\partial \sigma_{0}$ and $\sigma_{1}\star\partial M\subset \partial M$. The manifold $M'$ obtained from $M$ by elementary shelling from $\sigma_{1}$ is the closure of $M-T$. $M$ is also said to be obtained from $M'$ by an elementary inverse shelling.
\end{definition}

\begin{theorem}\cite[Theorem 5.10]{Lic99}\label{thm:lickorish}
Two connected combinatorial $n$-manifolds with non-empty boundary are piecewise linearly homeomorphic if and only if they are related by a sequence of elementary shellings, inverse shellings and a simplicial isomorphism.
\end{theorem}

\begin{theorem}\cite[page 352, Theorem 2.1.2]{Tur94}\label{thm:relativepachner}
Two triangulations of a compact $3$-manifold $M$ coinciding in $\partial M$ are related by a finite sequence of bistellar moves and an ambient isotopy in $M$ identical on $\partial M$.

\end{theorem}

When the dimension $n=3$, the shelling(resp. inverse shelling) with respect to a tetrahedron $T$ can be classified to the following $3$-types depending on the number of faces of $T$ are being attached to the resulting(original) manifold $M$.
\begin{itemize}
\item \textbf{Type 1:}
$$
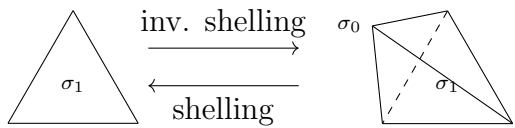
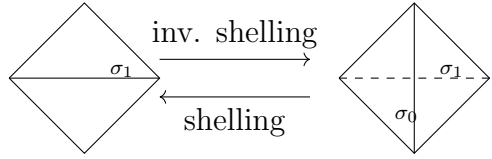
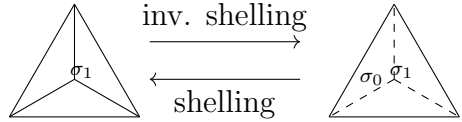
\vcenter{\hbox{
\begin{tikzpicture}
\coordinate (A) at (0, 1);
\coordinate (B) at (-0.866, -.5);
\coordinate (C) at (0.866, -.5);
\draw (0, 0) node{\tiny{$\sigma_{1}$}};
\draw (A)--(B);
\draw (C)--(A);
\draw (B)--(C);
\draw[->] (1,0.5)--(3,0.5); \node[above] at (2,0.5) {inv. shelling};
\draw[<-] (1,0)--(3,0); \node[below] at (2,0) {shelling};
\end{tikzpicture}}}
\vcenter{\hbox{
\begin{tikzpicture}
\coordinate (A) at (0, 1);
\coordinate (B) at (-0.866, -.5);
\coordinate (C) at (0.866, -.5);
\draw (0, 0) node{\tiny{$\sigma_{1}$}};
\draw[dashed] (A)--(B);
\draw (C)--(A);
\draw (B)--(C);
\draw (-1, 0.8) node[left]{\tiny{$\sigma_0$}}--(A);
\draw (-1, 0.8)--(B);
\draw (-1, 0.8)--(C);
\end{tikzpicture}}}
$$
\item \textbf{Type 2:}
$$
\vcenter{\hbox{
\begin{tikzpicture}
\draw (-1, 0)--(0, 1)--(1, 0)--(0, -1)--(-1, 0);
\draw (0.5, 0.1) node{\tiny{$\sigma_{1}$}};
\draw (-1, 0)--(1, 0);
\draw[->] (1,0.25)--(3,0.25); \node[above] at (2,0.25) {inv. shelling};
\draw[<-] (1,-0.25)--(3,-0.25); \node[below] at (2,-0.25) {shelling};
\end{tikzpicture}}}
\vcenter{\hbox{
\begin{tikzpicture}
\draw (-1, 0)--(0, 1)--(1, 0)--(0, -1)--(-1, 0);
\draw (0.5, 0.1) node{\tiny{$\sigma_{1}$}};
\draw[dashed] (-1, 0)--(1, 0);
\draw(0, -1)--(0, 1);
\draw (-0.1, -0.5) node{\tiny{$\sigma_0$}};
\end{tikzpicture}}}
$$
\item \textbf{Type 3:}
$$
\vcenter{\hbox{
\begin{tikzpicture}
\coordinate (A) at (0, 1);
\coordinate (B) at (-0.866, -.5);
\coordinate (C) at (0.866, -.5);
\draw (0.1, 0.1) node{\tiny{$\sigma_{1}$}};
\draw (A)--(B);
\draw (C)--(A);
\draw (B)--(C);
\draw (0, 0)--(A);
\draw (0, 0)--(B);
\draw (0, 0)--(C);
\draw[->] (1,0.5)--(3,0.5); \node[above] at (2,0.5) {inv. shelling};
\draw[<-] (1,0)--(3,0); \node[below] at (2,0) {shelling};
\end{tikzpicture}}}
\vcenter{\hbox{
\begin{tikzpicture}
\coordinate (A) at (0, 1);
\coordinate (B) at (-0.866, -.5);
\coordinate (C) at (0.866, -.5);
\draw (0.1, 0.1) node{\tiny{$\sigma_{1}$}};
\draw (A)--(B);
\draw (C)--(A);
\draw (B)--(C);
\draw[dashed] (0, 0)--(A);
\draw[dashed] (0, 0)--(B);
\draw[dashed] (0, 0)--(C);
\draw (-0.3, 0) node{\tiny{$\sigma_{0}$}};
\end{tikzpicture}}}
$$
\end{itemize}
In light of Theorem \ref{thm:lickorish}, we only need to show $Z(M, \Sigma, \Gamma;\Delta)$ is invariant under the shelling and inverse shelling moves.

\begin{proof}[Proof of Theorem \ref{thm:trigindep}]
Let $\Delta$ and $\Delta'$ be two triangulations of $(M, \Sigma, \Gamma)$, i.e., triangulations of the closure of $R_{B}(M, \Sigma)$ in generic position with $\Gamma$. By Theorem \ref{thm:lickorish}, $\Delta$ and $\Delta'$ can be related by a sequence of elementary shelling and inverse elementary shelling moves. In order to prove the theorem, we only need to show $Z(M, \Sigma, \Gamma;\Delta)=Z(M, \Sigma, \Gamma; \Delta')$ when $\Delta'$ is obtained from $\Delta$ by one elementary shelling from $\sigma_1$ and a simplicial isomorphism $\tau$ that maps the closure of $R_{B}(M, \Sigma)-t_{0}$ to $R_{B}(M, \Sigma)$. We denote the set of inner simplices of $\Delta$ (resp.~$\Delta'$) by $V, E, F, T$ (resp.~$V', E', F', T'$).

\textbf{Type 1.} Suppose $\sigma_1$ is a $2$-simplex. Index the simplices such that $V=V'$, $E'=E$, $F'=F-\{\sigma_1\}$. Let $\{t_0, t_1\}$ (resp.~$\{t_1^{'}\}$) be the set of tetrahedra adjacent to $\sigma_{1}$, then $T-\{t_0, t_1\}=T'-\{t_1^{'}\}$. In addition, $t_1^{'}=t_{0}\cup t_{1}$ in $R_{B}(M, \Sigma)$. By definition, 
\begin{align*}
Z(M, \Sigma, \Gamma; \Delta)=&\zeta^{-|V|+|E|-|F|+|T|}\left\langle\bigotimes_{t\in T}\Gamma_{t}\,, \ \bigotimes_{f\in F}\star_{f}\right\rangle \\
=&\zeta^{-|V|+|E|-|F|+|T|}\left\langle \Gamma_{t_{0}}\otimes \Gamma_{t_1}\otimes\bigotimes_{t\in T'}\Gamma_{t}\,, \ (\bigotimes_{f\in F'}\star_{f})\otimes \star_{\sigma_1}\right\rangle \\
=&\zeta^{-|V|+|E|-|F|+|T|}\left\langle \Gamma_{t_{0}}\otimes \Gamma_{t_1}\,, \left\langle \bigotimes_{t\in T'}\Gamma_{t}\,, \ \bigotimes_{f\in F'}\star_{f}\right\rangle \otimes \star_{\sigma_1}\right\rangle \\
\end{align*}

Notice $t_{0}\cup t_{1}$ is homeomorphic to a $3$-ball in $R_{B}$, by Theorem \ref{thm:B-handlebody}, one has
$$\left\langle \Gamma_{t_{0}}\otimes \Gamma_{t_1}, -\otimes \star_{\sigma_1}\right\rangle=\langle \Gamma_{t_{1}}^{'}, -\rangle.$$
In fact, when evaluating the left hand side, we need to multiply by $\zeta^2$ for the two $B$-colored 3-balls, while on the right hand side, we multiply by two copies of $\zeta$: one from the one $B$-colored 3-ball and one from Move 2. As a consequence, they cancel with each other.

Therefore,
$$Z(M, \Sigma, \Gamma; \Delta)=\zeta^{-|V'|+|E'|-|F'|+|T'|}\left\langle \Gamma_{t_1^{'}}\,, \left\langle \bigotimes_{t\in T'}\Gamma_{t}\,, \ \bigotimes_{f\in F'}\star_{f}\right\rangle \right\rangle=Z(M, \Sigma, \Gamma; \Delta').$$

\textbf{Type 2.} Suppose $\sigma_1$ is an $1$-simplex. Index the simplices such that $V'=V$, $E'=E-\{\sigma_1\}$. For faces and tetrahedra, let $\{f_0, f_1, \ldots, f_n\}$ (resp.~$\{f_2',\ldots, f_{n}'\}$) be the set of faces adjacent to $\sigma_1$ in $\Delta$ (resp.~$\Delta'$), and $\{t_0, t_1, \ldots, t_n\}$ (resp.~$\{t_1', \ldots, t_n'\}$) the set of tetrahedra adjacent to $\Sigma_1$ in $\Delta$ (resp.~$\Delta'$). Then 
$$
F^{\perp} := F'-\{f_2',\ldots, f_{n}'\}=F-\{f_0, f_1, \ldots, f_n\}
$$
and
$$ 
T^{\perp} := T'-\{t_1', \ldots, t_n'\}=T-\{t_0, t_1, \ldots, t_n\}\,.
$$ 
In addition, $\displaystyle \bigcup_{1\le i\le n} t_i'=\bigcup_{0\le i\le n}t_i$ in $R_{B}(M, \Sigma)$.  By definition, we have

\begin{align*}
&Z(M, \Sigma, \Gamma; \Delta)\\
=&\zeta^{-|V|+|E|-|F|+|T|}\left\langle\bigotimes_{t\in T}\Gamma_{t}\,, \ \bigotimes_{f\in F}\star_{f}\right\rangle \\
=&\zeta^{-|V|+|E|-|F|+|T|}\left\langle \left(\bigotimes_{0 \le i\le n}\Gamma_{t_i} \right) \otimes \bigotimes_{t\in T^\perp}\Gamma_{t}\,,  \left(\bigotimes_{f\in F^\perp}\star_{f}\right) \otimes \bigotimes_{0\le i\le n}\star_{f_i}\right\rangle\\
=&\zeta^{-|V|+|E|-|F|+|T|}\left\langle \bigotimes_{0 \le i\le n}\Gamma_{t_i}\,, \left\langle \bigotimes_{t\in T^\perp} \Gamma_{t}\,, \ \bigotimes_{f\in F^\perp} \star_{f}\right\rangle \otimes \bigotimes_{0\le i\le n}\star_{f_i}\right\rangle\,. \\
\end{align*}
Notice $\displaystyle \bigcup_{0\le i\le n}t_i$ is homeomorphic to a $3$-ball and $\sigma_1$ is an untwisted arc embedded in it. By Proposition \ref{pro:stringgenus}, one has
$$\left\langle \bigotimes_{0 \le i\le n}\Gamma_{t_i}\,, - \otimes \bigotimes_{0\le i\le n}\star_{f_i}\right\rangle=\left\langle \bigotimes_{1 \le i\le n}\Gamma_{t_i'}\,, - \otimes \bigotimes_{2\le i\le n}\star_{f_i'}\right\rangle$$
As before, when evaluating the left hand side, we need to multiply by $\zeta$ for the $(n+1)$ $B$-colored 3-balls, while on the right hand side, we multiply $\zeta$ as well, the $\zeta$ appear in applying Move $1$ is cancelled by applying one Move $2$ to make the $B$-colored region homeomorphic to a $3$-ball.

Therefore,
\begin{align*}
Z(M, \Sigma, \Gamma; \Delta)=&\zeta^{-|V|+|E|-|F|+|T|}\left\langle \bigotimes_{1 \le i\le n}\Gamma_{t_i'}\,, \left\langle \bigotimes_{t\in T^\perp}\Gamma_{t}\,, \ \bigotimes_{f\in F^\perp} \star_{f}\right\rangle \otimes \bigotimes_{2\le i\le n}\star_{f_i'}\right\rangle\\
=&Z(M, \Sigma, \Gamma; \Delta').   
\end{align*}

\textbf{Type 3.} Let $\Delta''$ be a triangulation of $M$ by applying Type $1$ inverse shelling with respect to boundary face $\sigma_{0}$. Then the triangulation of $\Delta''$ and $\Delta'$ are identical on the boundary. By Theorem \ref{thm:relativepachner}, $\Delta''$ and $\Delta'$ are related by a sequence of bistella moves. Thus
$$Z(M, \Sigma, \Gamma;\Delta')=Z(M, \Sigma, \Gamma;\Delta'')$$
by the identical argument as in \cite{Tur94}. On the other hand, $\Delta''$ and $\Delta$ are related by a shelling move of Type $1$. Therefore,
$$Z(M, \Sigma, \Gamma;\Delta')=Z(M, \Sigma, \Gamma;\Delta'')=Z(M, \Sigma, \Gamma;\Delta).$$

Theorem \ref{thm:trigindep} shows that we established a unique well-defined partition function 
$$Z(M, \Sigma, \Gamma):=Z(M, \Sigma, \Gamma; \Delta)$$
on $\mathcal{C}$-decorated 3-alterfolds. To prove Theorem \ref{thm:partition function}, we only need to show this partition function has the claimed properties.
\end{proof}

\begin{proof}[Proof of Theorem \ref{thm:partition function}]
We show the partition function satisfies all properties listed in Theorems \ref{thm:B-handlebody} and $\ref{thm:partition function}$.

\textbf{Disjoint union:} Suppose $\Delta$ and $\Delta'$ are triangulations of $(M, \Sigma, \Gamma)$ and $(M', \Sigma', \Gamma')$ respectively. Then $\Delta\sqcup\Delta'$ is a triangulation of $(M\sqcup M', \Sigma \sqcup \Sigma', \Gamma\sqcup \Gamma')$. The property follows by evaluating $(M\sqcup M', \Sigma \sqcup \Sigma', \Gamma\sqcup \Gamma')$ with triangulation $\Delta\sqcup\Delta'$.

\textbf{Homeomorphism:} If two 3-alterfold has homeomorphic $B$-colored region and diagrams on the separating surface, one can choose homemorphic triangulations, thus evaluate to the same scalar.

\textbf{Planar Graphical Calculus:} Suppose If $\Gamma'$ and $\Gamma$ are diagrams on $\Sigma$ that are identical outside of a contractible region $D$, and $\Gamma_{D}:=\Gamma\cap D$ equals to $\Gamma_{D}':=\Gamma'\cap D$ as morphisms in $\mathcal{C}$. We can choose a triangulation $\Delta$ of $R_{B}(M, \Sigma)$ such that the region $D$ is in the interior of a boundary face. Therefore, All closed diagrams appeared in $(M, \Sigma^{\Delta}, \Gamma^{\Delta})$ and $(M, \Sigma^{\Delta}, \Gamma'^{\Delta})$ are equal as morphisms in $\mathcal{C}$.

\textbf{Move $i$:} Suppose $(M, \Sigma', \Gamma')$ is derived from $(M, \Sigma, \Gamma)$ by Move $i$. We may choose a triangulation $\Delta$ of $(M, \Sigma, \Gamma)$ such that the $i$-handle described in Move $i$ retracts to one of the $i$-simplex belongs to the triangulation(relative to the boundary). Since $(M, \Sigma, \Gamma)$ appeared in the steps of evaluating $M$. The property follows.

\end{proof}

\begin{remark}\label{rem:homeingore}
In the proof of Theorem \ref{thm:partition function}, we see that given a triangulation on an alterfold, $Z$ can be computed by evaluating diagrams on $B$-colored balls after applying Moves $0$-$2$ to the $B$-colored region. Then Theorem \ref{thm:trigindep} already implies the Homeomorphism property in Theorem \ref{thm:partition function}. In Theorem \ref{thm:partition function2}, we will describe another complete set of Moves that uniquely determines the partition function.
\end{remark}

\begin{theorem}[Embedded Surface Invariants]\label{thm:surface}
Suppose $\Sigma\subset M$ is an embedded closed surface of an oriented closed $3$-manifold $M$,
We have that $Z(M, \Sigma\times\{-\epsilon, \epsilon\}, \emptyset)$ is a topological invariant of the pair $(\Sigma, M)$, where $(M, \Sigma\times\{-\epsilon, \epsilon\}, \emptyset)$ is the alterfold by blowing up $\Sigma$ to an $A$-colored region.
\end{theorem}
\begin{proof}
    It follows from Theorem \ref{thm:partition function}.
\end{proof}

\begin{theorem}[Invariants for 3-manifold with boundary]\label{thm:3mboundary}
Let $M$ be an oriented compact $3$-manifold with boundary, with $\mathcal{C}$-diagram $\Gamma$ on its boundary. 
Let $M'$ be an arbitrary $3$-alterfold with its $B$-colored region homeomorphic to $M$. 
Then the quantity $Z(M', \partial M, \Gamma)$ is a topological invariant of $(M, \Gamma)$.
\end{theorem}
\begin{proof}
It is an immediate consequence of the homeomorphisms property.
\end{proof}

\begin{remark}

In \cite{FarSch22}, Farnsteiner and Schweigert studied $\cC$-decorated 3-manifolds with free boundaries and embedded $\cZ(\cC)$-colored tensor networks, together with a choice of graph skeleton. 
By assuming Hypothesis 16, they argued the existence of certain partition function that is independent of the choice of the skeleton. In particular, one gets invariants of $\cC$-decorated 3-manifolds with free boundary under assumptions. 
In Theorem \ref{thm:3mboundary}, the existence of the partition function proposed in \cite{FarSch22} is established for $\cC$-decorated 3-manifolds with free boundary without assumption. 
Moreover, by Theorem \ref{thm:equivcenter}, if there is a $\cZ(\cC)$-colored tensor network in the 3-manifold with boundary, we can replace it by a $\cC$-decorated $A$-colored multihandlebody and close the free boundary (arbitrarily with $A$-colored multihandlebodies) to obtain a 3-alterfold. Then by Theorem \ref{thm:partition function}, we get the general partition function proposed in \cite{FarSch22} without assumption. 
\end{remark}

\begin{remark}
Kawahigashi, Sato, Wakui \cite{KSW05} also defined a partition function on 3-manifold with free boundary.
The partition function depends on a triangulation of the free boundary, i.e. the separating surface in 3-alterfold.
They also did the computation when the free boundary is torus. 
Theorem \ref{thm:partition function} shows that we are able to do computation for arbitrary separating surfaces.
\end{remark}

\begin{theorem}[Handle Slides]\label{thm:handleslides}
Let $(M, \Sigma, \Gamma)$ be a $\mathcal{C}$-decorated 3-alterfold and $C\subset \Gamma$ be an $\Omega$-colored circle. Suppose $(M, \Sigma, \Gamma')$ is established from handle sliding another curve of $\Gamma$ along $C$. Then
$$Z(M, \Sigma, \Gamma)=Z(M, \Sigma, \Gamma').$$
\end{theorem}

\begin{proof} Graphically, we have
$$
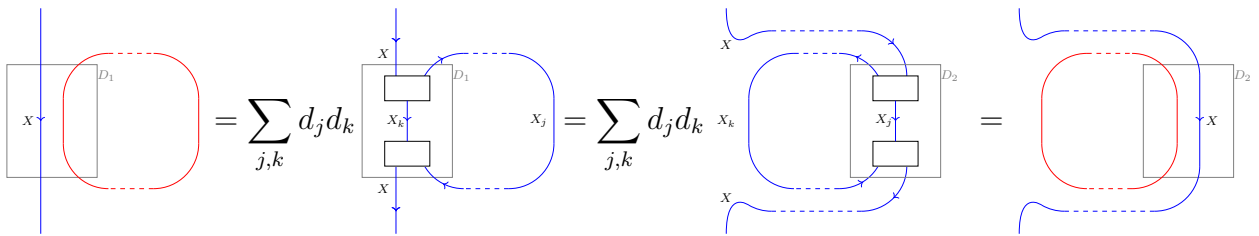
\vcenter{\hbox{\scalebox{0.6}{
\begin{tikzpicture}
\begin{scope}
\draw[gray](-1.75, -1.25) rectangle (0.25, 1.25);
\draw[gray] (0.45, 1) node{\tiny$D_1$};
\end{scope}
\draw[blue, ->-=0.5] (-1, 2.5)--(-1, -2.5);
\draw (-1, 0) node[left]{\tiny{$X$}};
\begin{scope}[color=red, xshift=1cm]
\draw (-1.5, -0.5)--(-1.5, 0.5);
\draw (-1.5, 0.5) arc (180:90:1);
\draw[dashed] (-0.5, 1.5)--(0.5, 1.5);
\draw (0.5, 1.5) arc (90:0:1);
\draw (1.5, 0.5)--(1.5, -0.5);
\draw (1.5, -0.5) arc (0:-90:1);
\draw[dashed] (0.5, -1.5)--(-0.5, -1.5);
\draw (-0.5, -1.5) arc (-90:-180:1);
\end{scope}
\end{tikzpicture}
}}}
=\sum_{j, k}d_{j}d_{k}
\vcenter{\hbox{\scalebox{0.6}{
\begin{tikzpicture}
\begin{scope}
\draw[gray](-1.75, -1.25) rectangle (0.25, 1.25);
\draw[gray] (0.45, 1) node{\tiny$D_1$};
\end{scope}
\draw[blue, ->-=0.15, ->-=0.9] (-1, 2.5)--(-1, -2.5);
\draw (-1, 1.5) node[left]{\tiny{$X$}};
\draw (-1, -1.5) node[left]{\tiny{$X$}};
\begin{scope}[color=blue, xshift=1cm]
\draw[->-=0.7] (-1.5, 0.5) arc (180:90:1);
\draw[dashed] (-0.5, 1.5)--(0.5, 1.5);
\draw (0.5, 1.5) arc (90:0:1);
\draw (1.5, 0.5)--(1.5, -0.5);
\draw (1.5, -0.5) arc (0:-90:1);
\draw[dashed] (0.5, -1.5)--(-0.5, -1.5);
\draw[->-=0.3] (-0.5, -1.5) arc (-90:-180:1);
\end{scope}
\draw (2.5, 0) node[left] {\tiny {$X_{j}$}};
\fill[white] (-1.5, -1) rectangle (0, 1);
\draw (-1.25, 0.45) rectangle (-0.25, 1);
\draw (-0.75, 0.725){};
\draw (-1.25, -0.45) rectangle (-0.25, -1);
\draw (-0.75, -0.725){};
\draw[blue, ->-=0.5] (-0.75, 0.45)--(-0.75,-0.45);
\draw (-1, 0) node{\tiny {$X_k$}};
\end{tikzpicture}}}}
=\sum_{j, k}d_{j}d_{k}\vcenter{\hbox{\scalebox{0.6}{
\begin{tikzpicture}
\begin{scope}[color=blue, xshift=1cm]
\draw (-1.5, -0.5)--(-1.5, 0.5);
\draw (-1.5, 0.5) arc (180:90:1);
\draw[dashed] (-0.5, 1.5)--(0.5, 1.5);
\draw[-<-=0.3] (0.5, 1.5) arc (90:0:1);
\draw[-<-=0.7] (1.5, -0.5) arc (0:-90:1);
\draw[dashed] (0.5, -1.5)--(-0.5, -1.5);
\draw (-0.5, -1.5) arc (-90:-180:1);
\end{scope}
\draw (-1, 0) node[]{\tiny $X_{k}$};
\begin{scope}[color=blue, xshift=1cm]
\draw (-2, 2.5) ..controls (-2, 1.5) and (-1.5, 1.866)..(-1.5, 1.866);
\draw (-1, 2) arc (90:120:1);
\draw[dashed] (-1, 2)--(1, 2);
\draw[->-=0.5] (1, 2) arc (90:0:1);
\draw[->-=0.5] (2, 1)--node[right]{\tiny {$X$}}(2, -1);
\draw[->-=0.5] (2, -1) arc (0:-90:1);
\draw[dashed] (1, -2)--(-1, -2);
\begin{scope}[yscale=-1]
\draw (-2, 2.5) ..controls (-2, 1.5) and (-1.5, 1.866)..(-1.5, 1.866);
\draw (-1, 2) arc (90:120:1);
\end{scope}
\end{scope}
\begin{scope}[xshift=3.5cm]
\fill[white] (-1.5, -1) rectangle (0, 1);
\draw (-1.25, 0.45) rectangle (-0.25, 1);
\draw (-0.75, 0.725){};
\draw (-1.25, -0.45) rectangle (-0.25, -1);
\draw (-0.75, -0.725){};
\draw[blue, ->-=0.5] (-0.75, 0.45)--(-0.75,-0.45);
\draw (-1, 0) node{\tiny{$X_j$}};
\end{scope}
\draw (-1, 1.7) node{\tiny{$X$}};
\draw (-1, -1.7) node{\tiny{$X$}};
\begin{scope}[xshift=3.5cm]
\draw[gray](-1.75, -1.25) rectangle (0.25, 1.25);
\draw[gray] (0.45, 1) node{\tiny$D_2$};
\end{scope}
\end{tikzpicture}
}}}
=
\vcenter{\hbox{\scalebox{0.6}{
\begin{tikzpicture}
\begin{scope}[xshift=3.5cm]
\draw[gray](-1.75, -1.25) rectangle (0.25, 1.25);
\draw[gray] (0.45, 1) node{\tiny$D_2$};
\end{scope}
\begin{scope}[color=red, xshift=1cm]
\draw (-1.5, -0.5)--(-1.5, 0.5);
\draw (-1.5, 0.5) arc (180:90:1);
\draw[dashed] (-0.5, 1.5)--(0.5, 1.5);
\draw (0.5, 1.5) arc (90:0:1);
\draw (1.5, 0.5)--(1.5, -0.5);
\draw (1.5, -0.5) arc (0:-90:1);
\draw[dashed] (0.5, -1.5)--(-0.5, -1.5);
\draw (-0.5, -1.5) arc (-90:-180:1);
\end{scope}
\begin{scope}[color=blue, xshift=1cm]
\draw (-2, 2.5) ..controls (-2, 1.5) and (-1.5, 1.866)..(-1.5, 1.866);
\draw (-1, 2) arc (90:120:1);
\draw[dashed] (-1, 2)--(1, 2);
\draw (1, 2) arc (90:0:1);
\draw[->-=0.5] (2, 1)--(2, -1);
\draw (2, -1) arc (0:-90:1);
\draw[dashed] (1, -2)--(-1, -2);
\begin{scope}[yscale=-1]
\draw (-2, 2.5) ..controls (-2, 1.5) and (-1.5, 1.866)..(-1.5, 1.866);
\draw (-1, 2) arc (90:120:1);
\end{scope}
\end{scope}
\draw (3, 0) node[right]{\tiny$X$};
\end{tikzpicture}
}}}
$$
The first equality comes from applying Planar Graphical Calculus on the contractible region $D_{1}$ and $D_{2}$ respectively indicated by the gray box in the picture. The middle equality comes from isotopy of diagram $\Gamma$. 
\end{proof}

\begin{example}
The following is an example of handle sliding for a curve:
$$\PipeSlide$$
\end{example}

\begin{example}
We compute the alterfold of $B$-colored $3$-sphere $(\bS^{3}(B), \emptyset, \emptyset)$. The $3$-sphere can be decompose to one $0$-handle and one $3$-handle, therefore:
$$Z(\bS^{3}(B), \emptyset, \emptyset)=\mu^{-1}\zeta^{-1}Z(\bS^{3}, \partial \bB^{3}, \emptyset)=\mu^{-1}Z(\bS^{3}(A), \emptyset, \emptyset)=\mu^{-1},$$
where the first equality comes from applying Move $0$ and the second equality comes from applying Move $1$.
\end{example}

\subsection{Link Surgery}
Let $L$ be a framed link embedded in the $B$-colored region of a $\cC$-decorated $3$-alterfold $(M, \Sigma, \Gamma)$, we define two operation changes the decorated alterfold:
\begin{itemize}
    \item \textbf{Link surgery:} Removing a tubular neighborhood $L_{\epsilon}$ of $L$ and gluing back a solid torus switching the longitude and the meridian(with respect to the framing of $L$). We denote the resulting alterfold by $(M_{L}, \Sigma, \Gamma)$.
    \item \textbf{Alterfold surgery:} Changing the coloring of a tubular neighborhood $L_{\epsilon}$ of $L$ to $A$, and put an $\Omega$-colored circle on the longitude of each component of $\partial L_{\epsilon}$(with respect to the framing of $L$). We denote the resulting alterfold by $(M, \Sigma\sqcup \partial L_{\epsilon}, \Gamma\sqcup \Gamma_{L})$.
\end{itemize}
We remark that the link surgery move changes the topology of the ambient manifold $M$ while the alterfold surgery does not. We define the following surgery move interchanging these two resulting decorated alterfold.

The next theorem discusses the relation between the quantity $Z(M_{L}, \Sigma, \Gamma)$ and $Z(M, \Sigma\sqcup \partial L_{\epsilon}, \Gamma\sqcup \Gamma_{L})$.

\begin{theorem}\label{thm:linksurgery}
Let $(M, \Sigma, \Gamma)$ be a $\mathcal{C}$-decorated $3$-alterfold, and $L\subset M$ is a framed link embedded in the $B$-colored region. Then
$$Z(M_L, \Sigma, \Gamma)=\mu^{-|L|}Z(M, \Sigma\sqcup \partial L_{\epsilon}, \Gamma\sqcup \Gamma_{L})$$
where $(M_L, \Sigma, \Gamma)$ and $Z(M, \Sigma\sqcup \partial L_{\epsilon}, \Gamma\sqcup \Gamma_{L})$ are defined as above.
\end{theorem}

\begin{proof}
By the definition of link surgery, $R_{B}(M_{L}, \Sigma)$ is obtained from $R_{B}(M, \Sigma)$ by taking out the tubular neighborhood $L_{\epsilon}$ of $L$, then for each link component, gluing back a solid torus $\mathbb{T}_i$ in a way that the meridian of the torus is identified to the longitude of $L_{\epsilon}$. Let $\mathbb{T}:=\sqcup_{i} \mathbb{T}_i$, we have $R_{B}(M_{L}, \Sigma)\setminus \mathbb{T}$ homeomorphic to $R_{B}(M, \Sigma)\setminus L_{\epsilon}$. Since each $T_i$ decomposes to a $0$-handle and a $1$-handle. Applying Move $0$ and Move $1$ to each $\mathbb{T}_i$, we obtain a $\mathcal{C}$-decorated 3-alterfold 
$$(M_{L}, \Sigma', \Gamma'):=(M', \Sigma\sqcup \partial \mathbb{T}, \Gamma\sqcup \Gamma_\mathbb{T})$$
where $\Gamma_\mathbb{T}$ are union of $\Omega$ colored circles on the meridian of each $\mathbb{T}_i$ coming from Move $1$. 

Notice $R_{B}(M', \Sigma')$ is homeomorphic to $R_{B}(M, \Sigma\sqcup \partial L_{\epsilon})$. In addition, this homeomorphism maps the meridians of $\mathbb{T}$ to the longitudes of $L_{\epsilon}$. Therefore, by the homeomorphism property in Theorem \ref{thm:partition function}, we have that
$$Z(M, \Sigma\sqcup \partial L_{\epsilon}, \Gamma\cup \Gamma_{L})=Z(M', \Sigma', \Gamma')=\mu^{|L|}Z(M', \Sigma, \Gamma).$$
\end{proof}

\begin{corollary}
Let $L$ and $L'$ be two framed links in $M$.
If $M_{L}\cong M_{L'}$, then 
$$\mu^{-|L|}Z(M, \partial L_{\epsilon}, \Gamma_{L})=\mu^{-|L'|}Z(M, \partial L_{\epsilon}', \Gamma_{L'}).$$
In particular, this defines a surgery invariant.
\end{corollary}

\begin{proof}
It follows from the partition function $Z$ defines a topological invariant for manifold (considered as alterfold with $B$-color).
\end{proof}

\begin{remark}
In Section $4$, we will show the category of $A$-colored tubes is braided equivalent to the Drinfeld center $\mathcal{Z(C)}$ of $\mathcal{C}$, under this equivalence, solid torus with a $\Omega$-colored longitude is identified as the $\Omega$-color of $\mathcal{Z(C)}$. In the case of $M=\bS^{3}$, this is the Reshetikhin-Turaev invariant of $\mathcal{Z(C)}$.
\end{remark}

\begin{corollary}\label{cor:TVlink}
Let $L$ be a framed link in $\bS^3$. 
Then
$$\mu^{-|L|}Z(\bS^{3}, \partial L_{\epsilon}, \Gamma_{L})=TV_{\mathcal{C}}(\bS_{L}^{3}).$$
\end{corollary}
\begin{proof}
    It follows from Proposition \ref{prop:turaev-viro} and Theorem \ref{thm:linksurgery}.
\end{proof}

To end this subsection, we give an alternative equivalent characterization of the partition function using link surgery:
\begin{theorem}\label{thm:partition function2}
Let $\cC$ be a spherical fusion category over $\k$, and $\zeta$ be a nonzero scalar. There exist a unique partition function $Z$ from $\cC$-decorated $3$-alterfolds to the ground field $\k$ subject to \textbf{Disjoint Union, Planar Graphical Calculus, Move 0-3} described in Theorem \ref{thm:partition function} and
\begin{itemize}
    \item \textbf{$\bS^{3}$ Normalization:} Let $\bS^3(A)$ be the $A$-colored $3$-sphere,
$$Z(\bS^{3}(A), \emptyset, \emptyset)=1.$$
    \item \textbf{Surgery Move:} Let $L$ be a framed link in the $B$-colored region of $\cC$-decorated $3$-alterfold $(M, \Sigma, \Gamma)$. Then
    $$Z(M_L, \Sigma, \Gamma)=\mu^{-|L|}Z(M, \Sigma\sqcup \partial L_{\epsilon}, \Gamma\sqcup \Gamma_{L})$$
    where $(M_L, \Sigma, \Gamma)$ and $(M, \Sigma\sqcup \partial L_{\epsilon}, \Gamma\sqcup \Gamma_{L})$ are $3$-alterfold derived from link surgery and alterfold surgery along $L$ in $M$ respectively.
\end{itemize}
\end{theorem}
\begin{proof}
We only need to show these moves imply the Normalization property in Theorem \ref{thm:partition function} by Remark \ref{rem:homeingore}. 
In order to show this, it suffices to show for all connected oriented $3$-manifold $M$, $Z(M(A), \emptyset, \emptyset)=1$ where $M(A)$ is $A$-colored 3-manifold $M$ in the alterfold.

Without loss of generality, we may assume $M=\bS^{3}_{L}$, i.e., can be obtained by doing link surgery along $L=L_1\sqcup L_{2}\cdots \sqcup L_{n}$ in $\bS^{3}$. Let $L'=L'_1\sqcup L'_{2}\cdots \sqcup L'_{n}$ be disjoint union of $0$-framed unknot in $\bS^{3}\setminus L$, each of which links only with the corresponding components of $L$ as follows
$$
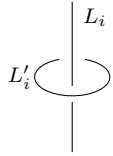
\begin{tikzpicture}
\draw (0, -1) to (0, 1);
\path[fill=white](-0.17, -0.33) rectangle (0.17, -0.13);
\draw (0, 0) [partial ellipse=-250:70:0.5 and 0.25];
\draw (0, 0.8) node[right]{\tiny$L_i$};
\draw (-0.7, 0) node{\tiny$L'_i$};
\end{tikzpicture}$$
We compute 
$$Z(\bS^{3}, \partial L_{\epsilon}\sqcup \partial L'_{\epsilon}, \Gamma_{L}\sqcup\Gamma_{L'})=\mu^{2n}Z(\bS^{3}(B), \emptyset, \emptyset)=\mu^{2n-1}\ne 0.$$
The first equality comes applying Surgery Move to $L\sqcup L'$, note that by Kirby calculus \cite{Kir78,FenRou79}, $\bS^3_{L\sqcup L'}\cong\bS^{3}$. 
Moreover, note that each component of $\bS^{3}_{L}-(\bS^{3}-L_{\epsilon})$ is a $B$-colored genus-1 handlebody, so by applying a Move $0$ and a Move $1$ to change the color of each such handlebody to $A$, we have
\begin{equation}\label{eq:surgery-L}
Z(\bS^{3}_{L}, -\partial (\bS^{3}-L_{\epsilon})\sqcup \partial L'_{\epsilon}, \Gamma_{L}\sqcup\Gamma_{L'}) = \mu^{n} Z(\bS^{3}_{L}, \partial L'_{\epsilon}, \Gamma_{L'})\,.
\end{equation}
Since the alterfolds $(\bS^{3}_{L}, -\partial (\bS^{3}-L_{\epsilon})\sqcup \partial L'_{\epsilon}, \Gamma_{L}\sqcup\Gamma_{L'})$ and $(\bS^{3}, \partial L_{\epsilon}\sqcup \partial L'_{\epsilon}, \Gamma_{L}\sqcup\Gamma_{L'})$ have exactly the same $B$-colored region (together with the $\cC$-decoration on the separating surfaces), so one can apply Moves $0$-$3$ to change their color to $A$. 
As a consequence,
\begin{equation}\label{eq:changecolor}
Z(\bS^{3}_{L}, -\partial (\bS^{3}-L_{\epsilon})\sqcup \partial L'_{\epsilon}, \Gamma_{L}\sqcup\Gamma_{L'})=\mu^{2n-1}Z(\bS^{3}_{L}(A), \emptyset, \emptyset)\,.
\end{equation}
Now by Equation \eqref{eq:surgery-L} and the Surgery Move, we have
\begin{align*}
& Z(\bS^{3}_{L}, -\partial (\bS^{3}-L_{\epsilon})\sqcup \partial L'_{\epsilon}, \Gamma_{L}\sqcup\Gamma_{L'})
= \mu^{n}Z(\bS^{3}_{L}, \partial L'_{\epsilon}, \Gamma_{L'})\\
=& Z(\bS^{3}, \partial L_{\epsilon}\sqcup \partial L'_{\epsilon}, \Gamma_{L}\sqcup\Gamma_{L'})
= \mu^{2n-1}\,.
\end{align*}
Finally, comparing with Equation \eqref{eq:changecolor}, we have $Z(\bS^{3}_{L}(A), \emptyset, \emptyset)=1$.
\end{proof}

\subsection{Ideal Triangulation and Invariants of Pseudo-Manifolds}
The Turaev-Viro invariant has been extended to pseudo $3$-manifold in \cite{BenPet96}. In this section, we discuss the relation between the 3-alterfold and the pseudo $3$-manifold. The partition function recovers the Turaev-Viro invariants for pseudo $3$-manifold.

By a \textit{pseudo $3$-manifold}, we mean a topological space $M$ such that at each point $P \in M$, its neighborhood $U_{P}$ is homeomorphic to a cone of an closed compact oriented surface $\Sigma_{P}$. A point $P \in M$ is called \textit{singular} if $\Sigma_{P}$ is not homeomorphic to $\bS^2$. 
In particular, a $3$-manifold is a pseudo $3$-manifold with no singular points. On the other hand, one can obtain a unique $3$-manifold with boundary up to homeomorphism by removing a tubular neighborhood for every singular point.

A triangulation $\Delta$ of a pseudo 3-manifold $M$ consists of a disjoint union $X=\bigsqcup_{t\in T} t$ and a collection of homeomorphisms $\Psi$ between pair of faces such that the  quotient space $X\slash \Psi$ is homeomorphic to $M$. 
The \textit{vertices, edges, faces} and \textit{tetrahedra} in $\Delta$ are respectively the quotients of vertices, edges, faces, and tetrahedra in $X$. 
From the definition, all singular points of $M$ must be vertices. 
Non-singular vertices are called \textit{inner}.

\begin{definition}\cite{CheYan18}
Let $\mathcal{C}$ be a spherical fusion category, and $\Delta$ be a triangulation of a pseudo $3$-manifold $M$. The Turaev-Viro invariant of $M$ is defined as the sum
$$TV_{\mathcal{C}}(M, \Delta)=\mu^{-|V|}\sum_{c}\dim(c)\left\langle\bigotimes_{t\in T}\Gamma_{t}^{c}\,, \ \bigotimes_{f\in F}\star_{f}^{c}\right\rangle$$
where $|V|$ is the number of inner vertices, and $\Gamma_{t}^{c}$, $\star_{c}^{f}$ and $\dim(c)$ are as in the proof of Proposition \ref{prop:turaev-viro}.
\end{definition}

\begin{remark}
The original definition in \cite{CheYan18} is of the case of representation category of quantum group $\mathcal{U}_{q}(\mathfrak{sl}_{2})$ and the formula is written in terms of $6j$-symbols.
The definition generalizes to arbitrary spherical fusion categories.
\end{remark}

\begin{theorem}\cite{CheYan18}\label{thm:pseudotv}
The quantity $TV_{\mathcal{C}}(M, \Delta)$ depends only on the pseudo-manifold $M$, not on the triangulation $\Delta$. \qed
\end{theorem}

\begin{remark}
In \cite{CheYan18}, the theorem is proved by checking the quantity is invariant under $0$-$2$ and $2$-$3$ move of triangulations. 
In our framework, we can interpret this quantity by the partition function evaluate on certain $\mathcal{C}$-decorated 3-alterfold and the invariance follows.
\end{remark}

\newcommand{\sing}{\Lambda}
\begin{proof}
We denote the set of singular points of $M$ by $\sing(M)$. For each singular point $P\in \sing(M)$, let $H_{P}$ be a multi-handlebody with its boundary homeomorphic to $\Sigma_{P}$ and $f_{P}:\Sigma_{P}\rightarrow \partial H_{P}$ be a homeomorphism. We define a $3$-manifold
$$M^{\dagger}:=(M\setminus \bigsqcup_{P \in \sing(M)}U_{P}) \sqcup (\bigsqcup_{P\in \sing(M)}H_{P})\slash \{f_{P}| P \in \sing(M)\},$$
We claim
$$TV_{\mathcal{C}}(M, \Delta)=\zeta^{\chi(M)-|\sing(M)|}Z(M^{\dagger}, \bigsqcup_{P\in \sing(M)}\Sigma_{P}, \emptyset)$$
where the $A$-colored regions are the union of multi-handlebodies $H_{P}$.

To prove the claim, we first compute the right hand side. By Theorem \ref{thm:partition function}, we can apply Moves 0-3 to change the $B$-colored region of the 3-alterfold into $3$-balls, each of which arises from a truncated tetrahedron in $\Delta$ associated to singular points. Note that in this process, we also change the $\cC$-decoration.

To get the left hand side, we apply Move $0$ to each inner vertex of $\Delta$. 
Then we apply Move $1$ to each edge (possibly truncated). Lastly, we apply Move $2$ to each face (possibly polygon after truncation). 
This process results in a 3-alterfold $(M^{\dagger}, \Sigma^{\Delta}, \Gamma=\emptyset^{\Delta})$ with diagram $\Gamma^{c}$ on the boundary of $B$-colored balls corresponds to the truncated tetrahedrons that are filled in with $\star_{f}^{c}$. Therefore, we have
\begin{align*}
Z(M^{\dagger}, \bigsqcup_{P\in \sing(M)}\Sigma_{P}, \emptyset)=&\mu^{-|V|}\zeta^{-|V|+|E|-|F|}Z(M^{\dagger}, \Sigma^{\Delta}, \Gamma)\\
=&\mu^{-|V|}\zeta^{-|V|+|E|-|F|}\sum_{c}\dim(c)\left\langle\bigotimes_{t\in T}\Gamma_{t}^{c}\,, \ \bigotimes_{f\in F}\star_{f}^{c}\right\rangle \zeta^{|T|}\\
=&TV_{\mathcal{C}}(M, \Delta)\zeta^{-|V|+|E|-|F|+|T|}\\
=&TV_{\mathcal{C}}(M, \Delta)\zeta^{-|V|+|E|-|F|+|T|-|\sing(M)|+|\sing(M)|}\\
=&\zeta^{-\chi(M)+\sing(M)}TV_{\mathcal{C}}(M, \Delta).
\end{align*}
The claim follows.

Since $\chi(M)$, $|\sing(M)|$ and $B$-colored region of $(M^{\dagger}, \bigsqcup_{P\in \sing(M)}\Sigma_{P}, \emptyset)$ depend only on the pseudo-manifold $M$, not on the triangulation $\Delta$. The theorem follows.
\end{proof}

\section{Topologized Drinfeld Center}
In this section, we define, for any spherical fusion category $\mathcal{C}$, the associated tube category $\cA$ using the partition function $Z$ defined in the previous section and show that the tube category is braided monoidal equivalent to the Drinfeld center $\cZ(\cC)$ of $\mathcal{C}$. 
We assume that in this section, the ambient space $M=\bS^{3}$.

\subsection{The Tube Category $\mathcal{A}$}
To define the category $\cA$, we first define the topological building blocks of $\cA$, namely, $\mathcal{C}$-disks and cornered $\mathcal{C}$-handlebodies.

\begin{definition}[$\mathcal{C}$-Disk]
 By a $\mathcal{C}$-disk $D$, we mean an oriented pointed disk with marked points (different from the base point) on the boundary colored by objects of $\mathcal{C}$.  We put a $\$$ sign on the base point and read the circle along the orientation. The boundary $\partial D$ of a $\mathcal{C}$-disk $D$ is called a $\mathcal{C}$-circle.
\end{definition}

For example, a $\mathcal{C}$-disk with marked points on the boundary colored by $X, Y, V$ is given by
 $$\raisebox{-0.4cm}
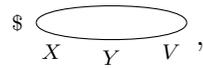
{\begin{tikzpicture}
\draw  (-1, 1.25) .. controls +(0, 0.3) and +(0, 0.3).. (1, 1.25);
\draw  (1,1.25)  .. controls +(0, -0.3) and +(0, -0.3).. (-1, 1.25) node [pos=0.5, below] {\tiny $Y$} node [pos=0.2, below] {\tiny $V$} node [pos=0.8, below] {\tiny $X$};
\draw (-1.25, 1.25) node{\tiny{$\$$}};
\end{tikzpicture}}\,,$$ 
where we orient the above disk counter-clockwisely. 

\begin{definition}[Cornered $\mathcal{C}$-Handlebody]
A cornered handlebody is a pair $\displaystyle \left(H, \bD=\bigsqcup_{i} D_i\right)$, where $H$ is a (multi-)handlebody and $\mathbb{D}$ is disjoint union of finitely many oriented pointed disks $D_j$ on the boundary $\partial{H}$ of $H$. 
    
By a cornered $\mathcal{C}$-handlebody $(H, \bD, \Gamma)$, we mean a cornered (multi-)handlebody $\displaystyle \left(H, \bD=\bigsqcup_{i} D_i\right)$ together with a tensor diagram $\Gamma$ of $\mathcal{C}$ such that $D_i$ is a $\cC$-disk, the edges of $\Gamma$ is in the complement of $\bD$, meeting the boundary circles $\partial D_i$ transversely at the marked points, and the color of the edges is compatible with the color of the marked points.
\end{definition}

\begin{definition}[Gluing of Cornered $\mathcal{C}$-Handlebody]
Suppose that $(H, \bD, \Gamma)$ and $(H', \bD', \Gamma')$ are two cornered $\mathcal{C}$-handlebodies.
Let $\bD_1 \subset \bD$ and $\bD_2 \subset \bD'$ be subsets of $\bD$ and $\bD'$, and suppose $f: \bD_1 \to \bD_2$ is an orientation reversing homeomorphism preserving boundary conditions (i.e., $\$$-sign, marked points and their colors). Then we can glue the two cornered $\mathcal{C}$-handlebodies via $f$ and the resulting cornered $\mathcal{C}$-handlebody is denoted by $(H\cup_f H', \bD\sqcup \bD'\backslash (\bD_1 \sqcup \bD_2), \Gamma \cup \Gamma')$.
\end{definition}

Now we can define a category using $\cC$-disks and cornered $\cC$-handlebodies embedded in $\bS^3 = \mathbb{R}^3\cup \{\infty\}$. 
For such cornered $\mathcal{C}$-handlebody $(H, \bD, \Gamma)$, we can view $(\bS^3, \partial H)$ as a 3-alterfold(with boundary) with the interior (resp.~exterior) of $H$ colored by $A$ (resp.~$B$). In particular, if $\partial H=\emptyset$, the $\mathcal{C}$-diagram $\Gamma$ is closed, $(M, \partial H, \Gamma)$ is a $\mathcal{C}$-decorated 3-alterfold.
We also assume that the image of $\bD$ in $H$ is a subset of $\mathbb{R}^2 \times \{0, 1\}$ with the restriction of the embedding of $\bD$ on its connected components to $\mathbb{R}^2\times \{1\}$ (resp.~$\mathbb{R}^2\times \{0\}$) orientation preserving (resp.~orientation reversing). For example, a cornered $\cC$-handlebody in $\bS^3$ can depicted as:
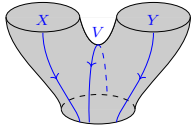
\begin{align*}
\vcenter{\hbox{
\scalebox{0.7}{\begin{tikzpicture}[scale=0.35, baseline=.5ex]
\draw (-3, 0) [partial ellipse=0:360:2 and 0.8];
\draw (3, 0) [partial ellipse=0:360:2 and 0.8];
\draw[dashed] (0, -5)[partial ellipse=0:180:2 and 0.8];
\draw (0, -5)[partial ellipse=180:360:2 and 0.8];
\draw (5, 0)..controls (5,-2) and (2,-4).. (2, -5);
\draw (-5, 0)..controls (-5,-2) and (-2,-4).. (-2, -5);
\draw (-1, 0) ..controls (-.5, -2) and (.5, -2)..(1, 0);
\path [fill=black, opacity=0.2] (5, 0)..controls (5,-2) and (2,-4).. (2, -5)--(-2, -5).. controls (-2,-4) and (-5,-2) .. (-5, 0);
\path [fill=black, opacity=0.2] (-3, 0) [partial ellipse=0:180:2 and 0.8] (3, 0) [partial ellipse=0:180:2 and 0.8] (0, -5)[partial ellipse=180:360:2 and 0.8]; 
\draw [fill=white] (-1, 0) ..controls (-.5, -2) and (.5, -2)..(1, 0);
\draw[blue, ->-=0.5] (3, -0.85) node [above] {\tiny{$Y$}} ..controls (3,-3) and (1,-5).. (1, -5.7);
\draw[blue, ->-=0.5] (-3, -0.85) node [above] {\tiny{$X$}}..controls (-3,-3) and (-1,-5).. (-1, -5.7);
\draw[blue, ->-=0.3] (0,-4.8)[partial ellipse=90:196:0.5 and 3.3];
\draw[blue, dashed] (0,-5)[partial ellipse=16:90:0.5 and 3.5] node [above]{\tiny $V$};
\end{tikzpicture}}}}
\end{align*}
where $\Gamma$ is the union of three strings.
We will simply use the following graph without color:
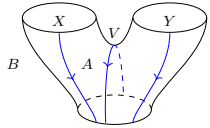
\begin{align*}
\vcenter{\hbox{
\scalebox{0.7}{\begin{tikzpicture}[scale=0.35, baseline=.5ex]
\draw (-3, 0) [partial ellipse=0:360:2 and 0.8];
\draw (3, 0) [partial ellipse=0:360:2 and 0.8];
\draw[dashed] (0, -5)[partial ellipse=0:180:2 and 0.8];
\draw (0, -5)[partial ellipse=180:360:2 and 0.8];
\draw (5, 0)..controls (5,-2) and (2,-4).. (2, -5);
\draw (-5, 0)..controls (-5,-2) and (-2,-4).. (-2, -5);
\draw (-1, 0) ..controls (-.5, -2) and (.5, -2)..(1, 0);
\draw[blue, ->-=0.5] (3, -0.85) node [above, black] {\tiny{$Y$}} ..controls (3,-3) and (1,-5).. (1, -5.7);
\draw[blue, ->-=0.5] (-3, -0.85) node [above, black] {\tiny{$X$}}..controls (-3,-3) and (-1,-5).. (-1, -5.7);
\draw[blue, ->-=0.3] (0,-4.8)[partial ellipse=90:196:0.5 and 3.3];
\draw[blue, dashed] (0,-5)[partial ellipse=16:90:0.5 and 3.5] node [above, black]{\tiny $V$};
\node at (-5.5, -2.5) {\tiny $B$}; \node at (-1.5, -2.5) {\tiny $A$};
\end{tikzpicture}}}}
\end{align*}
If there is no confusion, we will also suppress the labels $A$ and $B$.

\begin{definition}\label{def:A0} 
Let $\cA_0$ be a category such that
\begin{enumerate}[(1)]
\item Objects are formal direct sums of finite unions of horizontal $\mathcal{C}$-disks in the $xy$-plane of radius $\displaystyle \frac{1}{4}$ with their centers placed at integer points on the $x$-axis, and base points placed on the left side (i.e., at $\displaystyle \mathbb{Z}-\frac{1}{4}$ on the $x$-axis).  

\item The morphism space $\hom_{\cA_{0}}(O_1, O_2)$ between $\mathcal{C}$-disks $O_1$ and $O_2$ are $\k$-linear span of cornered $\mathcal{C}$-handlebodies $(H, \bD, \Gamma)$ embedded in $\mathbb{R}^2\times [0, 1]$ such that $\bD=O_1\cup \overline{O_2}$, $\bD \cap \left(\mathbb{R}^2 \times \{1\}\right)=O_1$, and $\bD \cap \left(\mathbb{R}^2 \times \{0\}\right)=O_2$, modulo isotopy relative to the corners. The morphism space $\hom_{\mathcal{A}_{0}}(\bigoplus_{i}O_i, \bigoplus_{j}O_j)$ equals to the set of matrices with the $ij$-th entry taking values in $\hom_{\mathcal{A}}(O_{i}, O_{j})$.

The composition of morphisms are defined to be the composition of the matrix multiplication and the vertical gluing of the cornered $\mathcal{C}$-handlebodies (and isotope to $\mathbb{R}^2 \times [0,1]$ afterwards).
\end{enumerate}
\end{definition}
\renewcommand{\CIRCLE}{O}

\begin{remark}
Suppose $X \in \cC$, there is an object in $\cA_0$ depicted as
$\vcenter{\hbox{\begin{tikzpicture}
\draw (-1, 1.25) .. controls +(0, 0.5) and +(0, 0.5).. (1, 1.25);
\draw (1,1.25)  .. controls +(0, -0.5) and +(0, -0.5).. (-1, 1.25) node [pos=0.5, below] {\tiny $X$};
\end{tikzpicture}}}$, which is denoted by $\CIRCLE_X$.
Notice that the empty set and unmarked $\mathcal{C}$-disk should be treated as different objects in $\cA_0$.
We denote by $\1_{\emptyset}$ the empty set and \CIRCLE$_{\1}$ the unmarked $\mathcal{C}$-disk. A morphism between them can be depicted as
$\raisebox{-0.4cm}{
\begin{tikzpicture}[scale=0.8]
\begin{scope}[shift={(0, -2.5)}]
\draw [ dashed] (-1, 1.25) .. controls +(0, 0.3) and +(0, 0.3).. (1, 1.25);
\draw    (1,1.25)  .. controls +(0, -0.3) and +(0, -0.3).. (-1, 1.25);
\end{scope}
\draw (-1, -1.25) .. controls +(0, 1.2) and +(0, 1.2).. (1, -1.25);
\end{tikzpicture}}\in \hom_{\cA_0}(\1_{\emptyset}, \CIRCLE_{\1})$.
\end{remark}

\begin{remark}
Suppose $X\in \mathcal{C}$, The identity morphism of $O_{X}$ is depicted as the cylinder
\raisebox{-1.5em}{\scalebox{0.5}{\begin{tikzpicture}[scale=0.3]
\draw (0,5) [partial ellipse=0:360:2 and 0.8];
\draw (-2, 5)--(-2, -4);
\draw (2, 5)--(2, -4);
\draw[dashed] (0,-4) [partial ellipse=0:180:2 and 0.8];
\draw (0,-4) [partial ellipse=180:360:2 and 0.8];
\draw[blue, ->-=0.5] (0, 4.2)--(0, -4.8) node [below, black] {\tiny $X$};
\end{tikzpicture}}}.
Note that, the orientation of the top and bottom corners are opposite to each other.
\end{remark}

\begin{remark}
Note that for any $X$, $Y \in \cC$, $O_{XY}\cong O_{YX}$ by the following isomorphisms.

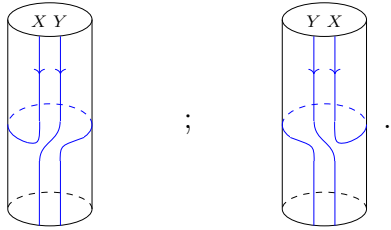
\[
\vcenter{\hbox{\scalebox{0.8}{
\begin{tikzpicture}[scale=0.35]
\draw (0,5) [partial ellipse=0:360:2 and 0.8];
\draw (-2, 5)--(-2, -4);
\draw (2, 5)--(2, -4);
\draw[dashed] (0,-4) [partial ellipse=0:180:2 and 0.8];
\draw (0,-4) [partial ellipse=180:360:2 and 0.8] node[below, white]{\tiny{$X$}};
\begin{scope}[yshift=0.01cm]
\draw[blue, ->-=0.2] (0.5, 4.2) node[above, black]{\tiny{$Y$}}--(0.5, -4.8);
\draw[blue, ->-=0.2] (-0.5, 4.2)node[above, black]{\tiny{$X$}}--(-0.5, -4.8);
\end{scope}
\begin{scope}[shift={(0, 0)}, xscale=5.4, yscale=4]
\draw [blue, dashed] (0,0) [partial ellipse=0:180:0.37 and 0.25];
\draw [white, line width=4pt] (0,0) [partial ellipse=270:250:0.37 and 0.25];
\draw [blue] (0,0) [partial ellipse=180:360:0.37 and 0.25];
\begin{scope}[shift={(0, -0.2)}]
\path [fill=white] (-0.2, -0.23) rectangle (0.3, 0.24);
\draw[blue] (0.09,-0.23)..controls +(0,0.2) and +(-0.2,-0.1)..(0.3,0.05);
\draw[blue] (-0.09,0.24)..controls (-0.09,-0) and (-0.09,-0.04)..(-0.2,-0.01);
\draw[blue] (-0.09, -0.23)..controls (-0.09, 0) and (0.09, 0)..(0.09, 0.24);
\end{scope}
\end{scope}
\end{tikzpicture}}}}
\hspace{1cm}
;
\hspace{1cm}
\vcenter{\hbox{\scalebox{0.8}{
\begin{tikzpicture}[scale=0.35]
\draw (0,5) [partial ellipse=0:360:2 and 0.8];
\draw (-2, 5)--(-2, -4);
\draw (2, 5)--(2, -4);
\draw[dashed] (0,-4) [partial ellipse=0:180:2 and 0.8];
\draw (0,-4) [partial ellipse=180:360:2 and 0.8] node[below, white]{\tiny{$X$}};
\begin{scope}[yshift=0.01cm]
\draw[blue, ->-=0.2] (0.5, 4.2) node[above, black]{\tiny{$X$}}--(0.5, -4.8);
\draw[blue, ->-=0.2] (-0.5, 4.2)node[above, black]{\tiny{$Y$}}--(-0.5, -4.8);
\end{scope}
\begin{scope}[shift={(0, 0)}, xscale=-5.4, yscale=4]
\draw [blue, dashed] (0,0) [partial ellipse=0:180:0.37 and 0.25];
\draw [white, line width=4pt] (0,0) [partial ellipse=270:250:0.37 and 0.25];
\draw [blue] (0,0) [partial ellipse=180:360:0.37 and 0.25];
\begin{scope}[shift={(0, -0.2)}]
\path [fill=white] (-0.2, -0.23) rectangle (0.3, 0.24);
\draw[blue] (0.09,-0.23)..controls +(0,0.2) and +(-0.2,-0.1)..(0.3,0.05);
\draw[blue] (-0.09,0.24)..controls (-0.09,-0) and (-0.09,-0.04)..(-0.2,-0.01);
\draw[blue] (-0.09, -0.23)..controls (-0.09, 0) and (0.09, 0)..(0.09, 0.24);
\end{scope}
\end{scope}
\end{tikzpicture}}}}.
\]
\end{remark}

Now we endow $\mathcal{A}_{0}$ with a monoidal structure via horizontal juxtaposition. 
\begin{enumerate}
    \item The tensor of objects is to put the $\mathcal{C}$-disks together disjointly. For example,
    \begin{align*}
    O_{X}\otimes O_{Y}=\vcenter{\hbox{
\begin{tikzpicture}
\draw (-1, 1.25) .. controls +(0, 0.5) and +(0, 0.5).. (1, 1.25);
\draw (1,1.25)  .. controls +(0, -0.5) and +(0, -0.5).. (-1, 1.25) node [pos=0.5, below] {\tiny $X$};
\begin{scope}[shift={(2.5, 0)}]
    \draw (-1, 1.25) .. controls +(0, 0.5) and +(0, 0.5).. (1, 1.25);
\draw (1,1.25)  .. controls +(0, -0.5) and +(0, -0.5).. (-1, 1.25) node [pos=0.5, below] {\tiny $Y$};
\end{scope}
\end{tikzpicture}}}\,.
\end{align*}
\item The tensor of morphisms is obtained from putting the cornered handlebodies together disjointly. For example,
$$
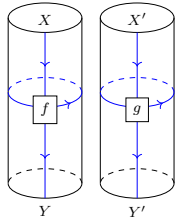
\scalebox{0.7}{\begin{tikzpicture}[scale=0.35]
\draw (0,5) [partial ellipse=0:360:2 and 0.8];
\draw (-2, 5)--(-2, -4);
\draw (2, 5)--(2, -4);
\draw[dashed] (0,-4) [partial ellipse=0:180:2 and 0.8];
\draw (0,-4) [partial ellipse=180:360:2 and 0.8];
\draw[blue, ->-=0.5] (0, 4.2) node[above,black]{\tiny{${X}$}}->(0, 0.5);
\draw[blue, ->-=0.5] (0, -0.5)--(0, -4.8)node[below, black]{\tiny{${Y}$}};
\draw [blue , ->-=0.5](0, .95) [partial ellipse=-75:0:2 and 0.8];
\draw [blue, dashed](0, .95) [partial ellipse=0:180:2 and 0.8];
\draw [blue](0, .95) [partial ellipse=180:255:2 and 0.8];
\node [draw, fill=white] (0, 0){\tiny $f$};
\begin{scope}[shift={(5, 0)}]
\draw (0,5) [partial ellipse=0:360:2 and 0.8];
\draw (-2, 5)--(-2, -4);
\draw (2, 5)--(2, -4);
\draw[dashed] (0,-4) [partial ellipse=0:180:2 and 0.8];
\draw (0,-4) [partial ellipse=180:360:2 and 0.8];
\draw[blue, ->-=0.5] (0, 4.2) node[above, black]{\tiny{${X'}$}}->(0, 0.5);
\draw[blue, ->-=0.5] (0, -0.5)--(0, -4.8)node[below, black]{\tiny{${Y'}$}};
\draw [blue, ->-=0.5](0, .95) [partial ellipse=-75:0:2 and 0.8];
\draw [blue, dashed](0, .95) [partial ellipse=0:180:2 and 0.8];
\draw [blue](0, .95) [partial ellipse=180:255:2 and 0.8];
\node [draw, fill=white] (0, 0){\tiny $g$};
\end{scope}
\end{tikzpicture}}
$$
\end{enumerate}
Note that the empty set is the tensor unit for the monoidal structure and $\cA_0$ is a strict monoidal category. Moreover, the rigidity of $\cA_0$ follows from the rigidity of $\mathcal{C}$, and the pivotal structure of $\cA_0$ is inherited from $\mathcal{C}$.
In this way, $\cA_0$ is a strict pivotal monoidal category, and for any endomorphism $f \in \hom_{\cA_0}(O, O)$ of an object $O \in \cA_0$, its pivotal trace, $\Tr_{\cA_{0}}(f) \in \hom_{\cA_{0}}(\1_{\emptyset},\1_{\emptyset})$, is defined in the usual way. 

As is mentioned above, any morphism in $\hom_{\cA_{0}}(\emptyset,\emptyset)$ is linear combination of cornered $\cC$-handlebodies with empty corner, each of which gives rise to a $\cC$-decorated 3-alterfold. Hence, we can extend the partition function $Z$ in Theorem \ref{thm:partition function} linearly to such morphisms, which inspires the following definition.

\begin{definition}
For any $O \in \cA_0$, and any $f \in \hom_{\cA_0}(O, O)$, we define the (numerical) trace of $f$, denoted by $\tr(f)$, to be
\[\tr(f) := Z(\Tr_{\cA_{0}}(f)) = Z(\ev_{O} \circ (f\otimes \id_{O^*}) \circ \coev_{O}) \in \k\,.\]
\textbf{We assume that $\zeta=1$ in this section.}
\end{definition}

For example, if $X \in \cC$, and $\vcenter{\hbox{
\begin{tikzpicture}[scale=0.2]
\draw (0,5) [partial ellipse=0:360:2 and 0.8];
\draw (-2, 5)--(-2, -4);
\draw (2, 5)--(2, -4);
\draw[dashed] (0,-4) [partial ellipse=0:180:2 and 0.8];
\draw (0,-4) [partial ellipse=180:360:2 and 0.8];
\draw[blue, ->-=0.5] (0, 4.2) node[above, black]{\tiny{${X}$}}->(0, 3);
\draw[blue, ->-=0.5] (0, -3.5)--(0, -4.8) node[below, black]{\tiny{${X}$}};
\node (0, 0){\tiny $\Gamma$};
\end{tikzpicture}}}\in \hom_{\cA_0}(\CIRCLE_X, \CIRCLE_X)$ is an endomorphism of $O_X$ with $\Gamma$ a $\cC$-diagram. 
Then 
\begin{align*}
\tr\left( \vcenter{\hbox{\scalebox{0.7}{
\begin{tikzpicture}[scale=0.3]
\draw (0,5) [partial ellipse=0:360:2 and 0.8];
\draw (-2, 5)--(-2, -4);
\draw (2, 5)--(2, -4);
\draw[dashed] (0,-4) [partial ellipse=0:180:2 and 0.8];
\draw (0,-4) [partial ellipse=180:360:2 and 0.8];
\draw[blue, ->-=0.5] (0, 4.2) node[above, black]{\tiny{${X}$}}->(0, 3);
\draw[blue, ->-=0.5] (0, -3.5)--(0, -4.8) node[below, black]{\tiny{${X}$}};
\node (0, 0){\tiny $\Gamma$};
\end{tikzpicture}}}}\right)=Z\left(\vcenter{\hbox{\scalebox{0.8}{ \begin{tikzpicture}[scale=0.6]
\draw [ dashed] (-1, 1.25) .. controls +(0, 0.3) and +(0, 0.3).. (1, 1.25);
\draw [ dashed]  (1,1.25)  .. controls +(0, -0.3) and +(0, -0.3).. (-1, 1.25);
\begin{scope}[shift={(0, -1)}]
\end{scope}
\node at (0, 0) {\tiny $\Gamma$};
\draw [blue, ->-=0.5] (0, 0.7) --(0, 1) .. controls +(0,1) and +(0, 1) .. (2.5, 1) --(2.5, -1.5) node [below] {\tiny $X$} (0, -1.2) --(0, -1.5) .. controls +(0,-0.6) and +(0, -0.6) .. (2.5, -1.5);
\begin{scope}[shift={(0, -2.5)}]
\draw [  dashed] (-1, 1.25) .. controls +(0, 0.3) and +(0, 0.3).. (1, 1.25);
\draw [  dashed]  (1,1.25)  .. controls +(0, -0.3) and +(0, -0.3).. (-1, 1.25);
\end{scope}
\draw  (-1, 1.25)--(-1, -1.25) (1, 1.25)--(1, -1.25);
\begin{scope}[shift={(2.5, 0)}]
\draw [ dashed] (-1, 1.25) .. controls +(0, 0.3) and +(0, 0.3).. (1, 1.25);
\draw [  dashed]  (1,1.25)  .. controls +(0, -0.3) and +(0, -0.3).. (-1, 1.25);
\begin{scope}[shift={(0, -2.5)}]
\draw [ dashed] (-1, 1.25) .. controls +(0, 0.3) and +(0, 0.3).. (1, 1.25);
\draw [dashed] (1,1.25)  .. controls +(0, -0.3) and +(0, -0.3).. (-1, 1.25);
\end{scope}
\draw  (-1, 1.25)--(-1, -1.25) (1, 1.25)--(1, -1.25);
\end{scope}
\draw (1, 1.25) .. controls +(0, 0.3) and +(0, 0.3) .. (1.5, 1.25);
\draw  (-1, 1.25) .. controls +(0, 1.4) and +(0, 1.4) .. (3.5, 1.25);
\draw (1, -1.25) .. controls +(0, -0.3) and +(0, -0.3) .. (1.5, -1.25);
\draw (-1, -1.25) .. controls +(0, -1.4) and +(0, -1.4) .. (3.5, -1.25);
\end{tikzpicture}}} }\right)\,.
\end{align*}

\begin{definition}[Tube category]
Let $\cA_1$ be the category $\cA_0$ modulo the negligible morphism with respect to the trace $\tr$, To be precise, the objects of $\cA_1$ and $\cA_0$ are the same and 
\[ \hom_{\cA_1}(O_1, O_2)=\hom_{\cA_0}(O_1, O_2)/\mathcal{I}(O_1, O_2)\]
where the negligible ideal
\begin{align*}
\mathcal{I}(O_1, O_2)=\{f\in \hom_{\cA_0}(O_1, O_2)\mid \tr(fg)=0, \forall g \in \hom_{\cA_0}(O_2, O_1)\}.
\end{align*}
The tube category $\mathcal{A}$ is defined to be the idempotent completion(Karoubi envelope) of $\cA_1$.
\end{definition}

\begin{remark}Notice $\1_{\emptyset}$ in $\cA_{0}$ descends to a tensor unit in $\cA$, and $\hom_{\cA}(\1_\emptyset, \1_\emptyset)\cong \k$. Thus, the trace $\tr$ descends to $\cA$ and coincides with the pivotal trace $\Tr_{\cA}$ of $\cA$. Consequently, the trace $\Tr_{\cA}$ is nondegenerate.

It can be seen that $\cA$ is a strict spherical monoidal category. Indeed, for any endomorphism $f \in \hom_{\cA_0}(O,O)$, $\Tr_{\cA_{0}}(f)$ and $\Tr_{\cA_{0}}(f^*)$ are isotopic $\mathcal{C}$-decorated 3-alterfolds, so by Theorem \ref{thm:partition function}, $\Tr_{\cA}(f) = \Tr_{\cA}(f^*)$.
\end{remark}

\begin{notation}
For morphisms in $\cA$, we will be using their representatives in $\cA_{0}$.
\end{notation}

\begin{theorem}\label{thm:tubemove}
Suppose $f, g$ are two morphisms in $\mathcal{A}$ that differ by one of the local moves described in Theorem \ref{thm:partition function}(Move $0, 1, 2, 3$, Homeomorphism relative to the boundary, Planar Graphical Calculus), then $f=g$.

\end{theorem}
\begin{proof}
It follows from Theorem \ref{thm:partition function}.
\end{proof}

\begin{proposition}\label{prop:objtensor}
Let $X$, $Y$ be objects in $\mathcal{C}$.
In the category $\cA$, we have
\[ O_X\otimes O_Y\cong \bigoplus_{i=0}^r O_{XX_i Y X_i^*}\,.\]
\end{proposition}
\begin{proof}
We define two morphisms $f_{X, Y}, f_{X, Y}'$ as follows:
\[f_{X,Y}=\sum_{i=0}^r 
\raisebox{1.5em}{\scalebox{0.75}{\begin{tikzpicture}[scale=0.35, baseline=.5ex]
\draw (-3, 0) [partial ellipse=0:360:2 and 0.8];
\draw (3, 0) [partial ellipse=0:360:2 and 0.8];
\draw[dashed] (0, -5)[partial ellipse=0:180:2 and 0.8];
\draw (0, -5)[partial ellipse=180:360:2 and 0.8];
\draw (5, 0)..controls (5,-2) and (2,-4).. (2, -5);
\draw (-5, 0)..controls (-5,-2) and (-2,-4).. (-2, -5);
\draw (-1, 0) ..controls (-.5, -2) and (.5, -2)..(1, 0);
\draw[blue, ->-=0.5] (3, -0.85) node [above, black] {\tiny{$Y$}} ..controls (3,-3) and (1,-5).. (1, -5.7);
\draw[blue, ->-=0.5] (-3, -0.85) node [above, black] {\tiny{$X$}}..controls (-3,-3) and (-1,-5).. (-1, -5.7);
\draw[blue, ->-=0.5] (0,-4.8)[partial ellipse=90:196:0.5 and 3.3];
\draw[blue, dashed] (0,-5)[partial ellipse=16:90:0.5 and 3.5] node [above, black]{\tiny $X_i$};
\end{tikzpicture}}}
\quad\text{and}\quad
f_{X,Y}'=\sum_{j=0}^r d_{j}
\raisebox{-1em}{\scalebox{0.75}{\begin{tikzpicture}[scale=0.35, baseline=.5ex]
\draw[dashed] (-3, 0) [partial ellipse=0:180:2 and 0.8];
\draw[dashed] (3, 0) [partial ellipse=0:180:2 and 0.8];
\draw (-3, 0) [partial ellipse=0:-180:2 and 0.8];
\draw (3, 0) [partial ellipse=0:-180:2 and 0.8];
\draw (0, 5)[partial ellipse=0:180:2 and 0.8];
\draw (0, 5)[partial ellipse=0:-180:2 and 0.8];
\draw (5, 0)..controls (5,2) and (2,4).. (2, 5);
\draw (-5, 0)..controls (-5,2) and (-2,4).. (-2, 5);
\draw (-1, 0)..controls (-.5, 2) and (.5, 2)..(1, 0);
\draw[blue, -<-=0.5] (3, -0.8) node [below, black] {\tiny{$Y$}} ..controls (3,1) and (1,3).. (1, 4.34);
\draw[blue,  -<-=0.5] (-3, -0.8) node [below, black] {\tiny{$X$}} ..controls (-3,1) and (-1,3).. (-1, 4.34);
\draw[blue, -<-=0.5] (0,5)[partial ellipse=270:192:0.5 and 3.5];
\draw[blue, dashed] (0,5)[partial ellipse=12:-90:0.5 and 3.5] node [black, below] {\tiny{$X_j$}};
\end{tikzpicture}}}.\]
We show they are inverse to each other by taking compositions.
\[f_{X, Y}f_{X, Y}'=
\sum_{{i, j}=0}^r d_{j}\ 
\scalebox{0.75}{\begin{tikzpicture}[scale=0.35,baseline=0em]
\draw[dashed] (0, -5)[partial ellipse=0:180:2 and 0.8];
\draw (0, -5)[partial ellipse=180:360:2 and 0.8];
\draw (5, 0)..controls (5,-2) and (2,-4).. (2, -5);
\draw (-5, 0)..controls (-5,-2) and (-2,-4).. (-2, -5);
\draw (-1, 0)..controls (-.5, -2) and (.5, -2)..(1, 0);
\draw[blue, ->-=0.5] (3, -0.8)..controls (3,-3) and (1,-5).. (1, -5.65) node[below, black]{\tiny{$Y$}};
\draw[blue, ->-=0.5] (-3, -0.8)..controls (-3,-3) and (-1,-5).. (-1, -5.65)node[below, black]{\tiny{$X$}};
\draw[blue, -<-=0.5] (0,-5)[partial ellipse=90:192:0.5 and 3.5];
\draw[blue, dashed] (0,-5)[partial ellipse=12:90:0.5 and 3.5];
\draw[dashed] (-3, 0) [partial ellipse=0:180:2 and 0.8];
\draw[dashed] (3, 0) [partial ellipse=0:180:2 and 0.8];
\draw (-3, 0) [partial ellipse=0:-180:2 and 0.8];
\draw (3, 0) [partial ellipse=0:-180:2 and 0.8];
\draw (0, 5)[partial ellipse=0:180:2 and 0.8];
\draw (0, 5)[partial ellipse=0:-180:2 and 0.8];
\draw (5, 0)..controls (5,2) and (2,4).. (2, 5);
\draw (-5, 0)..controls (-5,2) and (-2,4).. (-2, 5);
\draw (-1, 0)..controls (-.5, 2) and (.5, 2)..(1, 0);
\draw[blue] (3, -0.8)..controls (3,1) and (1,3).. (1, 4.34);
\draw[blue] (-3, -0.8)..controls (-3,1) and (-1,3).. (-1, 4.34);
\draw[blue, -<-=0.5] (0,5)[partial ellipse=270:192:0.5 and 3.5];
\draw[blue, dashed] (0,5)[partial ellipse=12:-90:0.5 and 3.5];
\draw (0, -0.8) node {\tiny{$X_i$}};
\draw (0, 0.8) node {\tiny{$X_j$}};
\end{tikzpicture}}
=\sum_{i=0}^r\ 
\scalebox{0.75}{\begin{tikzpicture}[yscale=0.35, xscale=0.5, baseline=4em]
\draw[dashed] (0, 0) [partial ellipse=0:180:2 and 0.8];
\draw (0, 0) [partial ellipse=0:-180:2 and 0.8];
\draw (0, 10) [partial ellipse=0:180:2 and 0.8];
\draw (0, 10) [partial ellipse=0:-180:2 and 0.8];
\draw (-2, 0)--(-2, 10);
\draw (2, 0)--(2, 10);
\draw[blue, -<-=0.5] (-1, -0.75)node[below, black]{\tiny{$X$}} --(-1, 9.35);
\draw[blue, -<-=0.5] (1, -0.75) node[below, black]{\tiny{$Y$}}--(1, 9.35);
\draw [blue, dashed, -<-=0.5] (0.5, 0.7)node[below, black]{\tiny{$X_i^*$}} --(0.5, 10.55);
\draw[blue, -<-=0.5] (0, -0.8) node[below, black]{\tiny{$X_i$}}--(0, 9.2);
\end{tikzpicture}}=\sum_{i=0}^r\id_{O_{XX_iYX_i^*}}
,\]
where the second equality comes from applying Move 2.

\begin{align*}
f_{X, Y}'f_{X, Y}=&
\sum_{{j}=0}^r d_{j}
\scalebox{0.7}{
\begin{tikzpicture}[scale=0.35, baseline=0]
\begin{scope}[yshift=-5cm]
\draw (0, 0.8) node{\tiny{$X_j$}};
\draw[dashed] (-3, 0) [partial ellipse=0:180:2 and 0.8];
\draw[dashed] (3, 0) [partial ellipse=0:180:2 and 0.8];
\draw (-3, 0) [partial ellipse=0:-180:2 and 0.8];
\draw (3, 0) [partial ellipse=0:-180:2 and 0.8];
\draw (5, 0)..controls (5,2) and (2,4).. (2, 5);
\draw (-5, 0)..controls (-5,2) and (-2,4).. (-2, 5);
\draw (-1, 0)..controls (-.5, 2) and (.5, 2)..(1, 0);
\draw[blue] (3, -0.8)node[below, black]{\tiny{$Y$}} ..controls (3,1) and (1,3).. (1, 4.34);
\draw[blue] (-3, -0.8) node[below, black]{\tiny{$X$}} ..controls (-3,1) and (-1,3).. (-1, 4.34);
\draw[blue, -<-=0.5] (0,5)[partial ellipse=270:192:0.5 and 3.5];
\draw[blue, dashed] (0,5)[partial ellipse=12:-90:0.5 and 3.5];
\end{scope}
\begin{scope}[yshift=5cm]
\draw (-3, 0) [partial ellipse=0:360:2 and 0.8];
\draw (3, 0) [partial ellipse=0:360:2 and 0.8];
\draw[dashed] (0, -5)[partial ellipse=0:180:2 and 0.8];
\draw (0, -5)[partial ellipse=180:360:2 and 0.8];
\draw (5, 0)..controls (5,-2) and (2,-4).. (2, -5);
\draw (-5, 0)..controls (-5,-2) and (-2,-4).. (-2, -5);
\draw (-1, 0)..controls (-.5, -2) and (.5, -2)..(1, 0);
\draw[blue, ->-=0.4] (3, -0.8)..controls (3,-3) and (1,-5).. (1, -5.65) ;
\draw[blue, ->-=0.4] (-3, -0.8)..controls (-3,-3) and (-1,-5).. (-1, -5.65);
\draw[blue, ->-=0.4] (0,-5)[partial ellipse=90:192:0.5 and 3.5];
\draw[blue, dashed] (0,-5)[partial ellipse=12:90:0.5 and 3.5];
\end{scope}
\end{tikzpicture}}
=
\scalebox{0.7}{
\begin{tikzpicture}[scale=0.35,baseline=0em]
\begin{scope}[yshift=-5cm]
\draw (0, 0.8) node{\tiny{$X_j$}};
\draw[dashed] (-3, 0) [partial ellipse=0:180:2 and 0.8];
\draw[dashed] (3, 0) [partial ellipse=0:180:2 and 0.8];
\draw (-3, 0) [partial ellipse=0:-180:2 and 0.8];
\draw (3, 0) [partial ellipse=0:-180:2 and 0.8];
\draw (5, 0)..controls (5,2) and (2,4).. (2, 5);
\draw (-5, 0)..controls (-5,2) and (-2,4).. (-2, 5);
\draw (-1, 0)..controls (-.5, 2) and (.5, 2)..(1, 0);
\draw[blue] (3, -0.8)node[below, black]{\tiny{$Y$}} ..controls (3,1) and (1,3).. (1, 4.34);
\draw[blue] (-3, -0.8) node[below, black]{\tiny{$X$}} ..controls (-3,1) and (-1,3).. (-1, 4.34);
\draw[red] (0,5)[partial ellipse=270:192:0.5 and 3.5];
\draw[red, dashed] (0,5)[partial ellipse=12:-90:0.5 and 3.5];
\end{scope}
\begin{scope}[yshift=5cm]
\draw(-3, 0) [partial ellipse=0:360:2 and 0.8];
\draw (3, 0) [partial ellipse=0:360:2 and 0.8];
\draw[dashed] (0, -5)[partial ellipse=0:180:2 and 0.8];
\draw (0, -5)[partial ellipse=180:360:2 and 0.8];
\draw (5, 0)..controls (5,-2) and (2,-4).. (2, -5);
\draw (-5, 0)..controls (-5,-2) and (-2,-4).. (-2, -5);
\draw (-1, 0)..controls (-.5, -2) and (.5, -2)..(1, 0);
\draw[blue, ->-=0.4] (3, -0.8)..controls (3,-3) and (1,-5).. (1, -5.65) ;
\draw[blue, ->-=0.4] (-3, -0.8)..controls (-3,-3) and (-1,-5).. (-1, -5.65);
\draw[red] (0,-5)[partial ellipse=90:192:0.5 and 3.5];
\draw[red, dashed] (0,-5)[partial ellipse=12:90:0.5 and 3.5];
\end{scope}
\end{tikzpicture}}
=
\scalebox{0.7}{\begin{tikzpicture}[scale=0.35,baseline=0em]
\draw (0,5) [partial ellipse=0:360:2 and 0.8];
\draw (-2, 5)--(-2, -4);
\draw (2, 5)--(2, -4);
\draw[dashed] (0,-4) [partial ellipse=0:180:2 and 0.8];
\draw (0,-4) [partial ellipse=180:360:2 and 0.8];
\draw[blue, ->-=0.5] (0, 4.2)--(0, -4.8) node [below, black] {\tiny $X$};
\begin{scope}[shift={(5, 0)}]
\draw (0,5) [partial ellipse=0:360:2 and 0.8];
\draw (-2, 5)--(-2, -4);
\draw (2, 5)--(2, -4);
\draw[dashed] (0,-4) [partial ellipse=0:180:2 and 0.8];
\draw (0,-4) [partial ellipse=180:360:2 and 0.8];
\draw[blue, ->-=0.5] (0, 4.2)--(0, -4.8) node [below, black] {\tiny $Y$};
\end{scope}
\end{tikzpicture}}
=\id_{O_X} \otimes \id_{O_Y},
\end{align*}
where the third equality comes from applying Move 1. This completes the proof of the proposition.
\end{proof}

\begin{corollary}\label{cor:singledisk}
Any object $O$ in $\cA_1$ is isomorphic to a finite direct sum $\displaystyle \bigoplus_{X}$\CIRCLE$_{X}$, i.e. a finite direct sum of $\mathcal{C}$-disks with a single object in $\mathcal{C}$.
\end{corollary}
\begin{proof}
Suppose $O\in \cA_{1}$ consists of a single $\mathcal{C}$-disk, with marked points colored by $Y_1, \ldots, Y_k\in \cC$. By the semisimplicity of $\cC$, there exist simple objects $\{U_i\}$ (possibly with multiplicity) such that $Y_1 \otimes \cdots \otimes Y_k \cong \bigoplus U_i$, i.e., there are morphisms $p_i:Y_1\otimes\cdots \otimes Y_k\rightarrow U_{i}$ and $q_i:U_i\rightarrow Y_1\otimes\cdots\otimes Y_k$ such that $p_iq_j=\delta_{ij}\id_{U_i}$ and
$\sum_{i}{q_ip_i}=\id_{Y_{1}\otimes\cdots\otimes Y_{k}}$. Suppose $\Gamma_p$ and $\Gamma_q$ are the planar $\cC$-diagram representing $\sum p_i$ and $\sum q_i$ respectively. We can embed them to cylinders and get morphisms $H_p$ and $H_q$ in $\cA_1$ respectively, such that $H_{q} H_{p}=\id_{O}$ and $H_{p}H_{q}=\bigoplus \id_{O_{U_i}}$. In other words, $O \cong \bigoplus O_{U_{i}} \cong O_{Y_{k} \otimes \cdots \otimes Y_{k}}$.

The general case follows from doing induction on the number of $\cC$-disks and apply proposition $\ref{prop:objtensor}$.
\end{proof}

\begin{remark}\label{rem:homspace}
By Corollary \ref{cor:singledisk}, we have that any morphism space in $\cA_1$ is linearly isomorphic to the direct sum of $\hom_{\cA}(\CIRCLE_X, \CIRCLE_Y)$ for $X, Y\in\mathcal{C}$.
\end{remark}

\begin{theorem}\label{lem:morphismform}
Any morphism $H$ in $\hom_{\cA}(\CIRCLE_X, \CIRCLE_Y)$ is of the following form (up to isomorphism):
$$
\sum_{i=0}^r d_i\vcenter{\hbox{\scalebox{0.7}{
\begin{tikzpicture}[scale=0.35]
\draw (0,5) [partial ellipse=0:360:2 and 0.8];
\draw (-2, 5)--(-2, -4);
\draw (2, 5)--(2, -4);
\draw[dashed] (0,-4) [partial ellipse=0:180:2 and 0.8];
\draw (0,-4) [partial ellipse=180:360:2 and 0.8];
\draw[blue, ->-=0.5] (0, 4.2) node[above, black]{\tiny{${X}$}}->(0, 0.5);
\draw[blue, ->-=0.5] (0, -0.5)--(0, -4.8)node[below, black]{\tiny{${Y}$}};
\draw [blue , ->-=0.5](0, .95) [partial ellipse=-75:0:2 and 0.8] node [pos=0.5, below, black] {\tiny $X_i$};
\draw [blue, dashed](0, .95) [partial ellipse=0:180:2 and 0.8];
\draw [blue](0, .95) [partial ellipse=180:255:2 and 0.8];
\node [draw, fill=white] (0, 0){\tiny $f_i$};
\end{tikzpicture}}}}, \text{ simply represented by }
\vcenter{\hbox{\scalebox{0.7}{
\begin{tikzpicture}[scale=0.35]
\draw (0,5) [partial ellipse=0:360:2 and 0.8];
\draw (-2, 5)--(-2, -4);
\draw (2, 5)--(2, -4);
\draw[dashed] (0,-4) [partial ellipse=0:180:2 and 0.8];
\draw (0,-4) [partial ellipse=180:360:2 and 0.8];
\draw[blue, ->-=0.5] (0, 4.2) node[above, black]{\tiny{${X}$}}->(0, 0.5);
\draw[blue, ->-=0.5] (0, -0.5)--(0, -4.8)node[below, black]{\tiny{${Y}$}};
\draw [red](0, .95) [partial ellipse=-75:0:2 and 0.8];
\draw [red, dashed](0, .95) [partial ellipse=0:180:2 and 0.8];
\draw [red](0, .95) [partial ellipse=180:255:2 and 0.8];
\node [draw, fill=white] (0, 0){\tiny $f$};
\end{tikzpicture}}}}\,,
$$
where $\displaystyle f=\sum_{i=}^r f_i$ and $f_i\in \hom_{\mathcal{C}}(XX_i, X_iY)$.
\end{theorem}
\begin{proof}
We first claim, by finitely many steps of local moves, one can make an arbitrary cornered $\mathcal{C}$-handlebody $H$ in $\hom_{\cA_1}(O_X, O_Y)$ homeomorphic to a cylinder. It can be seens as follows. Without loss of generality, we may assume $H$ is connected, otherwise, one can apply Move $1$ to connect to components.

Taking a spine of the handlebody $H$ (i.e., an embedded graph that $H$ retracts to) results in a knotted graph $G\subset \mathbb{R}\times [0, 1]$ which admit a planar presentation. To each crossing, we apply Move $1$ to result in an unknotted cornered handlebody. Then, we apply Move $2$ to reduce the genus of the handlebody to $0$.
Therefore, we can assume $H$ in the theorem is a cylinder.

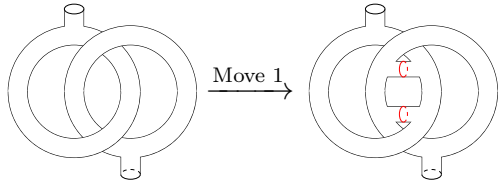
\[
\vcenter{\hbox{\scalebox{0.5}{
\begin{tikzpicture}[scale=1]
\draw [line width=0.53cm] (0, 1.5)--(0, 2.2);
\draw [line width=0.53cm] (0,0) [partial ellipse=-0.1:180.1:1.5 and 1.5];
\draw [white, line width=0.5cm] (0,0) [partial ellipse=-0.1:180.1:1.5 and 1.5];
\begin{scope}[shift={(1.5, 0)}]
\draw [line width=0.53cm] (0,0) [partial ellipse=-0.1:180.1:1.5 and 1.5];
\draw [white, line width=0.5cm] (0,0) [partial ellipse=-0.1:180.1:1.5 and 1.5];
\draw [line width=0.53cm] (0, -1.5)--(0, -2.2);
\end{scope} % upper
\begin{scope}[shift={(1.5, 0)}]
\draw [line width=0.53cm] (0,0) [partial ellipse=180:360:1.5 and 1.5];
\draw [white, line width=0.5cm] (0,0) [partial ellipse=178:362:1.5 and 1.5];
\draw [white, line width=0.5cm] (0, -1.48)--(0, -2.22);
\draw (0, -2.2) [partial ellipse=180:360:0.26 and 0.13];
\draw[dashed] (0, -2.2) [partial ellipse=0:180:0.26 and 0.13];
\end{scope}
\draw [line width=0.53cm] (0,0) [partial ellipse=180:360:1.5 and 1.5];
\draw [white, line width=0.5cm] (0,0) [partial ellipse=178:362:1.5 and 1.5];
\draw [white, line width=0.5cm] (0, 1.48)--(0, 2.22);
\draw (0, 2.2) [partial ellipse=0:360:0.26 and 0.13];
\end{tikzpicture}}}}
\xrightarrow{\text{Move 1}}
\vcenter{\hbox{\scalebox{0.5}{
\begin{tikzpicture}[scale=1]
\draw [line width=0.43cm] (0.3, 0.6)--(1.2, 0.6);
\draw [line width=0.43cm] (0.3, -0.6)--(1.2, -0.6);
\draw [line width=0.53cm] (0, 1.5)--(0, 2.2);
\draw [line width=0.53cm] (0,0) [partial ellipse=-0.1:180.1:1.5 and 1.5];
\draw [white, line width=0.5cm] (0,0) [partial ellipse=-0.1:180.1:1.5 and 1.5];
\begin{scope}[shift={(1.5, 0)}]
\draw [line width=0.53cm] (0,0) [partial ellipse=-0.1:180.1:1.5 and 1.5];
\draw [white, line width=0.5cm] (0,0) [partial ellipse=-0.1:180.1:1.5 and 1.5];
\draw [line width=0.53cm] (0, -1.5)--(0, -2.2);
\end{scope} % upper
\begin{scope}[shift={(1.5, 0)}]
\draw [line width=0.53cm] (0,0) [partial ellipse=180:360:1.5 and 1.5];
\draw [white, line width=0.5cm] (0,0) [partial ellipse=178:362:1.5 and 1.5];
\draw [white, line width=0.5cm] (0, -1.48)--(0, -2.22);
\draw (0, -2.2) [partial ellipse=180:360:0.26 and 0.13];
\draw[dashed] (0, -2.2) [partial ellipse=0:180:0.26 and 0.13];
\end{scope}
\draw [line width=0.53cm] (0,0) [partial ellipse=180:360:1.5 and 1.5];
\draw [white, line width=0.5cm] (0,0) [partial ellipse=178:362:1.5 and 1.5];
\draw [white, line width=0.5cm] (0, 1.48)--(0, 2.22);
\draw (0, 2.2) [partial ellipse=0:360:0.26 and 0.13];
\draw [line width=0.4cm, white] (0.3, 0.6)--(1.2, 0.6);
\draw [line width=0.4cm, white] (0.3, -0.6)--(1.2, -0.6);
\draw[red] (0.75, 0.6) [partial ellipse=90:270:0.107 and 0.21];
\draw[red, dashed] (0.75, 0.6) [partial ellipse=-90:90:0.107 and 0.21];
\draw[red] (0.75, -0.6) [partial ellipse=90:270:0.107 and 0.21];
\draw[red, dashed] (0.75, -0.6) [partial ellipse=-90:90:0.107 and 0.21];
\end{tikzpicture}}}}
\]
Suppose $\Gamma$ is a tensor diagram on a cylinder $H$.
We take a vertical line $\gamma$ on the surface of the cylinder such that $\gamma$ is disjoint from the coupons of the tensor diagram $\Gamma$ and all strings intersect  $\gamma$ transversely.
Then by Lemma \ref{lem:unit1}, we obtained a tensor diagram $\Gamma'$ such that there is a single string colored by $X_i$ intersecting $\gamma$ transversely. This complete the proof.
\end{proof}

For any $X$, $Y \in \cC$, let 
$$\tube(X,Y) := \bigoplus_{i=0}^r \hom_{\mathcal{C}}(XX_i, X_iY)\,.$$
Define a map $\cyl: \tube(X,Y) \to \hom_{\cA}(\CIRCLE_X, \CIRCLE_Y)$ by
\begin{align}\label{eq:cyl}
\cyl(f) :=
\vcenter{\hbox{\scalebox{0.6}{
\begin{tikzpicture}[scale=0.35]
\draw (0,5) [partial ellipse=0:360:2 and 0.8];
\draw (-2, 5)--(-2, -4);
\draw (2, 5)--(2, -4);
\draw[dashed] (0,-4) [partial ellipse=0:180:2 and 0.8];
\draw (0,-4) [partial ellipse=180:360:2 and 0.8];
\draw[blue, ->-=0.5] (0, 4.2) node[above, black]{\tiny{${X}$}}->(0, 0.5);
\draw[blue, ->-=0.5] (0, -0.5)--(0, -4.8)node[below, black]{\tiny{${Y}$}};
\draw [red](0, .95) [partial ellipse=-75:0:2 and 0.8];
\draw [red, dashed](0, .95) [partial ellipse=0:180:2 and 0.8];
\draw [red](0, .95) [partial ellipse=180:255:2 and 0.8];
\node [draw, fill=white] (0, 0){\tiny $f$};
\end{tikzpicture}}}}\,,
\quad\forall f \in \tube(X,Y)\,.
\end{align}
Clearly, $\cyl$ is linear. In particular, if $g\in \hom_{\mathcal{C}}(X, Y)=\hom_{\cC}(X\1_{\cC}, \1_{\cC}Y)$, then
$\cyl(g)=\vcenter{\hbox{\scalebox{0.6}
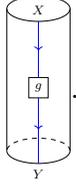
{
\begin{tikzpicture}[scale=0.35]
\draw (0,5) [partial ellipse=0:360:2 and 0.8];
\draw (-2, 5)--(-2, -4);
\draw (2, 5)--(2, -4);
\draw[dashed] (0,-4) [partial ellipse=0:180:2 and 0.8];
\draw (0,-4) [partial ellipse=180:360:2 and 0.8];
\draw[blue, ->-=0.5] (0, 4.2) node[above, black]{\tiny{${X}$}}->(0, 0.5);
\draw[blue, ->-=0.5] (0, -0.5)--(0, -4.8)node[below, black]{\tiny{${Y}$}};
\node [draw, fill=white] (0, 0){\tiny $g$};
\end{tikzpicture}}}}$. Therefore, if $\displaystyle f=\sum_{i=0}^r f_i \in \tube(X,Y)$ with $f_i\in \hom_{\mathcal{C}}(XX_i, X_iX)$, then by applying Move 2, we have $\Tr_{\cA}(\cyl(f))=\Tr_{\mathcal{C}}(f_0)$.

Since for any $X, Y \in \cC$, the restriction of $\cyl$ on $\hom_{\mathcal{C}}(X, Y)$ preserves the composition of morphisms, and $\cyl(\id_X) = \id_{O_{X}}$, it gives rise to a functor from $\mathcal{C}$ to $\cA$ sending $X \in \cC$ to $O_X \in \cA$.

\begin{remark}
Let $\Theta_{X}=\displaystyle \bigoplus_{i=0}^r\hom_{\mathcal{C}}(XX_i, XX_i)$.
We define a trace $\Tr_{\Theta}$ on $\Theta_X$ by
\begin{align*}
    \Tr_{\Theta}(f) := \sum_{i=0}^r d_i\Tr_{\mathcal{C}}(f_i),
\end{align*}
where $\displaystyle f=\sum_{i=0}^r f_i$ and $f_i\in \hom_{\mathcal{C}}(XX_i, XX_i)$.
\end{remark}

\begin{lemma}\label{lem:nondegenerate}
The map $\cyl$ is a linear isomorphism between $\tube(X,Y)$ and $\hom_{\cA}(\CIRCLE_X, \CIRCLE_Y)$ for any $X, Y\in \mathcal{C}$.
\end{lemma}
\begin{proof}
By Theorem \ref{lem:morphismform}, $\cyl$ is surjective, so it suffices to show that $\cyl$ is injective. Suppose $f\in \tube(X,Y)$ and $g\in \tube(Y,X)$.
By the Move 2, we have
\begin{align*}
    \Tr_{\cA}(\cyl(f) \cyl(g))=\Tr_{\Theta}(f \rho(g)),
\end{align*}
where 
\begin{align*}
\rho(g)=\raisebox{-1cm}{
\begin{tikzpicture}
\draw [blue, -<-=0.1,, -<-=0.9] (0.2, -1)--(0.2, 0.4).. controls +(0, 0.5) and +(0, 0.5).. (0.8,0.4)--(0.8, -1);
\draw [blue, ->-=0.1, , ->-=0.9] (-0.2, 1)--(-0.2, -0.4).. controls +(0, -0.5) and +(0, -0.5).. (-0.8,-0.4)--(-0.8, 1);
\draw [fill=white] (-0.5, -0.4) rectangle (0.5, 0.4);
\node at (0, 0) {$g$};
\end{tikzpicture}}.
\end{align*}
Since $\Tr_{\cC}$ is nondegenerate, and $\rho$ is an isomorphism, so for all $f\in \tube(X,Y)$, there exists $g\in \tube(Y,X)$ such that $\Tr_{\cA}(\cyl(f) \cyl(g))=\Tr_{\Theta}(f \rho(g))\ne 0$. Thus, $\cyl$ is injective, and we are done.
\end{proof}

\begin{proposition}\label{prop:localfinite}
The category $\cA$ is locally finite, i.e. the morphism spaces are finite-dimensional.
\end{proposition}

\begin{proof}
Note that any object in $\cA$ is a subobject in $\cA_1$.
By Corollary \ref{cor:singledisk}, it suffices to consider the morphism space $\hom_{\cA}(\CIRCLE_X, \CIRCLE_Y)$ for some $X, Y\in\mathcal{C}$.
By Lemma \ref{lem:nondegenerate}, we have that $\dim \hom_{\cA}(\CIRCLE_X, \CIRCLE_Y)= \dim T(X,Y)$. 
Note that $ \dim \hom_{\mathcal{C}}(X X_i, X_i Y)$ is finite. 
Thus $\cA$ is locally finite.
\end{proof}

\begin{remark}
Note that the tube category $\mathcal{A}$ has a canonical braid structure, which is defined to be the braiding of tubes. For instance, the braid map $c_{O_{X}, O_{Y}}$ is depicted as
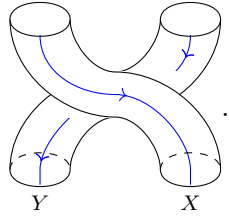
\begin{align*}
\vcenter{\hbox{
\begin{tikzpicture}
\draw  (1.6, 0) .. controls +(0, 0.3) and +(0, 0.3) .. (2.4, 0);
\draw (1.6, 0) .. controls +(0,-0.3) and +(0, -0.3) .. (2.4, 0);
\draw  (1.6, 0) .. controls +(0, -0.3) and +(0.3, 0) ..(1, -0.7) .. controls +(-0.7, 0) and +(0,0.7) ..(-0.4, -2);
\draw  (2.4, 0) .. controls +(0, -0.7) and +(0.7, 0) ..(1, -1.3) .. controls +(-0.3, 0) and +(0,0.3) ..(0.4, -2);
\draw [blue,->-=0.1, ->-=0.9] (2, -0.21) .. controls +(0, -0.5) and +(0.5, 0) ..(1, -1) .. controls +(-0.4, 0) and +(0,0.4) ..(0, -2)--(0, -2.2) node [below, black] {\tiny $Y$};
\draw [ dashed] (-0.4, -2) .. controls +(0, 0.3) and +(0, 0.3) .. (0.4, -2);
\draw (-0.4, -2) .. controls +(0,-0.3) and +(0, -0.3) .. (0.4, -2); % right tube
\path [fill=white]  (0.15,-0.69) rectangle (1.85, -1.31);
\draw (-0.4, 0) .. controls +(0, 0.3) and +(0, 0.3) .. (0.4, 0);
\draw  (-0.4, 0) .. controls +(0,-0.3) and +(0, -0.3) .. (0.4, 0);
\draw  (0.4, 0) .. controls +(0, -0.3) and +(-0.3, 0) ..(1, -0.7) .. controls +(0.7, 0) and +(0,0.7) ..(2.4, -2);
\draw  (-0.4, 0) .. controls +(0, -0.7) and +(-0.7, 0) ..(1, -1.3) .. controls +(0.3, 0) and +(0,0.3) ..(1.6, -2);
\draw [blue,->-=0.5] (0, -0.21) .. controls +(0, -0.5) and +(-0.5, 0) ..(1, -1) .. controls +(0.5, 0) and +(0,0.4) ..(2, -2)--(2, -2.2) node [below, black] {\tiny $X$};
\draw [dashed] (1.6, -2) .. controls +(0, 0.3) and +(0, 0.3) .. (2.4, -2);
\draw  (1.6, -2) .. controls +(0,-0.3) and +(0, -0.3) .. (2.4, -2); % left tube
\end{tikzpicture}}}.
\end{align*}
Therefore, $\cA$ can be realized as a braided monoidal category equipped with this canonical braid structure.
\end{remark}

Later we will see $\cA$ is a braided fusion category, and the next lemma actually implies the modularity of $\cA$.

\begin{lemma}\label{lem:zkirby}
In the tube category $\cA$, we have that 
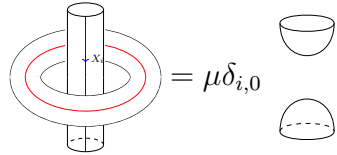
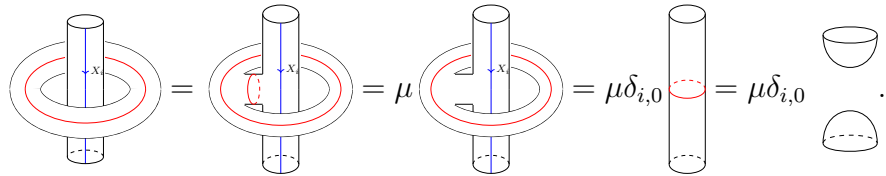
\begin{align*}
\vcenter{\hbox{\scalebox{0.5}{
\begin{tikzpicture}[xscale=0.8, yscale=0.6]
\draw [line width=0.83cm] (0,0) [partial ellipse=0:180.1:2 and 1.5];
\draw [white, line width=0.8cm] (0,0) [partial ellipse=0:180.1:2 and 1.5];
\draw [red] (0,0) [partial ellipse=0:180:2 and 1.5];
\path [fill=white](-0.65, 0) rectangle (0.65, 2.5);
\begin{scope}[shift={(0,3)}]
\draw (0,0) [partial ellipse=0:360:0.6 and 0.3];
\end{scope}
\draw (-0.6, 3)--(-0.6, 0) (0.6, 3)--(0.6, 0); 
\draw [blue, -<-=0.3] (0, 0)--(0, 2.7) node[pos=0.3, right, black]{\tiny$X_{i}$};% upper one
\draw [blue] (0, 0)--(0, -3.3);
\draw (-0.6, -3)--(-0.6, 0) (0.6, -3)--(0.6, 0);
\draw [line width=0.83cm] (0,0) [partial ellipse=180:360:2 and 1.5];
\draw [white, line width=0.8cm] (0,0) [partial ellipse=178:362:2 and 1.5];
\draw [red] (0,0) [partial ellipse=178:362:2 and 1.5];
\begin{scope}[shift={(0,-3)}]
\draw [dashed](0,0) [partial ellipse=0:180:0.6 and 0.3];
\draw (0,0) [partial ellipse=180:360:0.6 and 0.3];
\end{scope}
\end{tikzpicture}}}}
=\mu\delta_{i, 0}\vcenter{\hbox{ \scalebox{0.6}{
\begin{tikzpicture}[xscale=0.6,yscale=0.8]
\begin{scope}[shift={(0, -1.5)}]
\draw [ dashed] (-1, 1.25) .. controls +(0, 0.3) and +(0, 0.3).. (1, 1.25);
\draw    (1,1.25)  .. controls +(0, -0.3) and +(0, -0.3).. (-1, 1.25);
\draw (-1, 1.25) .. controls +(0, 1.2) and +(0, 1.2).. (1, 1.25);
\end{scope}
\begin{scope}[shift={(0, 1.5)}]
\draw (-1, 1.25) .. controls +(0, 0.3) and +(0, 0.3).. (1, 1.25);
\draw    (1,1.25)  .. controls +(0, -0.3) and +(0, -0.3).. (-1, 1.25);
\draw (-1, 1.25) .. controls +(0, -1.2) and +(0, -1.2).. (1, 1.25);
\end{scope}
\end{tikzpicture}}}}
\end{align*}
\end{lemma}
\begin{proof}
We compute,
\begin{align*}
\vcenter{\hbox{\scalebox{0.5}{
\begin{tikzpicture}[xscale=0.8, yscale=0.6]
\draw [line width=0.83cm] (0,0) [partial ellipse=0:180.1:2 and 1.5];
\draw [white, line width=0.8cm] (0,0) [partial ellipse=0:180.1:2 and 1.5];
\draw [red] (0,0) [partial ellipse=0:180:2 and 1.5];
\path [fill=white](-0.65, 0) rectangle (0.65, 2.5);
\begin{scope}[shift={(0,3)}]
\draw (0,0) [partial ellipse=0:360:0.6 and 0.3];
\end{scope}
\draw (-0.6, 3)--(-0.6, 0) (0.6, 3)--(0.6, 0); 
\draw [blue, -<-=0.3] (0, 0)--(0, 2.7) node[pos=0.3, right, black]{\tiny$X_{i}$};% upper one
\draw [blue] (0, 0)--(0, -3.3);
\draw (-0.6, -3)--(-0.6, 0) (0.6, -3)--(0.6, 0);
\draw [line width=0.83cm] (0,0) [partial ellipse=180:360:2 and 1.5];
\draw [white, line width=0.8cm] (0,0) [partial ellipse=178:362:2 and 1.5];
\draw [red] (0,0) [partial ellipse=178:362:2 and 1.5];
\begin{scope}[shift={(0,-3)}]
\draw [dashed](0,0) [partial ellipse=0:180:0.6 and 0.3];
\draw (0,0) [partial ellipse=180:360:0.6 and 0.3];
\end{scope}
\end{tikzpicture}}}}
=
\vcenter{\hbox{\scalebox{0.5}{
\begin{tikzpicture}[scale=0.8]
\draw [line width=0.63cm] (0,0) [partial ellipse=0:180.1:2 and 1.2];
\draw [white, line width=0.6cm] (0,0) [partial ellipse=0:180.1:2 and 1.2];
\draw [red] (0,0) [partial ellipse=0:180:2 and 1.2];
\path [fill=white](-0.65, 0) rectangle (0.65, 2.5);
\begin{scope}[shift={(0,2.5)}]
\draw (0,0) [partial ellipse=0:360:0.6 and 0.3];
\end{scope}
\draw (-0.6, 2.5)--(-0.6, 0) (0.6, 2.5)--(0.6, 0); 
\draw [blue, -<-=0.3] (0, 0)--(0, 2.2) node[pos=0.3, right, black]{\tiny$X_{i}$};% upper one
\draw (-0.6, -2.5)--(-0.6, 0) (0.6, -2.5)--(0.6, 0);
\draw [blue] (0, 0)--(0, -2.8);
\draw [line width=0.63cm] (0,0) [partial ellipse=180:360:2 and 1.2];
\draw [white, line width=0.6cm] (0,0) [partial ellipse=178:362:2 and 1.2];
\draw [red] (0,0) [partial ellipse=178:362:2 and 1.2];
\begin{scope}[shift={(0,-2.5)}]
\draw [dashed](0,0) [partial ellipse=0:180:0.6 and 0.3];
\draw (0,0) [partial ellipse=180:360:0.6 and 0.3];
\end{scope}
\path [fill=white](-1.7, -0.5) rectangle (-0.5, 0.5);
\draw (-1.2, 0.5)--(-0.6, 0.5) (-1.2, -0.5)--(-0.6, -0.5);
\draw [red] (-0.9, 0) [partial ellipse=90:270:0.2 and 0.5] ;
\draw [red, dashed] (-0.9, 0) [partial ellipse=-90:90:0.2 and 0.5] ;
\end{tikzpicture}}}}
=\mu
\vcenter{\hbox{\scalebox{0.5}{
\begin{tikzpicture}[scale=0.8]
\draw [line width=0.63cm] (0,0) [partial ellipse=0:180.1:2 and 1.2];
\draw [white, line width=0.6cm] (0,0) [partial ellipse=0:180.1:2 and 1.2];
\draw [red] (0,0) [partial ellipse=0:180:2 and 1.2];
\path [fill=white](-0.65, 0) rectangle (0.65, 2.5);
\begin{scope}[shift={(0,2.5)}]
\draw (0,0) [partial ellipse=0:360:0.6 and 0.3];
\end{scope}
\draw (-0.6, 2.5)--(-0.6, 0) (0.6, 2.5)--(0.6, 0); 
\draw [blue, -<-=0.3] (0, 0)--(0, 2.2) node[pos=0.3, right, black]{\tiny$X_{i}$};% upper one
\draw (-0.6, -2.5)--(-0.6, 0) (0.6, -2.5)--(0.6, 0);
\draw [blue] (0, 0)--(0, -2.8);
\draw [line width=0.63cm] (0,0) [partial ellipse=180:360:2 and 1.2];
\draw [white, line width=0.6cm] (0,0) [partial ellipse=178:362:2 and 1.2];
\draw [red] (0,0) [partial ellipse=178:362:2 and 1.2];
\begin{scope}[shift={(0,-2.5)}]
\draw [dashed](0,0) [partial ellipse=0:180:0.6 and 0.3];
\draw (0,0) [partial ellipse=180:360:0.6 and 0.3];
\end{scope}
\path [fill=white](-1.7, -0.5) rectangle (-0.5, 0.5);
\draw (-1.2, 0.5)--(-0.6, 0.5) (-1.2, -0.5)--(-0.6, -0.5);
\end{tikzpicture}}}}
=\mu\delta_{i, 0}
\vcenter{\hbox{\scalebox{0.5}{
\begin{tikzpicture}[scale=0.8]
\begin{scope}[shift={(0,2.5)}]
\draw (0,0) [partial ellipse=0:360:0.6 and 0.3];
\end{scope}
\draw (-0.6, 2.5)--(-0.6, 0) (0.6, 2.5)--(0.6, 0); % upper one
\draw (-0.6, -2.5)--(-0.6, 0) (0.6, -2.5)--(0.6, 0);
\begin{scope}[shift={(0,0)}]
\draw [dashed, red](0,0) [partial ellipse=0:180:0.6 and 0.3];
\draw [red] (0,0) [partial ellipse=180:360:0.6 and 0.3];
\end{scope}
\begin{scope}[shift={(0,-2.5)}]
\draw [dashed](0,0) [partial ellipse=0:180:0.6 and 0.3];
\draw (0,0) [partial ellipse=180:360:0.6 and 0.3];
\end{scope}
\end{tikzpicture}}}}
=\mu\delta_{i, 0}\vcenter{\hbox{ \scalebox{0.6}{
\begin{tikzpicture}[xscale=0.6,yscale=0.8]
\begin{scope}[shift={(0, -1.5)}]
\draw [ dashed] (-1, 1.25) .. controls +(0, 0.3) and +(0, 0.3).. (1, 1.25);
\draw    (1,1.25)  .. controls +(0, -0.3) and +(0, -0.3).. (-1, 1.25);
\draw (-1, 1.25) .. controls +(0, 1.2) and +(0, 1.2).. (1, 1.25);
\end{scope}
\begin{scope}[shift={(0, 1.5)}]
\draw (-1, 1.25) .. controls +(0, 0.3) and +(0, 0.3).. (1, 1.25);
\draw    (1,1.25)  .. controls +(0, -0.3) and +(0, -0.3).. (-1, 1.25);
\draw (-1, 1.25) .. controls +(0, -1.2) and +(0, -1.2).. (1, 1.25);
\end{scope}
\end{tikzpicture}}}}.
\end{align*}
The first equality comes from applying Move $1$, the second equality comes from handle slides,  the third comes from applying Move $2$, and the last equality comes from applying Move $1$ in the opposite direction.
This completes the proof.
\end{proof}

\begin{remark}
By Lemma \ref{lem:zkirby}, for any object $X\in \mathcal{C}$, we have that
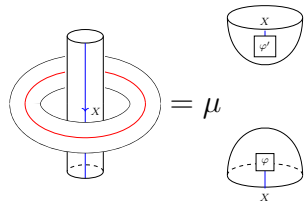
\begin{align*}
\vcenter{\hbox{\scalebox{0.5}{
\begin{tikzpicture}[xscale=0.8, yscale=0.6]
\draw [line width=0.83cm] (0,0) [partial ellipse=0:180.1:2 and 1.5];
\draw [white, line width=0.8cm] (0,0) [partial ellipse=0:180.1:2 and 1.5];
\draw [red] (0,0) [partial ellipse=0:180:2 and 1.5];
\path [fill=white](-0.65, 0) rectangle (0.65, 2.5);
\begin{scope}[shift={(0,3)}]
\draw (0,0) [partial ellipse=0:360:0.6 and 0.3];
\end{scope}
\draw (-0.6, 3)--(-0.6, 0) (0.6, 3)--(0.6, 0); 
\draw [blue] (0, 0)--(0, 2.7);% upper one
\draw (-0.6, -3)--(-0.6, 0) (0.6, -3)--(0.6, 0);
\draw [blue, ->-=0.1] (0, 0)--(0, -3.3) node [right, pos=0.1, black] {\tiny $X$};
\draw [line width=0.83cm] (0,0) [partial ellipse=180:360:2 and 1.5];
\draw [white, line width=0.8cm] (0,0) [partial ellipse=178:362:2 and 1.5];
\draw [red] (0,0) [partial ellipse=178:362:2 and 1.5];
\begin{scope}[shift={(0,-3)}]
\draw [dashed](0,0) [partial ellipse=0:180:0.6 and 0.3];
\draw (0,0) [partial ellipse=180:360:0.6 and 0.3];
\end{scope}
\end{tikzpicture}}}}
=\mu\vcenter{\hbox{\scalebox{0.5}{
\begin{tikzpicture}[xscale=1,yscale=1.4]
\begin{scope}[shift={(0, -1.5)}]
\draw [ dashed] (-1, 1.25) .. controls +(0, 0.3) and +(0, 0.3).. (1, 1.25);
\draw    (1,1.25)  .. controls +(0, -0.3) and +(0, -0.3).. (-1, 1.25);
\draw (-1, 1.25) .. controls +(0, 1.2) and +(0, 1.2).. (1, 1.25);
\draw [blue] (0, 1)node [below, black] {\tiny $X$} --(0,1.5) ;
\node [draw, fill=white] at(0, 1.5) {\tiny $\varphi$};
\end{scope}
\begin{scope}[shift={(0, 1.5)}]
\draw (-1, 1.25) .. controls +(0, 0.3) and +(0, 0.3).. (1, 1.25);
\draw    (1,1.25)  .. controls +(0, -0.3) and +(0, -0.3).. (-1, 1.25);
\draw (-1, 1.25) .. controls +(0, -1.2) and +(0, -1.2).. (1, 1.25);
\draw [blue] (0, 1)node [above, black] {\tiny $X$} --(0,0.8);
\node [draw, fill=white] at(0, 0.7) {\tiny $\varphi'$};
\end{scope}
\end{tikzpicture}}}}
\end{align*}
where $\varphi$ is a basis of $\hom( \1, X)$ and $\varphi'$ is its dual basis.
\end{remark}

\subsection{Equivalence of Tube Category and the Drinfeld Center}
In this section, we show the tube category $\cA$ is equivalent to the Drinfeld center $\mathcal{Z(C)}$ as a braided monoidal category.

We define an assignment $F_0$ from $\cA_1$ to $\mathcal{C}$ as follows. On the object level, $F_{0}(O_{Y_1Y_2\cdots Y_k})=\oplus_{i}X_i^{*}Y_1Y_2\cdots Y_kX_i$, and $F_0(O\otimes O')=F_{0}(O)\otimes F_{0}(O')$. On the morphism level, to each $\cC$-handlebody $H: O\rightarrow O'$ as a morphism in $\mathcal{\cA}_0$ we define a $F_0(H)$ to be a $\mathcal{C}$-diagram in a rectangular region in the $y=1$ plane with marked points on the top edge and bottom edge respectively as follows.

First isotope $H$ such that the $y$-coordinate of $H$ is less than $1$. For any $\cC$-disk $D_n^1$ centered at $(n,0,1)$, connect it with the disk centered at $(n,1,2)$ by the mapping cylinder of a homeomorphism of $D_n^1$ which stretches $\partial D_n^1$ between the \$-sign and the first marked point it encounters (going counter-clockwisely), so that their distance on the circle is larger than $\pi/4$. For the ``lower'' $\cC$-disks, i.e., those with center on the line $z=0$, $y=0$, we attach them to the $y=1$ plane in the similar way so that the centers of their images are on the line $z=-1$, $y=1$.

Then to each upper(resp. lower) disk attached to $(n, 1, 2)$(resp. $(n', 1, -1)$), put a lower(resp. upper) semi-circle colored by a simple object $X_n$(resp. $X_{n'}$) on the $y=1$ plane, centered at $(n, 1, 2)$(resp. $(n', -1, -1)$) with radius $1/3$, multiply a factor of $d_{X_{n}}$.  Summing over all possible colorings. This result in a cornered $\mathcal{C}$- handlebody $\tilde{H}$ with its corner a rectangular $\mathcal{C}$-disk.

Lastly, we apply finite steps of Move $1$ and Move $2$ to make it a planar $\mathcal{C}$-diagram on the $y=1$ plane, which defines a morphism in $\hom_{C}(F(O), F(O'))$. The existence of such moves can be seen by a similar argument as in the proof of Theorem $\ref{lem:morphismform}$.

To show $F_0$ well-defined, we only need to show that different sequences of Moves result in the same morphism in $\hom(F_0(O), F_0(O'))$. For all $g\in \hom(F_0(O'), F_0(O))$, let $\boxed{g}$ be the rectangle containing a string diagram representing $g$ with a single coupon. Gluing the boundary of $\boxed{g}$ and $\tilde{H}$ result in a 3-alterfold $\boxed{g}\#\tilde{H}$. By definition,
$$Z(\boxed{g}\# \tilde{H})=\tr_{\mathcal{C}}(F_{0}(H)g).$$
On one hand, the partition function is invariant under Move $1$ and Move $2$. On the other hand, the trace on the right hand defines a non-degenerate form. Therefore, $F_{0}(H)$ defines a unique morphism in $\hom(F_0(O), F_0(O'))$.

In the rest of the paper, we will present $F_0(H)$ by the cornered handlebody $\tilde{H}$ with a single rectangle corner. This will simplify the computation. E.g.
\begin{align*}
F_0\left(
\vcenter{\hbox{\scalebox{0.7}{
\begin{tikzpicture}[scale=0.35]
\draw (0,5) [partial ellipse=0:360:2 and 0.8];
\draw (-2, 5)--(-2, -4);
\draw (2, 5)--(2, -4);
\draw[dashed] (0,-4) [partial ellipse=0:180:2 and 0.8];
\draw (0,-4) [partial ellipse=180:360:2 and 0.8];
\draw[blue, ->-=0.5] (0, 4.2) node[above, black]{\tiny{${X}$}}->(0, 0.5);
\draw[blue, ->-=0.5] (0, -0.5)--(0, -4.8)node[below, black]{\tiny{${Y}$}};
\draw [red](0, .95) [partial ellipse=-75:0:2 and 0.8];
\draw [red, dashed](0, .95) [partial ellipse=0:180:2 and 0.8];
\draw [red](0, .95) [partial ellipse=180:255:2 and 0.8];
\node [draw, fill=white] (0, 0){\tiny $f$};
\end{tikzpicture}}}}
\right)
=\sum_{i, j=0}^rd_{i}^{1/2} d_j^{1/2}
\vcenter{\hbox{\scalebox{0.7}{\begin{tikzpicture}[scale=0.35]
\draw [dashed] (0,5) [partial ellipse=0:360:2 and 0.8];
\draw (-2, 5)--(-2, -4);
\draw (2, 5)--(2, -4);
\draw[dashed] (0,-4) [partial ellipse=0:180:2 and 0.8];
\draw [dashed] (0,-4) [partial ellipse=180:360:2 and 0.8];
\draw[blue, ->-=0.5] (0, 7) node[above, black]{\tiny{${X}$}}->(0, 0.5);
\draw[blue, ->-=0.5] (0, -0.5)--(0, -6)node[below, black]{\tiny{${Y}$}};
\draw [red](0, .95) [partial ellipse=-75:0:2 and 0.8];
\draw [red, dashed](0, .95) [partial ellipse=0:180:2 and 0.8];
\draw [red](0, .95) [partial ellipse=180:255:2 and 0.8];
\node [draw, fill=white] (0, 0){\tiny $f$};
\draw [blue] (0, -6) [partial ellipse=40:0:2.5 and 4] node [below, black] {\tiny $X_{i}$};
\draw [blue, dashed, ->-=0.6] (0, -6) [partial ellipse=135:45:2.5 and 4];
\draw [blue] (0, -6) [partial ellipse=135:180:2.5 and 4] node [below, black] {\tiny $X_{i}^*$};
\draw [blue] (0, 7) [partial ellipse=-40:0:2.5 and 4] node [above, black] {\tiny $X_{j}$};
\draw [blue, dashed, ->-=0.6] (0, 7) [partial ellipse=-45:-135:2.5 and 4];
\draw [blue] (0, 7) [partial ellipse=-140:-180:2.5 and 4] node [above, black] {\tiny $X_{j}^*$};
\draw (-4,-6) rectangle (4, 7);
\end{tikzpicture}}}}
\end{align*}
The summand on the right-hand side is understood as a $xz$-plane projection of the following $3$-d diagram.
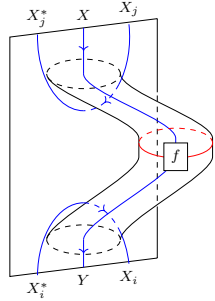
\begin{align*}
\vcenter{\hbox{\scalebox{0.7}{\begin{tikzpicture}[scale=0.35]
\draw (4, -5)--(-4,-6)--(-4, 7) --(4, 8)  (4, -5)--( 4, -2.5) (4, 4)--(4, 8);
\draw [dashed]( 4, -2.5)--(4, 4);
\draw [dashed] (0,5) [partial ellipse=0:360:2 and 0.8];
\draw (-2, 5)..controls +(0,-1) and +(0,1)..(3,1.5)--(3,0.5)..controls +(0,-1) and +(0,1)..(-2,-4);
\draw (2, 5)..controls +(0,-1) and +(0,1)..(7,1.5)--(7,0.5)..controls +(0,-1) and +(0,1)..(2,-4);
\draw [blue] (0, 5)..controls +(0,-1) and +(0,1)..(5,1.5)--(5,0.5)..controls +(0,-1) and +(0,1)..(0,-4);
\draw[dashed] (0,-4) [partial ellipse=0:180:2 and 0.8];
\draw [dashed] (0,-4) [partial ellipse=180:360:2 and 0.8];
\draw[blue, ->-=0.5] (0, 7.5) node[above, black]{\tiny{${X}$}}->(0, 5);
\draw[blue, ->-=0.5] (0, -4)--(0, -5.5)node[below, black]{\tiny{${Y}$}};
\draw [red](5, 1.25) [partial ellipse=-75:0:2 and 0.8];
\draw [red, dashed](5, 1.25) [partial ellipse=0:180:2 and 0.8];
\draw [red](5, 1.25) [partial ellipse=180:255:2 and 0.8];
\node [draw, fill=white] at (5, 0.5){\tiny $f$};
\draw [blue] (0, -6) [partial ellipse=40:10:2.5 and 4] node [below, black] {\tiny $X_{i}$};
\draw [blue,dashed, ->-=0.5] (0, -6) [partial ellipse=90:40:2.5 and 4];
\draw [blue] (0, -6) [partial ellipse=90:179:2.5 and 4] node [below, black] {\tiny $X_{i}^*$};
\draw [blue] (0, 7) [partial ellipse=-40:10:2.5 and 4] node [above, black] {\tiny $X_{j}$};
\draw [blue,dashed, ->-=0.5] (0, 7) [partial ellipse=-40:-90:2.5 and 4];
\draw [blue] (0, 7) [partial ellipse=-90:-182:2.5 and 4] node [above, black] {\tiny $X_{j}^*$};
\end{tikzpicture}}}}
\end{align*}

The following is an example of the image of a cornered $\mathcal{C}$-handlebody with $3$ upper disks and $2$ lower disks.

\begin{align*}
F_0\left(
\vcenter{\hbox{\begin{tikzpicture}[xscale=0.25, yscale=0.35]
\draw (-3, 0) [partial ellipse=0:360:2 and 0.8];
\draw (3, 0) [partial ellipse=0:360:2 and 0.8];
\draw (9, 0) [partial ellipse=0:360:2 and 0.8];
\draw[dashed] (0, -5)[partial ellipse=0:180:2 and 0.8];
\draw (0, -5)[partial ellipse=180:360:2 and 0.8];
\draw[dashed] (6, -5)[partial ellipse=0:180:2 and 0.8];
\draw (6, -5)[partial ellipse=180:360:2 and 0.8];
%\draw (5, 0)..controls (5,-2) and (2,-4).. (2, -5);
\draw (11, 0)..controls (11,-2) and (8,-4).. (8, -5);
\draw (-5, 0)..controls (-5,-2) and (-2,-4).. (-2, -5);
\draw (-1, 0) ..controls (-.5, -2) and (.5, -2)..(1, 0);
\draw (5, 0) ..controls (5.5, -2) and (6.5, -2)..(7, 0);
\draw (2, -5) ..controls (2.5, -3) and (3.5, -3)..(4, -5);
\node at (3, -2.5) { $\Gamma$};
\end{tikzpicture}}}\right)
=\sum_{i_1, \ldots, i_5=0}^r \prod_{j=1}^5 d_{i_j}^{1/2}
\vcenter{\hbox{\scalebox{0.5}{\begin{tikzpicture}[yscale=0.35, xscale=0.4]
\draw [dashed](-3, 0) [partial ellipse=0:360:2 and 0.8];
\draw [dashed](-3, 3) [partial ellipse=0:360:2 and 0.8];
\draw [blue, dashed, ->-=0.6] (-3, 5) [partial ellipse=-40:-140:2.5 and 4];
\draw [blue] (-3, 5) [partial ellipse=-140:-180:2.5 and 4] node [above, black] {\tiny $X_{i_1}^*$};
\draw [blue] (-3, 5) [partial ellipse=-40:0:2.5 and 4] node [above, black] {\tiny $X_{i_1}$};
\draw (-5, 0)--(-5, 3) (-1, 0)--(-1, 3);
\draw [blue,-<-=0.7] (-3,-0.8)--(-3,5) node [above, black] {\tiny $Y_{1}$};
\begin{scope}[shift={(6, 0)}]
\draw [dashed](-3, 0) [partial ellipse=0:360:2 and 0.8];
\draw [dashed](-3, 3) [partial ellipse=0:360:2 and 0.8];
\draw [blue, dashed, ->-=0.6] (-3, 5) [partial ellipse=-40:-140:2.5 and 4];
\draw [blue] (-3, 5) [partial ellipse=-140:-180:2.5 and 4]node [above, black] {\tiny $X_{i_2}^*$};
\draw [blue] (-3, 5) [partial ellipse=-40:0:2.5 and 4] node [above, black] {\tiny $X_{i_2}$};
\draw (-5, 0)--(-5, 3) (-1, 0)--(-1, 3);
\draw [blue, -<-=0.7] (-3,-0.8)--(-3,5) node [above, black] {\tiny $Y_{2}$};
\end{scope}
\begin{scope}[shift={(12, 0)}]
\draw [dashed](-3, 0) [partial ellipse=0:360:2 and 0.8];
\draw [dashed](-3, 3) [partial ellipse=0:360:2 and 0.8];
\draw [blue, dashed, ->-=0.6] (-3, 5) [partial ellipse=-40:-140:2.5 and 4];
\draw [blue] (-3, 5) [partial ellipse=-140:-180:2.5 and 4] node [above, black] {\tiny $X_{i_3}^*$};
\draw [blue] (-3, 5) [partial ellipse=-40:0:2.5 and 4] node [above, black] {\tiny $X_{i_3}$};
\draw (-5, 0)--(-5, 3) (-1, 0)--(-1, 3);
\draw [blue, -<-=0.7] (-3,-0.8)--(-3,5) node [above, black] {\tiny $Y_{3}$};
\end{scope}
\begin{scope}[shift={(3, -5)}]
\draw [dashed](-3, 0) [partial ellipse=0:360:2 and 0.8];
\draw [dashed](-3, -3) [partial ellipse=0:360:2 and 0.8];
\draw [blue, dashed, ->-=0.6] (-3, -5) [partial ellipse=40:140:2.5 and 4];
\draw [blue] (-3, -5) [partial ellipse=140:180:2.5 and 4]node [below, black] {\tiny $X_{i_4}^*$};
\draw [blue] (-3, -5) [partial ellipse=40:0:2.5 and 4]node [below, black] {\tiny $X_{i_4}$};
\draw (-5, 0)--(-5, -3) (-1, 0)--(-1, -3);
\draw [blue, ->-=0.9] (-3,-0.8)--(-3,-5) node [below, black] {\tiny $Y_{4}$};
\end{scope}
%\draw[dashed] (0, -5)[partial ellipse=0:180:2 and 0.8];
%\draw [dashed] (0, -5)[partial ellipse=180:360:2 and 0.8];
%\draw[dashed] (6, -5)[partial ellipse=0:180:2 and 0.8];
%\draw (6, -5)[partial ellipse=180:360:2 and 0.8];
\begin{scope}[shift={(9, -5)}]
\draw [dashed](-3, 0) [partial ellipse=0:360:2 and 0.8];
\draw [dashed](-3, -3) [partial ellipse=0:360:2 and 0.8];
\draw [blue, dashed, ->-=0.6] (-3, -5) [partial ellipse=40:140:2.5 and 4];
\draw [blue] (-3, -5) [partial ellipse=140:180:2.5 and 4] node [below, black] {\tiny $X_{i_5}^*$};
\draw [blue] (-3, -5) [partial ellipse=40:0:2.5 and 4]node [below, black] {\tiny $X_{i_5}$};
\draw (-5, 0)--(-5, -3) (-1, 0)--(-1, -3);
\draw [blue, ->-=0.9] (-3,-0.8)--(-3,-5) node [below, black] {\tiny $Y_{5}$};
\end{scope}
\draw (11, 0)..controls (11,-2) and (8,-4).. (8, -5);
\draw (-5, 0)..controls (-5,-2) and (-2,-4).. (-2, -5);
\draw (-1, 0) ..controls (-.5, -2) and (.5, -2)..(1, 0);
\draw (5, 0) ..controls (5.5, -2) and (6.5, -2)..(7, 0);
\draw (2, -5) ..controls (2.5, -3) and (3.5, -3)..(4, -5);
\node at (3, -2.5) { $\Gamma$};
\draw (-7, -10) rectangle (13, 5);
\end{tikzpicture}}}}.
\end{align*}
In a summary, we see that the topological meaning of $F_0$ is to glue the morphism onto the $xz$-plane.
We say $F_0$ is a {\it handle-on} assignment.

\begin{remark}
The factor $d_i^{1/2}$ is not essential in this construction. If $d_i^{1/2}$ is not in the ground field $\k$, one can replace this factor by $d_i$ to each upper(or lower) half circles colored by $X_i$.
\end{remark}

\begin{proposition}\label{prop:f0functor}
    The assignment $F_0:\cA_1\to \mathcal{C}$ is a strict monoidal functor which induces a strict monoidal functor from $\mathcal{A}$ to $\cC$.
    We say $F_0$ is a handle-on functor from $\cA_1$ to $\mathcal{C}$.
\end{proposition}
\begin{proof}
It clear that $F_0(\id_{O})=\id_{F_0(O)}$, this can be seen by applying Move $2$ to each cylinder components of $\id_O$.

Then we show $F_0$ preserves the composition. Suppose $H_1$ and $H_2$ are morphisms in $\hom(O, O')$ and $\hom(O', O'')$ respectively. To each $\cC$-disks of $O'$, we apply isotopy and a Move $1$ to $F(H_2)F(H_1)$ as follows to derive $F(H_2H_1)$.
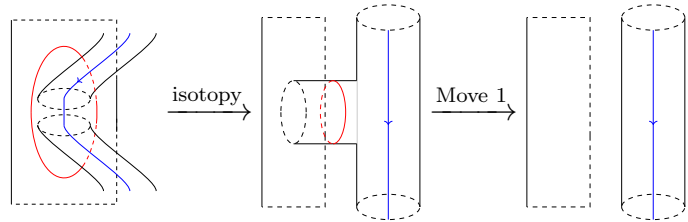
\begin{align*}
\vcenter{\hbox{\scalebox{0.5}{\begin{tikzpicture}[scale=0.35]
\draw  (-4,-12)--(-4, 2)  (4, -6.5)--( 4, -2.5);
\draw [dashed]( 4, -2.5)--(4, 2)--(-4, 2) ( 4, -5.5)--(4, -12)--(-4,-12);
%\draw [dashed] (0,5) [partial ellipse=0:360:2 and 0.8];
\draw (3,1)..controls +(0,-1) and +(0,1)..(-2,-4);
\draw (7,1)..controls +(0,-1) and +(0,1)..(2,-4);
\draw [blue, ->-=0.6] (5,1)..controls +(0,-1) and +(0,1)..(0,-4)--(0, -6);
\draw[dashed] (0,-4) [partial ellipse=0:180:2 and 0.8];
\draw [dashed] (0,-4) [partial ellipse=180:360:2 and 0.8];
\draw [red] (0, -5) [partial ellipse=55:180:2.5 and 5];
\draw [red] (0, -5) [partial ellipse=-55:-180:2.5 and 5];
\draw [red] (0, -5) [partial ellipse=10:-10:2.5 and 5];
\draw [red,dashed] (0, -5) [partial ellipse=55:10:2.5 and 5];
\draw [red,dashed] (0, -5) [partial ellipse=-55:-10:2.5 and 5];
\begin{scope}[shift={(0, -2)}]
\draw[dashed] (0,-4) [partial ellipse=0:180:2 and 0.8];
\draw [dashed] (0,-4) [partial ellipse=180:360:2 and 0.8];
\draw (-2, -4) ..controls +(0,-1) and +(0,1)..(3,-9);
\draw (2,-4) ..controls +(0,-1) and +(0,1)..(7,-9);
\draw [blue] (0,-4) ..controls +(0,-1) and +(0,1)..(5,-9);
\end{scope}
\end{tikzpicture}}}}
\xrightarrow{\text{isotopy}}
\vcenter{\hbox{\scalebox{0.6}{\begin{tikzpicture}[scale=0.35]
\draw  (-2,-10)--(-2, 2)  ;
\draw [dashed]( 2, -2.5)--(2, 2)--(-2, 2) ( 2, -5.5)--(2, -10)--(-2,-10)(2, -5.5)--( 2, -2.5) ;
%\draw [dashed] (0,5) [partial ellipse=0:360:2 and 0.8];
\draw (4,2)--(4,-10)(8,2)--(8,-10);
\draw [blue, ->-=0.5] (6,1.2)--(6,-10.8);
\draw[dashed] (0,-4) [partial ellipse=0:180:0.8 and 2];
\draw [dashed] (0,-4) [partial ellipse=180:360:0.8 and 2];
\draw [white, thick] (4, -2)--(4, -6);
\draw (0, -2)--(4,-2) (0, -6)--(4, -6);
\draw[red] (2.5,-4) [partial ellipse=90:-90:0.8 and 2];
\draw[red, dashed] (2.5,-4) [partial ellipse=90:270:0.8 and 2];
\draw [ dashed](6,2) [partial ellipse=180:360:2 and 0.8];
\draw[ dashed] (6,2) [partial ellipse=0:180:2 and 0.8];
\draw [ dashed] (6,-10) [partial ellipse=180:360:2 and 0.8];
\draw[ dashed] (6,-10) [partial ellipse=0:180:2 and 0.8];
\end{tikzpicture}}}}
\xrightarrow{\text{Move 1}}
\vcenter{\hbox{\scalebox{0.6}{\begin{tikzpicture}[scale=0.35]
\draw  (-2,-10)--(-2, 2)  ;
\draw [dashed]( 2, -2.5)--(2, 2)--(-2, 2) ( 2, -5.5)--(2, -10)--(-2,-10)(2, -5.5)--( 2, -2.5) ;
%\draw [dashed] (0,5) [partial ellipse=0:360:2 and 0.8];
\draw (4,2)--(4,-10)(8,2)--(8,-10);
\draw [blue, ->-=0.5] (6,1.2)--(6,-10.8);
\draw [dashed](6,2) [partial ellipse=180:360:2 and 0.8];
\draw[ dashed] (6,2) [partial ellipse=0:180:2 and 0.8];
\draw [ dashed] (6,-10) [partial ellipse=180:360:2 and 0.8];
\draw[ dashed] (6,-10) [partial ellipse=0:180:2 and 0.8];
\end{tikzpicture}}}}
\end{align*}
Therefore, $F_0$ is a functor. Moreover, $F_0$ is a strict monoidal functor, since by definition, $F_0$ preserves tensor product and the tensor unit, and the structure morphisms are identities.
\end{proof}

\begin{remark}
It is easy to see that $F_{0}$ admit an adjoint functor $I:\mathcal{C}\rightarrow \mathcal{A}$, defined by $I(Y)=O_{Y}$, and on the morphism level,
\begin{align*}
I\left(
\vcenter{\hbox{\scalebox{0.6}{
\begin{tikzpicture}[scale=0.35]
\draw[blue, ->-=0.5] (0, 4.2) node[above, black]{\tiny{$X$}}->(0, 0.5);
\draw[blue, ->-=0.5] (0, -0.5)--(0, -4.8) node[below, black]{\tiny{$Y$}};
\node [draw, fill=white] (0, 0){\tiny $f$};
\end{tikzpicture}}}}\right)
=
\vcenter{\hbox{\scalebox{0.6}{
\begin{tikzpicture}[scale=0.35]
\draw (0,5) [partial ellipse=0:360:2 and 0.8];
\draw (-2, 5)--(-2, -4);
\draw (2, 5)--(2, -4);
\draw[dashed] (0,-4) [partial ellipse=0:180:2 and 0.8];
\draw (0,-4) [partial ellipse=180:360:2 and 0.8];
\draw[blue, ->-=0.5] (0, 4.2) node[above, black]{\tiny{${X}$}}->(0, 0.5);
\draw[blue, ->-=0.5] (0, -0.5)--(0, -4.8)node[below, black]{\tiny{${Y}$}};
\node [draw, fill=white] (0, 0){\tiny $f$};
\end{tikzpicture}}}}.
\end{align*}
This can be seen as follows, since all objects in $\mathcal{A}$ is isomorphic to $O_{X}$ for some object $X\in \cC$. We check
$$\hom_{\cA}(O_{X}, I(Y))=\hom_{\cA}(O_{X}, O_Y)=\hom_{\cC}(\bigoplus_{ i}X_i^{*}XX_i, Y)=\hom_{\cC}(F_0(O_X), Y).$$

In particular, $I(\id_{\1_{\mathcal{C}}})$ is the empty tube. 
\end{remark}
Now we recall the following lemma for semisimplicity of an algebra.
\begin{lemma}[Lemma 5.10 of \cite{Mug03a}]\label{lem:semisimple}
Let $K$ be a finite-dimensional $\k$-algebra with a nondegenerate trace map. If the trace map vanished on nilpotent elements, then $K$ is semisimple. 
\end{lemma}

\begin{proposition}\label{prop:semisimple}
The tube category $\cA$ is semisimple.
\end{proposition}
\begin{proof}
We prove the proposition by showing all endomorphism algebra of $\cA_1$ is semisimple. By Remark \ref{rem:homspace}, we only have to consider the morphism space $\hom_{\cA}(\CIRCLE_X, \CIRCLE_X)$.
By Lemma \ref{lem:nondegenerate}, we assume that $\cyl(f)\in\hom_{\cA}(\CIRCLE_X, \CIRCLE_X)$ is nilpotent.
We only have to show $\Tr_{\cA}(\cyl(f))=0$.
Note that $F_0$ is a functor, we have that $F_0(\cyl(f))$ is nilpotent.
By the argument of Lemma \ref{lem:nondegenerate} again, we have that
\begin{align*}
\Tr_{\cA}(\cyl(f))=\Tr_{\mathcal{C}}(F_0(\cyl(f)))=0,
\end{align*}
By Lemma \ref{lem:semisimple}, we see that $\hom_{\cA}(\CIRCLE_X, \CIRCLE_X)$ is semisimple and $\cA$ is semisimple.
\end{proof}
\begin{remark}\label{rem:modularity}
    By Proposition \ref{prop:semisimple}, and the fact that $\hom_{\cA}(\1, \1)=\k$,  $\mathcal{A}$ is a braided spherical fusion category. Notice that Lemma \ref{lem:zkirby} implies that the M\"uger center $\mathcal{Z}_{2}(\mathcal{A})$ is trivial.
    Therefore, $\cA$ is a modular tensor category (see \cite{EGNO15} Proposition 8.20.12 for equivalent definitions of modularity).

    The invertibility of $S$-matrix defines a nondegenerate bilinear form on the Grothendieck ring $\mathcal{K}_{0}(\mathcal{A})$. Thus Lemma \ref{lem:zkirby} also implies $\sum_{X_{i}\in \Irr{\mathcal{C}}}d_{i}I(X_i)$ equals to the $\Omega$-color of $\mathcal{A}$ in $\mathcal{K}_{0}(\mathcal{A})$.
\end{remark}

Now we lift $F_0: \mathcal{A}\rightarrow \mathcal{C}$ to a braided monoidal functor $F:\cA\rightarrow \mathcal{Z(C)}$. We first define $F: \cA_1\rightarrow \mathcal{Z(C)}$, then $F$ naturally extends to $\mathcal{A}$ by the semisimplicity of $\mathcal{Z(C)}$.

To lift $F_0$ to a functor to $\mathcal{Z(C)}$, we assign each object $F_{0}(O)$ and $V\in \cC$ a morphims.
$$e_{F_{0}(O)}(V): F_{0}(O)\otimes V \rightarrow V\otimes F_{0}(O)$$
that presented by a diagram with a string colored by $V$ cross under the tubes representing $\id_{O}$. For instance, $e_{O_{X}}(V)$ is represented by 
\begin{align*}
\sum_{i,j}d_i^{1/2}d_j^{1/2}
\vcenter{\hbox{\scalebox{0.6}{\begin{tikzpicture}[scale=0.35]
\draw (4, -5)--(-4,-6)--(-4, 7) --(4, 8)  (4, -5)--( 4, -2.5) (4, 4)--(4, 8);
\draw [dashed]( 4, -2.5)--(4, 4);
\draw [dashed] (0,5) [partial ellipse=0:360:2 and 0.8];
\draw (-2, 5)..controls +(0,-1) and +(0,1)..(3,1.5)--(3,0.5)..controls +(0,-1) and +(0,1)..(-2,-4);
\draw (2, 5)..controls +(0,-1) and +(0,1)..(7,1.5)--(7,0.5)..controls +(0,-1) and +(0,1)..(2,-4);
\draw [blue] (0, 5)..controls +(0,-1) and +(0,1)..(5,1.5)--(5,0.5)..controls +(0,-1) and +(0,1)..(0,-4);
\draw[dashed] (0,-4) [partial ellipse=0:180:2 and 0.8];
\draw [dashed] (0,-4) [partial ellipse=180:360:2 and 0.8];
\draw[blue, ->-=0.5] (0, 7.5) node[above, black]{\tiny{${X}$}}->(0, 5);
\draw[blue, ->-=0.5] (0, -4)--(0, -5.5)node[below, black]{\tiny{${X}$}};
% \draw [red](5, 1.25) [partial ellipse=-75:0:2 and 0.8];
% \draw [red, dashed](5, 1.25) [partial ellipse=0:180:2 and 0.8];
% \draw [red](5, 1.25) [partial ellipse=180:255:2 and 0.8];
% \node [draw, fill=white] at (5, 0.5){\tiny $f$};
\draw [blue] (0, -6) [partial ellipse=40:10:2.5 and 4] node [below, black] {\tiny $X_{i}$};
\draw [blue,dashed, ->-=0.5] (0, -6) [partial ellipse=90:40:2.5 and 4];
\draw [blue] (0, -6) [partial ellipse=90:179:2.5 and 4] node [below, black] {\tiny $X_{i}^*$};
\draw [blue] (0, 7) [partial ellipse=-40:10:2.5 and 4] node [above, black] {\tiny $X_{j}$};
\draw [blue,dashed, ->-=0.5] (0, 7) [partial ellipse=-40:-90:2.5 and 4];
\draw [blue] (0, 7) [partial ellipse=-90:-182:2.5 and 4] node [above, black] {\tiny $X_{j}^*$};
\draw [blue] (0, 4) [partial ellipse=-90:-40:3.5 and 3];
\draw [blue,dashed] (0, 4) [partial ellipse=0:-40:3.5 and 3] ;
\draw [blue] (3.5, 4)--(3.5, 8) node [above, black]{\tiny $V$};
\draw [blue] (0, -3) [partial ellipse=90:180:3.2 and 4];
\draw [blue, ->-=0.1] (-3.2, -3)--(-3.2, -6);
\end{tikzpicture}}}}\,.
\end{align*}
Clearly, $e_{F_{0}(O)}(-)$ is a half-braiding in the sense of Section 2.2, since morphism in $\mathcal{C}$ can slide under the tubes, and the hexagon condition is obvious.
\begin{proposition}\label{prop:ffunctor}
The assignment
    \begin{align*}
        F:\mathcal{A}_1&\rightarrow \mathcal{Z(C)} \\
        O&\mapsto (F_{0}(O), e_{F_0(O)})
    \end{align*}
defines a braided monoidal functor $F$ from $\mathcal{A}_1$ to $\mathcal{Z(C)}$, which extend to a braided monoidal functor from $\cA$ to $\mathcal{Z(C)}$.
\end{proposition}
\begin{proof}
We first show for all morphism $f\in \hom_{\cA}(O, O')$, $F_0(f)$ is a morphism in $\mathcal{Z(C)}$, i.e., we need to show for all $V\in \mathcal{C}$,
$$(\id_V\otimes F_0(f))e_{F_{0}(O)}=e_{F_{0}(O')}(F_0(f)\otimes \id_V).$$
This can be seen by the similar argument as in the proof of Proposition \ref{prop:f0functor}. Moreover, since $F_0$ is monoidal, $F$ is also monoidal. It remains to show $F$ is braided.

For all $O, O'\in \cA$, by definition, we have $F(c_{O, O'})=\tilde{c}_{O, O'}$ which is represented by the exactly same $\cC$-diagram in $\mathbb{R}^3$ as $e_{F_0(O)}(F_0(O'))$. In fact, in view of the previous discussion, the following picture is sufficient to prove $F$ is braided.
\begin{align*}
F(c_{O_X, O_Y})=F\left(\vcenter{\hbox{\scalebox{0.7}{
\begin{tikzpicture}
\draw  (1.6, 0) .. controls +(0, 0.3) and +(0, 0.3) .. (2.4, 0);
\draw (1.6, 0) .. controls +(0,-0.3) and +(0, -0.3) .. (2.4, 0);
\draw  (1.6, 0) .. controls +(0, -0.3) and +(0.3, 0) ..(1, -0.7) .. controls +(-0.7, 0) and +(0,0.7) ..(-0.4, -2);
\draw  (2.4, 0) .. controls +(0, -0.7) and +(0.7, 0) ..(1, -1.3) .. controls +(-0.3, 0) and +(0,0.3) ..(0.4, -2);
\draw [blue,->-=0.1, ->-=0.9] (2, -0.21) .. controls +(0, -0.5) and +(0.5, 0) ..(1, -1) .. controls +(-0.4, 0) and +(0,0.4) ..(0, -2)--(0, -2.2) node [below] {\tiny $Y$};
\draw [ dashed] (-0.4, -2) .. controls +(0, 0.3) and +(0, 0.3) .. (0.4, -2);
\draw (-0.4, -2) .. controls +(0,-0.3) and +(0, -0.3) .. (0.4, -2); % right tube
\path [fill=white]  (0.15,-0.69) rectangle (1.85, -1.31);
\draw (-0.4, 0) .. controls +(0, 0.3) and +(0, 0.3) .. (0.4, 0);
\draw  (-0.4, 0) .. controls +(0,-0.3) and +(0, -0.3) .. (0.4, 0);
\draw  (0.4, 0) .. controls +(0, -0.3) and +(-0.3, 0) ..(1, -0.7) .. controls +(0.7, 0) and +(0,0.7) ..(2.4, -2);
\draw  (-0.4, 0) .. controls +(0, -0.7) and +(-0.7, 0) ..(1, -1.3) .. controls +(0.3, 0) and +(0,0.3) ..(1.6, -2);
\draw [blue,->-=0.5] (0, -0.21) .. controls +(0, -0.5) and +(-0.5, 0) ..(1, -1) .. controls +(0.5, 0) and +(0,0.4) ..(2, -2)--(2, -2.2) node [below] {\tiny $X$};
\draw [dashed] (1.6, -2) .. controls +(0, 0.3) and +(0, 0.3) .. (2.4, -2);
\draw  (1.6, -2) .. controls +(0,-0.3) and +(0, -0.3) .. (2.4, -2); % left tube
\end{tikzpicture}}}}\right)
=\sum_{i,j,k, \ell=0}^r (d_id_jd_k d_{\ell})^{1/2}
\vcenter{\hbox{\scalebox{0.6}{
\begin{tikzpicture}[xscale=1.2]
\begin{scope}[shift={(2, 0)}]
\draw [dashed](0,0) [partial ellipse=0:360:0.4 and 0.2];
\draw [dashed](0,0.8) [partial ellipse=0:360:0.4 and 0.2];
\draw (0.4, 0)--(0.4, 0.8) (-0.4, 0)--(-0.4, 0.8) ;
\draw [blue, dashed, ->-=0.6](0,1.2) [partial ellipse=-140:-40:0.6 and 0.9];
\draw [blue](0,1.2) [partial ellipse=-40:0:0.6 and 0.9] node [above, black] {\tiny $X_j$};
\draw [blue](0,1.2) [partial ellipse=-140:-180:0.6 and 0.9] node [above, black] {\tiny $X_j^*$};
\draw [blue] (0, -0.2)--(0, 1.2) node [above, black] {\tiny $Y$};
\end{scope}
\draw  (1.6, 0) .. controls +(0, -0.3) and +(0.3, 0) ..(1, -0.7) .. controls +(-0.7, 0) and +(0,0.7) ..(-0.4, -2);
\draw  (2.4, 0) .. controls +(0, -0.7) and +(0.7, 0) ..(1, -1.3) .. controls +(-0.3, 0) and +(0,0.3) ..(0.4, -2);
\draw [blue,->-=0.1, ->-=0.9] (2, -0.21) .. controls +(0, -0.5) and +(0.5, 0) ..(1, -1) .. controls +(-0.4, 0) and +(0,0.4) ..(0, -2)--(0, -2.2) ;
\path [fill=white]  (0.15,-0.69) rectangle (1.85, -1.31);
\draw [dashed](0,0) [partial ellipse=0:360:0.4 and 0.2];
\draw [dashed](0,0.8) [partial ellipse=0:360:0.4 and 0.2];
\draw (0.4, 0)--(0.4, 0.8) (-0.4, 0)--(-0.4, 0.8) ;
\draw [blue, dashed, ->-=0.6](0,1.2) [partial ellipse=-140:-40:0.6 and 0.9];
\draw [blue](0,1.2) [partial ellipse=-40:0:0.6 and 0.9] node [above, black] {\tiny $X_i$};
\draw [blue](0,1.2) [partial ellipse=-140:-180:0.6 and 0.9] node [above, black] {\tiny $X_i^*$};
\draw [blue] (0, -0.2)--(0, 1.2) node [above, black] {\tiny $X$};
\draw  (0.4, 0) .. controls +(0, -0.3) and +(-0.3, 0) ..(1, -0.7) .. controls +(0.7, 0) and +(0,0.7) ..(2.4, -2);
\draw  (-0.4, 0) .. controls +(0, -0.7) and +(-0.7, 0) ..(1, -1.3) .. controls +(0.3, 0) and +(0,0.3) ..(1.6, -2);
\draw [blue,->-=0.5] (0, -0.21) .. controls +(0, -0.5) and +(-0.5, 0) ..(1, -1) .. controls +(0.5, 0) and +(0,0.4) ..(2, -2)--(2, -2.2);
\begin{scope}[shift={(0, -2)}]
\draw [dashed](0,0) [partial ellipse=0:360:0.4 and 0.2];
\draw [dashed](0,-0.8) [partial ellipse=0:360:0.4 and 0.2];
\draw (0.4, 0)--(0.4, -0.8) (-0.4, 0)--(-0.4, -0.8) ;
\draw [blue, dashed, ->-=0.6](0,-1.2) [partial ellipse=140:40:0.6 and 0.9];
\draw [blue](0,-1.2) [partial ellipse=40:0:0.6 and 0.9] node [below, black] {\tiny $X_k$};
\draw [blue](0,-1.2) [partial ellipse=140:180:0.6 and 0.9] node [below, black] {\tiny $X_k^*$};
\draw [blue] (0, 0.2)--(0, -1.2);
\end{scope} % left tube
\begin{scope}[shift={(2, -2)}]
\draw [dashed](0,0) [partial ellipse=0:360:0.4 and 0.2];
\draw [dashed](0,-0.8) [partial ellipse=0:360:0.4 and 0.2];
\draw (0.4, 0)--(0.4, -0.8) (-0.4, 0)--(-0.4, -0.8) ;
\draw [blue, dashed, ->-=0.6](0,-1.2) [partial ellipse=140:40:0.6 and 0.9];
\draw [blue](0,-1.2) [partial ellipse=40:0:0.6 and 0.9] node [below, black] {\tiny $X_\ell$};
\draw [blue](0,-1.2) [partial ellipse=140:180:0.6 and 0.9] node [below, black] {\tiny $X_\ell^*$};
\draw [blue] (0, 0.2)--(0, -1.2);
\end{scope}
\draw (-0.8,-3.2) rectangle (2.8,1.2);
\end{tikzpicture}}}}=e_{F_0(O_X)}(F_0(O_Y)).
\end{align*}

\end{proof}

To prove $F$ is an equivalence between $\cA$ and $\mathcal{Z(C)}$, we define an assignment $G$ on morphisms from $\hom_{\mathcal{Z(C)}}((X, e_X), (Y, e_Y))$ to $\hom_{\cA}(O_X, O_Y)$ by
\begin{align*}
G\left(
\vcenter{\hbox{\scalebox{0.8}{
\begin{tikzpicture}[scale=0.35]
\draw[blue, ->-=0.5] (0, 4.2) node[above, black]{\tiny{${(X, e_X)}$}}->(0, 0.5);
\draw[blue, ->-=0.5] (0, -0.5)--(0, -4.8) node[below, black]{\tiny{${(Y, e_Y)}$}};
\node [draw, fill=white] (0, 0){\tiny $f$};
\end{tikzpicture}}}}\right)
=\frac{1}{\mu}
\vcenter{\hbox{\scalebox{0.8}{
\begin{tikzpicture}[scale=0.35]
\draw (0,5) [partial ellipse=0:360:2 and 0.8];
\draw (-2, 5)--(-2, -4);
\draw (2, 5)--(2, -4);
\draw[dashed] (0,-4) [partial ellipse=0:180:2 and 0.8];
\draw (0,-4) [partial ellipse=180:360:2 and 0.8];
\draw[blue, ->-=0.5] (0, 4.2) node[above, black]{\tiny{${X}$}}->(0, 0.5);
\draw[blue, ->-=0.5] (0, -0.5)--(0, -4.8)node[below, black]{\tiny{${Y}$}};
\draw [red](0, 3) [partial ellipse=-75:0:2 and 0.8];
\draw [red, dashed](0, 3) [partial ellipse=0:180:2 and 0.8];
\draw [red](0, 3) [partial ellipse=180:255:2 and 0.8];
\node [draw, fill=white] (0, 0){\tiny $f$};
% \draw [red](0, -2) [partial ellipse=-75:0:2 and 0.8];
% \draw [red, dashed](0, -2) [partial ellipse=0:180:2 and 0.8];
% \draw [red](0, -2) [partial ellipse=180:255:2 and 0.8];
\end{tikzpicture}}}},
\end{align*}
where $f\in \hom_{\mathcal{Z(C)}}((X, e_X), (Y, e_Y))$ and the crossing in the diagram denote the the summation of morphisms $e_{X}(X_i)$ (see Notation in Theorem \ref{lem:morphismform}).

\begin{lemma}\label{lem:G-fg}
The assignment $G$ preserves compositions of morphisms.
\end{lemma}
\begin{proof}
For any for $f\in \hom _{\mathcal{Z(C)}}((W, e_W), (Y, e_Y))$ and $g\in \hom _{\mathcal{Z(C)}}((X, e_X), (W, e_W) )$, we have
\begin{align*}
G(fg)=\frac{1}{\mu}
\vcenter{\hbox{\scalebox{0.7}{
\begin{tikzpicture}[scale=0.35]
\draw (0,5) [partial ellipse=0:360:2 and 0.8];
\draw (-2, 5)--(-2, -4);
\draw (2, 5)--(2, -4);
\draw[dashed] (0,-4) [partial ellipse=0:180:2 and 0.8];
\draw (0,-4) [partial ellipse=180:360:2 and 0.8];
\draw[blue, ->-=0.5] (0, 4.2) node[above, black]{\tiny{${X}$}}->(0, 0);
\draw[blue, ->-=0.5] (0, 0)--(0, -4.8)node[below, black]{\tiny{${Y}$}};
\draw [red](0, 3) [partial ellipse=-75:0:2 and 0.8];
\draw [red, dashed](0, 3) [partial ellipse=0:180:2 and 0.8];
\draw [red](0, 3) [partial ellipse=180:255:2 and 0.8];
\node [draw, fill=white] at (0, -1){\tiny $f$};
\node [draw, fill=white] at (0, 1){\tiny $g$};
\end{tikzpicture}}}}
=
\frac{1}{\mu^2}
\vcenter{\hbox{\scalebox{0.7}{
\begin{tikzpicture}[scale=0.35]
\draw (0,5) [partial ellipse=0:360:2 and 0.8];
\draw (-2, 5)--(-2, -4);
\draw (2, 5)--(2, -4);
\draw[dashed] (0,-4) [partial ellipse=0:180:2 and 0.8];
\draw (0,-4) [partial ellipse=180:360:2 and 0.8];
\draw[blue, ->-=0.3] (0, 4.2) node [above, black]{\tiny{${X}$}}->(0, 0);
\draw[blue, ->-=0.8] (0, 0)--(0, -4.8)node[below, black]{\tiny{${Y}$}};
\draw [red](0, 3.1) [partial ellipse=-75:0:2 and 0.8];
\draw [red, dashed](0, 3.1) [partial ellipse=0:180:2 and 0.8];
\draw [red](0, 3.1) [partial ellipse=180:255:2 and 0.8];
\draw [red](0, 0.5) [partial ellipse=-75:0:2 and 0.8];
\draw [red, dashed](0, 0.5) [partial ellipse=0:180:2 and 0.8];
\draw [red](0, 0.5) [partial ellipse=180:255:2 and 0.8];
\node [draw, fill=white] at (0, -1.5){\tiny $f$};
\node [draw, fill=white] at (0, 1.2){\tiny $g$};
\end{tikzpicture}}}}=G(f)G(g)\,,
\end{align*}
as $f$ and $g$ are morphisms in $\mathcal{Z(C)}$. 
\end{proof}

In particular, for any $(X,e_X) \in \mathcal{Z(C)}$,  $G(\id_{(X, e_X)}) = \mu^{-1}\cyl(\sum_{i}e_X(X_i))$ is an idempotent in $\cA_1$. Now for any $(X,e_X) \in \mathcal{Z(C)}$, define $G(X,e_X) := (O_X, G(\id_{(X, e_X)}) \in \cA$, then $G(\id_{(X, e_X)}) = \id_{(G(X, e_X)}$. So together with Lemma \ref{lem:G-fg}, we have $G: \mathcal{Z(C)} \to \cA$ is a functor.

We shall prove $GF\cong \id_{\cA}$ and $FG\cong \id_{\mathcal{Z(C)}}$ by the following two lemma. 
The prove is similar to Theorem 6.4 of \cite{Kir11}.

\begin{lemma}\label{lem:GF}
The functor $GF$ is isomorphic to the identity functor $\text{id}_{\cA}$.
\end{lemma}
\begin{proof}
By Corollary \ref{cor:singledisk}, any object in $\cA$ is isomorphic to a subobject of some $O_{X}$. Moreover, by Remark \ref{rem:homspace}, morphisms between objects of form $O_X$ are of form $\cyl(f)$. Thus, we only need to construct a natural isomorphism $h: GF\to \id_{\cA}$ for objects of the form $(O_{X}, \cyl(p_X))$, where $p_X \in T(X,X)$ is a morphism such that $\cyl(p_X)$ is an idempotent in $\cA_1$.

By Move 1, for any $f \in T(X,Y)$, we have
\begin{equation}\label{GFequiv}
GF(\cyl(f))=\mu^{-1}\sum_{i,j=0}^r (d_i d_j)^{1/2}
\vcenter{\hbox{\scalebox{0.7}{
\begin{tikzpicture}[scale=0.8]
\begin{scope}[shift={(1.5,1.5)}]
\draw (-1, 1.25) .. controls +(0, 0.3) and +(0, 0.3).. (1, 1.25);
\draw  (1,1.25)  .. controls +(0, -0.3) and +(0, -0.3).. (-1, 1.25);
\end{scope}
\draw (0.5, 2.75) .. controls +(0, -0.5) and +(0, 0.5) .. (-1, 1.25);
\draw (2.5, 2.75) .. controls +(0, -0.5) and +(0, 0.5) .. (4, 1.25);
\draw (1, 1.25)  .. controls +(0, 0.5) and +(0, 0.5) ..  (2, 1.25);
\draw [blue, dashed] (1.5, 1.65) .. controls +(0.2, 0) and +(0, -0.3).. (1.8, 3) node [above, black] {\tiny $X_i$};
\draw [blue] (1.5, 1.65) .. controls +(-0.2, 0) and +(0, -0.3).. (1.2, 2.5);
\draw [dashed] (-1, 1.25) .. controls +(0, 0.3) and +(0, 0.3).. (1, 1.25);
\draw [dashed]  (1,1.25)  .. controls +(0, -0.3) and +(0, -0.3).. (-1, 1.25);
\begin{scope}[shift={(0, -2.5)}]
\draw  [dashed] (-1, 1.25) .. controls +(0, 0.3) and +(0, 0.3).. (1, 1.25);
\draw [dashed]  (1,1.25)  .. controls +(0, -0.3) and +(0, -0.3).. (-1, 1.25);
\end{scope}
\begin{scope}[shift={(3, 0)}]
\draw [dashed] (-1, 1.25) .. controls +(0, 0.3) and +(0, 0.3).. (1, 1.25);
\draw [dashed]  (1,1.25)  .. controls +(0, -0.3) and +(0, -0.3).. (-1, 1.25);
\end{scope}
\begin{scope}[shift={(3, -2.5)}]
\draw [dashed] (-1, 1.25) .. controls +(0, 0.3) and +(0, 0.3).. (1, 1.25);
\draw [dashed]  (1,1.25)  .. controls +(0, -0.3) and +(0, -0.3).. (-1, 1.25);
\end{scope}
\begin{scope}[shift={(0, -1)}]
\draw [red, dashed] (-1, 1.25) .. controls +(0, 0.3) and +(0, 0.3).. (1, 1.25);
\draw [red]  (1,1.25)  .. controls +(0, -0.3) and +(0, -0.3).. (-1, 1.25);
\end{scope}
\node [draw, fill=white] at (0, 0) {\tiny ${f}$};
\draw [blue, -<-=0.2, ->-=0.8](0, 0.3) --(0, 1) .. controls +(0, 0.5) and +(0, -0.5).. (0.9, 2.55) node [above, black] {\tiny $X$}  (0, -0.3) --(0, -1.5) .. controls +(0, -0.5) and +(0, 0.5).. (0.9, -2.95)node [below, black] {\tiny $Y$};
\draw  (-1, 1.25)--(-1, -1.25) (1, 1.25)--(1, -1.25);
\begin{scope}[shift={(3, 0)}]
\begin{scope}[shift={(0, -1)}]
\draw [red, dashed] (-1, 1.25) .. controls +(0, 0.3) and +(0, 0.3).. (1, 1.25);
\draw [red]  (1,1.25)  .. controls +(0, -0.3) and +(0, -0.3).. (-1, 1.25);
\end{scope}
\draw  (-1, 1.25)--(-1, -1.25) (1, 1.25)--(1, -1.25);
\end{scope}
\begin{scope}[shift={(1.5,-4)}]
\draw [dashed] (-1, 1.25) .. controls +(0, 0.3) and +(0, 0.3).. (1, 1.25);
\draw   (1,1.25)  .. controls +(0, -0.3) and +(0, -0.3).. (-1, 1.25);
\end{scope}
\draw  (0.5, -2.75) .. controls +(0, 0.5) and +(0, -0.5) .. (-1, -1.25);
\draw  (2.5, -2.75) .. controls +(0, 0.5) and +(0, -0.5) .. (4, -1.25);
\draw  (1, -1.25)  .. controls +(0, -0.5) and +(0, -0.5) ..  (2, -1.25);
\draw [blue] (1.5, -1.65) .. controls +(0.2, 0) and +(0, 0.3).. (1.8, -2.95) node [below, black] {\tiny $X_j$};
\draw [blue, dashed] (1.5, -1.65) .. controls +(-0.2, 0) and +(0, 0.3).. (1.2, -2.5);
\end{tikzpicture}}}}
=\mu^{-1}\sum_{i,j=0}^r (d_id_j)^{1/2}
\vcenter{\hbox{\scalebox{0.7}{
\begin{tikzpicture}[scale=0.35]
\draw (0,5) [partial ellipse=0:360:2 and 0.8];
\draw (-2, 5)--(-2, -4);
\draw (2, 5)--(2, -4);
\draw[dashed] (0,-4) [partial ellipse=0:180:2 and 0.8];
\draw (0,-4) [partial ellipse=180:360:2 and 0.8];
\draw[blue, ->-=0.5] (0, 4.2) node[above, black]{\tiny{${X}$}}->(0, 0.5);
\draw[blue, ->-=0.5] (0, -0.5)--(0, -4.8)node[below, black]{\tiny{${Y}$}};
\draw [red](0, 0.95) [partial ellipse=-75:0:2 and 0.8];
\draw [red, dashed](0, 0.95) [partial ellipse=0:180:2 and 0.8];
\draw [red](0, 0.95) [partial ellipse=180:255:2 and 0.8];
\node [draw, fill=white] (0, 0){\tiny $f$};
\draw [blue,dashed] (2, -2).. controls +(-0.2, 0.1) and +(0.2, 0.3).. (0.8, -3.3);
\draw [blue] (2, -2).. controls +(0,-0.4) and +(0.2, 0.2).. (0.8, -4.7) node [below, black] {\tiny $X_j$};
\draw [blue] (2, 3).. controls +(-0.2, -0.1) and +(0.2, -0.3).. (0.8, 4.3);
\draw [blue,dashed] (2, 3).. controls +(0,0.4) and +(0.2, -0.2).. (0.8, 5.7) node [above, black] {\tiny $X_i$};
\end{tikzpicture}}}}\,.
\end{equation}
For simplicity, we denote $GF(\cyl(f))$ by $\mu^{-1}\vcenter{\hbox{
\begin{tikzpicture}[scale=0.25]
\draw (0,5) [partial ellipse=0:360:2 and 0.8];
\draw (-2, 5)--(-2, -4);
\draw (2, 5)--(2, -4);
\draw[dashed] (0,-4) [partial ellipse=0:180:2 and 0.8];
\draw (0,-4) [partial ellipse=180:360:2 and 0.8];
\draw[blue, ->-=0.5] (0, 4.2) node[above, black]{\tiny{${X}$}}->(0, 0.5);
\draw[blue, ->-=0.5] (0, -0.5)--(0, -4.8)node[below, black]{\tiny{${Y}$}};
\draw [red](0, 0.95) [partial ellipse=-75:0:2 and 0.8];
\draw [red, dashed](0, 0.95) [partial ellipse=0:180:2 and 0.8];
\draw [red](0, 0.95) [partial ellipse=180:255:2 and 0.8];
\node [draw, fill=white] (0, 0){\tiny $f$};
\draw [red,dashed] (2, -2).. controls +(-0.2, 0.1) and +(0.2, 0.3).. (0.8, -3.3);
\draw [red] (2, -2).. controls +(0,-0.4) and +(0.2, 0.2).. (0.8, -4.7);
\draw [red] (2, 3).. controls +(-0.2, -0.1) and +(0.2, -0.3).. (0.8, 4.3);
\draw [red,dashed] (2, 3).. controls +(0,0.4) and +(0.2, -0.2).. (0.8, 5.7);
\end{tikzpicture}}}$. Using this notation, it is clear that the following two morphisms
\begin{align*}
h_{p_X}=\mu^{-1}
\vcenter{\hbox{\scalebox{0.7}{
\begin{tikzpicture}[scale=0.35]
\draw (0,5) [partial ellipse=0:360:2 and 0.8];
\draw (-2, 5)--(-2, -4);
\draw (2, 5)--(2, -4);
\draw[dashed] (0,-4) [partial ellipse=0:180:2 and 0.8];
\draw (0,-4) [partial ellipse=180:360:2 and 0.8];
\draw[blue, ->-=0.5] (0, 4.2) node[above, black]{\tiny{${X}$}}->(0, 0.5);
\draw[blue, ->-=0.5] (0, -0.5)--(0, -4.8)node[below, black]{\tiny{${X}$}};
\draw [red](0, 0.95) [partial ellipse=-75:0:2 and 0.8];
\draw [red, dashed](0, 0.95) [partial ellipse=0:180:2 and 0.8];
\draw [red](0, 0.95) [partial ellipse=180:255:2 and 0.8];
\node [draw, fill=white] (0, 0){\tiny $p_X$};
%\draw [red,dashed] (2, -2).. controls +(-0.2, 0.1) and +(0.2, 0.3).. (0.8, -3.3);
%\draw [red] (2, -2).. controls +(0,-0.4) and +(0.2, 0.2).. (0.8, -4.7);
\draw [red] (2, 3).. controls +(-0.2, -0.1) and +(0.2, -0.3).. (0.8, 4.3);
\draw [red,dashed] (2, 3).. controls +(0,0.4) and +(0.2, -0.2).. (0.8, 5.7);
\end{tikzpicture}}}}\,, \text{ and }
h_{p_X}'=
\vcenter{\hbox{\scalebox{0.7}{
\begin{tikzpicture}[scale=0.35]
\draw (0,5) [partial ellipse=0:360:2 and 0.8];
\draw (-2, 5)--(-2, -4);
\draw (2, 5)--(2, -4);
\draw[dashed] (0,-4) [partial ellipse=0:180:2 and 0.8];
\draw (0,-4) [partial ellipse=180:360:2 and 0.8];
\draw[blue, ->-=0.5] (0, 4.2) node[above, black]{\tiny{${X}$}}->(0, 0.5);
\draw[blue, ->-=0.5] (0, -0.5)--(0, -4.8)node[below, black]{\tiny{${X}$}};
\draw [red](0, 0.95) [partial ellipse=-75:0:2 and 0.8];
\draw [red, dashed](0, 0.95) [partial ellipse=0:180:2 and 0.8];
\draw [red](0, 0.95) [partial ellipse=180:255:2 and 0.8];
\node [draw, fill=white] (0, 0){\tiny $p_X$};
\draw [red,dashed] (2, -2).. controls +(-0.2, 0.1) and +(0.2, 0.3).. (0.8, -3.3);
\draw [red] (2, -2).. controls +(0,-0.4) and +(0.2, 0.2).. (0.8, -4.7);
\end{tikzpicture}}}}
\end{align*}
between $(O_X, \cyl(p_{X}))$ and $GF(O_X, \cyl(p_{X}))=(GF(O_X), GF(\cyl(p_X))$ are inverse to each other as morphisms in $\cA$, as the semicircles compose to an $\Omega$-colored circle, which evaluates to $\mu$ by planar graphical calculus.

Finally , for any morphism $\cyl(f)\in \hom_{\cA}((O_X, \cyl(p_X)), (O_Y, \cyl(q_Y)))$, we have
\begin{align*}
h_{q_{Y}}\circ GF(\cyl({f})) \circ h_{p_{X}}'=
\vcenter{\hbox{\scalebox{0.7}{
\begin{tikzpicture}[scale=0.35]
\draw (0,5) [partial ellipse=0:360:2 and 0.8];
\draw (-2, 5)--(-2, -4);
\draw (2, 5)--(2, -4);
\draw[dashed] (0,-4) [partial ellipse=0:180:2 and 0.8];
\draw (0,-4) [partial ellipse=180:360:2 and 0.8];
\draw[blue, ->-=0.3] (0, 4.2) node[above, black]{\tiny{${X}$}}->(0, 0.5);
\draw[blue, ->-=0.8] (0, -0.5)--(0, -4.8)node[below, black]{\tiny{${Y}$}};
\begin{scope}[shift={(0, -2)}]
\draw [red](0, 0.95) [partial ellipse=-75:0:2 and 0.8];
\draw [red, dashed](0, 0.95) [partial ellipse=0:180:2 and 0.8];
\draw [red](0, 0.95) [partial ellipse=180:255:2 and 0.8];
\node [draw, fill=white] (0, 0){\tiny $q_Y$};
\end{scope}
\begin{scope}[shift={(0, 0)}]
\draw [red](0, 0.95) [partial ellipse=-75:0:2 and 0.8];
\draw [red, dashed](0, 0.95) [partial ellipse=0:180:2 and 0.8];
\draw [red](0, 0.95) [partial ellipse=180:255:2 and 0.8];
\node [draw, fill=white] (0, 0){\tiny $f$};
\end{scope}
\begin{scope}[shift={(0, 2)}]
\draw [red](0, 0.95) [partial ellipse=-75:0:2 and 0.8];
\draw [red, dashed](0, 0.95) [partial ellipse=0:180:2 and 0.8];
\draw [red](0, 0.95) [partial ellipse=180:255:2 and 0.8];
\node [draw, fill=white] (0, 0){\tiny $p_X$};
\end{scope}
\end{tikzpicture}}}}
=\cyl(f)\,,
\end{align*}
which implies the naturality of $h$. This completes the proof of the lemma.
\end{proof}

\begin{lemma}\label{lem:FG}
The functor $FG$ is isomorphic to the identity functor $\id_{\mathcal{Z(C)}}$.
\end{lemma}

\begin{proof}
Let $f:(X, e_{X})\rightarrow (Y, e_{Y})$ be a morphism in $\mathcal{Z(C)}$, we compute $FG(f)$ as follows
\begin{align*}
FG(f)
&=\frac{1}{\mu}\sum_{i, j=0}^rd_{i}^{1/2}d_{j}^{1/2}
\vcenter{\hbox{\scalebox{0.6}{\begin{tikzpicture}[scale=0.35]
\draw [dashed] (0,5) [partial ellipse=0:360:2 and 0.8];
\draw (-2, 5)--(-2, -4);
\draw (2, 5)--(2, -4);
\draw[dashed] (0,-4) [partial ellipse=0:180:2 and 0.8];
\draw [dashed] (0,-4) [partial ellipse=180:360:2 and 0.8];
\draw[blue, ->-=0.5] (0, 7) node[above, black]{\tiny{${X}$}}->(0, 0.5);
\draw[blue, ->-=0.5] (0, -0.5)--(0, -6)node[below, black]{\tiny{${Y}$}};
\draw [red](0, 2) [partial ellipse=-75:0:2 and 0.8];
\draw [red, dashed](0, 2) [partial ellipse=0:180:2 and 0.8];
\draw [red](0, 2) [partial ellipse=180:255:2 and 0.8];
% \draw [red](0, -2) [partial ellipse=-75:0:2 and 0.8];
% \draw [red, dashed](0, -2) [partial ellipse=0:180:2 and 0.8];
% \draw [red](0, -2) [partial ellipse=180:255:2 and 0.8];
\node [draw, fill=white] at (0, 0){\tiny $f$};
\draw [blue] (0, -6) [partial ellipse=40:0:2.5 and 4] node [below, black] {\tiny $X_{i}$};
\draw [blue, dashed, ->-=0.6] (0, -6) [partial ellipse=135:45:2.5 and 4];
\draw [blue] (0, -6) [partial ellipse=135:180:2.5 and 4] node [below, black] {\tiny $X_{i}^*$};
\draw [blue] (0, 7) [partial ellipse=-40:0:2.5 and 4] node [above, black] {\tiny $X_{j}$};
\draw [blue, dashed, ->-=0.6] (0, 7) [partial ellipse=-45:-135:2.5 and 4];
\draw [blue] (0, 7) [partial ellipse=-140:-180:2.5 and 4] node [above, black] {\tiny $X_{j}^*$};
\draw (-4,-6) rectangle (4, 7);
\end{tikzpicture}}}}
=\frac{1}{\mu}
\sum_{i, j=0}^rd_{i}^{1/2}d_j^{1/2} \vcenter{\hbox{\scalebox{0.6}{
\begin{tikzpicture}[scale=0.35]
\draw[dashed] (0,5) [partial ellipse=0:360:2 and 0.8];
\draw (-2, 5)--(-2, -4);
\draw (2, 5)--(2, -4);
\draw[dashed] (0,-4) [partial ellipse=0:180:2 and 0.8];
\draw[dashed] (0,-4) [partial ellipse=180:360:2 and 0.8];
\draw[blue, ->-=0.4] (0, 7) node [above, black]{\tiny{${X}$}}--(0, 0.5);
\draw[blue, ->-=0.6] (0, -0.5)--(0, -6)node[below, black]{\tiny{${Y}$}};
\begin{scope}[yshift=1cm]
\draw [red](0, .95) [partial ellipse=-75:0:2 and 0.8];
\draw [red, dashed](0, .95) [partial ellipse=0:180:2 and 0.8];
\draw [red](0, .95) [partial ellipse=180:255:2 and 0.8];
\end{scope}
\node [draw, fill=white] at (0, 0) {\tiny ${f}$};
\draw [blue,->-=0.3] (2, 7).. controls (1.7, 7) and (1.4, 1.5) ..(0.4, 1.5);
\draw [blue] (-2, 7).. controls (-1.7, 7) and (-1.4, 1.5) ..(-0.4, 1.5);
\begin{scope}[yscale=-1]
\draw [blue, -<-=0.33] (2, 6).. controls (1.7, 6) and (1.4, 1.5) ..(0.4, 1.5);
\draw [blue] (-2, 6).. controls (-1.7, 6) and (-1.4, 1.5) ..(-0.4, 1.5);
\end{scope}
\draw (-3, 7) node [above]{\tiny{$X_{i}$}};
\draw (-3, -6) node [below]{\tiny{$X_{j}$}};
\draw (-4, -6) rectangle (4, 7);
\end{tikzpicture}}}}
=\frac{1}{\mu}\sum_{i, j}d_{i}^{1/2}d_{j}^{1/2}
\vcenter{\hbox{\scalebox{0.6}{\begin{tikzpicture}[scale=0.35]
\draw[blue, ->-=0.5] (0, 6) node[above, black]{\tiny{${X}$}}->(0, 0);
\draw[blue, ->-=0.5] (0, 0)--(0, -6)node[below, black]{\tiny{${Y}$}};
\node [draw, fill=white] at (0, 0){\tiny $f$};
\draw [blue,-<-=0.5] (0, -6) [partial ellipse=80:0:2.5 and 4] node [below, black] {\tiny $X_{i}$};
\draw [blue] (0, -6) [partial ellipse=100:180:2.5 and 4] node [below, black] {\tiny $X_{i}^*$};
\draw [blue,-<-=0.5] (0, 6) [partial ellipse=-80:0:2.5 and 4] node [above, black] {\tiny $X_{j}$};
\draw [blue] (0, 6) [partial ellipse=-100:-180:2.5 and 4] node [above, black] {\tiny $X_{j}^*$};
\end{tikzpicture}}}}.
\end{align*}
The second equality comes from sliding the blue arcs to the front along the red curves in the middle. The last equality comes from Move 2 with respect to the middle hole.

We define $\upsilon_{X, e_{X}}: (X, e_{X})\rightarrow GF(X, e_{X})$ and $\upsilon_{X, e_{X}}': GF(X, e_{X})\rightarrow (X, e_{X})$ as follows
$$\upsilon_{X, e_{X}}=\sum_j d_j^{1/2}
\vcenter{\hbox{ \scalebox{0.6}{\begin{tikzpicture}[scale=0.35]
\draw[blue, ->-=0.5] (0, 6) node[above, black]{\tiny{${X}$}}->(0, -6);
\draw [blue,-<-=0.5] (0, 6) [partial ellipse=-80:0:2.5 and 4] node [above, black] {\tiny $X_{j}$};
\draw [blue] (0, 6) [partial ellipse=-100:-180:2.5 and 4] node [above, black] {\tiny $X_{j}^*$};
\end{tikzpicture}}}}
\text{ and }
\upsilon_{X, e_{X}}'=
\mu^{-1}\sum_{i}d_{i}^{1/2}
\vcenter{\hbox{\scalebox{0.6}{\begin{tikzpicture}[scale=0.35]
\draw[blue, ->-=0.5] (0, 6) node[above, black]{\tiny{${X}$}}->(0, -6);
\draw [blue,-<-=0.5] (0, -6) [partial ellipse=80:0:2.5 and 4] node [below, black] {\tiny $X_{i}$};
\draw [blue] (0, -6) [partial ellipse=100:180:2.5 and 4] node [below, black] {\tiny $X_{i}^*$};
\end{tikzpicture}}}}
$$
By the planar graphical calculus, we have that
\begin{align*}
\upsilon_{X, e_{X}}\circ \upsilon_{X, e_{X}}'=FG(\id_{(X, e_X)}), \quad \upsilon_{X, e_{X}}'\circ \upsilon_{X, e_{X}}=\id_{(X, e_X)}.
\end{align*}
Lastly, the following equation implies the naturality of $\upsilon$, which completes the proof:
\[\upsilon_{Y, e_{Y}}'\circ FG(f)\circ \upsilon_{X, e_{X}}=
\frac{1}{\mu^{2}}
\vcenter{\hbox{ \scalebox{0.6}{\begin{tikzpicture}[scale=0.35]
\draw[blue, ->-=0.5] (0, 6) node[above, black]{\tiny{${X}$}}->(0, 0);
\draw[blue, ->-=0.5] (0, 0)--(0, -6)node[below, black]{\tiny{${Y}$}};
\node [draw, fill=white] at (0, 0){\tiny $f$};
\begin{scope}[shift={(0, 3)}]
\draw [red] (0.2, -1.5) .. controls +(1.6, 0) and +(1.6, 0) .. (0.2, 1.5);
\draw [red] (-0.2, -1.5) .. controls +(-1.6, 0) and +(-1.6, 0) .. (-0.2, 1.5);
\end{scope}
\begin{scope}[shift={(0, -3)}]
\draw [red] (0.2, -1.5) .. controls +(1.6, 0) and +(1.6, 0) .. (0.2, 1.5);
\draw [red] (-0.2, -1.5) .. controls +(-1.6, 0) and +(-1.6, 0) .. (-0.2, 1.5);
\end{scope}
\end{tikzpicture}}}}
=f\,.\qedhere\]
\end{proof}

\begin{theorem}\label{thm:equivcenter}
The tube category $\cA$ is braided monoidal equivalent to the Drinfeld center $\mathcal{Z(C)}$.
\end{theorem}
\begin{proof}
By Lemma \ref{lem:GF} and \ref{lem:FG}, we see that  $F$ is an equivalence of categories. Moreover, by Proposition \ref{prop:ffunctor}, $F$, being a braided strict monoidal functor, is an equivalence of braided monoidal categories. In addition, $F$ is linear and exact by definition. Therefore, $F$ is an equivalence between braided fusion categories.
\end{proof}

\begin{remark}
It is important to notice that by our topologized Drinfeld center, the idempotents in $\hom_{\cA}(O, O)$ naturally give half-braidings in $\cC$.
\end{remark}

\begin{remark}
By the fact that $\cA$ is a modular tensor category, Theorem \ref{thm:equivcenter} presents a topological proof that the Drinfeld center is a modular tensor category.
\end{remark}

\subsection{TV Invariants and RT Invariants}

By Theorem \ref{thm:equivcenter}, we see that $\cA$ can be recognized as a topological definition of the Drinfeld center.
By Remark \ref{rem:modularity}, we see that $\vcenter{\hbox{\scalebox{0.6}{
\begin{tikzpicture}[xscale=0.8, yscale=0.6]
\draw [line width=0.83cm] (0,0) [partial ellipse=0:360.1:2 and 1.5];
\draw [white, line width=0.8cm] (0,0) [partial ellipse=-1:361.1:2 and 1.5];
\draw [red] (0,0) [partial ellipse=0:360:2 and 1.5];
\end{tikzpicture}}}}$ is the $\Omega$ color of $\cA$.
By identifying the tube category $\cA$ and the Drinfeld center,  we have

\begin{theorem}\label{thm:RT=TV}
Suppose $\mathcal{C}$ is a spherical fusion category and $M$ is a  compact oriented $3$-manifold  without boundary.
Then
$$TV_{\mathcal{C}}(M)=RT_{\mathcal{Z(C)}}(M).$$
\end{theorem}
\begin{proof}
Suppose that $M=\bS^3_L$, where $L$ is a framed link.
Then by Theorem \ref{thm:equivcenter} and Corollary \ref{cor:TVlink}, we have that
\begin{align*}
RT_{\mathcal{Z(C)}}(M)=\mu^{-|L|}Z(\bS^{3}, \partial L_{\epsilon}, \Gamma_{L})=TV_{\mathcal{C}}(\bS_{L}^{3})=TV_{\mathcal{C}}(M).
\end{align*}
This completes the proof.
\end{proof}
\begin{remark}
In the case $\mathcal{C}$ is a modular tensor category, it is well-known\cite{EGNO15, Mug03b} that $\mathcal{Z(C)}$ is braided equivalent to the Deligne tensor product $\mathcal{C}\boxtimes \mathcal{C}^{\text{rev}}$, as a corollary, one has
$$TV_{\mathcal{C}}(M)=RT_{\mathcal{C}}(M)RT_{\mathcal{C}}(-M).$$
\end{remark}

\begin{remark}
For Morita equivalent spherical fusion categories $\mathcal{C}$ and $\mathcal{D}$, it is well-known\cite{EGNO15, Mug03b} that $\mathcal{Z(C)}$ and $\mathcal{Z(D)}$ are braided equivalent. As a corollary, one has
$$TV_{\mathcal{C}}(M)=TV_{\mathcal{D}}(M).$$
\end{remark}

\begin{theorem}\label{thm:pseudotv=rt}
Suppose $\mathcal{C}$ is a spherical fusion category and $M$ is a pseudo $3$-manifold with singular locus $M_{sing}$. Suppose $M\setminus M_{sing}$ is homeomorphic to $\bS^{3}\setminus L$ for a link $L\subset \bS^{3}$. Then
\begin{align*}
TV_{\mathcal{C}} (M)= RT_{\mathcal{Z(C)}} (\bS^3, L_{I(\1)}),
\end{align*}
where $L_{I(\1)}$ is the $I(\1)$-colored framed link $L$.
\end{theorem}
\begin{proof}
By Theorem \ref{thm:pseudotv} and the assumption that $\zeta = 1$, we have that 
\begin{align*}
TV_{\mathcal{C}} (M) = Z(\bS^3, \partial L_\epsilon, \emptyset)=RT_{\mathcal{Z(C)}} (\bS^3, L_{I(\1)}).
\end{align*}
This completes the proof of the theorem.
\end{proof}

As an example, consider $\mathcal{C} = \cC(\mathfrak{sl}_2, q)$, the modular tensor category obtained from the quantum group $\mathcal{U}_{q}(\mathfrak{sl}_2)$, where $q$ is a $2r$-th root of unity (see \cite{Row05f} for details). After identifying $\mathcal{Z(C)}$ with $\mathcal{C}\boxtimes\mathcal{C}^{rev}$, $I(\1)$ is identified as $\bigoplus_{0\le i\le r-2} V_i\boxtimes V_{i}^{*}$ where $V_i$ corresponds to the $(i+1)$-dimensional irreducible representation of $\mathfrak{sl}_2$. On the other hand, $\langle L_{V_i}\rangle_{\mathcal{C}}$ can be computed using $J_{L, i}(q^2)$, the $i$-th colored Jones polynomial evaluated at $q^2$. Consequently, we have
\begin{corollary}[Theorem 1.1 in \cite{DKY18}]
Suppose $\mathcal{C}=\cC(\mathfrak{sl}_2, q)$ and $q$ is a primitive $2r$-th root of unity. Let $L$ be a framed link in $\bS^{3}$ and $\bS^{3}\setminus L$ be the pseudo $3$-manifold as in the previous theorem. Then
\begin{equation*} 
TV_{\mathcal{C}} (\bS^3 \setminus L) =\mu^{-1}\sum_{0\le i\le r-2}J_{L,i}(q^2)\overline{J_{L, i}(q^2)}\,. \qedhere
\end{equation*}
\end{corollary}
\begin{remark}
In \cite{BenPet96} \cite{CheYan18} and \cite{DKY18}, they only considered the case that $\mathcal{C}$ equals to the quantum group associated to representation theory of $SU(2)$ and $SO(3)$, the modularity of quantum groups is essential in their proof. Theorem \ref{thm:pseudotv=rt} generalize their result to spherical fusion categories.
\end{remark}

\section{Generalized Frobenius-Schur Indicators}

The classical Frobenius-Schur (FS) indicators for representations of finite groups were introduced more than a century ago and has been studied and generalized ever since. 
The FS indicators were generalized to modules of semisimple Hopf algebras by Linchenko-Montgomery \cite{LinMon20} and Kashina-Sommerh\"{a}user-Zhu \cite{KSZ06}, and the generalizations to semisimple quasi-Hopf algebras were studied by Mason-Ng \cite{MasNg05} and Schauenburg \cite{Sch04}. 
FS indicators in categorical settings were also considered and studied by many, e.g., Fuchs-Ganchev-Szlach\'{a}nyi-Vescerny\'{e}s \cite{FGSzV99} and Fuchs-Schweigert \cite{FucSch03}. 
In \cite{NgSch07}, Ng and Schauenburg defined higher FS indicators for pivotal categories and unified the aforementioned versions of indicators. 

In \cite{SomZhu12}, Sommerh\"{a}user and Zhu used an $\SL_2(\bZ)$-equivariant indicator to prove the congruence subgroup theorem for semisimple factorizable Hopf algebras. 
Ng and Schauenburg in \cite{NgSch10} defined generalized FS-indicators for pivotal categories that extends the equivariant indicators of Sommerh\"{a}user-Zhu, and used the equivariance of the generalized FS-indicator to successfully prove the congruence subgroup theorem for modular tensor categories in full generality. 
Although the equivariance of the generalized FS-indicator is one of the key ingredient in the proof of the congruence subgroup theorem, there has long been a lack of intuitive explanation. 

In this section, we present $\mathcal{C}$-decorated 3-alterfolds whose partition functions equals to generalized Frobenius-Schur indicators. The diagrammatic presentation implies $\SL(2, \mathbb{Z})$-equivariance immediately.

\subsection{Topologized Frobenius-Schur Indicators}
We briefly recall the definition of the generalized FS-indicators for a spherical fusion category introduced in \cite{NgSch10}, according to which it suffices to assume that the category is strict pivotal.

Let $\mathcal{C}$ be a spherical fusion category over $\k$ that is strict pivotal. For any $V \in \mathcal{C}$ and $m \in \bZ$, we follow \cite[Sec.~2]{NgSch10} and set
\[V^m = \begin{cases}
\1\,, & \text{if $m = 0$}\,,\\ 
V \otimes V^{m-1} & \text{if $m > 0$}\,,\\
(V^*)^{-m} & \text{if $m < 0$}\,.
\end{cases}\]
Note that $V^{-m} = (V^m)^*$ by definition and the strictness assumption. For any pair $(m, \ell) \in \bZ^2$, denote
\[J_{m,\ell}(V) := 
\begin{cases}
V^{-\ell} \otimes V^m \otimes V^{\ell} = (V^{-\ell} \otimes V^{\ell}) \otimes V^{m} \xrightarrow{\ev\otimes \id} V^m & \text{if $m\ell \ge 0$}\,,\\
V^{-\ell} \otimes V^m \otimes V^{\ell} = V^m \otimes (V^{-\ell} \otimes V^\ell) \xrightarrow{\id \otimes \ev} V^{m} & \text{if $m\ell \le 0$}\,.
\end{cases}\]
Then we use $J_{m, \ell}$ to define a map $E^{(m, \ell)}_{(X, e_X), V}: \hom_{\cC}(X, V^m) \to\hom_{\cC}(X, V^m)$ for any $(X, e_X) \in \cZ(\cC)$, $V \in \cC$ and $(m, \ell) \in \bZ^2$ as follows. To any morphism $f \in \hom_{\cC}(X, V^m)$, we assign the composition

\[\begin{split}
E^{(m, \ell)}_{(X, e_X), V}(f):
& X \xrightarrow{\ \id \otimes \coev\ } X \otimes V^{-\ell} \otimes V^{\ell} \xrightarrow{\ e_X(V^{-\ell})  \ot \id \ } \\
& \qquad\quad V^{-\ell} \ot X \ot V^{\ell} \xrightarrow{\ \id \ot f \ot \id\ } V^{-\ell} \ot V^m \ot V^{\ell} \xrightarrow{\ J_{m, \ell}(V)\ } V^m\,.
\end{split}\]
It is easy to see that $E^{(m, \ell)}_{(X, e_X), V}$ is linear, and the {\it generelized Frobenius-Schur indicator} (generalized FS-indicator) attached to $(X, e_X) \in \cZ(\cC)$, $V \in \cC$ and $(m, \ell) \in \bZ^2$, denoted by $\nu_{m, \ell}^{(X, e_X)}(V)$, is defined to be the trace of this linear map:
\[\nu_{(m, \ell)}^{(X, e_X)}(V) := \Tr_{\k}\left(E^{(m, \ell)}_{(X, e_X), V}\right)\,.\]

\begin{example}\label{example:nu}
The semisimplicity of $\cC$ allows one to depict the generalized Frobenius-Schur indicators. Here, we present the indicators in planar graphical calculus (cf.~\cite[Proposition 3.1]{NgSch10}), and compare it with our quantum invariant later. Let $(X, e_X) \in \cZ(\cC)$, $V \in \cC$ and $(m, \ell) \in \bZ^2$ be as above, we have the following cases. When $m = 0$, for any $\ell \in \bZ$, we have 
\[\nu_{(0, \ell)}^{(X, e_X)}(V) = % \sum_{\varphi \in \ONB(X, \1)}
\raisebox{-2.5em}{\scalebox{0.6}{\begin{tikzpicture}
\node[draw, inner xsep = 0.8em] (A) at (0,0) {$\varphi'$};
\node[draw, inner xsep = 0.8em] (B) at (0,-2.2) {$\varphi$};
\draw[line width=5mm, blue] (0,-2) circle [radius = 1.2];
\draw[line width=4.5mm, white] (0,-2) circle [radius = 1.2];
\node[fill=white, inner ysep = 5mm, inner xsep = 2mm] at (0, -0.9) {};
\draw (A.south)--(B.north); \node at (-0.25, -1.55) {$X$};
\node at (-1.2, -2) {$\ell$}; \node at (-1.6, -3) {$V^\ell$};
\end{tikzpicture}}}\,,\]
where the blue lines are understood as $\ell$ lines. As before, we use crossings to represent the half-braiding $e_X(V^{\ell})$, and the summation over $\varphi \in \ONB(X, \1)$ is suppressed.

When $m \ne 0$, by Eq.~(2.3) in \cite{NgSch10}, $\nu_{(m, \ell)}^{(X, e_X)}(V) = \nu_{(-m, -\ell)}^{(X, e_X)}(V^*)$, so without loss of generality, we can assume that $m > 0$. Then we can write $\ell = -qm + r$, for some $q \in \bZ$ and $0 \le r \le m-1$. By \cite[Lemma 2.7]{NgSch10}, we have

\[\nu_{m,\ell}^{(X,e_X)}(V)
= % \sum_{\varphi\in\ONB(X, V^m)}\ 
\raisebox{-5.5em}{\scalebox{0.5}{\begin{tikzpicture}
\node[draw, inner xsep = 2em] (A) at (0,0) {$\varphi'$};
\node[draw] (B) at (0,-2.5) {$\theta_X^q$};
\node[draw, inner xsep = 2em] (C) at (0,-4) {$\varphi$};
\draw[line width=5mm, blue] (C.210)
..controls +(0,-0.7) and +(0,-0.7)..($(C.210)+(-1,0)$)--++(0,2.5) node (C1) {}
..controls +(0,0.3) and +(-0.3,0)..($(C1)+(0.5,0.5)$)--++(2,0) node (C2) {}
..controls +(0.3,0) and +(0,0.3)..($(C2)+(0.5,-0.5)$)--++(0,-2.5) node (C3) {}
..controls +(0,-0.6) and +(0,-0.6)..($(C3)+(0.8,0)$)--++(0,4.7) node (C4) {}
..controls +(0,0.3) and +(0.3,0)..($(C4)+(-0.5,0.5)$)--++(-0.8,0)
..controls +(-0.4,0) and +(0,0.4)..(A.34);
\draw[line width=4.5mm, white] (C.210)
..controls +(0,-0.7) and +(0,-0.7)..($(C.210)+(-1,0)$)--++(0,2.5) node (C1) {}
..controls +(0,0.3) and +(-0.3,0)..($(C1)+(0.5,0.5)$)--++(2,0) node (C2) {}
..controls +(0.3,0) and +(0,0.3)..($(C2)+(0.5,-0.5)$)--++(0,-2.5) node (C3) {}
..controls +(0,-0.6) and +(0,-0.6)..($(C3)+(0.8,0)$)--++(0,4.7) node (C4) {}
..controls +(0,0.3) and +(0.3,0)..($(C4)+(-0.5,0.5)$)--++(-0.8,0) node (C5) {}
..controls +(-0.4,0) and +(0,0.4)..(A.34); \node at (-0.95,-4.8) {$r$};
\draw[line width=6mm, blue] (C.320)--++(0,-0.5)
..controls +(0,-1.5) and +(0,-1.5)..($(C.320)+(3.1,-0.5)$)--++(0,6) node (D1) {}
..controls +(0,0.5) and +(0.5,0)..($(D1)+(-1,0.8)$) --++(-1.9,0)
..controls +(-0.9,0) and +(0,0.7)..(A.135);
\draw[line width=5.5mm, white] (C.320)--++(0,-0.5)
..controls +(0,-1.5) and +(0,-1.5)..($(C.320)+(3.1,-0.5)$)--++(0,6) node (D1) {}
..controls +(0,0.5) and +(0.5,0)..($(D1)+(-1,0.8)$) --++(-1.9,0)
..controls +(-0.9,0) and +(0,0.7)..(A.135); \node at (1.9,-5.8) {$m-r$};
\node[fill=white, inner ysep=6mm, inner xsep = 2mm] at (0,-1.2) {};
\draw (A.south)--(B.north); \draw(B.south)--(C.north); \node at (0.3,-3.3) {$X$};
\node at (4.2,-1) {$V^m$};
\end{tikzpicture}}}
\quad
= % \sum_{\varphi\in\ONB(X, V^m)}\ 
\raisebox{-5em}{\scalebox{0.5}{\begin{tikzpicture}
\node[draw, inner xsep = 2em] (A) at (0,0) {$\varphi'$};
\node[draw] (B) at (0,-2.5) {$\theta_X^q$};
\node[draw, inner xsep = 2em] (C) at (0,-4) {$\varphi$};
\draw[line width=5mm, blue] (C.210)
..controls +(0,-0.7) and +(0,-0.7)..($(C.210)+(-1,0)$)--++(0,2.5) node (C1) {}
..controls +(0,0.3) and +(-0.3,0)..($(C1)+(0.5,0.5)$)--++(2,0) node (C2) {}
..controls +(0.3,0) and +(0,-0.3)..($(C2)+(0.5,0.5)$)--++(0,1.3) node (C3) {}
..controls +(0,0.45) and +(0,0.45)..($(C3)+(-1,0)$)
..controls +(0,-0.1) and +(0,0.1)..(A.33.5);
\draw[line width=4.5mm, white] (C.210)
..controls +(0,-0.7) and +(0,-0.7)..($(C.210)+(-1,0)$)--++(0,2.5) node (C1) {}
..controls +(0,0.3) and +(-0.3,0)..($(C1)+(0.5,0.5)$)--++(2,0) node (C2) {}
..controls +(0.3,0) and +(0,-0.3)..($(C2)+(0.5,0.5)$)--++(0,1.3) node (C3) {}
..controls +(0,0.45) and +(0,0.45)..($(C3)+(-1,0)$)
..controls +(0,-0.1) and +(0,0.1)..(A.33.5);
\draw[line width=6mm, blue] (C.320)
..controls +(0,-1.7) and +(0,-1.7)..($(C.320)+(-2.6,0)$)--++(0,4.6) node (D1) {}
..controls +(0,0.3) and +(-0.3,0)..($(D1)+(0.5,0.5)$)--++(1,0) node (D2) {}
..controls +(0.3,0) and +(0,0.3)..($(D2)+(0.5,-0.485)$);
\draw[line width=5.5mm, white] (C.320)
..controls +(0,-1.7) and +(0,-1.7)..($(C.320)+(-2.6,0)$)--++(0,4.6) node (D1) {}
..controls +(0,0.3) and +(-0.3,0)..($(D1)+(0.5,0.5)$)--++(1,0) node (D2) {}
..controls +(0.3,0) and +(0,0.3)..($(D2)+(0.5,-0.485)$);
\node[fill=white, inner ysep=6mm, inner xsep = 2mm] at (0,-1.2) {};
\draw (A.south)--(B.north); \draw(B.south)--(C.north);
\node at (0.3,-3.3) {$X$}; \node at (-3,-3) {$V^m$};
\node at (-0.95,-4.8) {$r$}; \node at (-0.95,-5.5) {$m-r$};
\end{tikzpicture}}}\,.\]
\end{example}

Now we relate the generalized FS-indicator with our quantum invariant. Let $L$ be a Hopf link embedded in $\bS^3$, and let $L_{\epsilon}$ be a tubular neighborhood of $L \subset \bS^3$. Let $\Sigma_L := \partial (\overline{\bS^3\setminus L_{\epsilon}})$ be boundary surface oriented so that the $A$-colored region of $(\bS^3, \Sigma_L)$ is $L_{\epsilon}$. For any $(X, e_X) \in \cZ(\cC)$, $V \in \cC$ and $(m,\ell) \in \bZ^2$, we will define a $\cC$-decorated 3-alterfold using $(\bS^3, \Sigma_L)$ that is related to $\nu_{(m,\ell)}^{(X, e_X)}(V)$. We start by describing some $\cC$-colored graphs on the torus with the following notations.

When $m \ne 0$, then its absolute value $|m| > 0$, and we write $\ell = -q\cdot m+\sgn(m)\cdot r$ for some $q \in \bZ$ and $0 \le r \le |m|-1$, where $\sgn(m) = m/|m|$ is the sign of $m$. In this case, let $\tilde{V} = V^{\sgn(m)}$. When $m=0$, let $\tilde{V} = V^{\sgn(\ell)}$, $q = 0$ and $r = |\ell|$. Let $\Gamma_{(m,r)}(V)$ and $\Phi(X, e_X)$ be two $\cC$-colored graphs such that

\begin{align}\label{eq:Gamma-and-Phi}
(\bS^3, \mathbb{T}, \Gamma_{(m, r)}(V)) = 
\vcenter{\hbox{\scalebox{0.8}{
\begin{tikzpicture}[scale=0.8]
\draw [double distance=0.8cm] (0,0) [partial ellipse=-0.1:180.1:2 and 1.5];
\draw[blue] (0,0) [partial ellipse=0:180:1.8 and 1.3];
\draw[blue] (0,0) [partial ellipse=0:180:2.2 and 1.7];
\draw [blue] (0,0) [partial ellipse=0:180:2.1 and 1.6];
\draw [line width=0.83cm] (0,0) [partial ellipse=180:360:2 and 1.5];
\draw [white, line width=0.8cm] (0,0) [partial ellipse=178:362:2 and 1.5];
\draw[blue] (0,0) [partial ellipse=178:362:1.8 and 1.3];
\draw[blue] (0,0) [partial ellipse=178:362:2.2 and 1.7];
\draw [blue] (0,0) [partial ellipse=178:362:2.1 and 1.6];
\begin{scope}[shift={(-1.98, 0.2)}]
\draw [blue, dashed](0,0) [partial ellipse=0:180:0.5 and 0.3];
\draw [blue] (0,0) [partial ellipse=180:360:0.5 and 0.3];
\end{scope} 
\begin{scope}[shift={(-1.99, -0.1)}]
\draw [blue, dotted](0,0) [partial ellipse=180:145:0.5 and 0.3];
\draw [blue, dotted](0,0) [partial ellipse=0:45:0.5 and 0.3];
\draw [blue] (0,0) [partial ellipse=180:360:0.5 and 0.3];
\end{scope} 
\begin{scope}[shift={(-1.98, -0.2)}]
\draw [fill=white] (-0.28,-0.3)--(-0.31, 0.3)--(0.36,0.3)--(0.39,-0.3)--(-0.28,-0.3);
\end{scope}
\node at (-1.8,-2) {\tiny $\tilde{V}$};
\node at (-0.4,1.4) {\tiny $|m|$};
\end{tikzpicture}}}}
\,,\quad
(\bS^3, \mathbb{T}, \Phi(X, e_X))=
\vcenter{\hbox{\scalebox{0.8}{
\begin{tikzpicture}[scale=0.8]
\begin{scope}
\draw [double distance=0.8cm] (0,0) [partial ellipse=-0.1:180.1:2 and 1.5];
\draw [red] (0,0) [partial ellipse=0:180:2 and 1.5];
\end{scope} % upper
\begin{scope}
\draw [line width=0.83cm] (0,0) [partial ellipse=180:360:2 and 1.5];
\draw [white, line width=0.8cm] (0,0) [partial ellipse=178:362:2 and 1.5];
\draw [red] (0,0) [partial ellipse=178:362:2 and 1.5];
\end{scope}
\begin{scope}[shift={(2, 0)}]
\draw [blue, dashed](0,0) [partial ellipse=0:180:0.5 and 0.3]; 
\node at (-0.75,0) {\tiny $X$};
\draw [white, line width=4pt] (0,0) [partial ellipse=270:290:0.5 and 0.3];
\fill[white] (-0.05,-0.3) circle[radius=2pt];
\draw [blue, ->-=0.3] (0,0) [partial ellipse=180:360:0.5 and 0.3];
\fill [blue] (0.2,-0.275) circle[radius=2pt];
\end{scope}
\node at (-1.8,-2) {\tiny \textcolor{white}{$(\tilde{V})$}};
\end{tikzpicture}}}}\,.
\end{align}
In $(\bS^3, \mathbb{T}, \Gamma_{(m,r)}(V))$, all strings are colored by $\tilde{V}$. When $m \ne 0$, the box is filled by the diagram 
$\raisebox{-1em}{\begin{tikzpicture}[scale=1.5]
\draw (-0.5, -0.4) rectangle (0.5,0.4);
\draw [blue, ->-=0.5] (0.3, 0.4) .. controls +(0, -0.15) and +(-0.1, 0) .. (0.5, 0.2);
\draw [blue, ->-=0.45] (-0.1, 0.4) .. controls +(0.08, -0.25) and +(-0.2, 0) .. (0.5, -0.2);
\draw [blue, ->-=0.5] (-0.3, 0.4) --(0.3, -0.4);
\draw [blue, -<-=0.5] (-0.3, -0.4) .. controls +(0, 0.15) and +(0.1, 0) .. (-0.5, -0.2);
\draw [blue, -<-=0.5] (0.1,- 0.4) .. controls +(-0.1, 0.25) and +(0.2, 0) .. (-0.5, 0.2);
\draw [decorate, decoration = {brace}] (-0.3,0.42) --  (0.3,0.42) node [pos=0.5, above] {\tiny $|m|$} ;
\draw [decorate, decoration = {brace}] (0.52,0.2) --  (0.52,-0.2) node [pos=0.5, right] {\tiny $r$} ;
\end{tikzpicture}}$.
When $m=0$, it is filled with 
$\raisebox{-1em}{\begin{tikzpicture}[scale=1.5]
\draw (-0.5, -0.4) rectangle (0.5,0.4);
\draw[blue, ->-=0.5] (-0.5,0.2)--(0.5,0.2);\draw[blue, ->-=0.5] (-0.5,-0.2)--(0.5,-0.2);
\draw [decorate, decoration = {brace}] (0.52,0.2) --  (0.52,-0.2) node [pos=0.5, right] {\tiny $r$} ;
\end{tikzpicture}}$
of $r$ parallel $V$-colored strings, which can be viewed as a marginal case for the $\cC$-decoration of the box when $m \ne 0$. Therefore, $\Gamma_{(m,r)}(V)$ is a $\tilde{V}$-colored multicurve on $\mathbb{T}$ of the homology class $r$-times the meridian and $|m|$-times the longitude. We also call such a multicurve an $(|m|, r)$-curve on the torus\footnote{The homology class of curve in the complement torus is $|m|$ times the meridian and $-r$ times the longitude.}
In $(\bS^3, \mathbb{T}, \Phi(X, e_X))$, the blue bullet represents $\theta_X^q$. For fixed $V$ and $(X, e_X)$, we simply write $\Gamma_{(m,r)}$ for $\Gamma_{(m,r)}(V)$, and $\Phi$ for $\Phi(X, e_X)$. 

Now we define a $\cC$-decorated 3-alterfold $(\bS^3, \Sigma_L, \Gamma_{(m,r)} \sqcup \Phi)$ by the following picture

\begin{equation}\label{eq:gamma-L}
(\bS^3, \Sigma_L, \Gamma_{(m,r)}(V) \sqcup \Phi(X, e_X)) 
% = (\bS^3, \Sigma_L, \Gamma_{(m,r)} \sqcup \Phi) 
:= 
\vcenter{\hbox{\scalebox{0.8}{
\begin{tikzpicture}[scale=0.8]
\draw [double distance=0.8cm] (0,0) [partial ellipse=-0.1:180.1:2 and 1.5];
\draw[blue] (0,0) [partial ellipse=0:180:1.8 and 1.3];
\draw[blue] (0,0) [partial ellipse=0:180:2.2 and 1.7];
\draw [blue] (0,0) [partial ellipse=0:180:2.1 and 1.6];
\begin{scope}[shift={(2, 0)}]
\draw [double distance=0.8cm] (0,0) [partial ellipse=-0.1:180.1:2 and 1.5];
\draw [red] (0,0) [partial ellipse=0:180:2 and 1.5];
\end{scope} % upper
\begin{scope}[shift={(2, 0)}]
\draw [line width=0.83cm] (0,0) [partial ellipse=180:360:2 and 1.5];
\draw [white, line width=0.8cm] (0,0) [partial ellipse=178:362:2 and 1.5];
\draw [red] (0,0) [partial ellipse=178:362:2 and 1.5];
\end{scope}
\draw [line width=0.83cm] (0,0) [partial ellipse=180:360:2 and 1.5];
\draw [white, line width=0.8cm] (0,0) [partial ellipse=178:362:2 and 1.5];
\draw[blue] (0,0) [partial ellipse=178:362:1.8 and 1.3];
\draw[blue] (0,0) [partial ellipse=178:362:2.2 and 1.7];
\draw [blue] (0,0) [partial ellipse=178:362:2.1 and 1.6];
\begin{scope}[shift={(-1.98, 0.2)}]
\draw [blue, dashed](0,0) [partial ellipse=0:180:0.5 and 0.3];
\draw [blue] (0,0) [partial ellipse=180:360:0.5 and 0.3];
\end{scope} 
\begin{scope}[shift={(-1.99, -0.1)}]
\draw [blue, dashed](0,0) [partial ellipse=180:145:0.5 and 0.3];
\draw [blue, dashed](0,0) [partial ellipse=0:45:0.5 and 0.3];
\draw [blue] (0,0) [partial ellipse=180:360:0.5 and 0.3];
\end{scope} 
\begin{scope}[shift={(-1.98, -0.2)}]
\draw [fill=white] (-0.28,-0.3)--(-0.31, 0.3)--(0.36,0.3)--(0.39,-0.3)--(-0.28,-0.3);
\end{scope}
% lower 
\begin{scope}[shift={(4, 0)}]
\draw [blue, dashed](0,0) [partial ellipse=0:180:0.5 and 0.3]; 
\node at (-0.75,0) {\tiny $X$};
\draw [white, line width=4pt] (0,0) [partial ellipse=270:290:0.5 and 0.3];
\fill[white] (-0.05,-0.3) circle[radius=2pt];
\draw [blue, ->-=0.3] (0,0) [partial ellipse=180:360:0.5 and 0.3];
\fill [blue] (0.2,-0.275) circle[radius=2pt];
\end{scope}
\node at (-1.8,-2.2) {\tiny $\tilde{V}$}; \node at (-0.4,1.4) {\tiny $|m|$};
\end{tikzpicture}}}}\,,
\end{equation}
where the torus on the left, denoted by $\mathbb{T}_{V}$ for later use, has $\Gamma_{(m,r)}$ on it, and that on the right, denoted by $\mathbb{T}_{(X, e_X)}$ has $\Phi$ on it. Of course, $\Sigma_L = \mathbb{T}_V \sqcup \mathbb{T}_{(X, e_X)}$. 

As an example, consider the case when $m = 0$. Then
\begin{align*}
(\bS^3, \Sigma_L, \Gamma_{(0,r)} \sqcup \Phi) =
\vcenter{\hbox{\scalebox{0.8}{
\begin{tikzpicture}[scale=0.8]
\draw [double distance=0.8cm] (0,0) [partial ellipse=-0.1:180.1:2 and 1.5];
\begin{scope}[shift={(2, 0)}]
\draw [double distance=0.8cm] (0,0) [partial ellipse=-0.1:180.1:2 and 1.5];
\draw [red] (0,0) [partial ellipse=0:180:2 and 1.5];
\end{scope} % upper
\begin{scope}[shift={(2, 0)}]
\draw [line width=0.83cm] (0,0) [partial ellipse=180:360:2 and 1.5];
\draw [white, line width=0.8cm] (0,0) [partial ellipse=178:362:2 and 1.5];
\draw [red] (0,0) [partial ellipse=178:362:2 and 1.5];
\end{scope}
\draw [line width=0.83cm] (0,0) [partial ellipse=180:360:2 and 1.5];
\draw [white, line width=0.8cm] (0,0) [partial ellipse=178:362:2 and 1.5];
\begin{scope}[shift={(-1.98, 0.2)}]
\draw [blue, dotted](0,0) [partial ellipse=0:180:0.5 and 0.3];
\draw [blue, ->-=0.4] (0,0) [partial ellipse=180:360:0.5 and 0.3];
\end{scope} 
\begin{scope}[shift={(-1.99, 0)}]
\draw [blue, dotted](0,0) [partial ellipse=0:180:0.5 and 0.3];
\draw [blue, ->-=0.4] (0,0) [partial ellipse=180:360:0.5 and 0.3];
\end{scope} 
\begin{scope}[shift={(-1.97, -0.2)}]
\draw [blue, dotted](0,0) [partial ellipse=0:180:0.51 and 0.3];
\draw [blue, ->-=0.4] (0,0) [partial ellipse=180:360:0.51 and 0.3];
\end{scope} 
% lower 
\begin{scope}[shift={(4, 0)}]
\draw [blue, dashed](0,0) [partial ellipse=0:180:0.5 and 0.3]; 
\node at (-0.75,0) {\tiny $X$};
\draw [white, line width=4pt] (0,0) [partial ellipse=270:290:0.5 and 0.3];
\fill[white] (-0.05,-0.3) circle[radius=2pt];
\draw [blue, ->-=0.3] (0,0) [partial ellipse=180:360:0.5 and 0.3];
\end{scope}
\draw [decorate, decoration = {brace}] (-2.55,-0.3) --  (-2.55,0.3) node [pos=0.5, left] {\tiny $\tilde{V}^{|\ell|}$} ;
\end{tikzpicture}}}}\,.
\end{align*}

\begin{theorem}\label{thm:nu-and-z}
Let $\mathcal{C}$ be a spherical fusion category. For any $(X, e_X) \in \cZ(\cC)$, $V \in \cC$ and $(m, \ell) \in \bZ^2$, 

\[\nu_{(m, \ell)}^{(X, e_X)}(V)=\frac{1}{\mu} Z(\bS^3, \Sigma_L, \Gamma_{(m, r)}\sqcup \Phi)\,.\]
\end{theorem}
\begin{proof}
We assume that $m >0$.
Now we evaluate graphically the right hand side and obtain that

\begin{align*}
&\frac{1}{\mu} Z(\bS^3, \Sigma_L, \Gamma_{(m, r)}\sqcup \Phi)\xlongequal{\text{Move 1}}
\frac{1}{\mu}\vcenter{\hbox{\scalebox{0.8}{
\begin{tikzpicture}[scale=0.8]
\draw [double distance=0.8cm] (0,0) [partial ellipse=-0.1:180.1:2 and 1.5];
\draw [blue, double distance=0.4cm] (0,0) [partial ellipse=0:180:2 and 1.5];
\draw [blue] (0,0) [partial ellipse=0:180:2.1 and 1.6];
\node at (-0.4,1.4) {\tiny $m$};
\begin{scope}[shift={(2, 0)}]
\draw [double distance=0.8cm] (0,0) [partial ellipse=-0.1:180.1:2 and 1.5];
\draw [red] (0,0) [partial ellipse=0:180:2 and 1.5];
\end{scope} % upper
\begin{scope}[shift={(2, 0)}]
\draw [line width=0.83cm] (0,0) [partial ellipse=180:360:2 and 1.5];
\draw [white, line width=0.8cm] (0,0) [partial ellipse=178:362:2 and 1.5];
\draw [red] (0,0) [partial ellipse=178:362:2 and 1.5];
\end{scope}
\draw [double distance=0.8cm] (0,0) [partial ellipse=180:360:2 and 1.5];
\draw [blue, double distance=0.4cm] (0,0) [partial ellipse=180:360:2 and 1.5];
\draw [blue] (0,0) [partial ellipse=180:360:2.1 and 1.6];
%\draw [blue, double distance=0.13cm] (0,0) [partial ellipse=180:360:2 and 1.5];
\begin{scope}[shift={(-1.98, 0.2)}]
\draw [blue, dashed](0,0) [partial ellipse=0:180:0.5 and 0.3];
\draw [blue] (0,0) [partial ellipse=180:360:0.5 and 0.3];
\end{scope} 
% \begin{scope}[shift={(-2, 0)}]
% \draw [blue, dashed](0,0) [partial ellipse=0:180:0.5 and 0.3];
% \draw [blue] (0,0) [partial ellipse=180:360:0.5 and 0.3];
% \end{scope} 
\begin{scope}[shift={(-1.98, -0.2)}]
\draw [blue, dashed](0,0) [partial ellipse=0:180:0.5 and 0.3];
\draw [blue] (0,0) [partial ellipse=180:360:0.5 and 0.3];
\end{scope} 
\begin{scope}[shift={(-2, -0.2)}]
\draw [fill=white] (-0.3, -0.33) rectangle (0.4, 0.34);
\end{scope}% lower 
\begin{scope}[shift={(4, 0)}]
\draw [blue, dashed](0,0) [partial ellipse=0:180:0.5 and 0.3] node [below left, black] {\tiny $X$} ;
\draw [white, line width=4pt] (0,0) [partial ellipse=270:250:0.5 and 0.3];
\draw [blue, -<-=0.3] (0,0) [partial ellipse=180:360:0.5 and 0.3] node [pos=0.7] {\tiny $\bullet$};
\end{scope}
\draw  [line width=0.5cm] (0.55, 0)--(1.45, 0);
\draw  [line width=0.47cm, white] (0.5, 0)--(1.5, 0);
 \draw [red] (1,0) [partial ellipse=90:270:0.15 and 0.3];
  \draw [red,dashed] (1,0) [partial ellipse=90:-90:0.15 and 0.3];
\end{tikzpicture}}}} 
= 
\vcenter{\hbox{\scalebox{0.8}{
\begin{tikzpicture}[scale=0.8]
\draw [double distance=0.8cm] (0,0) [partial ellipse=-0.1:180.1:2 and 1.5];
\draw [blue, double distance=0.4cm] (0,0) [partial ellipse=0:180:2 and 1.5];
\draw [blue] (0,0) [partial ellipse=0:180:2.1 and 1.6];
\node at (-0.4,1.4) {\tiny $m$};
%\draw [blue, double distance=0.13cm] (0,0) [partial ellipse=0:180:2 and 1.5];
\begin{scope}[shift={(2, 0)}]
\draw [double distance=0.8cm] (0,0) [partial ellipse=-0.1:180.1:2 and 1.5];
\draw [red] (0,0) [partial ellipse=0:180:2 and 1.5];
\end{scope} % upper
\begin{scope}[shift={(2, 0)}]
\draw [line width=0.83cm] (0,0) [partial ellipse=180:360:2 and 1.5];
\draw [white, line width=0.8cm] (0,0) [partial ellipse=178:362:2 and 1.5];
\draw [red] (0,0) [partial ellipse=178:362:2 and 1.5];
\end{scope}
\draw [double distance=0.8cm] (0,0) [partial ellipse=180:360:2 and 1.5];
\draw [blue, double distance=0.4cm] (0,0) [partial ellipse=180:360:2 and 1.5];
\draw [blue] (0,0) [partial ellipse=180:360:2.1 and 1.6];
%\draw [blue, double distance=0.13cm] (0,0) [partial ellipse=180:360:2 and 1.5];
\begin{scope}[shift={(-1.98, 0.2)}]
\draw [blue, dashed](0,0) [partial ellipse=0:180:0.5 and 0.3];
\draw [blue] (0,0) [partial ellipse=180:360:0.5 and 0.3];
\end{scope} 
% \begin{scope}[shift={(-2, 0)}]
% \draw [blue, dashed](0,0) [partial ellipse=0:180:0.5 and 0.3];
% \draw [blue] (0,0) [partial ellipse=180:360:0.5 and 0.3];
% \end{scope} 
\begin{scope}[shift={(-1.98, -0.2)}]
\draw [blue, dashed](0,0) [partial ellipse=0:180:0.5 and 0.3];
\draw [blue] (0,0) [partial ellipse=180:360:0.5 and 0.3];
\end{scope} 
\begin{scope}[shift={(-2, -0.2)}]
\draw [fill=white] (-0.3, -0.33) rectangle (0.4, 0.34);
\end{scope}% lower 
\begin{scope}[shift={(4, 0)}]
\draw [blue, dashed](0,0) [partial ellipse=0:180:0.5 and 0.3] node [below left, black] {\tiny $X$};
\draw [white, line width=4pt] (0,0) [partial ellipse=270:250:0.5 and 0.3];
\draw [blue, -<-=0.3] (0,0) [partial ellipse=180:360:0.5 and 0.3] node [pos=0.7] {\tiny $\bullet$};
\end{scope}
\draw  [line width=0.5cm] (0.55, 0)--(1.45, 0);
\draw  [line width=0.47cm, white] (0.5, 0)--(1.5, 0);
% \draw [red] (1,0) [partial ellipse=90:270:0.15 and 0.3];
% \draw [red,dashed] (1,0) [partial ellipse=90:-90:0.15 and 0.3];
\end{tikzpicture}}}} \\
\xlongequal{\text{Move 2}}&
\vcenter{\hbox{\scalebox{0.8}{
\begin{tikzpicture}[scale=0.8]
\draw [double distance=0.8cm] (0,0) [partial ellipse=0:360:2 and 1.5];
\draw [blue, double distance=0.4cm] (0,0) [partial ellipse=0:360:2 and 1.5];
\draw [blue] (0,0) [partial ellipse=0:360:2.1 and 1.6];
\node at (-0.4,1.4) {\tiny $m$};
%\draw [blue, double distance=0.13cm] (0,0) [partial ellipse=0:360:2 and 1.5];
\begin{scope}[shift={(-1.98, 0.2)}]
\draw [blue, dashed](0,0) [partial ellipse=0:180:0.5 and 0.3];
\draw [blue] (0,0) [partial ellipse=180:360:0.5 and 0.3];
\end{scope} 
% \begin{scope}[shift={(-2, 0)}]
% \draw [blue, dashed](0,0) [partial ellipse=0:180:0.5 and 0.3];
% \draw [blue] (0,0) [partial ellipse=180:360:0.5 and 0.3];
% \end{scope} 
\begin{scope}[shift={(-1.98, -0.2)}]
\draw [blue, dashed](0,0) [partial ellipse=0:180:0.5 and 0.3];
\draw [blue] (0,0) [partial ellipse=180:360:0.5 and 0.3];
\end{scope} 
\begin{scope}[shift={(-2, -0.2)}]
\draw [fill=white] (-0.3, -0.33) rectangle (0.4, 0.34);
\end{scope}
\draw [white, line width=0.45cm] (0,0) [partial ellipse=-35:35:2 and 1.5];
\begin{scope}[shift={(2, 0)}]
\draw [red, dashed](0,0) [partial ellipse=0:180:0.5 and 0.3];
\draw [red] (0,0) [partial ellipse=180:360:0.5 and 0.3];
\draw [white, line width=4pt] (0,0) [partial ellipse=270:250:0.5 and 0.3];
p\end{scope}
\draw [blue, -<-=0.5] (0,0) [partial ellipse=-35:35:2 and 1.5] node [pos=0.5, right] {\tiny $X$} node [pos=0.7] {\tiny $\bullet$};
\draw [fill=white] (1.3, 0.6) rectangle (2, 1.1); \node at (1.65, 0.85) {\tiny $\varphi$};
\draw [fill=white] (1.3, -0.6) rectangle (2, -1.1); \node at (1.65, -0.85) {\tiny $\varphi'$};
\end{tikzpicture}}}}
\xlongequal{\text{Move 2}}
\vcenter{\hbox{\scalebox{0.8}{
\begin{tikzpicture}[scale=0.8]
\draw (0,0) [partial ellipse=0:360:2.2 and 2.2];
\draw [blue, double distance=0.4cm] (0,0) [partial ellipse=0:360:1.6 and 1.6];
\draw [blue] (0,0) [partial ellipse=0:360:1.7 and 1.7];
\node at (-0.4,1.45) {\tiny $m$};
%\draw [blue, double distance=0.13cm] (0,0) [partial ellipse=0:360:1.6 and 1.6];
% \draw [blue, dashed](0,0) [partial ellipse=0:180:2.2 and 0.6];
% \draw [blue] (0,0) [partial ellipse=180:360:2.2 and 0.6];
\begin{scope}[shift={(0, 0.2)}]
\draw [blue, dashed](0,0) [partial ellipse=0:180:2.18 and 0.6];
\draw [blue] (0,0) [partial ellipse=180:360:2.18 and 0.6];
\end{scope}
\begin{scope}[shift={(0, -0.2)}]
\draw [blue, dashed](0,0) [partial ellipse=0:180:2.18 and 0.6];
\draw [blue] (0,0) [partial ellipse=180:360:2.18 and 0.6];
\end{scope}
\draw [fill=white] (-1.9, -0.8) rectangle (-1, 0);
\draw [white, double distance=0.45cm] (0,0) [partial ellipse=50:-50:1.6 and 1.6];
\draw [blue, ->-=0.5] (0,0) [partial ellipse=50:-50:1.6 and 1.6] node [pos=0.5, left] {\tiny $X$}  node [pos=0.7] {\tiny $\bullet$};
\draw [fill=white] (0.5, 1) rectangle (1.4, 1.5); \node at (0.95, 1.25) {\tiny $\varphi$};
\draw [fill=white] (0.5, -1) rectangle (1.4, -1.5); \node at (0.95, -1.25) {\tiny $\varphi'$};
\end{tikzpicture}}}}
= 
\raisebox{-5em}{\scalebox{0.5}{\begin{tikzpicture}
\node[draw, inner xsep = 2em] (A) at (0,0) {$\varphi$};
\node[draw, inner xsep = 2em] (C) at (0,-4) {$\varphi'$};
\draw[line width=5mm, blue] (C.210)
..controls +(0,-0.7) and +(0,-0.7)..($(C.210)+(-1,0)$)--++(0,2.5) node (C1) {}
..controls +(0,0.3) and +(-0.3,0)..($(C1)+(0.5,0.5)$)--++(2,0) node (C2) {}
..controls +(0.3,0) and +(0,-0.3)..($(C2)+(0.5,0.5)$)--++(0,1.3) node (C3) {}
..controls +(0,0.45) and +(0,0.45)..($(C3)+(-1,0)$)
..controls +(0,-0.1) and +(0,0.1)..(A.33.5);
\draw[line width=4.5mm, white] (C.210)
..controls +(0,-0.7) and +(0,-0.7)..($(C.210)+(-1,0)$)--++(0,2.5) node (C1) {}
..controls +(0,0.3) and +(-0.3,0)..($(C1)+(0.5,0.5)$)--++(2,0) node (C2) {}
..controls +(0.3,0) and +(0,-0.3)..($(C2)+(0.5,0.5)$)--++(0,1.3) node (C3) {}
..controls +(0,0.45) and +(0,0.45)..($(C3)+(-1,0)$)
..controls +(0,-0.1) and +(0,0.1)..(A.33.5);
\draw[line width=6mm, blue] (C.320)
..controls +(0,-1.7) and +(0,-1.7)..($(C.320)+(-2.6,0)$)--++(0,4.6) node (D1) {}
..controls +(0,0.3) and +(-0.3,0)..($(D1)+(0.5,0.5)$)--++(1,0) node (D2) {}
..controls +(0.3,0) and +(0,0.3)..($(D2)+(0.5,-0.485)$);
\draw[line width=5.5mm, white] (C.320)
..controls +(0,-1.7) and +(0,-1.7)..($(C.320)+(-2.6,0)$)--++(0,4.6) node (D1) {}
..controls +(0,0.3) and +(-0.3,0)..($(D1)+(0.5,0.5)$)--++(1,0) node (D2) {}
..controls +(0.3,0) and +(0,0.3)..($(D2)+(0.5,-0.485)$);
\node[fill=white, inner ysep=6mm, inner xsep = 2mm] at (0,-1.2) {};
\draw (A.south)--(C.north);
\node at (0.3,-3.3) {$X$}; \node at (-3,-3) {$V^m$};
\node at (-0.95,-4.8) {$r$}; \node at (-0.95,-5.5) {$m-r$};
\node (B) at (0,-2.5) {\textcolor{blue}{$\bullet$}};
\end{tikzpicture}}}
=\nu_{m, \ell}^{(X, e_X)}(V),
\end{align*}
where the rectangle is filled with
$\raisebox{-1em}{\begin{tikzpicture}[scale=1.5]
\draw (-0.5, -0.4) rectangle (0.5,0.4);
\draw [blue, ->-=0.5] (0.3, 0.4) .. controls +(0, -0.15) and +(-0.1, 0) .. (0.5, 0.2);
\draw [blue, ->-=0.45] (-0.1, 0.4) .. controls +(0.08, -0.25) and +(-0.2, 0) .. (0.5, -0.2);
\draw [blue, ->-=0.5] (-0.3, 0.4) --(0.3, -0.4);
\draw [blue, -<-=0.5] (-0.3, -0.4) .. controls +(0, 0.15) and +(0.1, 0) .. (-0.5, -0.2);
\draw [blue, -<-=0.5] (0.1,- 0.4) .. controls +(-0.1, 0.25) and +(0.2, 0) .. (-0.5, 0.2);
\draw [decorate, decoration = {brace}] (-0.3,0.42) --  (0.3,0.42) node [pos=0.5, above] {\tiny $m$} ;
\draw [decorate, decoration = {brace}] (0.52,0.2) --  (0.52,-0.2) node [pos=0.5, right] {\tiny $r$} ;
\end{tikzpicture}}$, and the second equality follows from handle slide. This completes the proof.
\end{proof}

\subsection{$S$-matrix}
We will study the topologized of $S$-matrix in our framework.

\begin{lemma}\label{lem:unit4}
Suppose that $(X, e_X), (Y, e_Y)$ are simple objects in $\mathcal{Z}(\mathcal{C})$.
Then we have
\begin{align*}
\vcenter{\hbox{\scalebox{0.6}{
\begin{tikzpicture}[scale=0.8]
\begin{scope}[shift={(1.5,1.5)}]
\draw (-1, 1.25) .. controls +(0, 0.3) and +(0, 0.3).. (1, 1.25);
\draw  (1,1.25)  .. controls +(0, -0.3) and +(0, -0.3).. (-1, 1.25);
\end{scope}
\draw (0.5, 2.75) .. controls +(0, -0.5) and +(0, 0.5) .. (-1, 1.25);
\draw (2.5, 2.75) .. controls +(0, -0.5) and +(0, 0.5) .. (4, 1.25);
\draw (1, 1.25)  .. controls +(0, 0.5) and +(0, 0.5) ..  (2, 1.25);
\draw [blue, dashed] (1.5, 1.65) .. controls +(0.2, 0) and +(0, -0.3).. (1.8, 3) node [above, black] {\tiny $X$};
\draw [blue, -<-=0.5] (1.5, 1.65) .. controls +(-0.2, 0) and +(0, -0.3).. (1.2, 2.5);
\draw [dashed] (-1, 1.25) .. controls +(0, 0.3) and +(0, 0.3).. (1, 1.25);
\draw [dashed]  (1,1.25)  .. controls +(0, -0.3) and +(0, -0.3).. (-1, 1.25);
\begin{scope}[shift={(0, -2.5)}]
\draw  [dashed] (-1, 1.25) .. controls +(0, 0.3) and +(0, 0.3).. (1, 1.25);
\draw [dashed]  (1,1.25)  .. controls +(0, -0.3) and +(0, -0.3).. (-1, 1.25);
\end{scope}
\begin{scope}[shift={(3, 0)}]
\draw [dashed] (-1, 1.25) .. controls +(0, 0.3) and +(0, 0.3).. (1, 1.25);
\draw [dashed]  (1,1.25)  .. controls +(0, -0.3) and +(0, -0.3).. (-1, 1.25);
\end{scope}
\begin{scope}[shift={(3, -2.5)}]
\draw [dashed] (-1, 1.25) .. controls +(0, 0.3) and +(0, 0.3).. (1, 1.25);
\draw [dashed]  (1,1.25)  .. controls +(0, -0.3) and +(0, -0.3).. (-1, 1.25);
\end{scope}
\draw  (-1, 1.25)--(-1, -1.25) (1, 1.25)--(1, -1.25);
\begin{scope}[shift={(3, 0)}]
\draw  (-1, 1.25)--(-1, -1.25) (1, 1.25)--(1, -1.25);
\end{scope}
\begin{scope}[shift={(1.5,-4)}]
\draw [dashed] (-1, 1.25) .. controls +(0, 0.3) and +(0, 0.3).. (1, 1.25);
\draw   (1,1.25)  .. controls +(0, -0.3) and +(0, -0.3).. (-1, 1.25);
\end{scope}
\draw  (0.5, -2.75) .. controls +(0, 0.5) and +(0, -0.5) .. (-1, -1.25);
\draw  (2.5, -2.75) .. controls +(0, 0.5) and +(0, -0.5) .. (4, -1.25);
\draw  (1, -1.25)  .. controls +(0, -0.5) and +(0, -0.5) ..  (2, -1.25);
\draw [blue, ->-=0.5] (1.5, -1.65) .. controls +(0.2, 0) and +(0, 0.3).. (1.8, -2.95) node [below, black] {\tiny $Y$};
\draw [blue, dashed] (1.5, -1.65) .. controls +(-0.2, 0) and +(0, 0.3).. (1.2, -2.5);
\draw [red](1.5,0) [partial ellipse=105:275:1 and 2];
\draw [red](1.5,0) [partial ellipse=-75:95:1 and 2];
\end{tikzpicture}}}}
=\frac{\mu}{d(X)}\delta_{(X, e_X), (Y, e_Y)}
\vcenter{\hbox{\scalebox{0.7}{
\begin{tikzpicture}[scale=0.35]
\draw (0,5) [partial ellipse=0:360:2 and 0.8];
\draw (-2, 5)--(-2, -4);
\draw (2, 5)--(2, -4);
\draw[dashed] (0,-4) [partial ellipse=0:180:2 and 0.8];
\draw (0,-4) [partial ellipse=180:360:2 and 0.8];
\draw[blue, ->-=0.5] (-0.5, 4.2)--(-0.5, -4.8)node[below, black]{\tiny{${X}$}};
\draw[blue,dashed, ->-=0.5] (0.5, 5.8)--(0.5, -3.2)node[below, black]{\tiny{${X^*}$}};
\end{tikzpicture}}}}
\end{align*}
\end{lemma}
\begin{proof}
By Move 2, we see that the left hand side is 
\begin{align*}
\vcenter{\hbox{\scalebox{0.7}{
\begin{tikzpicture}[yscale=0.35, xscale=0.6]
\draw (0,5) [partial ellipse=0:360:2 and 0.8];
\draw (-2, 5)--(-2, -4);
\draw (2, 5)--(2, -4);
\draw[dashed] (0,-4) [partial ellipse=0:180:2 and 0.8];
\draw (0,-4) [partial ellipse=180:360:2 and 0.8];
\draw[blue,dashed, ->-=0.2] (0.5, 5.8) node [above, black] {\tiny $X^*$}--(0.5, -3.2)node[below, black]{\tiny{$Y^*$}} node [pos=0.5, draw, fill=white] {\tiny $\varphi'$};
\draw[blue, ->-=0.2] (-0.5, 4.2)--(-0.5, -4.8) node[below, black]{\tiny{$Y$}} node [pos=0.5, draw, fill=white] {\tiny $\varphi$};
\begin{scope}[shift={(-0.5, -0.3)}]
\draw [red] (0, 0) [partial ellipse=95:265:1 and 2];
\draw [red] (0, 0) [partial ellipse=85:-85:1 and 2];
\end{scope}
\end{tikzpicture}}}}
=
\frac{\mu}{d(X)}\delta_{(X, e_X), (Y, e_Y)}
\vcenter{\hbox{\scalebox{0.7}{
\begin{tikzpicture}[scale=0.35]
\draw (0,5) [partial ellipse=0:360:2 and 0.8];
\draw (-2, 5)--(-2, -4);
\draw (2, 5)--(2, -4);
\draw[dashed] (0,-4) [partial ellipse=0:180:2 and 0.8];
\draw (0,-4) [partial ellipse=180:360:2 and 0.8];
\draw[blue, ->-=0.5] (-0.5, 4.2)--(-0.5, -4.8)node[below, black]{\tiny{${X}$}};
\draw[blue,dashed, ->-=0.5] (0.5, 5.8)--(0.5, -3.2)node[below, black]{\tiny{${X^*}$}};
\end{tikzpicture}}}}
\end{align*}
where the equality follows from the property of projector $P_{\mathcal{C}}$.
\end{proof}

\begin{proposition}\label{prop:centeridempotent}
Suppose that $(X, e_X), (Y, e_Y)$ are simple objects in $\mathcal{Z}(\mathcal{C})$.
Then for all $V\in \cC$ we have that 
\begin{align*}
\vcenter{\hbox{\scalebox{0.5}{
\begin{tikzpicture}[xscale=0.8, yscale=0.6]
\draw [double distance=0.8cm] (0,0) [partial ellipse=-0.1:180.1:2 and 1.5];
\draw [red] (0,0) [partial ellipse=0:180:2 and 1.5];
\path [fill=white](-0.65, 0) rectangle (0.65, 2.5);
\begin{scope}[shift={(0,3)}]
\draw (0,0) [partial ellipse=0:360:0.6 and 0.3];
\end{scope}
\draw (-0.6, 3)--(-0.6, 0) (0.6, 3)--(0.6, 0); 
\draw [blue, -<-=0.5] (0, 0)--(0, 2.7) node [pos=0.5, right, black] {\tiny $V$};% upper one
\draw (-0.6, -3)--(-0.6, 0) (0.6, -3)--(0.6, 0);
\draw [blue] (0, 0)--(0, -3.3);
\draw [line width=0.83cm] (0,0) [partial ellipse=180:360:2 and 1.5];
\draw [white, line width=0.8cm] (0,0) [partial ellipse=178:362:2 and 1.5];
\draw [red] (0,0) [partial ellipse=178:362:2 and 1.5];
\begin{scope}[shift={(-2, 0)}]
\draw [blue, dashed](0,0) [partial ellipse=0:180:0.5 and 0.3] node [below left, black] {\tiny $X$};
\draw [white, line width=4pt] (0,0) [partial ellipse=270:290:0.5 and 0.3];
\draw [blue, ->-=0.3] (0,0) [partial ellipse=180:360:0.5 and 0.3];
\end{scope}
\begin{scope}[shift={(0, -5)}] % The down one
\draw [double distance=0.8cm] (0,0) [partial ellipse=-0.1:180.1:2 and 1.5];
\draw [red] (0,0) [partial ellipse=0:180:2 and 1.5];
\path [fill=white](-0.65, 0) rectangle (0.65, 2.5);
\draw (-0.6, 3)--(-0.6, 0) (0.6, 3)--(0.6, 0); 
\draw [blue] (0, 0)--(0, 2.7);% upper one
\draw (-0.6, -3)--(-0.6, 0) (0.6, -3)--(0.6, 0);
\draw [blue] (0, 0)--(0, -3.3);
\draw [line width=0.83cm] (0,0) [partial ellipse=180:360:2 and 1.5];
\draw [white, line width=0.8cm] (0,0) [partial ellipse=178:362:2 and 1.5];
\draw [red] (0,0) [partial ellipse=178:362:2 and 1.5];
\begin{scope}[shift={(0,-3)}]
\draw [dashed](0,0) [partial ellipse=0:180:0.6 and 0.3];
\draw (0,0) [partial ellipse=180:360:0.6 and 0.3];
\end{scope}
\begin{scope}[shift={(-2, 0)}]
\draw [blue, dashed](0,0) [partial ellipse=0:180:0.5 and 0.3] node [below left, black] {\tiny $Y$};
\draw [white, line width=4pt] (0,0) [partial ellipse=270:290:0.5 and 0.3];
\draw [blue, ->-=0.3] (0,0) [partial ellipse=180:360:0.5 and 0.3];
\end{scope}
\end{scope}
\end{tikzpicture}}}}
=\frac{\mu^2}{d(X)} \delta_{(X, e_X), (Y, e_Y)}
\vcenter{\hbox{\scalebox{0.5}{
\begin{tikzpicture}[xscale=0.8, yscale=0.6]
\draw [line width=0.83cm] (0,0) [partial ellipse=-0.1:180.1:2 and 1.5];
\draw [white, line width=0.8cm] (0,0) [partial ellipse=-0.1:180.1:2 and 1.5];
\draw [red] (0,0) [partial ellipse=0:180:2 and 1.5];
\path [fill=white](-0.65, 0) rectangle (0.65, 2.5);
\begin{scope}[shift={(0,3)}]
\draw (0,0) [partial ellipse=0:360:0.6 and 0.3];
\end{scope}
\draw (-0.6, 3)--(-0.6, 0) (0.6, 3)--(0.6, 0); 
\draw [blue, -<-=0.5] (0, 0)--(0, 2.7) node [pos=0.5, right, black] {\tiny $V$};% upper one
\draw (-0.6, -3)--(-0.6, 0) (0.6, -3)--(0.6, 0);
\draw [blue] (0, 0)--(0, -3.3);
\draw [line width=0.83cm] (0,0) [partial ellipse=180:360:2 and 1.5];
\draw [white, line width=0.8cm] (0,0) [partial ellipse=178:362:2 and 1.5];
\draw [red] (0,0) [partial ellipse=178:362:2 and 1.5];
\begin{scope}[shift={(0,-3)}]
\draw [dashed](0,0) [partial ellipse=0:180:0.6 and 0.3];
\draw (0,0) [partial ellipse=180:360:0.6 and 0.3];
\end{scope}
\begin{scope}[shift={(-2, 0)}]
\draw [blue, dashed](0,0) [partial ellipse=0:180:0.5 and 0.3] node [below left, black] {\tiny $X$};
\draw [white, line width=4pt] (0,0) [partial ellipse=270:290:0.5 and 0.3];
\draw [blue, ->-=0.3] (0,0) [partial ellipse=180:360:0.5 and 0.3];
\end{scope}
\end{tikzpicture}}}}
\end{align*}
\end{proposition}
\begin{proof}
Isotope the left hand side, we have
\begin{align*}
&\vcenter{\hbox{\scalebox{0.5}{
\begin{tikzpicture}[xscale=0.8, yscale=0.6]
\draw [double distance=0.8cm] (0,0) [partial ellipse=-0.1:180.1:2 and 1.5];
\draw [double distance=0.8cm] (0,0) [partial ellipse=-0.1:180.1:4.2 and 3.6];
\draw [red] (0,0) [partial ellipse=0:180:2 and 1.5];
\draw [red] (0,0) [partial ellipse=0:180:4.2 and 3.6];
\path [fill=white](-0.65, 0) rectangle (0.65, 4.5);
\begin{scope}[shift={(0,5)}]
\draw (0,0) [partial ellipse=0:360:0.6 and 0.3];
\end{scope}
\draw (-0.6, 5)--(-0.6, 0) (0.6, 5)--(0.6, 0); 
\draw [blue, -<-=0.5] (0, 0)--(0, 4.7) node [pos=0.5, right, black] {\tiny $V$};% upper one
\draw (-0.6, -5)--(-0.6, 0) (0.6, -5)--(0.6, 0);
\draw [blue] (0, 0)--(0, -5.3);
\draw [line width=0.83cm] (0,0) [partial ellipse=180:360:2 and 1.5];
\draw [white, line width=0.8cm] (0,0) [partial ellipse=178:362:2 and 1.5];
\draw [line width=0.83cm] (0,0) [partial ellipse=180:360:4.2 and 3.6];
\draw [white, line width=0.8cm] (0,0) [partial ellipse=178:362:4.2 and 3.6];
\draw [red] (0,0) [partial ellipse=178:362:2 and 1.5];
\draw [red] (0,0) [partial ellipse=178:362:4.2 and 3.6];
\begin{scope}[shift={(0,-5)}]
\draw [dashed](0,0) [partial ellipse=0:180:0.6 and 0.3];
\draw (0,0) [partial ellipse=180:360:0.6 and 0.3];
\end{scope}
\begin{scope}[shift={(-2, 0)}]
\draw [blue, dashed](0,0) [partial ellipse=0:180:0.5 and 0.3] node [below left, black] {\tiny $X$};
\draw [white, line width=4pt] (0,0) [partial ellipse=270:290:0.5 and 0.3];
\draw [blue, ->-=0.3] (0,0) [partial ellipse=180:360:0.5 and 0.3];
\end{scope}
\begin{scope}[shift={(-4.2, 0)}]
\draw [blue, dashed](0,0) [partial ellipse=0:180:0.5 and 0.3] node [below left, black] {\tiny $Y$};
\draw [white, line width=4pt] (0,0) [partial ellipse=270:290:0.5 and 0.3];
\draw [blue, ->-=0.3] (0,0) [partial ellipse=180:360:0.5 and 0.3];
\end{scope}
\end{tikzpicture}}}}
\xlongequal{\text{Move 1}}
\vcenter{\hbox{\scalebox{0.5}{
\begin{tikzpicture}[xscale=0.8, yscale=0.6]
\draw [double distance=0.8cm] (0,0) [partial ellipse=-0.1:180.1:2 and 1.5];
\draw [double distance=0.8cm] (0,0) [partial ellipse=-0.1:180.1:4.2 and 3.6];
\draw [red] (0,0) [partial ellipse=0:180:2 and 1.5];
\draw [red] (0,0) [partial ellipse=0:180:4.2 and 3.6];
\path [fill=white](-0.65, 0) rectangle (0.65, 4.5);
\begin{scope}[shift={(0,5)}]
\draw (0,0) [partial ellipse=0:360:0.6 and 0.3];
\end{scope}
\draw (-0.6, 5)--(-0.6, 0) (0.6, 5)--(0.6, 0); 
\draw [blue, -<-=0.5] (0, 0)--(0, 4.7) node [pos=0.5, right, black] {\tiny $V$};% upper one
\draw (-0.6, -5)--(-0.6, 0) (0.6, -5)--(0.6, 0);
\draw [blue] (0, 0)--(0, -5.3);
\draw [line width=0.83cm] (0,0) [partial ellipse=180:360:2 and 1.5];
\draw [white, line width=0.8cm] (0,0) [partial ellipse=178:362:2 and 1.5];
\draw [line width=0.83cm] (0,0) [partial ellipse=180:360:4.2 and 3.6];
\draw [white, line width=0.8cm] (0,0) [partial ellipse=178:362:4.2 and 3.6];
\draw [red] (0,0) [partial ellipse=178:362:2 and 1.5];
\draw [red] (0,0) [partial ellipse=178:362:4.2 and 3.6];
\begin{scope}[shift={(0,-5)}]
\draw [dashed](0,0) [partial ellipse=0:180:0.6 and 0.3];
\draw (0,0) [partial ellipse=180:360:0.6 and 0.3];
\end{scope}
\begin{scope}[shift={(-2, 0)}]
\draw [blue, dashed](0,0) [partial ellipse=0:180:0.5 and 0.3] node [below left, black] {\tiny $X$};
\draw [white, line width=4pt] (0,0) [partial ellipse=270:290:0.5 and 0.3];
\draw [blue, ->-=0.3] (0,0) [partial ellipse=180:360:0.5 and 0.3];
\end{scope}
\begin{scope}[shift={(-4.2, 0)}]
\draw [blue, dashed](0,0) [partial ellipse=0:180:0.5 and 0.3] node [below left, black] {\tiny $Y$};
\draw [white, line width=4pt] (0,0) [partial ellipse=270:290:0.5 and 0.3];
\draw [blue, ->-=0.3] (0,0) [partial ellipse=180:360:0.5 and 0.3];
\end{scope}
\path [fill=white] (-3.8, 0.5) rectangle (-2.01, 1.3);
\draw (-3.6, 0.5)--(-2.4, 0.5) (-3.4, 1.3)--(-2.01, 1.3);
\begin{scope}[shift={(-3,0.9)}]
\draw [red, dashed](0,0) [partial ellipse=90:-90:0.2 and 0.4];
\draw [red](0,0) [partial ellipse=90:270:0.2 and 0.4];
\end{scope}
\end{tikzpicture}}}}
\xlongequal{\text{Handle slide}}
\mu \vcenter{\hbox{\scalebox{0.5}{
\begin{tikzpicture}[xscale=0.8, yscale=0.6]
\draw [double distance=0.8cm] (0,0) [partial ellipse=-0.1:180.1:2 and 1.5];
\draw [double distance=0.8cm] (0,0) [partial ellipse=-0.1:180.1:4.2 and 3.6];
\draw [red] (0,0) [partial ellipse=0:180:2 and 1.5];
\draw [red] (0,0) [partial ellipse=0:180:4.2 and 3.6];
\path [fill=white](-0.65, 0) rectangle (0.65, 4.5);
\begin{scope}[shift={(0,5)}]
\draw (0,0) [partial ellipse=0:360:0.6 and 0.3];
\end{scope}
\draw (-0.6, 5)--(-0.6, 0) (0.6, 5)--(0.6, 0); 
\draw [blue, -<-=0.5] (0, 0)--(0, 4.7) node [pos=0.5, right, black] {\tiny $V$};% upper one
\draw (-0.6, -5)--(-0.6, 0) (0.6, -5)--(0.6, 0);
\draw [blue] (0, 0)--(0, -5.3);
\draw [line width=0.83cm] (0,0) [partial ellipse=180:360:2 and 1.5];
\draw [white, line width=0.8cm] (0,0) [partial ellipse=178:362:2 and 1.5];
\draw [line width=0.83cm] (0,0) [partial ellipse=180:360:4.2 and 3.6];
\draw [white, line width=0.8cm] (0,0) [partial ellipse=178:362:4.2 and 3.6];
\draw [red] (0,0) [partial ellipse=178:362:2 and 1.5];
\draw [red] (0,0) [partial ellipse=178:362:4.2 and 3.6];
\begin{scope}[shift={(0,-5)}]
\draw [dashed](0,0) [partial ellipse=0:180:0.6 and 0.3];
\draw (0,0) [partial ellipse=180:360:0.6 and 0.3];
\end{scope}
\begin{scope}[shift={(-2, 0)}]
\draw [blue, dashed](0,0) [partial ellipse=0:180:0.5 and 0.3] node [below left, black] {\tiny $X$};
\draw [white, line width=4pt] (0,0) [partial ellipse=270:290:0.5 and 0.3];
\draw [blue, ->-=0.3] (0,0) [partial ellipse=180:360:0.5 and 0.3];
\end{scope}
\begin{scope}[shift={(-4.2, 0)}]
\draw [blue, dashed](0,0) [partial ellipse=0:180:0.5 and 0.3] node [below left, black] {\tiny $Y$};
\draw [white, line width=4pt] (0,0) [partial ellipse=270:290:0.5 and 0.3];
\draw [blue, ->-=0.3] (0,0) [partial ellipse=180:360:0.5 and 0.3];
\end{scope}
\path [fill=white] (-3.8, 0.5) rectangle (-2.01, 1.3);
\draw (-3.6, 0.5)--(-2.4, 0.5) (-3.4, 1.3)--(-2.01, 1.3);
\end{tikzpicture}}}}
\end{align*}
By Lemma \ref{lem:unit4}, we see that the proposition is true.
\end{proof}

\begin{proposition}\label{prop:identitydecomposition}
We have that
\begin{align}\label{eq:centeriddecom}
\sum_{i=0}^r \sum_{(X, e_X)\in \Irr_{\mathcal{Z}}}\frac{d(X)}{\mu^{2}}
\vcenter{\hbox{\scalebox{0.6}{
\begin{tikzpicture}[xscale=0.8, yscale=0.6]
\draw [line width=0.83cm] (0,0) [partial ellipse=-0.1:180.1:2 and 1.5];
\draw [white, line width=0.8cm] (0,0) [partial ellipse=-0.1:180.1:2 and 1.5];
\draw [red] (0,0) [partial ellipse=0:180:2 and 1.5];
\path [fill=white](-0.65, 0) rectangle (0.65, 2.5);
\begin{scope}[shift={(0,3)}]
\draw (0,0) [partial ellipse=0:360:0.6 and 0.3];
\end{scope}
\draw (-0.6, 3)--(-0.6, 0) (0.6, 3)--(0.6, 0); 
\draw [blue, -<-=0.5] (0, 0)--(0, 2.7) node [pos=0.5, right, black] {\tiny $X_i$};% upper one
\draw (-0.6, -3)--(-0.6, 0) (0.6, -3)--(0.6, 0);
\draw [blue] (0, 0)--(0, -3.3);
\draw [line width=0.83cm] (0,0) [partial ellipse=180:360:2 and 1.5];
\draw [white, line width=0.8cm] (0,0) [partial ellipse=179:361:2 and 1.5];
\draw [red] (0,0) [partial ellipse=179:361:2 and 1.5];
\begin{scope}[shift={(0,-3)}]
\draw [dashed](0,0) [partial ellipse=0:180:0.6 and 0.3];
\draw (0,0) [partial ellipse=180:360:0.6 and 0.3];
\end{scope}
\begin{scope}[shift={(-2, 0)}]
\draw [blue, dashed](0,0) [partial ellipse=0:180:0.5 and 0.3] node [below left, black] {\tiny $X$};
\draw [white, line width=4pt] (0,0) [partial ellipse=270:290:0.5 and 0.3];
\draw [blue, ->-=0.3] (0,0) [partial ellipse=180:360:0.5 and 0.3];
\end{scope}
\end{tikzpicture}}}}
=\sum_{i=0}^r\vcenter{\hbox{\scalebox{0.6}{
\begin{tikzpicture}[scale=0.35]
\draw (0,5) [partial ellipse=0:360:2 and 0.8];
\draw (-2, 5)--(-2, -4);
\draw (2, 5)--(2, -4);
\draw[dashed] (0,-4) [partial ellipse=0:180:2 and 0.8];
\draw (0,-4) [partial ellipse=180:360:2 and 0.8] node[below, white]{\tiny{$X$}};
\draw[blue,->-=0.5] (0, 4.2)--(0, -4.8) node [pos=0.5, right, black] {\tiny $X_i$};
\end{tikzpicture}}}},
\end{align}
where $\Irr_{\mathcal{Z}}$ is the set of all non-isomorphic simple objects in $\mathcal{Z(C)}$.
Moreover, each summand
$$\frac{d(X)}{\mu^2}\sum_{i=0}^r\vcenter{\hbox{\scalebox{0.6}{
\begin{tikzpicture}[xscale=0.8, yscale=0.6]
\draw [line width=0.83cm] (0,0) [partial ellipse=-0.1:180.1:2 and 1.5];
\draw [white, line width=0.8cm] (0,0) [partial ellipse=-0.1:180.1:2 and 1.5];
\draw [red] (0,0) [partial ellipse=0:180:2 and 1.5];
\path [fill=white](-0.65, 0) rectangle (0.65, 2.5);
\begin{scope}[shift={(0,3)}]
\draw (0,0) [partial ellipse=0:360:0.6 and 0.3];
\end{scope}
\draw (-0.6, 3)--(-0.6, 0) (0.6, 3)--(0.6, 0); 
\draw [blue, -<-=0.5] (0, 0)--(0, 2.7) node [pos=0.5, right, black] {\tiny $X_i$};% upper one
\draw (-0.6, -3)--(-0.6, 0) (0.6, -3)--(0.6, 0);
\draw [blue] (0, 0)--(0, -3.3);
\draw [line width=0.83cm] (0,0) [partial ellipse=180:360:2 and 1.5];
\draw [white, line width=0.8cm] (0,0) [partial ellipse=178:362:2 and 1.5];
\draw [red] (0,0) [partial ellipse=178:362:2 and 1.5];
\begin{scope}[shift={(0,-3)}]
\draw [dashed](0,0) [partial ellipse=0:180:0.6 and 0.3];
\draw (0,0) [partial ellipse=180:360:0.6 and 0.3];
\end{scope}
\begin{scope}[shift={(-2, 0)}]
\draw [blue, dashed](0,0) [partial ellipse=0:180:0.5 and 0.3] node [below left, black] {\tiny $X$};
\draw [white, line width=4pt] (0,0) [partial ellipse=270:290:0.5 and 0.3];
\draw [blue, ->-=0.3] (0,0) [partial ellipse=180:360:0.5 and 0.3];
\end{scope}
\end{tikzpicture}}}}$$
 is a minimal central idempotent for $(X, e_X)\in \Irr_{\mathcal{Z}}$.
\end{proposition}
\begin{proof}
We first show each summand on the left hand side is nonzero. Indeed, the trace of each summand in $\cA$ is given by
\begin{align*}
\frac{d(X)}{\mu^{2}}
\vcenter{\hbox{\scalebox{0.5}{
\begin{tikzpicture}[scale=0.8]
\draw [line width=0.83cm] (0,0) [partial ellipse=-0.1:180.1:2 and 1.5];
\draw [white, line width=0.8cm] (0,0) [partial ellipse=-0.1:180.1:2 and 1.5];
\draw [red] (0,0) [partial ellipse=0:180:2 and 1.5];
\begin{scope}[shift={(2, 0)}]
\draw [line width=0.83cm] (0,0) [partial ellipse=-0.1:180.1:2 and 1.5];
\draw [white, line width=0.8cm] (0,0) [partial ellipse=-0.1:180.1:2 and 1.5];
\draw [red] (0,0) [partial ellipse=0:180:2 and 1.5];
\end{scope} % upper
\begin{scope}[shift={(2, 0)}]
\draw [line width=0.83cm] (0,0) [partial ellipse=180:360:2 and 1.5];
\draw [white, line width=0.8cm] (0,0) [partial ellipse=178:362:2 and 1.5];
\draw [red] (0,0) [partial ellipse=178:362:2 and 1.5];
\end{scope}
\draw [line width=0.83cm] (0,0) [partial ellipse=180:360:2 and 1.5];
\draw [white, line width=0.8cm] (0,0) [partial ellipse=178:362:2 and 1.5];
\draw [red] (0,0) [partial ellipse=178:362:2 and 1.5];
\begin{scope}[shift={(-2, 0)}]
\draw [blue, dashed](0,0) [partial ellipse=0:180:0.5 and 0.3] node [below left, black] {\tiny $X$};
\draw [white, line width=4pt] (0,0) [partial ellipse=270:290:0.5 and 0.3];
\draw [blue, ->-=0.3] (0,0) [partial ellipse=180:360:0.5 and 0.3];
\end{scope} % lower 
\end{tikzpicture}}}}
=\frac{d(X)}{\mu^{2}}\mu^{2} d(X)=d(X)^{2}\ne 0.
\end{align*}
By Proposition \ref{prop:centeridempotent}, summands on the left hand side of Equation \eqref{eq:centeriddecom} are orthogonal idempotents in $\text{End}_\mathcal{A}(I(\bigoplus_i X_{i}))$, where the summation is over $\Irr(\cC)$. Also notice the outer ring can be moved up and down, thus commute with morphisms on the center tube. i.e., they are central idempotents in $\text{End}_\mathcal{A}(I(\bigoplus X_{i}))$. 

On the other hand, since $\mathcal{A}$ is equivalent to $\mathcal{Z(C)}$, there are at most $|\Irr(\mathcal{Z})|$ minimal central idempotents in $\text{End}_\mathcal{A}(\bigoplus_{i}I(X_{i}))$(arbitrary endomorphism algebra in $\mathcal{A}$, actually). Therefore, the summands on the left hand side are minimal central idempotents, and add up to the identity morphism.
\end{proof}

\begin{corollary}
Suppose $\mathcal{C}$ is a modular tensor category.
Then 
\begin{align*}
\sum_{j,k=0}^r \frac{d_j d_k}{\mu^2}\vcenter{\hbox{ \scalebox{0.6}{
\begin{tikzpicture}[xscale=0.8, yscale=0.6]
\draw [line width=0.83cm] (0,0) [partial ellipse=-0.1:180.1:2 and 1.5];
\draw [white, line width=0.8cm] (0,0) [partial ellipse=-0.1:180.1:2 and 1.5];
\draw [red] (0,0) [partial ellipse=0:180:2 and 1.5];
\path [fill=white](-0.65, 0) rectangle (0.65, 2.5);
\begin{scope}[shift={(0,3)}]
\draw (0,0) [partial ellipse=0:360:0.6 and 0.3];
\end{scope}
\draw (-0.6, 3)--(-0.6, 0) (0.6, 3)--(0.6, 0); 
\draw [blue, -<-=0.5] (0, 0)--(0, 2.7) node [pos=0.5, right, black] {\tiny $V$};% upper one
\draw (-0.6, -3)--(-0.6, 0) (0.6, -3)--(0.6, 0);
\draw [blue] (0, 0)--(0, -3.3);
\draw [line width=0.83cm] (0,0) [partial ellipse=180:360:2 and 1.5];
\draw [white, line width=0.8cm] (0,0) [partial ellipse=179:361:2 and 1.5];
\draw [red] (0,0) [partial ellipse=179:361:2 and 1.5];
\begin{scope}[shift={(0,-3)}]
\draw [dashed](0,0) [partial ellipse=0:180:0.6 and 0.3];
\draw (0,0) [partial ellipse=180:360:0.6 and 0.3];
\end{scope}
\begin{scope}[shift={(-2, 0)}]
\draw [blue, dashed](0,0) [partial ellipse=0:180:0.5 and 0.3] node [below left, black] {\tiny $X_k$};
\draw [white, line width=4pt] (0,0) [partial ellipse=270:290:0.5 and 0.3];
\draw [blue, ->-=0.3] (0,0) [partial ellipse=180:360:0.5 and 0.3];
\end{scope}
\begin{scope}[shift={(2, 0)}]
\draw [blue, dashed](0,0) [partial ellipse=0:180:0.5 and 0.3] node [pos=0.5, above right, black] {\tiny $X_j$};
\draw [blue, -<-=0.5] (0,0) [partial ellipse=180:260:0.5 and 0.3];
\draw [blue] (0,0) [partial ellipse=280:360:0.5 and 0.3];
\end{scope}
\end{tikzpicture}}}}
=\vcenter{\hbox{\scalebox{0.6}{
\begin{tikzpicture}[scale=0.35]
\draw (0,5) [partial ellipse=0:360:2 and 0.8];
\draw (-2, 5)--(-2, -4);
\draw (2, 5)--(2, -4);
\draw[dashed] (0,-4) [partial ellipse=0:180:2 and 0.8];
\draw (0,-4) [partial ellipse=180:360:2 and 0.8] node[below, white]{\tiny{$X$}};
\draw[blue, ->-=0.5] (0, 4.2)--(0, -4.8) node [pos=0.5, right, black] {\tiny $V$};
\end{tikzpicture}}}}
\end{align*}
\end{corollary}
\begin{proof}
We leave it to the readers.
\end{proof}

By definition in \cite{BakKir01}, the $S$-matrix entries of $\cA$ ,or equivalently, of $\mathcal{Z(C)}$, are traces of double braiding on pairs of simple objects. More precisely, for simple objects $(X, e_X), (Y, e_Y)$ in $\mathcal{Z}(\mathcal{C})$, we have

\begin{align}\label{eq:S-matrix}
S_{(X, e_X), (Y^*, e_{Y^*})}=\frac{1}{\mu}
\vcenter{\hbox{\scalebox{0.7}{
\begin{tikzpicture}[scale=0.8]
\draw [double distance=0.8cm] (0,0) [partial ellipse=-0.1:180.1:2 and 1.5];
\draw [blue, -<-=0.7] (0,0) [partial ellipse=0:180:2 and 1.5];
\begin{scope}[shift={(2, 0)}]
\draw [double distance=0.8cm] (0,0) [partial ellipse=-0.1:180.1:2 and 1.5];
\draw [blue, -<-=0.5] (0,0) [partial ellipse=0:180:2 and 1.5];
\end{scope} % upper
\begin{scope}[shift={(2, 0)}]
\draw [line width=0.83cm] (0,0) [partial ellipse=180:360:2 and 1.5];
\draw [white, line width=0.8cm] (0,0) [partial ellipse=178:362:2 and 1.5];
\draw [blue] (0,0) [partial ellipse=178:362:2 and 1.5];
\end{scope}
\draw [line width=0.83cm] (0,0) [partial ellipse=180:360:2 and 1.5];
\draw [white, line width=0.8cm] (0,0) [partial ellipse=178:362:2 and 1.5];
\draw [blue] (0,0) [partial ellipse=178:362:2 and 1.5];
\begin{scope}[shift={(-2, 0)}]
\draw [red, dashed](0,0) [partial ellipse=0:180:0.5 and 0.3]; 
\node[below left, black] at (0.6,-0.8) {\tiny $X$};
\draw [red] (0,0) [partial ellipse=180:255:0.5 and 0.3];
\draw [red] (0,0) [partial ellipse=285:360:0.5 and 0.3];
\end{scope} % lower 
\begin{scope}[shift={(0, 0)}]
\draw [red, dashed](0,0) [partial ellipse=0:180:0.5 and 0.3]; 
\node[below left, black] at (4.2,-0.8) {\tiny $Y$};
\draw [red] (0,0) [partial ellipse=180:255:0.5 and 0.3];
\draw [red] (0,0) [partial ellipse=285:360:0.5 and 0.3];
\end{scope}
\end{tikzpicture}}}}
=\frac{1}{\mu}
\vcenter{\hbox{\scalebox{0.7}{
\begin{tikzpicture}[scale=0.8]
\draw [line width=0.83cm] (0,0) [partial ellipse=0:180:2 and 1.5];
\draw [line width=0.8cm, white] (0,0) [partial ellipse=-0.2:180.2:2 and 1.5];
\draw [red] (0,0) [partial ellipse=0:180:2 and 1.5];
\begin{scope}[shift={(2, 0)}]
\draw [line width=0.83cm] (0,0) [partial ellipse=0:180:2 and 1.5];
\draw [line width=0.8cm, white] (0,0) [partial ellipse=-0.2:180.2:2 and 1.5];
\draw [red] (0,0) [partial ellipse=0:180:2 and 1.5];
\end{scope} % upper
\begin{scope}[shift={(2, 0)}]
\draw [line width=0.83cm] (0,0) [partial ellipse=180:360:2 and 1.5];
\draw [white,line width=0.8cm] (0,0) [partial ellipse=178:362:2 and 1.5];
\draw [red] (0,0) [partial ellipse=178:362:2 and 1.5];
\end{scope}
\draw [line width=0.83cm] (0,0) [partial ellipse=180:360:2 and 1.5];
\draw [white, line width=0.8cm] (0,0) [partial ellipse=178:362:2 and 1.5];
\draw [red] (0,0) [partial ellipse=178:362:2 and 1.5];
\begin{scope}[shift={(-2, 0)}]
\draw [blue, dashed](0,0) [partial ellipse=0:180:0.5 and 0.3] node [below left, black] {\tiny $X$};
\draw [white, line width=4pt] (0,0) [partial ellipse=270:290:0.5 and 0.3];
\draw [blue, ->-=0.3] (0,0) [partial ellipse=180:360:0.5 and 0.3];
\end{scope} % lower 
\begin{scope}[shift={(0, 0)}]
\draw [blue, dashed](0,0) [partial ellipse=0:180:0.5 and 0.3] node [below left, black] {\tiny $Y$};
\draw [white, line width=4pt] (0,0) [partial ellipse=270:290:0.5 and 0.3];
\draw [blue, -<-=0.7] (0,0) [partial ellipse=180:360:0.5 and 0.3];
\end{scope}
\end{tikzpicture}}}}\,,
\end{align}
where the second equality follows from
similar argument as in the proof of Proposition \ref{prop:centeridempotent}. Notice the complement of a tubular neighborhood of a Hopf link is homeomorphic to $\mathbb{T}\times I$, The mapping class of a homeomorphism of $\mathbb{T}$ mapping $(m, l)$ to $(-l, m)$ extends to $\mathbb{T}\times I$, which defines a homeomorphism between the $B$-colored region of the second term and the $B$-colored region of the third term. This is an other way to see the second equality of Equation \eqref{eq:S-matrix}.

\begin{proposition}\label{prop:Stransform}
Suppose that $(X, e_X)\in \mathcal{Z(C)}$ is simple.
Then for all $V\in \cC$ we have that
\begin{align}\label{eq:Sinverse}
\vcenter{\hbox{\scalebox{0.7}{
\begin{tikzpicture}[xscale=0.8, yscale=0.6]
\draw [line width=0.83cm] (0,0) [partial ellipse=-0.1:180.1:2 and 1.5];
\draw [white, line width=0.8cm] (0,0) [partial ellipse=-0.1:180.1:2 and 1.5];
\draw [blue] (0,0) [partial ellipse=0:180:2 and 1.5];
\path [fill=white](-0.65, 0) rectangle (0.65, 2.5);
\begin{scope}[shift={(0,3)}]
\draw (0,0) [partial ellipse=0:360:0.6 and 0.3];
\end{scope}
\draw (-0.6, 3)--(-0.6, 0) (0.6, 3)--(0.6, 0); 
\draw [blue, -<-=0.5] (0, 0)--(0, 2.7) node [pos=0.5, right, black] {\tiny $V$};% upper one
\draw (-0.6, -3)--(-0.6, 0) (0.6, -3)--(0.6, 0);
\draw [blue, ->-=0.5] (0, 0)--(0, -3.3) node [pos=0.5, right, black] {\tiny $V$};
\draw [line width=0.83cm] (0,0) [partial ellipse=180:360:2 and 1.5];
\draw [white, line width=0.8cm] (0,0) [partial ellipse=179:361:2 and 1.5];
\draw [blue, ->=0.3] (0,0) [partial ellipse=170:361:2 and 1.5] node [pos=0.7, below left, black] {\tiny $X$};
\begin{scope}[shift={(0,-3)}]
\draw [dashed](0,0) [partial ellipse=0:180:0.6 and 0.3];
\draw (0,0) [partial ellipse=180:360:0.6 and 0.3];
\end{scope}
\begin{scope}[shift={(-2, 0)}]
\draw [red, dashed](0,0) [partial ellipse=0:180:0.5 and 0.3];
%\draw [white, line width=4pt] (0,0) [partial ellipse=270:290:0.5 and 0.3];
\draw [red] (0,0) [partial ellipse=180:270:0.5 and 0.3];
\draw [red] (0,0) [partial ellipse=285:360:0.5 and 0.3];
\end{scope}
\end{tikzpicture}}}}
=
\sum_{(Y, e_Y)\in \Irr_{\mathcal{Z}}} \mu^{-1} S_{(X, e_X), (Y, e_Y)}
\vcenter{\hbox{ \scalebox{0.7}{
\begin{tikzpicture}[xscale=0.8, yscale=0.6]
\draw [line width=0.83cm] (0,0) [partial ellipse=-0.1:180.1:2 and 1.5];
\draw [white, line width=0.8cm] (0,0) [partial ellipse=-0.1:180.1:2 and 1.5];
\draw [red] (0,0) [partial ellipse=0:180:2 and 1.5];
\path [fill=white](-0.65, 0) rectangle (0.65, 2.5);
\begin{scope}[shift={(0,3)}]
\draw (0,0) [partial ellipse=0:360:0.6 and 0.3];
\end{scope}
\draw (-0.6, 3)--(-0.6, 0) (0.6, 3)--(0.6, 0); 
\draw [blue, -<-=0.5] (0, 0)--(0, 2.7) node [pos=0.5, right, black] {\tiny $V$};% upper one
\draw (-0.6, -3)--(-0.6, 0) (0.6, -3)--(0.6, 0);
\draw [blue] (0, 0)--(0, -3.3);
\draw [line width=0.83cm] (0,0) [partial ellipse=180:360:2 and 1.5];
\draw [white, line width=0.8cm] (0,0) [partial ellipse=179:361:2 and 1.5];
\draw [red] (0,0) [partial ellipse=179:361:2 and 1.5];
\begin{scope}[shift={(0,-3)}]
\draw [dashed](0,0) [partial ellipse=0:180:0.6 and 0.3];
\draw (0,0) [partial ellipse=180:360:0.6 and 0.3];
\end{scope}
\begin{scope}[shift={(-2, 0)}]
\draw [blue, dashed](0,0) [partial ellipse=0:180:0.5 and 0.3] node [below left, black] {\tiny $Y$};
\draw [white, line width=4pt] (0,0) [partial ellipse=270:290:0.5 and 0.3];
\draw [blue, -<-=0.3] (0,0) [partial ellipse=180:360:0.5 and 0.3];
\end{scope}
\end{tikzpicture}}}},
\end{align}
\end{proposition}
\begin{proof}
Note that the left hand side is an element in the center of the algebra 
 $\displaystyle \End_{\cA}(O_V)$.
 By Proposition \ref{prop:identitydecomposition}, we see that it is a linear combination of minimal central idempotents.
 The rest argument is similar to the one in Proposition \ref{prop:centeridempotent}.
We leave it to the readers.
\end{proof}

\subsection{Equivariance}
The modular group $\SL_2(\bZ)$ naturally acts on pairs of integers $(m,\ell) \in \bZ^2$ by multiplication from right, i.e., $(m, \ell)\mathfrak{g} = (am+c\ell, bm+d\ell)$ for any $\mathfrak{g} =
\begin{pmatrix} a & b \\ c&d \end{pmatrix} \in \SL_2(\bZ)$. On the other hand, $\SL_2(\bZ)$ acts on the fusion algebra $K_0(\cZ(\cC))$ by extending linearly the assignments 

\[\mathfrak{t}(X, e_X):=\theta_{(X,e_X)}(X, e_X)\,,\quad \mathfrak{s}(X, e_X) := \mu^{-1} \sum_{(Y,e_Y) \in \Irr(\cZ(\cC))} S_{(X, e_X), (Y, e_Y)} (Y, e_Y)\] 
for $(X, e_X) \in \Irr_{\cZ}$, where $\mathfrak{t} =
\begin{pmatrix}  1 & 1\\ 0&1\end{pmatrix}$ and $\mathfrak{s} =
\begin{pmatrix}  0 & -1\\ 1 & 0 \end{pmatrix}$ form a set of generators of $\SL_2(\bZ)$. It is shown in \cite{NgSch10} that the generalized FS-indicator is equivariant, i.e., the action of $\SL_2(\bZ)$ on $\bZ^2$ is compatible with that on $K_0(\cZ(\cC))$. More precisely, the equivariance of $\nu$ says $\nu_{(m, \ell)\mathfrak{g}}^{(X,e_X)}(V) = \nu_{(m,\ell)}^{\tilde{\mathfrak{g}}(X,e_X)}(V)$ for all $(m,\ell)\in\bZ^2$, $(X, e_X) \in \cZ(\cC)$, $V \in \cC$ and $\mathfrak{g} \in \SL_2(\bZ)$, where $\tilde{\mathfrak{g}} = \mathfrak{jgj}$ with $\mathfrak{j} =\begin{pmatrix}  1 & 0 \\ 0 & -1
\end{pmatrix}$. Note that $\tilde{\mathfrak{t}} = \mathfrak{t}^{-1}$  and $\tilde{\mathfrak{s}} = \mathfrak{s}^{-1}$.

In this section, we give a topological proof of the equivariance of $\nu$ using quantum invariants, and argue that the conjugation by $\mathfrak{j}$ in the equation is a consequence of orientation reversing homeomorphisms.

We start the proof of the $\mathfrak{t}$-equivariance with the following lemma.

\begin{lemma}\label{lem:twist}
For all $X\in \mathcal{C}$, we have
\begin{equation}\label{eq:fulltwist}
\vcenter{\hbox{\scalebox{0.6}{
\begin{tikzpicture}[xscale=0.8, yscale=0.6]
\draw [line width=0.83cm] (0,0) [partial ellipse=-0.1:180.1:2 and 1.5];
\draw [white, line width=0.8cm] (0,0) [partial ellipse=-0.1:180.1:2 and 1.5];
\draw [red] (0,0) [partial ellipse=0:180:2 and 1.5];
\path [fill=white](-0.65, 0) rectangle (0.65, 2.5);
\begin{scope}[shift={(0,3)}]
\draw (0,0) [partial ellipse=0:360:0.6 and 0.3];
\end{scope}
\draw (-0.6, 3)--(-0.6, 0) (0.6, 3)--(0.6, 0); 
\draw [blue, -<-=0.5] (0, 0)--(0, 2.7);% upper one
\draw (0.3, 1.3) node{\tiny{$X$}};
\draw (-0.6, -3)--(-0.6, 0) (0.6, -3)--(0.6, 0);
\draw [blue] (0, 0)--(0, -3.3);
\draw [line width=0.83cm] (0,0) [partial ellipse=180:360:2 and 1.5];
\draw [white, line width=0.8cm] (0,0) [partial ellipse=179:361:2 and 1.5];
\draw [red] (0,0) [partial ellipse=179:361:2 and 1.5];
\begin{scope}[shift={(0,-3)}]
\draw [dashed](0,0) [partial ellipse=0:180:0.6 and 0.3];
\draw (0,0) [partial ellipse=180:360:0.6 and 0.3];
\end{scope}
\begin{scope}[shift={(-2, 0)}, scale=1.4]
\draw [red, dashed] (0,0) [partial ellipse=0:180:0.37 and 0.25];
\draw [white, line width=4pt] (0,0) [partial ellipse=270:250:0.37 and 0.25];
\draw [red] (0,0) [partial ellipse=180:360:0.37 and 0.25];
\begin{scope}[shift={(0, -0.2)}]
\path [fill=white] (-0.2, -0.23) rectangle (0.3, 0.24);
\draw[red] (0.12,-0.23)..controls +(0,0.2) and +(-0.2,-0.1)..(0.3,0.05);
\draw[red] (0,0.24)..controls +(0,-0.2) and +(0.2,-0.05)..(-0.2,-0.01);
\end{scope}
\end{scope}
\end{tikzpicture}}}}
=\mu\vcenter{\hbox{\scalebox{0.6}{
\begin{tikzpicture}[scale=0.35]
\draw (0,5) [partial ellipse=0:360:2 and 0.8];
\draw (-2, 5)--(-2, -4);
\draw (2, 5)--(2, -4);
\draw[dashed] (0,-4) [partial ellipse=0:180:2 and 0.8];
\draw (0,-4) [partial ellipse=180:360:2 and 0.8] node[below, white]{\tiny{$X$}};
\draw[blue, ->-=0.2] (0, 4.2)--(0, -4.8);
\draw (0.6, 2.5) node{\tiny{$X$}};
\begin{scope}[shift={(0, 0)}, xscale=-5.4, yscale=4]
\draw [blue, dashed] (0,0) [partial ellipse=0:180:0.37 and 0.25];
\draw [white, line width=4pt] (0,0) [partial ellipse=270:250:0.37 and 0.25];
\draw [blue] (0,0) [partial ellipse=180:360:0.37 and 0.25];
\begin{scope}[shift={(0, -0.2)}]
\path [fill=white] (-0.2, -0.23) rectangle (0.3, 0.24);
\draw[blue] (0,-0.23)..controls +(0,0.2) and +(-0.2,-0.1)..(0.3,0.05);
\draw[blue] (0,0.24)..controls +(0,-0.2) and +(0.2,-0.05)..(-0.2,-0.01);
\end{scope}
\end{scope}
\end{tikzpicture}}}}.\end{equation}
In addition, if $(X, e_X)\in \mathcal{Z(C)}$, we have
\begin{equation}\label{eq:centerfulltwist}
\vcenter{\hbox{\scalebox{0.6}{
\begin{tikzpicture}[xscale=0.8, yscale=0.6]
\draw [line width=0.83cm] (0,0) [partial ellipse=-0.1:180.1:2 and 1.5];
\draw [white, line width=0.8cm] (0,0) [partial ellipse=-0.1:180.1:2 and 1.5];
\draw [red] (0,0) [partial ellipse=0:180:2 and 1.5];
\path [fill=white](-0.65, 0) rectangle (0.65, 2.5);
\begin{scope}[shift={(0,3)}]
\draw (0,0) [partial ellipse=0:360:0.6 and 0.3];
\end{scope}
\draw (-0.6, 3)--(-0.6, 0) (0.6, 3)--(0.6, 0); 
\draw [blue, -<-=0.5] (0, 0)--(0, 2.7);% upper one
\draw (0.3, 1.3) node{\tiny{$X$}};
\draw (-0.6, -3)--(-0.6, 0) (0.6, -3)--(0.6, 0);
\draw [blue] (0, 0)--(0, -3.3);
\draw [line width=0.83cm] (0,0) [partial ellipse=180:360:2 and 1.5];
\draw [white, line width=0.8cm] (0,0) [partial ellipse=179:361:2 and 1.5];
\draw [red] (0,0) [partial ellipse=179:361:2 and 1.5];
\begin{scope}[shift={(0,-3)}]
\draw [dashed](0,0) [partial ellipse=0:180:0.6 and 0.3];
\draw (0,0) [partial ellipse=180:360:0.6 and 0.3];
\end{scope}
\begin{scope}[shift={(-2, 0)}, scale=1.4]
\draw [red, dashed] (0,0) [partial ellipse=0:180:0.37 and 0.25];
\draw [white, line width=4pt] (0,0) [partial ellipse=270:250:0.37 and 0.25];
\draw [red] (0,0) [partial ellipse=180:360:0.37 and 0.25];
\begin{scope}[shift={(0, -0.2)}]
\path [fill=white] (-0.2, -0.23) rectangle (0.3, 0.24);
\draw[red] (0.12,-0.23)..controls +(0,0.2) and +(-0.2,-0.1)..(0.3,0.05);
\draw[red] (0,0.24)..controls +(0,-0.2) and +(0.2,-0.05)..(-0.2,-0.01);
\end{scope}
\end{scope}
\begin{scope}[shift={(0, 0)}, scale=0.3]
\draw [red](0, 0) [partial ellipse=-75:0:2 and 0.8];
\draw [red, dashed](0, 0) [partial ellipse=0:180:2 and 0.8];
\draw [red](0, 0) [partial ellipse=180:255:2 and 0.8];
\end{scope}
\end{tikzpicture}}}}
=\mu\vcenter{\hbox{\scalebox{0.6}{
\begin{tikzpicture}[scale=0.35]
\draw (0,5) [partial ellipse=0:360:2 and 0.8];
\draw (-2, 5)--(-2, -4);
\draw (2, 5)--(2, -4);
\draw[dashed] (0,-4) [partial ellipse=0:180:2 and 0.8];
\draw (0,-4) [partial ellipse=180:360:2 and 0.8] node[below, white]{\tiny{$X$}};
\draw[blue, ->-=0.2] (0, 4.2)--(0, -4.8);
\draw (0.6, 2.5) node{\tiny{$X$}};
\begin{scope}[shift={(0, 0)}, xscale=-5.4, yscale=4]
\draw [blue, dashed] (0,0) [partial ellipse=0:180:0.37 and 0.25];
\draw [white, line width=4pt] (0,0) [partial ellipse=270:250:0.37 and 0.25];
\draw [blue] (0,0) [partial ellipse=180:360:0.37 and 0.25];
\begin{scope}[shift={(0, -0.2)}]
\path [fill=white] (-0.2, -0.23) rectangle (0.3, 0.24);
\draw[blue] (0,-0.23)..controls +(0,0.2) and +(-0.2,-0.1)..(0.3,0.05);
\draw[blue] (0,0.24)..controls +(0,-0.2) and +(0.2,-0.05)..(-0.2,-0.01);
\end{scope}
\end{scope}
\begin{scope}[shift={(0, -2)}]
\draw [red](0, 0) [partial ellipse=-75:0:2 and 0.8];
\draw [red, dashed](0, 0) [partial ellipse=0:180:2 and 0.8];
\draw [red](0, 0) [partial ellipse=180:255:2 and 0.8];
\end{scope}
\end{tikzpicture}}}}
=
\mu\vcenter{\hbox{\scalebox{0.6}{
\begin{tikzpicture}[scale=0.35]
\draw (0,5) [partial ellipse=0:360:2 and 0.8];
\draw (-2, 5)--(-2, -4);
\draw (2, 5)--(2, -4);
\draw[dashed] (0,-4) [partial ellipse=0:180:2 and 0.8];
\draw (0,-4) [partial ellipse=180:360:2 and 0.8] node[below, white]{\tiny{$X$}};
\draw[blue, ->-=0.2] (0, 4.2)--(0, -4.8);
\draw (0.6, 2.5) node{\tiny{$X$}};
\begin{scope}[xscale=-0.6, yscale=0.6, xshift=0.5cm]
\fill[white] (-1.5, -2) rectangle (1.5, 2);
\draw[blue] (-.5, 2) ..controls (-0.5,1) and (-0,-1)..(1, -1);
\draw[blue] (1, 0) [partial ellipse=-90:90:.5 and 1];
\fill[white] (-0.25,-0.5) rectangle (0.25, 0.5);
\draw[blue] (-.5, -2) ..controls (-0.5,-1) and (-0,1)..(1, 1);
\end{scope}
\begin{scope}[shift={(0, -2)}]
\draw [red](0, 0) [partial ellipse=-75:0:2 and 0.8];
\draw [red, dashed](0, 0) [partial ellipse=0:180:2 and 0.8];
\draw [red](0, 0) [partial ellipse=180:255:2 and 0.8];
\end{scope}
\end{tikzpicture}}}}=\mu^2 G(\theta_{(X, e_X)}).
\end{equation}
\end{lemma}
\begin{proof}
Equation \eqref{eq:fulltwist} follows from a similar argument as in the proofs of Theorem \ref{thm:nu-and-z}.

For Equation \eqref{eq:centerfulltwist}, the first equality comes from Equation \eqref{eq:fulltwist} and the second one comes from handle sliding the blue string along the red circle.
\end{proof}

\begin{proposition}[$\mathfrak{t}$-Equivariance]\label{prop:t-equiv}
Suppose $V\in\mathcal{C}$ and $(X, e_X)\in \mathcal{Z(C)}$. Then for any $(m,\ell) \in \bZ^2$, we have
\begin{align*}
\nu_{(m,\ell)\mathfrak{t}}^{(X, e_X)}(V)=\nu_{(m,\ell)}^{\tilde{\mathfrak{t}}(X, e_X)}(V).
\end{align*}
\end{proposition}
\begin{proof}
As is explained above, we can assume without loss of generality that $m \ge 0$. 
Write $\ell = -qm + r$ for $0 \le r \le m-1$ when $m > 0$, and set $q = 0$, $r = |\ell|$ when $m = 0$. In addition, by the linearity of $\nu$, we can assume that $(X, e_X) \in \Irr_{\cZ}$. 

Let $\Theta$ be a $\cC$-colored graph on a torus $\mathbb{T}_{\Theta} \subset \bS^3$ containing an $\Omega$-colored $(1,1)$-curve which is the result of the right-hand Dehn twist along the meridian of $\mathbb{T}_{\Theta}$. Consider the following $\cC$-decorated 3-alterfolds in $\bS^3$ with $\Xi = \Theta\sqcup\Gamma_{(m,r)}(V)\sqcup\Phi(X, e_X)$:

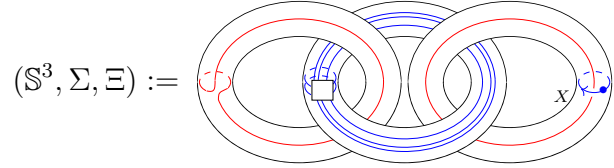
\begin{align*}\label{eq:action1}
(\bS^3, \Sigma, \Xi) := \vcenter{\hbox{\scalebox{0.8}{
\begin{tikzpicture}[scale=0.7]
\draw [double distance=0.57cm] (0,0) [partial ellipse=-0.1:180.1:2 and 1.5];
\draw [blue] (0,0) [partial ellipse=0:180:2.05 and 1.55];
\draw [blue] (0,0) [partial ellipse=0:180:2.15 and 1.65];
\draw [blue] (0,0) [partial ellipse=0:180:1.85 and 1.35];
\begin{scope}[shift={(2.5, 0)}]
\draw [double distance=0.57cm] (0,0) [partial ellipse=-0.1:180.1:2 and 1.5];
\draw [red] (0,0) [partial ellipse=0:180:2 and 1.5];
\end{scope} % upper
\begin{scope}[shift={(2.5, 0)}]
\draw [line width=0.6cm] (0,0) [partial ellipse=180:360:2 and 1.5];
\draw [white, line width=0.57cm] (0,0) [partial ellipse=178:362:2 and 1.5];
\draw [red] (0,0) [partial ellipse=178:362:2 and 1.5];
\end{scope}
\begin{scope}[shift={(-2.5, 0)}]
\draw [double distance=0.57cm] (0,0) [partial ellipse=-0.1:180.1:2 and 1.5];
\draw [red] (0,0) [partial ellipse=0:180:2 and 1.5];
\end{scope} % upper
\begin{scope}[shift={(-2.5, 0)}]
\draw [line width=0.6cm] (0,0) [partial ellipse=180:360:2 and 1.5];
\draw [white, line width=0.57cm] (0,0) [partial ellipse=178:362:2 and 1.5];
\draw [red] (0,0) [partial ellipse=178:362:2 and 1.5];
\end{scope}
\draw [line width=0.6cm] (0,0) [partial ellipse=180:360:2 and 1.5];
\draw [white, line width=0.57cm] (0,0) [partial ellipse=178:362:2 and 1.5];
\draw [blue] (0,0) [partial ellipse=178:362:2.05 and 1.55];
\draw [blue] (0,0) [partial ellipse=178:362:2.15 and 1.65];
\draw [blue] (0,0) [partial ellipse=178:362:1.85 and 1.35];
\begin{scope}[shift={(-2, -0.1)}]
\draw [blue, dashed](0,0) [partial ellipse=0:180:0.37 and 0.25];
\draw [blue,] (0,0) [partial ellipse=180:360:0.37 and 0.25];
\end{scope} 
\begin{scope}[shift={(-2, 0.1)}]
\draw [blue, dashed](0,0) [partial ellipse=0:180:0.37 and 0.25];
\draw [blue,] (0,0) [partial ellipse=180:360:0.37 and 0.25];
\end{scope} 
\begin{scope}[shift={(-2, -0.2)}]
\draw [fill=white] (-0.2, -0.23) rectangle (0.3, 0.24);
\end{scope}% lower 
\begin{scope}[shift={(4.5, 0)}]
\draw [blue, dashed](0,0) [partial ellipse=0:180:0.37 and 0.25] node [below left, black] {\tiny $X$};
\draw [white, line width=4pt] (0,0) [partial ellipse=270:250:0.37 and 0.25];
\draw [blue, -<-=0.3] (0,0) [partial ellipse=180:360:0.37 and 0.25] node [pos=0.7] {\tiny $\bullet$};
\end{scope}
\begin{scope}[shift={(-4.5, 0)}]
\draw [red, dashed] (0,0) [partial ellipse=0:180:0.37 and 0.25];
\draw [white, line width=4pt] (0,0) [partial ellipse=270:250:0.37 and 0.25];
\draw [red] (0,0) [partial ellipse=180:360:0.37 and 0.25];
\begin{scope}[shift={(0, -0.2)}]
\path [fill=white] (-0.2, -0.23) rectangle (0.3, 0.24);
\draw[red] (0.08,-0.23)..controls +(0,0.2) and +(-0.2,-0.1)..(0.3,0.05);
\draw[red] (0,0.24)..controls +(0,-0.2) and +(0.2,-0.05)..(-0.2,-0.01);
\end{scope}
\end{scope}
\end{tikzpicture}}}}
\end{align*}
and 
\begin{align*}
(\bS^3, \Sigma', f(\Xi)) := 
\vcenter{\hbox{\scalebox{0.7}
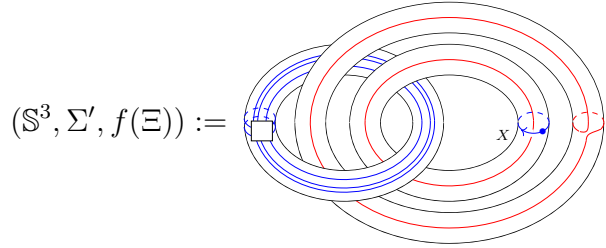
{
\begin{tikzpicture}[scale=0.8]
\draw [double distance=0.57cm] (0,0) [partial ellipse=-0.1:180.1:2 and 1.5];
\draw [blue] (0,0) [partial ellipse=0:180:2.05 and 1.55];
\draw [blue] (0,0) [partial ellipse=0:180:2.15 and 1.65];
\draw [blue] (0,0) [partial ellipse=0:180:1.85 and 1.35];
\begin{scope}[shift={(2.5, 0)}]
\draw [double distance=0.57cm] (0,0) [partial ellipse=-0.1:180.1:2 and 1.5];
\draw [red] (0,0) [partial ellipse=0:180:2 and 1.5];
\end{scope} 
\begin{scope}[shift={(2.5, 0)}]
\draw [double distance=0.57cm] (0,0) [partial ellipse=-0.1:180.1:3.3 and 2.5];
\draw [red] (0,0) [partial ellipse=0:180:3.3 and 2.5];
\end{scope} % upper
\begin{scope}[shift={(2.5, 0)}]
\draw [line width=0.6cm] (0,0) [partial ellipse=180:360:2 and 1.5];
\draw [white, line width=0.57cm] (0,0) [partial ellipse=178:362:2 and 1.5];
\draw [red] (0,0) [partial ellipse=178:362:2 and 1.5];
\end{scope}
\begin{scope}[shift={(2.5, 0)}]
\draw [line width=0.6cm] (0,0) [partial ellipse=180:360:3.3 and 2.5];
\draw [white, line width=0.57cm] (0,0) [partial ellipse=178:362:3.3 and 2.5];
\draw [red] (0,0) [partial ellipse=178:362:3.3 and 2.5];
\end{scope}
\draw [line width=0.6cm] (0,0) [partial ellipse=180:360:2 and 1.5];
\draw [white, line width=0.57cm] (0,0) [partial ellipse=178:362:2 and 1.5];
\draw [blue] (0,0) [partial ellipse=178:362:2.05 and 1.55];
\draw [blue] (0,0) [partial ellipse=178:362:2.15 and 1.65];
\draw [blue] (0,0) [partial ellipse=178:362:1.85 and 1.35];
\begin{scope}[shift={(-2, -0.1)}]
\draw [blue, dashed](0,0) [partial ellipse=0:180:0.37 and 0.25];
\draw [blue,] (0,0) [partial ellipse=180:360:0.37 and 0.25];
\end{scope} 
\begin{scope}[shift={(-2, 0.1)}]
\draw [blue, dashed](0,0) [partial ellipse=0:180:0.37 and 0.25];
\draw [blue,] (0,0) [partial ellipse=180:360:0.37 and 0.25];
\end{scope} 
\begin{scope}[shift={(-2, -0.2)}]
\draw [fill=white] (-0.2, -0.23) rectangle (0.3, 0.24);
\end{scope}% lower 
\begin{scope}[shift={(4.5, 0)}]
\draw [blue, dashed](0,0) [partial ellipse=0:180:0.37 and 0.25] node [below left, black] {\tiny $X$};
\draw [white, line width=4pt] (0,0) [partial ellipse=270:250:0.37 and 0.25];
\draw [blue, -<-=0.3] (0,0) [partial ellipse=180:360:0.37 and 0.25] node [pos=0.7] {\tiny $\bullet$};
\end{scope}
\begin{scope}[shift={(5.8, 0)}]
\draw [red, dashed] (0,0) [partial ellipse=0:180:0.37 and 0.25];
\draw [white, line width=4pt] (0,0) [partial ellipse=270:250:0.37 and 0.25];
\draw [red] (0,0) [partial ellipse=180:360:0.37 and 0.25];
\begin{scope}[shift={(0, -0.2)}]
\path [fill=white] (-0.2, -0.23) rectangle (0.3, 0.24);
% \draw [red] (0, 0.24) .. controls +(0, -0.2) and +(-0.2, -0.1) .. (0.3, 0.05);
% \draw [red] (-0.2, -0.01) .. controls +(0.2, -0.1) and +(-0.1, 0.2) .. (-0.02, -0.23);
\draw[red] (-0.05,-0.23)..controls +(0,0.2) and +(-0.2,-0.1)..(0.3,0.05);
\draw[red] (0,0.24)..controls +(0,-0.2) and +(0.2,-0.05)..(-0.2,-0.01);
\end{scope}
\end{scope}
\end{tikzpicture}}}}\,,
\end{align*}
where $f: R_B(\bS^3, \Sigma) \to R_B(\bS^3, \Sigma')$ is a homeomorphism of 3-manifolds with boundary such that $\Sigma' = f(\Sigma) \cong \mathbb{T}_\Theta \sqcup \mathbb{T}_V \sqcup \mathbb{T}_{(X, e_X)}$. By Theorem \ref{thm:partition function}, we have $Z(\bS^3, \Sigma, \Xi) = Z(\bS^3, \Sigma', f(\Xi))$. Now we express the two sides of the above equation in terms of generalized FS-indicators.

By Lemma \ref{lem:twist}, we have
\begin{align*}
  Z(\bS^3, \Sigma, \Xi) = \mu Z(\bS^3, \Sigma_L, \tilde{\Gamma}_{(m,m+r)} \sqcup \Phi)= 
\mu Z\left(\vcenter{\hbox{\scalebox{0.7}{
\begin{tikzpicture}[scale=0.8]
\draw [double distance=0.57cm] (0,0) [partial ellipse=-0.1:180.1:2 and 1.5];
\draw [blue] (0,0) [partial ellipse=0:180:2.05 and 1.55];
\draw [blue] (0,0) [partial ellipse=0:180:2.15 and 1.65];
\draw [blue] (0,0) [partial ellipse=0:180:1.85 and 1.35];
\begin{scope}[shift={(2.5, 0)}]
\draw [double distance=0.57cm] (0,0) [partial ellipse=-0.1:180.1:2 and 1.5];
\draw [red] (0,0) [partial ellipse=0:180:2 and 1.5];
\end{scope} % upper
\begin{scope}[shift={(2.5, 0)}]
\draw [line width=0.6cm] (0,0) [partial ellipse=180:360:2 and 1.5];
\draw [white, line width=0.57cm] (0,0) [partial ellipse=178:362:2 and 1.5];
\draw [red] (0,0) [partial ellipse=178:362:2 and 1.5];
\end{scope}
\draw [line width=0.6cm] (0,0) [partial ellipse=180:360:2 and 1.5];
\draw [white, line width=0.57cm] (0,0) [partial ellipse=178:362:2 and 1.5];
\draw [blue] (0,0) [partial ellipse=178:362:2.05 and 1.55];
\draw [blue] (0,0) [partial ellipse=178:362:2.15 and 1.65];
\draw [blue] (0,0) [partial ellipse=178:362:1.85 and 1.35];
\begin{scope}[shift={(-2, -0.1)}]
\draw [blue, dashed](0,0) [partial ellipse=0:180:0.37 and 0.25];
\draw [blue,] (0,0) [partial ellipse=180:360:0.37 and 0.25];
\end{scope} 
\begin{scope}[shift={(-2, 0.1)}]
\draw [blue, dashed](0,0) [partial ellipse=0:180:0.37 and 0.25];
\draw [blue,] (0,0) [partial ellipse=180:360:0.37 and 0.25];
\end{scope} 
\begin{scope}[shift={(-2, -0.2)}]
\draw [fill=white] (-0.2, -0.23) rectangle (0.3, 0.24);
\end{scope}% lower 
\begin{scope}[shift={(4.5, 0)}]
\draw [blue, dashed](0,0) [partial ellipse=0:180:0.37 and 0.25] node [below left, black] {\tiny $X$};
\draw [white, line width=4pt] (0,0) [partial ellipse=270:250:0.37 and 0.25];
\draw [blue, -<-=0.3] (0,0) [partial ellipse=180:360:0.37 and 0.25] node [pos=0.7] {\tiny $\bullet$};
\end{scope}
\end{tikzpicture}}}}\right)
\end{align*}
where $\tilde{\Gamma}_{(m,m+r)}$ is the $\cC$-colored graph on $\mathbb{T}_{V}$ obtained from changing the $V$-colored $(m,r)$-curve in $\Gamma_{m,r}$ into a $V$-colored $(m, m+r)$-curve. The box in the right hand side of the above equation is understood as 
$\raisebox{-1em}{\begin{tikzpicture}[scale=1.5]
\draw (-0.5, -0.5) rectangle (0.5,0.5);
\draw [blue, ->-=0.5] (0.3, 0.5) .. controls +(0, -0.15) and +(-0.1, 0) .. (0.5, 0.2);
\draw [blue, ->-=0.45] (-0.1, 0.5) .. controls +(0.08, -0.25) and +(-0.2, 0) .. (0.5, -0.2);
\draw [blue, ->-=0.5] (-0.5, 0.4) --(0.5, -0.4);
\draw [blue, -<-=0.5] (-0.3, -0.5) .. controls +(0, 0.15) and +(0.1, 0) .. (-0.5, -0.2);
\draw [blue, -<-=0.5] (0.1,- 0.5) .. controls +(-0.1, 0.25) and +(0.2, 0) .. (-0.5, 0.2);
\draw [decorate, decoration = {brace}] (-0.1,0.52) --  (0.3,0.52) node [pos=0.5, above] {\tiny $m$} ;
\draw [decorate, decoration = {brace}] (0.52,0.2) --  (0.52,-0.4) node [pos=0.5, right] {\tiny $m+r$} ;
\end{tikzpicture}}$. By a similar argument as in Theorem \ref{thm:nu-and-z}, we have
\begin{align*}
Z(\bS^3, \Sigma, \Xi)=Z(\bS^3, \Sigma_L, \tilde{\Gamma}_{(m,m+r)} \sqcup \Phi)=\mu^2 \nu_{(m, m+\ell)}^{(X, e_X)}(V) =\mu^2 \nu_{(m, \ell)\mathfrak{t}}^{(X, e_X)}(V)\,.
\end{align*}

For $(\bS^3, \Sigma', f(\Xi))$, by Lemma \ref{lem:twist}, Theorem \ref{thm:nu-and-z} and the linearity of $\nu$, we have 
\begin{align*}
Z(\bS^3, \Sigma', f(\Xi)) = \mu \theta_{(X, e_X)}^{-1} Z(\bS^3, \Sigma_L, \Gamma_{m, \ell} \sqcup \Phi)
=\mu^2\theta_{(X,e_X)}^{-1} \nu_{(m, \ell)}^{(X, e_X)}(V)
=\mu^2 \nu_{(m, \ell)}^{\tilde{\mathfrak{t}}(X, e_X)}(V)\,.
\end{align*}
Therefore, 
\[
\nu_{(m, \ell)\mathfrak{t}}^{(X, e_X)}(V) = \frac{1}{\mu^2}Z(\bS^3, \Sigma, \Gamma) = \frac{1}{\mu^2}Z(\bS^3, \Sigma', \Gamma') = \nu_{(m, \ell)}^{\tilde{\mathfrak{t}}(X, e_X)}(V)\,.\qedhere
\]
\end{proof}

\begin{remark}
Our topological proof gives a unified treatment on the $\mathfrak{t}$-equivariance of the indicators, i.e., there is no need to discuss whether $m=0$ or not. In particular, this avoids the use of $\SL_2(\bZ)$-relations in the case when $m=0$ as in \cite{NgSch10}, because Dehn twists themselves know these relations.
\end{remark}

\begin{remark}
In the above proof, we argued that the $\mathfrak{t}$ on the left hand side is a result of handle sliding along a curve {\it linked} to $\mathbb{T}_V$ and the  $\tilde{\mathfrak{t}}$ on the right hand side is given by handle slide along a curve {\it parallel} to $\mathbb{T}_{(X, e_X)}$. This can also be explained as the extension of a right-hand Dehn twist on $\mathbb{T}_V$ in the $B$-colored region restricts to a left-hand Dehn twist on $\mathbb{T}_{(X,e_X)}$. This phenomenon happens in general, which is the key to the equivariance of the generalized FS-indicator.
\end{remark}

For any $V \in \cC$ and $(m, \ell) \in \bZ^2$, if $m \ne 0$, write $\ell = -qm +\sgn(m)r$ as before. Define $\tilde{\Gamma}_{(m, \ell)}(V)$, or simply, $\tilde{\Gamma}_{(m, \ell)}$, to be the $\cC$-colored diagram on a torus $\mathbb{T}$ obtained from performing right-hand Dehn twist $q$-times (if $q < 0$, this means we need to apply left-hand Dehn twist $|q|$-times). If $m=0$, define $\tilde{\Gamma}_{(0, \ell)}(V) = \Gamma_{(0, \ell)}(V)$. In particular, $\tilde{\Gamma}_{(m, \ell)}$ is a $\tilde{V}$-colored $(|m|, |\ell|)$-curve on the torus. We depict $(\bS^3, \mathbb{T}, \tilde{\Gamma}_{(m,\ell)}(V))$ by 
\[(\bS^3, \mathbb{T}, \tilde{\Gamma}_{(m, r)}(V)) = (\bS^3, \mathbb{T}, \tilde{\Gamma}_{(m, r)})=
\vcenter{\hbox{\scalebox{0.8}{
\begin{tikzpicture}[scale=0.8]
\draw [double distance=0.77cm] (0,0) [partial ellipse=-0.1:180.1:2 and 1.5];
\draw[blue] (0,0) [partial ellipse=0:180:1.8 and 1.3];
\draw[blue] (0,0) [partial ellipse=0:180:2.2 and 1.7];
\draw [blue] (0,0) [partial ellipse=0:180:2.1 and 1.6];
\draw [line width=0.8cm] (0,0) [partial ellipse=180:360:2 and 1.5];
\draw [white, line width=0.77cm] (0,0) [partial ellipse=178:362:2 and 1.5];
\draw[blue] (0,0) [partial ellipse=178:362:1.8 and 1.3];
\draw[blue] (0,0) [partial ellipse=178:362:2.2 and 1.7];
\draw [blue] (0,0) [partial ellipse=178:362:2.1 and 1.6];
\begin{scope}[shift={(-1.98, 0.2)}]
\draw [blue, dashed](0,0) [partial ellipse=0:180:0.5 and 0.3];
\draw [blue] (0,0) [partial ellipse=180:360:0.5 and 0.3];
\end{scope} 
\begin{scope}[shift={(-1.99, -0.1)}]
\draw [blue, dashed](0,0) [partial ellipse=180:145:0.5 and 0.3];
\draw [blue, dashed](0,0) [partial ellipse=0:45:0.5 and 0.3];
\draw [blue] (0,0) [partial ellipse=180:360:0.5 and 0.3];
\end{scope} 
\begin{scope}[shift={(-1.98, -0.2)}]
\draw [fill=white] (-0.28,-0.3)--(-0.31, 0.3)--(0.36,0.3)--(0.39,-0.3)--(-0.28,-0.3);
\node at (0.05,0) {\tiny $\sim$};
\end{scope}
\node at (-1.8,-2) {\tiny $\tilde{V}$};
\node at (-0.4,1.4) {\tiny $|m|$};
\end{tikzpicture}}}}\]
where the tilded box is understood as filled with strings colored by $\tilde{V}$ such that the result of the filling is the $(|m|, |\ell|)$-curve in the definition of $\tilde{\Gamma}_{(m,\ell)}$.

\begin{corollary}\label{cor:general-m-ell}
For any $V \in \cC$, $(X, e_X) \in \cZ(\cC)$, and $(m, \ell) \in \bZ^2$, we have
\[Z(\bS^3, \Sigma_L, \tilde{\Gamma}_{(m,\ell)}(V) \sqcup \Phi(X, e_X)) = Z(\bS^3, \mathbb{T}_{V} \sqcup \mathbb{T}_{(X, e_X)}, \tilde{\Gamma}_{(m,\ell)} \sqcup \Phi) =  \mu\nu_{(m,\ell)}^{(X, e_X)}(V)\,,\]
where in $\mathbb{T}_V$, the torus containig $\tilde{\Gamma}_{(m,\ell)}(V)$, is on the left.
\end{corollary}
\begin{proof}
  The statement follows directly from Theorem \ref{thm:nu-and-z} and Proposition \ref{prop:t-equiv}.
\end{proof}

Now we are able to show the $\mathfrak{s}$-equivariance of $\nu$. 

\begin{proposition}[$\mathfrak{s}$-equivariance]\label{prop:s-equiv}
  For any $V \in \cC$, $(X, e_X) \in \cZ(\cC)$, and $(m, \ell) \in \bZ^2$, we have
\[\nu_{(m,\ell)\mathfrak{s}}^{(X, e_X)}(V)= \nu_{(m,\ell)}^{\tilde{\mathfrak{s}}(X,e_X)}(V)\,.\]
\end{proposition}
\begin{proof}
We will use the notations in Corollary \ref{cor:general-m-ell}. Let $f$ be an orientation preserving homeomorphism on the torus representing the mapping class $\mathfrak{s} \in \SL_2(\bZ)$, i.e., it first switches the meridian and longitude, then (to preserve orientation) reverses the orientation of the new longitude. Then the $\cC$-decorated 3-alterfolds $(\bS^3, f(\mathbb{T}_{V}) \sqcup \mathbb{T}_{(X, e_X)}, f(\tilde{\Gamma}_{(m,\ell)}) \sqcup \Phi)$ is homeomorphic to $(\bS^3, \mathbb{T}_{V} \sqcup f'(\mathbb{T}_{(X, e_X)}), \tilde{\Gamma}_{(m,\ell)} \sqcup f'(\Phi))$, where $f'$ is obtained by ``pushing'' $f$ to $R_B(\bS^3, \mathbb{T}_V, \tilde{\Gamma}_{(m,\ell)}) \cong \mathbb{T}\times I$, and then ``restract'' to $\mathbb{T}_{(X, e_X)}$. To be more precise, let $j$ be the orientation reversing homeomorphism on a torus $\mathbb{T}$ that identifies two boundaries of $\mathbb{T}\times I$ that is degree 1 on the meridian direction, and degree on the longitude direction. The induced action of $j$ on the homology group $H_1(\mathbb{T}, \bZ)$ is represented by $\mathfrak{j} =\begin{pmatrix}  1 & 0 \\ 0 & -1\end{pmatrix}$.  Then $f'$ is a conjugation of $f$ by $j$ from $\mathbb{T}_{(X, e_X)}$ to $\mathbb{T}_V$, and its mapping class is exactly $\tilde{\mathfrak{s}}$.

Therefore, by Theorem \ref{thm:partition function}, 
\begin{align*}
Z(\bS^3, f(\mathbb{T}_{V}) \sqcup \mathbb{T}_{(X, e_X)}, f(\tilde{\Gamma}_{(m,\ell)}) \sqcup \Phi) = Z(\bS^3, \mathbb{T}_{V} \sqcup f'(\mathbb{T}_{(X, e_X)}), \tilde{\Gamma}_{(m,\ell)} \sqcup f'(\Phi))
\end{align*}
On the one hand, by Corollary \ref{cor:general-m-ell}, the left hand side equals to 
\begin{align*}
Z(\bS^3, f(\mathbb{T}_{V}) \sqcup \mathbb{T}_{(X, e_X)}, \tilde{\Gamma}_{(\ell, -m)} \sqcup \Phi) = \mu\nu_{(m, \ell)\mathfrak{s}}^{(X, e_X)}(V)\,.
\end{align*}
On the other hand, 
\begin{align*}
Z(\bS^3, \mathbb{T}_{V} \sqcup f'(\mathbb{T}_{(X, e_X)}), \tilde{\Gamma}_{(m,\ell)} \sqcup f'(\Phi)) = 
\vcenter{\hbox{\scalebox{0.8}{
\begin{tikzpicture}[scale=0.8]
\draw [double distance=0.6cm] (0,0) [partial ellipse=-0.1:180.1:2 and 1.5];
\draw [blue] (0,0) [partial ellipse=0:180:2.05 and 1.55];
\draw [blue] (0,0) [partial ellipse=0:180:2.15 and 1.65];
\draw [blue] (0,0) [partial ellipse=0:180:1.85 and 1.35];
\begin{scope}[shift={(2.5, 0)}]
\draw [double distance=0.6cm] (0,0) [partial ellipse=-0.1:180.1:2 and 1.5];
\draw [blue] (0,0) [partial ellipse=0:180:2 and 1.5];
\end{scope} % upper
\begin{scope}[shift={(2.5, 0)}]
\draw [line width=0.63cm] (0,0) [partial ellipse=180:360:2 and 1.5];
\draw [white, line width=0.6cm] (0,0) [partial ellipse=178:362:2 and 1.5];
\draw [blue, ->-=0.6] (0,0) [partial ellipse=178:362:2 and 1.5] node [pos=0.4, below, black] {\tiny $X$};
\end{scope}
\draw [line width=0.63cm] (0,0) [partial ellipse=180:360:2 and 1.5];
\draw [white, line width=0.6cm] (0,0) [partial ellipse=178:362:2 and 1.5];
\draw [blue] (0,0) [partial ellipse=178:362:2.05 and 1.55];
\draw [blue] (0,0) [partial ellipse=178:362:2.15 and 1.65];
\draw [blue] (0,0) [partial ellipse=178:362:1.85 and 1.35];
\begin{scope}[shift={(-2, -0.1)}]
\draw [blue, dashed](0,0) [partial ellipse=0:180:0.37 and 0.25];
\draw [blue,] (0,0) [partial ellipse=180:360:0.37 and 0.25];
\end{scope} 
\begin{scope}[shift={(-2, 0.1)}]
\draw [blue, dashed](0,0) [partial ellipse=0:180:0.37 and 0.25];
\draw [blue,] (0,0) [partial ellipse=180:360:0.37 and 0.25];
\end{scope} 
\begin{scope}[shift={(-2, -0.2)}]
\draw [fill=white] (-0.2, -0.23) rectangle (0.3, 0.24);
\end{scope}% lower 
\begin{scope}[shift={(4.5, 0)}]
\draw [red, dashed](0,0) [partial ellipse=0:180:0.37 and 0.25];
\draw [red] (0,0) [partial ellipse=180:265:0.37 and 0.25];
\draw [red] (0,0) [partial ellipse=285:360:0.37 and 0.25];
\end{scope}
\end{tikzpicture}}}}
\end{align*}
and by Proposition \ref{prop:Stransform}, the latter equals to 
\begin{align*}
\sum_{(Y, e_Y)}  \mu^{-1} S_{(X, e_X), (Y, e_Y)}
\vcenter{\hbox{\scalebox{0.8}{
\begin{tikzpicture}[scale=0.8]
\draw [double distance=0.6cm] (0,0) [partial ellipse=-0.1:180.1:2 and 1.5];
\draw [blue] (0,0) [partial ellipse=0:180:2.05 and 1.55];
\draw [blue] (0,0) [partial ellipse=0:180:2.15 and 1.65];
\draw [blue] (0,0) [partial ellipse=0:180:1.85 and 1.35];
\begin{scope}[shift={(2.5, 0)}]
\draw [double distance=0.6cm] (0,0) [partial ellipse=-0.1:180.1:2 and 1.5];
\draw [red] (0,0) [partial ellipse=0:180:2 and 1.5];
\end{scope} % upper
\begin{scope}[shift={(2.5, 0)}]
\draw [line width=0.63cm] (0,0) [partial ellipse=180:360:2 and 1.5];
\draw [white, line width=0.6cm] (0,0) [partial ellipse=178:362:2 and 1.5];
\draw [red] (0,0) [partial ellipse=178:362:2 and 1.5];
\end{scope}
\draw [line width=0.63cm] (0,0) [partial ellipse=180:360:2 and 1.5];
\draw [white, line width=0.6cm] (0,0) [partial ellipse=178:362:2 and 1.5];
\draw [blue] (0,0) [partial ellipse=178:362:2.05 and 1.55];
\draw [blue] (0,0) [partial ellipse=178:362:2.15 and 1.65];
\draw [blue] (0,0) [partial ellipse=178:362:1.85 and 1.35];
\begin{scope}[shift={(-2, -0.1)}]
\draw [blue, dashed](0,0) [partial ellipse=0:180:0.37 and 0.25];
\draw [blue,] (0,0) [partial ellipse=180:360:0.37 and 0.25];
\end{scope} 
\begin{scope}[shift={(-2, 0.1)}]
\draw [blue, dashed](0,0) [partial ellipse=0:180:0.37 and 0.25];
\draw [blue,] (0,0) [partial ellipse=180:360:0.37 and 0.25];
\end{scope} 
\begin{scope}[shift={(-2, -0.2)}]
\draw [fill=white] (-0.2, -0.23) rectangle (0.3, 0.24);
\end{scope}% lower 
\begin{scope}[shift={(4.5, 0)}]
\draw [blue, dashed](0,0) [partial ellipse=0:180:0.37 and 0.25] node [below left, black] {\tiny $Y$};
\draw [white, line width=4pt] (0,0) [partial ellipse=270:250:0.37 and 0.25];
\draw [blue, ->-=0.3] (0,0) [partial ellipse=180:360:0.37 and 0.25] node [pos=0.7] {\tiny $\bullet$};
\end{scope}
\end{tikzpicture}}}}\,,
\end{align*}
where here and below, the summation over $(Y,e_Y)$ is understood as summing over simple objects in $\mathcal{Z(C)}$. The $\cC$-diagram on the right torus is understood as $\Phi(Y^*, e_{Y^*})$ since the $Y$-colored string is oriented differently from that in Equation \eqref{eq:Gamma-and-Phi}. Therefore, the diagram evaluates to
\begin{align*}
&\sum_{(Y, e_Y)}  \mu^{-1} S_{(X, e_X), (Y^*, e_{Y^*})} Z(\bS^3, \mathbb{T}_V \sqcup \mathbb{T}_{(Y,e_Y)}, \tilde{\Gamma}_{(m,\ell)}(V) \sqcup \Phi(Y,e_{Y}))\\ 
=&\mu\sum_{(Y, e_Y)}  \mu^{-1} S_{(X, e_X), (Y^*, e_{Y^*})} \nu_{(m,\ell)}^{(Y,e_Y)}(V)
=\mu\sum_{(Y, e_Y)}  \mu^{-1} S_{(Y,e_Y), (X^*, e_{X^*})} \nu_{(m,\ell)}^{(Y,e_Y)}(V)\\
=& \mu \nu_{(m,\ell)}^{\tilde{\mathfrak{s}}(X, e_X)}(V)\,,
\end{align*}
where the first equality follows from Corollary \ref{cor:general-m-ell}, the second equality is based on the well-known property of the $S$-matrix (see, for example, \cite[Equation (3.13)]{BakKir01}), and the last one follows from \cite[Equation (3.17)]{BakKir01} and the linearity of $\nu$. Since $\mu \ne 0$, we have $\nu_{(m,\ell)\mathfrak{s}}^{(X, e_X)}(V)= \nu_{(m,\ell)}^{\tilde{\mathfrak{s}}(X,e_X)}(V)$ as desired.
\end{proof}

Combining the above propositions, we obtain the $\SL_2(\bZ)$-equivariance of the generalized Frobenius-Schur indicators.

\begin{theorem}[$\SL_2(\bZ)$-equivariance]\label{thm:equivariance}
    For any $V \in \cC$, $(X, e_X) \in \cZ(\cC)$,  $(m, \ell) \in \bZ^2$, and $\mathfrak{g} \in \SL_2(\bZ)$, we have
\[\nu_{(m,\ell)\mathfrak{g}}^{(X, e_X)}(V)= \nu_{(m,\ell)}^{\tilde{\mathfrak{g}}(X,e_X)}(V)\,.\]
\end{theorem}
\begin{proof}
  Follows from the above and that conjugation by $\mathfrak{j}$ is automorphism of $\SL_2(\bZ)$. 
\end{proof}

\section{Summary and Further Researches}

In this section, we summarize the dictionary for the topologized notions related to the Drinfeld center and tube algebra in Figure \ref{tab:top}. 
Next, we discuss the future work on topological quantum field theory associated to the partition function $Z$, the generalized Frobenius-Schur indicator on higher genus handlebodies and quantum Fourier analysis respectively in the following.
\begin{figure}
\centering
\begin{tabular}{|c|c|c|c|c|}
\hline
     &\small  Drinfeld Center & \small Braiding & \small Half-Braiding & \small Minimal Central Idempotent\\
     \hline
    Topologized & Tube Category & $\vcenter{\hbox{\scalebox{0.5}{
\begin{tikzpicture}
\draw  (1.6, 0) .. controls +(0, 0.3) and +(0, 0.3) .. (2.4, 0);
\draw (1.6, 0) .. controls +(0,-0.3) and +(0, -0.3) .. (2.4, 0);
\draw  (1.6, 0) .. controls +(0, -0.3) and +(0.3, 0) ..(1, -0.7) .. controls +(-0.7, 0) and +(0,0.7) ..(-0.4, -2);
\draw  (2.4, 0) .. controls +(0, -0.7) and +(0.7, 0) ..(1, -1.3) .. controls +(-0.3, 0) and +(0,0.3) ..(0.4, -2);
\draw [blue,->-=0.1, ->-=0.9] (2, -0.21) .. controls +(0, -0.5) and +(0.5, 0) ..(1, -1) .. controls +(-0.4, 0) and +(0,0.4) ..(0, -2)--(0, -2.2) node [below, black] {\tiny $Y$};
\draw [ dashed] (-0.4, -2) .. controls +(0, 0.3) and +(0, 0.3) .. (0.4, -2);
\draw (-0.4, -2) .. controls +(0,-0.3) and +(0, -0.3) .. (0.4, -2); % right tube
\path [fill=white]  (0.15,-0.69) rectangle (1.85, -1.31);
\draw (-0.4, 0) .. controls +(0, 0.3) and +(0, 0.3) .. (0.4, 0);
\draw  (-0.4, 0) .. controls +(0,-0.3) and +(0, -0.3) .. (0.4, 0);
\draw  (0.4, 0) .. controls +(0, -0.3) and +(-0.3, 0) ..(1, -0.7) .. controls +(0.7, 0) and +(0,0.7) ..(2.4, -2);
\draw  (-0.4, 0) .. controls +(0, -0.7) and +(-0.7, 0) ..(1, -1.3) .. controls +(0.3, 0) and +(0,0.3) ..(1.6, -2);
\draw [blue,->-=0.5] (0, -0.21) .. controls +(0, -0.5) and +(-0.5, 0) ..(1, -1) .. controls +(0.5, 0) and +(0,0.4) ..(2, -2)--(2, -2.2) node [below, black] {\tiny $X$};
\draw [dashed] (1.6, -2) .. controls +(0, 0.3) and +(0, 0.3) .. (2.4, -2);
\draw  (1.6, -2) .. controls +(0,-0.3) and +(0, -0.3) .. (2.4, -2); % left tube
\end{tikzpicture}}}}$ & $\displaystyle \sum_{i,j}d_i^{1/2}d_j^{1/2}
\vcenter{\hbox{\scalebox{0.4}{\begin{tikzpicture}[scale=0.35]
\draw (4, -5)--(-4,-6)--(-4, 7) --(4, 8)  (4, -5)--( 4, -2.5) (4, 4)--(4, 8);
\draw [dashed]( 4, -2.5)--(4, 4);
\draw [dashed] (0,5) [partial ellipse=0:360:2 and 0.8];
\draw (-2, 5)..controls +(0,-1) and +(0,1)..(3,1.5)--(3,0.5)..controls +(0,-1) and +(0,1)..(-2,-4);
\draw (2, 5)..controls +(0,-1) and +(0,1)..(7,1.5)--(7,0.5)..controls +(0,-1) and +(0,1)..(2,-4);
\draw [blue] (0, 5)..controls +(0,-1) and +(0,1)..(5,1.5)--(5,0.5)..controls +(0,-1) and +(0,1)..(0,-4);
\draw[dashed] (0,-4) [partial ellipse=0:180:2 and 0.8];
\draw [dashed] (0,-4) [partial ellipse=180:360:2 and 0.8];
\draw[blue, ->-=0.5] (0, 7.5) node[above, black]{\tiny{${X}$}}->(0, 5);
\draw[blue, ->-=0.5] (0, -4)--(0, -5.5)node[below, black]{\tiny{${X}$}};
% \draw [red](5, 1.25) [partial ellipse=-75:0:2 and 0.8];
% \draw [red, dashed](5, 1.25) [partial ellipse=0:180:2 and 0.8];
% \draw [red](5, 1.25) [partial ellipse=180:255:2 and 0.8];
% \node [draw, fill=white] at (5, 0.5){\tiny $f$};
\draw [blue] (0, -6) [partial ellipse=40:10:2.5 and 4] node [below, black] {\tiny $X_{i}$};
\draw [blue,dashed, ->-=0.5] (0, -6) [partial ellipse=90:40:2.5 and 4];
\draw [blue] (0, -6) [partial ellipse=90:179:2.5 and 4] node [below, black] {\tiny $X_{i}^*$};
\draw [blue] (0, 7) [partial ellipse=-40:10:2.5 and 4] node [above, black] {\tiny $X_{j}$};
\draw [blue,dashed, ->-=0.5] (0, 7) [partial ellipse=-40:-90:2.5 and 4];
\draw [blue] (0, 7) [partial ellipse=-90:-182:2.5 and 4] node [above, black] {\tiny $X_{j}^*$};
\draw [blue] (0, 4) [partial ellipse=-90:-40:3.5 and 3];
\draw [blue,dashed] (0, 4) [partial ellipse=0:-40:3.5 and 3] ;
\draw [blue] (3.5, 4)--(3.5, 8) node [above, black]{\tiny $V$};
\draw [blue] (0, -3) [partial ellipse=90:180:3.2 and 4];
\draw [blue, ->-=0.1] (-3.2, -3)--(-3.2, -6);
\end{tikzpicture}}}}$ & $\displaystyle \frac{d(X)}{\mu^2}\sum_{i=0}^r\vcenter{\hbox{\scalebox{0.4}{
\begin{tikzpicture}[xscale=0.8, yscale=0.6]
\draw [line width=0.83cm] (0,0) [partial ellipse=-0.1:180.1:2 and 1.5];
\draw [white, line width=0.8cm] (0,0) [partial ellipse=-0.1:180.1:2 and 1.5];
\draw [red] (0,0) [partial ellipse=0:180:2 and 1.5];
\path [fill=white](-0.65, 0) rectangle (0.65, 2.5);
\begin{scope}[shift={(0,3)}]
\draw (0,0) [partial ellipse=0:360:0.6 and 0.3];
\end{scope}
\draw (-0.6, 3)--(-0.6, 0) (0.6, 3)--(0.6, 0); 
\draw [blue, -<-=0.5] (0, 0)--(0, 2.7) node [pos=0.5, right, black] {\tiny $X_i$};% upper one
\draw (-0.6, -3)--(-0.6, 0) (0.6, -3)--(0.6, 0);
\draw [blue] (0, 0)--(0, -3.3);
\draw [line width=0.83cm] (0,0) [partial ellipse=180:360:2 and 1.5];
\draw [white, line width=0.8cm] (0,0) [partial ellipse=178:362:2 and 1.5];
\draw [red] (0,0) [partial ellipse=178:362:2 and 1.5];
\begin{scope}[shift={(0,-3)}]
\draw [dashed](0,0) [partial ellipse=0:180:0.6 and 0.3];
\draw (0,0) [partial ellipse=180:360:0.6 and 0.3];
\end{scope}
\begin{scope}[shift={(-2, 0)}]
\draw [blue, dashed](0,0) [partial ellipse=0:180:0.5 and 0.3] node [below left, black] {\tiny $X$};
\draw [white, line width=4pt] (0,0) [partial ellipse=270:290:0.5 and 0.3];
\draw [blue, ->-=0.3] (0,0) [partial ellipse=180:360:0.5 and 0.3];
\end{scope}
\end{tikzpicture}}}}$ \\
\hline
\end{tabular}
\begin{tabular}{|c|c|c|c|c|}
     & $\Omega$-Color & \small $S$-Matrix Coefficients & \small FS Indicators & \small Twist\\
     \hline
    Topologized &  $\vcenter{\hbox{
 \begin{tikzpicture}[scale=0.5]
 \draw (0, 0) [partial ellipse=0:360:1 and 1];
\draw (0, 0) [partial ellipse=0:360:1.5 and 1.5];
 \draw[red] (0, 0) [partial ellipse=0:360:1.25 and 1.25];
 \end{tikzpicture}
 }}$ & $\frac{1}{\mu}
\vcenter{\hbox{\scalebox{0.5}{
\begin{tikzpicture}[scale=0.8]
\draw [line width=0.83cm] (0,0) [partial ellipse=0:180:2 and 1.5];
\draw [line width=0.8cm, white] (0,0) [partial ellipse=-0.2:180.2:2 and 1.5];
\draw [red] (0,0) [partial ellipse=0:180:2 and 1.5];
\begin{scope}[shift={(2, 0)}]
\draw [line width=0.83cm] (0,0) [partial ellipse=0:180:2 and 1.5];
\draw [line width=0.8cm, white] (0,0) [partial ellipse=-0.2:180.2:2 and 1.5];
\draw [red] (0,0) [partial ellipse=0:180:2 and 1.5];
\end{scope} % upper
\begin{scope}[shift={(2, 0)}]
\draw [line width=0.83cm] (0,0) [partial ellipse=180:360:2 and 1.5];
\draw [white,line width=0.8cm] (0,0) [partial ellipse=178:362:2 and 1.5];
\draw [red] (0,0) [partial ellipse=178:362:2 and 1.5];
\end{scope}
\draw [line width=0.83cm] (0,0) [partial ellipse=180:360:2 and 1.5];
\draw [white, line width=0.8cm] (0,0) [partial ellipse=178:362:2 and 1.5];
\draw [red] (0,0) [partial ellipse=178:362:2 and 1.5];
\begin{scope}[shift={(-2, 0)}]
\draw [blue, dashed](0,0) [partial ellipse=0:180:0.5 and 0.3] node [below left, black] {\tiny $X$};
\draw [white, line width=4pt] (0,0) [partial ellipse=270:290:0.5 and 0.3];
\draw [blue, ->-=0.3] (0,0) [partial ellipse=180:360:0.5 and 0.3];
\end{scope} % lower 
\begin{scope}[shift={(0, 0)}]
\draw [blue, dashed](0,0) [partial ellipse=0:180:0.5 and 0.3] node [below left, black] {\tiny $Y$};
\draw [white, line width=4pt] (0,0) [partial ellipse=270:290:0.5 and 0.3];
\draw [blue, -<-=0.7] (0,0) [partial ellipse=180:360:0.5 and 0.3];
\end{scope}
\end{tikzpicture}}}}$ & $\displaystyle \frac{1}{\mu}\vcenter{\hbox{\scalebox{0.6}{
\begin{tikzpicture}[scale=0.8]
\draw [double distance=0.8cm] (0,0) [partial ellipse=-0.1:180.1:2 and 1.5];
\draw[blue] (0,0) [partial ellipse=0:180:1.8 and 1.3];
\draw[blue] (0,0) [partial ellipse=0:180:2.2 and 1.7];
\draw [blue] (0,0) [partial ellipse=0:180:2.1 and 1.6];
\begin{scope}[shift={(2, 0)}]
\draw [double distance=0.8cm] (0,0) [partial ellipse=-0.1:180.1:2 and 1.5];
\draw [red] (0,0) [partial ellipse=0:180:2 and 1.5];
\end{scope} % upper
\begin{scope}[shift={(2, 0)}]
\draw [line width=0.83cm] (0,0) [partial ellipse=180:360:2 and 1.5];
\draw [white, line width=0.8cm] (0,0) [partial ellipse=178:362:2 and 1.5];
\draw [red] (0,0) [partial ellipse=178:362:2 and 1.5];
\end{scope}
\draw [line width=0.83cm] (0,0) [partial ellipse=180:360:2 and 1.5];
\draw [white, line width=0.8cm] (0,0) [partial ellipse=178:362:2 and 1.5];
\draw[blue] (0,0) [partial ellipse=178:362:1.8 and 1.3];
\draw[blue] (0,0) [partial ellipse=178:362:2.2 and 1.7];
\draw [blue] (0,0) [partial ellipse=178:362:2.1 and 1.6];
\begin{scope}[shift={(-1.98, 0.2)}]
\draw [blue, dashed](0,0) [partial ellipse=0:180:0.5 and 0.3];
\draw [blue] (0,0) [partial ellipse=180:360:0.5 and 0.3];
\end{scope} 
\begin{scope}[shift={(-1.99, -0.1)}]
\draw [blue, dashed](0,0) [partial ellipse=180:145:0.5 and 0.3];
\draw [blue, dashed](0,0) [partial ellipse=0:45:0.5 and 0.3];
\draw [blue] (0,0) [partial ellipse=180:360:0.5 and 0.3];
\end{scope} 
\begin{scope}[shift={(-1.98, -0.2)}]
\draw [fill=white] (-0.28,-0.3)--(-0.31, 0.3)--(0.36,0.3)--(0.39,-0.3)--(-0.28,-0.3);
\end{scope}
% lower 
\begin{scope}[shift={(4, 0)}]
\draw [blue, dashed](0,0) [partial ellipse=0:180:0.5 and 0.3]; 
\node at (-0.75,0) {\tiny $X$};
\draw [white, line width=4pt] (0,0) [partial ellipse=270:290:0.5 and 0.3];
\fill[white] (-0.05,-0.3) circle[radius=2pt];
\draw [blue, ->-=0.3] (0,0) [partial ellipse=180:360:0.5 and 0.3];
\fill [blue] (0.2,-0.275) circle[radius=2pt];
\end{scope}
\node at (-1.8,-2.2) {\tiny $\tilde{V}$}; \node at (-0.4,1.4) {\tiny $|m|$};
\end{tikzpicture}}}}$
& $\displaystyle \frac{1}{\mu}\vcenter{\hbox{\scalebox{0.5}{
\begin{tikzpicture}[scale=0.35]
\draw (0,5) [partial ellipse=0:360:2 and 0.8];
\draw (-2, 5)--(-2, -4);
\draw (2, 5)--(2, -4);
\draw[dashed] (0,-4) [partial ellipse=0:180:2 and 0.8];
\draw (0,-4) [partial ellipse=180:360:2 and 0.8] node[below, white]{\tiny{$X$}};
\draw[blue, ->-=0.2] (0, 4.2)--(0, -4.8);
\draw (0.6, 2.5) node{\tiny{$X$}};
\begin{scope}[shift={(0, 0)}, xscale=-5.4, yscale=4]
\draw [blue, dashed] (0,0) [partial ellipse=0:180:0.37 and 0.25];
\draw [white, line width=4pt] (0,0) [partial ellipse=270:250:0.37 and 0.25];
\draw [blue] (0,0) [partial ellipse=180:360:0.37 and 0.25];
\begin{scope}[shift={(0, -0.2)}]
\path [fill=white] (-0.2, -0.23) rectangle (0.3, 0.24);
\draw[blue] (0,-0.23)..controls +(0,0.2) and +(-0.2,-0.1)..(0.3,0.05);
\draw[blue] (0,0.24)..controls +(0,-0.2) and +(0.2,-0.05)..(-0.2,-0.01);
\end{scope}
\end{scope}
\begin{scope}[shift={(0, -2)}]
\draw [red](0, 0) [partial ellipse=-75:0:2 and 0.8];
\draw [red, dashed](0, 0) [partial ellipse=0:180:2 and 0.8];
\draw [red](0, 0) [partial ellipse=180:255:2 and 0.8];
\end{scope}
\end{tikzpicture}}}}$\\
\hline
\end{tabular}
\caption{The Dictionary for Topologized Notions}
\label{tab:top}
\end{figure}

It is natural to ask if our partition function $Z$ comes from a topological quantum field theory. 
The answer is yes. 
In a forthcoming paper, we show that by applying the universal construction (see \cite{BHMV2}), we will get a topological quantum field theory defined on the $\cC$-decorated cobordism category described as follows. 
The objects are bi-colored surfaces with $A$-$B$ colored regions separated by multi-curves with marked points colored by $\mathcal{C}$. 
The morphisms are $\mathcal{C}$ bi-colored cobordisms, i.e. 3-alterfold with bi-colored surfaces as its boundary. 
In this setting, our cobordism category is larger than the ones that been usually used in for $2+1$-dimensioncal TQFTs. Therefore, Different models(string-net, Levin-Wen, RT and TV) naturally embed in our TQFT so one can compare them in a unified setting. As a corollary, the Tureav-Viro TQFT of $\mathcal{C}$ equals to the Reshetikhin-Turaev TQFT of $\mathcal{Z(C)}$.

Theorem \ref{thm:nu-and-z} identifies the generalized Frobenius-Schur indicators as the partition functions on $\cC$-decorated bi-colored genus 1 handlebodies, and our TQFT described above gives rise to a straightforward generalization of the generalized Frobenius-Schur indicators to higher genus handlebodies, which we call {\it topological indicators}. 
Similar to the situation in Theorem \ref{thm:nu-and-z} in genus 1, consider two $\cC$-decorated bi-colored genus $g$ handlebodies in $\bS^3$ ``linked genus by genus'' (i.e., the handlebodies without coloring give rise to the genus $g$ Heegaard splitting of $\bS^3$) such that one of the handlebody is decorated with simple closed multicurves on the boundary surface colored by objects in $\cC$, and the other handlebody is given by a vector in the Reshetikhin-Turaev TQFT of $\mathcal{Z(C)}$ for the boundary surface of genus $g$ in the similar way as \cite[Chapter IV]{Tur94}. Color the interior of the above $\cC$-decorated linked handlebody by $A$, then we define the genus $g$ topological invariant associated to such a handlebody to be (up to scalar) the partition function of it. We will show that by the existence of our TQFT, the topological indicator is naturally equivariant with respect to the mapping class group action on the two boundary surfaces of the handlebodies, generalizing Theorem \ref{thm:equivariance}. The topological indicator is thus a pairing between the Reshetikhin-Turaev TQFT and the Turaev-Viro TQFT, using which we can study the interaction between these two TQFTs under our framework.

Suppose $\mathcal{C}$ is a unitary fusion category. 
In our framework, we have that $\mathcal{Z(C)}$ is braided monoidal equivalent to a unitary Drinfeld center. The topological quantum field theory associated to the quantum invariant $Z$ is also unitary which will be shown in the forthcoming paper. 
Quantum Fourier analysis \cite{JJLRW20} is a program studying Fourier analysis on quantum symmetries such as subfactors, unitary fusion categories, etc. Our framework can produce more interesting inequalities for quantum symmetries.

% \bibliographystyle{plain}
% \bibliography{TQFT}

\end{document}

%% file: tikz_command.tex
\usetikzlibrary{decorations.markings}

\tikzset{->-/.style={decoration={
  markings,
  mark=at position #1 with {\arrow{>}}},postaction={decorate}}}
\tikzset{-<-/.style={decoration={
  markings,
  mark=at position #1 with {\arrow{<}}},postaction={decorate}}}

\newcommand{\StabAndAfter}{\begin{tikzpicture}[ipe import, baseline, trim left]
  \draw
    (52.287, 0.1462)
     arc[start angle=-88.9742, end angle=91.0387, radius=16];
  \draw[shift={(51.714, 32.14)}, rotate=-178.9613]
    (0, 0)
     .. controls (8.0072, 0) and (8.0072, 4) .. (16.0072, 4);
  \draw[shift={(52.295, 0.145)}, rotate=-178.9613]
    (0, 0)
     .. controls (8.008, 0) and (8.008, -4) .. (16.008, -4);
  \draw[shift={(4.217, 3.854)}, rotate=-178.9613]
    (0, 0)
     .. controls (-2.6667, -8) and (-2.6667, -16) .. (0, -24);
  \draw[shift={(4.217, 3.854)}, rotate=-178.9613]
    (0, 0)
     .. controls (2.6667, -8) and (2.6667, -16) .. (0, -24);
  \draw
    (44.001, 17.6122)
     .. controls (48.001, 13.6122) and (56.001, 13.6122) .. (60.001, 17.6122);
  \draw
    (46.001, 16.1062)
     .. controls (50.001, 18.4435) and (54.001, 18.4435) .. (58.001, 16.1062);
  \draw
    (35.7818, 27.8492)
     -- (3.7818, 27.8492);
  \draw
    (4.2168, 3.8532)
     -- (36.2168, 3.8522);
  \draw[blue]
    (4.9708, 23.8482)
     .. controls (35.7818, 23.8492) and (43.7818, 19.8492) .. (46.0008, 16.1062);
  \draw[blue, dotted]
    (3.0298, 7.8492)
     .. controls (35.7818, 7.8492) and (43.7818, 11.8492) .. (46.0008, 16.1062);
  \draw[green]
    (51.7818, 15.8492) circle[radius=12];
  \draw[red]
    (58.0008, 16.1062)
     .. controls (59.7818, 19.8492) and (67.7818, 19.8492) .. (67.9978, 15.8492);
  \draw[red, dotted]
    (58.0018, 16.1072)
     .. controls (59.7818, 11.8492) and (67.7818, 11.8492) .. (67.9978, 15.8492);
  \draw
    (164.287, 0.1462)
     arc[start angle=-88.9742, end angle=91.0387, radius=16];
  \draw[shift={(163.714, 32.14)}, rotate=-178.9613]
    (0, 0)
     .. controls (8.0072, 0) and (8.0072, 4) .. (16.0072, 4);
  \draw[shift={(164.295, 0.145)}, rotate=-178.9613]
    (0, 0)
     .. controls (8.008, 0) and (8.008, -4) .. (16.008, -4);
  \draw[shift={(116.217, 3.854)}, rotate=-178.9613]
    (0, 0)
     .. controls (-2.6667, -8) and (-2.6667, -16) .. (0, -24);
  \draw[shift={(116.217, 3.854)}, rotate=-178.9613]
    (0, 0)
     .. controls (2.6667, -8) and (2.6667, -16) .. (0, -24);
  \draw
    (147.7818, 27.8492)
     -- (115.7818, 27.8492);
  \draw
    (116.2168, 3.8532)
     -- (148.2168, 3.8522);
  \draw[blue]
    (116.9708, 23.8492)
     .. controls (171.7818, 23.8492) and (179.7818, 19.8492) .. (179.7818, 15.8492);
  \draw[blue, dotted]
    (115.0298, 7.8482)
     .. controls (171.7818, 7.8492) and (179.7818, 11.8492) .. (179.7818, 15.8492);
  \draw[blue]
    (31.7818, 23.8492)
     -- (35.7813, 21.4812);
  \draw[blue]
    (35.7813, 21.4812)
     -- (31.7818, 19.8492);
  \draw[blue]
    (163.7818, 23.8492)
     -- (167.7818, 21.5352);
  \draw[blue]
    (163.7818, 19.8492)
     -- (167.7818, 21.5352);
  \node at (90,15) {$\longrightarrow$};
\end{tikzpicture}
}

\newcommand{\PipeSlide}{\begin{tikzpicture}[ipe import, baseline, trim left]
  \draw[shift={(0, 80)}, yscale=0.75]
    (0, 0)
     -- (0, -64);
  \draw
    (16, 80)
     -- (16, 32);
  \draw[red, dotted]
    (0, 56)
     arc[start angle=180, end angle=360, x radius=8, y radius=-4];
  \draw[red]
    (16, 56)
     arc[start angle=0, end angle=180, x radius=8, y radius=-4];
  \draw[blue]
    (48, 80)
     -- (16, 0);
  \draw[blue]
    (38, 62)
     -- (41.6, 64);
  \draw[blue]
    (42, 60)
     -- (41.6, 64);
  \draw[shift={(96, 80)}, yscale=0.75]
    (0, 0)
     -- (0, -64);
  \draw
    (112, 80)
     -- (112, 32);
  \draw[red, dotted]
    (96, 56)
     arc[start angle=180, end angle=360, x radius=8, y radius=-4];
  \draw[red]
    (112, 56)
     arc[start angle=0, end angle=180, x radius=8, y radius=-4];
  \draw[blue]
    (134, 62)
     -- (137.6, 64);
  \draw[blue]
    (138, 60)
     -- (137.6, 64);
  \draw[blue]
    (112, 44)
     .. controls (112, 40) and (108, 40) .. (106, 40)
     .. controls (104, 40) and (104, 40) .. (104, 36);
  \draw[blue]
    (96, 44)
     .. controls (96, 40) and (98, 40) .. (99, 40)
     .. controls (100, 40) and (100, 40) .. (100, 36);
  \draw[shift={(104, 36)}, yscale=0.75, blue]
    (0, 0)
     .. controls (0, -16) and (0, -16) .. (4, -16);
  \draw[shift={(100, 36)}, yscale=0.8, blue]
    (0, 0)
     .. controls (0, -20) and (0, -20) .. (8, -20);
  \draw[blue]
    (108, 24)
     .. controls (117.0667, 24) and (121.8667, 24.6667) .. (122.4, 26);
  \draw[blue]
    (108, 20)
     .. controls (116, 20) and (119.7333, 19.3333) .. (119.2, 18);
  \draw[blue]
    (112, 0)
     -- (119.2, 18);
  \draw[blue]
    (122.4, 26)
     -- (144, 80);
  \draw[blue, dotted]
    (96, 44)
     .. controls (96, 46.6667) and (98.6667, 48) .. (104, 48);
  \draw[blue, dotted]
    (112, 44)
     .. controls (112, 46.6667) and (109.3333, 48) .. (104, 48);
  \draw[shift={(208, 80)}, yscale=0.75]
    (0, 0)
     -- (0, -64);
  \draw
    (224, 80)
     -- (224, 32);
  \draw[red, dotted]
    (208, 56)
     arc[start angle=180, end angle=360, x radius=8, y radius=-4];
  \draw[red]
    (224, 56)
     arc[start angle=0, end angle=180, x radius=8, y radius=-4];
  \draw[blue]
    (246, 62)
     -- (249.6, 64);
  \draw[blue]
    (250, 60)
     -- (249.6, 64);
  \draw[blue]
    (220, 20)
     .. controls (228, 20) and (231.7333, 19.3333) .. (231.2, 18);
  \draw[blue]
    (224, 0)
     -- (231.2, 18);
  \draw[blue]
    (244.801, 52.002)
     .. controls (243.7337, 49.334) and (238.1333, 48) .. (228, 48);
  \draw[blue]
    (204, 48)
     .. controls (192, 48) and (192, 40) .. (192, 36);
  \draw[blue]
    (192, 36)
     .. controls (192, 28) and (192, 20) .. (204, 20);
  \draw[blue]
    (204, 20)
     -- (220, 20);
  \draw[blue]
    (244.801, 52.002)
     -- (256, 80);
\node at (160, 35) {$\longrightarrow$};
\node at (60, 35) {$\longrightarrow$};
\end{tikzpicture}
}

\newcommand{\KirbyZero}{\raise -4.5em
\hbox{
\begin{tikzpicture}
\begin{knot}
\begin{scope}
  \draw (-1,0.3) ellipse (0.3 and 0.1);
  \draw (-1.3,0.3)--(-1.3,-1.2); \draw (-0.7,0.3)--(-0.7,-1.2);
  \draw[blue, ->-=0.7] (-1, 0.2)--(-1,-1.2);
\end{scope}
  \draw (-2,-1) arc (180:360:1 and 0.4);
  \draw (-2,-1) arc (180:120:1 and 0.4);
  \draw (0,-1) arc (0:60:1 and 0.4);
  \draw (-2.4,-1) arc (180:360:1.4 and 0.8);
  \draw (-2.4,-1) arc (180:110:1.4 and 0.8);
  \draw (0.4,-1) arc (0:70:1.4 and 0.8);
  \draw[red] (-2.2,-1) arc (180:360:1.2 and 0.6);
  \draw[red] (-2.2,-1) arc (180:115:1.2 and 0.6);
  \draw[red] (0.2,-1) arc (0:65:1.2 and 0.6);
\begin{scope}
  \draw[dotted] (-1.3,-2.8) arc (180:0:0.3 and 0.1);
  \draw (-1.3,-2.8) arc (180:360:0.3 and 0.1);
  \draw (-1.3,-2)--(-1.3,-2.8); \draw (-0.7,-2)--(-0.7,-2.8);
  \draw[blue] (-1, -2)--(-1,-2.9);
\end{scope}
\begin{scope}[shift={(2.4,0)}]
  \draw (-2.4, -1) arc (180:0:0.2 and 0.1);
  \draw[dotted] (-2.4, -1) arc (180:360:0.2 and 0.1);
\end{scope}
\end{knot}
\end{tikzpicture}}
\quad \longrightarrow \quad
\raise -4.5em
\hbox{
\begin{tikzpicture}
\begin{knot}
  \draw (-1,0.3) ellipse (0.3 and 0.1);
  \draw (-1.3,0.3)--(-1.3,-0.7);
  \draw (-0.7,0.3)--(-0.7,-1.3);
  \draw[blue, ->-=0.7] (-1, 0.2)--(-1,-1.2);
  \strand (-1.3,-0.7) to[out=down,in=right] (-1.85,-0.9);
  \draw (-1.85,-0.9) arc (180:145:1.85 and 0.4);
  \strand (-1.85,-1.1) to[out=right,in=up] (-1.3,-1.3);
  \draw (-1.85,-1.1) arc (180:270:0.85 and 0.3);
  \draw[red] (-1.6,-0.9) arc (90:270:0.05 and 0.1);
  \draw[dotted, red] (-1.6,-0.9) arc (90:-90:0.05 and 0.1);
\begin{scope}
  \draw (0,-1) arc (0:60:1 and 0.4);
  \draw (0,-1) arc (0:-90:1 and 0.4); 
  \draw (-2.4,-1) arc (180:360:1.4 and 0.8);
  \draw (-2.4,-1) arc (180:110:1.4 and 0.8);
  \draw (0.4,-1) arc (0:70:1.4 and 0.8);
  \draw[red] (-2.2,-1) arc (180:360:1.2 and 0.6);
  \draw[red] (-2.2,-1) arc (180:115:1.2 and 0.6);
  \draw[red] (0.2,-1) arc (0:65:1.2 and 0.6);
\end{scope}
\begin{scope}
  \draw[dotted] (-1.3,-2.8) arc (180:0:0.3 and 0.1);
  \draw (-1.3,-2.8) arc (180:360:0.3 and 0.1);
  \draw (-1.3,-2)--(-1.3,-2.8); \draw (-0.7,-2)--(-0.7,-2.8);
  \draw[blue] (-1, -2)--(-1,-2.9);
\end{scope}
\begin{scope}[shift={(2.4,0)}]
  \draw (-2.4, -1) arc (180:0:0.2 and 0.1);
  \draw[dotted] (-2.4, -1) arc (180:360:0.2 and 0.1);
\end{scope}
\end{knot}
\end{tikzpicture}}
\quad \longrightarrow \quad
\raise -4.5em
\hbox{
\begin{tikzpicture}
\begin{knot}
  \draw (-1,0.3) ellipse (0.3 and 0.1);
  \draw (-1.3,0.3)--(-1.3,-0.7);
  \draw (-0.7,0.3)--(-0.7,-1.2);
  \draw[blue, ->-=0.7] (-1, 0.2)--(-1,-1.2);
  \strand (-1.3,-0.7) to[out=down,in=right] (-1.85,-0.9);
  \draw (-1.85,-0.9) arc (180:145:1.85 and 0.4);
  \strand (-1.85,-1.1) to[out=right,in=up] (-1.3,-1.3);  
  \draw (-1.85,-1.1) arc (180:270:0.85 and 0.3);
  \draw[green] (-1.75, -0.87) arc (180:150:1.75 and 0.25);
  \strand[green] (-1.75, -0.87)--(-1.5,-0.85);
  \draw[green] (-1.78, -1.15) arc (180:270:0.78 and 0.2);
  \strand[green] (-1.78,-1.15) to[out=right,in=left] (-1, -1.25);
  \draw[green] (-1,-1.25) arc (-90:50:0.65 and 0.25);
  \draw[green] (-1, -1.35) arc (-90:55:0.8 and 0.35);
\begin{scope}
  \draw (0,-1) arc (0:60:1 and 0.4);
  \draw (0,-1) arc (0:-90:1 and 0.4); 
  \draw (-2.4,-1) arc (180:360:1.4 and 0.8);
  \draw (-2.4,-1) arc (180:110:1.4 and 0.8);
  \draw (0.4,-1) arc (0:70:1.4 and 0.8);
  \draw[red] (-2.2,-1) arc (180:360:1.2 and 0.6);
  \draw[red] (-2.2,-1) arc (180:115:1.2 and 0.6);
  \draw[red] (0.2,-1) arc (0:65:1.2 and 0.6);
\end{scope}
\begin{scope}
  \draw[dotted] (-1.3,-2.8) arc (180:0:0.3 and 0.1);
  \draw (-1.3,-2.8) arc (180:360:0.3 and 0.1);
  \draw (-1.3,-2)--(-1.3,-2.8); \draw (-0.7,-2)--(-0.7,-2.8);
  \draw[blue] (-1, -2)--(-1,-2.9);
\end{scope}
\begin{scope}[shift={(2.4,0)}]
  \draw (-2.4, -1) arc (180:0:0.2 and 0.1);
  \draw[dotted] (-2.4, -1) arc (180:360:0.2 and 0.1);
\end{scope}
\end{knot}
\node[fill, white, inner sep = 1pt] at (-1.32,-1.265) {};
\fill[white] (-1.32, -1.256) circle (1pt);
\begin{scope}[shift={(-1.7,-1.15)}, xscale=0.14, yscale=0.14]
  \draw[green] (0,0)--(1,-1);
\end{scope}
\begin{scope}[shift={(-1.4,-1.2)}, xscale=0.14, yscale=0.14]
  \draw[green] (0,0)--(1,-1);
\end{scope}
\begin{scope}[shift={(-1.1,-1.25)}, xscale=0.1, yscale=0.1]
  \draw[green] (0,0)--(1,-1);
\end{scope}
\begin{scope}[shift={(-0.8,-1.24)}, xscale=0.09, yscale=0.09]
  \draw[green] (0,0)--(1,-1);
\end{scope}
\begin{scope}[shift={(-0.5,-1.16)}, xscale=0.08, yscale=0.08]
  \draw[green] (0,0)--(1,-1);
\end{scope}
\begin{scope}[shift={(-1.66,-0.79)}, xscale=0.07, yscale=0.07]
  \draw[green] (0,0)--(1,-1);
\end{scope}
\draw[green] (-0.4,-0.9)--(-0.3,-0.83);
\draw[green] (-0.37,-1.05)--(-0.215,-1.06);
\end{tikzpicture}}
\quad \longrightarrow \quad D\cdot \sum 
\raise -4.5em
\hbox{
\begin{tikzpicture}
\begin{knot}
  \node (A) [draw] at (-1, -0.8) {\tiny{$\varphi$}};
  \node (B) [draw] at (-1, -1.8) {\tiny{$\varphi$}};
  \draw (-1,0.3) ellipse (0.3 and 0.1);
  \draw (-1.3,0.3)--(-1.3,-2.8); \draw (-0.7,0.3)--(-0.7,-2.8);
  \draw[blue, ->-=0.7] (-1, 0.2)--(A.north);
  \draw[blue] (B.south)--(-1,-2.9);
\begin{scope}
  \draw[dotted] (-1.3,-2.8) arc (180:0:0.3 and 0.1);
  \draw (-1.3,-2.8) arc (180:360:0.3 and 0.1);
\end{scope}
\begin{scope}[red, shift={(0,1.5)}]
  \draw[dotted] (-1.3,-2.8) arc (180:0:0.3 and 0.1);
  \draw (-1.3,-2.8) arc (180:360:0.3 and 0.1);
\end{scope}
\end{knot}
\end{tikzpicture}
}}

%% file: tikzlibraryipe.code.tex
\usetikzlibrary{arrows.meta}

\makeatletter

\pgfdeclarearrow{
  name = ipe _linear,
  defaults = {
    length = +1bp,
    width  = +.666bp,
    line width = +0pt 1,
  },
  setup code = {
    \pgfarrowssetbackend{0pt}
    \pgfarrowssetvisualbackend{
      \pgfarrowlength\advance\pgf@x by-.5\pgfarrowlinewidth}
    \pgfarrowssetlineend{\pgfarrowlength}
    \ifpgfarrowreversed
      \pgfarrowssetlineend{\pgfarrowlength\advance\pgf@x by-.5\pgfarrowlinewidth}
    \fi
    \pgfarrowssettipend{\pgfarrowlength}
    \pgfarrowshullpoint{\pgfarrowlength}{0pt}
    \pgfarrowsupperhullpoint{0pt}{.5\pgfarrowwidth}
    \pgfarrowssavethe\pgfarrowlinewidth
    \pgfarrowssavethe\pgfarrowlength
    \pgfarrowssavethe\pgfarrowwidth
  },
  drawing code = {
    \pgfsetdash{}{+0pt}
    \ifdim\pgfarrowlinewidth=\pgflinewidth\else\pgfsetlinewidth{+\pgfarrowlinewidth}\fi
    \pgfpathmoveto{\pgfqpoint{0pt}{.5\pgfarrowwidth}}
    \pgfpathlineto{\pgfqpoint{\pgfarrowlength}{0pt}}
    \pgfpathlineto{\pgfqpoint{0pt}{-.5\pgfarrowwidth}}
    \pgfusepathqstroke
  },
  parameters = {
    \the\pgfarrowlinewidth,%
    \the\pgfarrowlength,%
    \the\pgfarrowwidth,%
  },
}

\pgfdeclarearrow{
  name = ipe _pointed,
  defaults = {
    length = +1bp,
    width  = +.666bp,
    inset  = +.2bp,
    line width = +0pt 1,
  },
  setup code = {
    \pgfarrowssetbackend{0pt}
    \pgfarrowssetvisualbackend{\pgfarrowinset}
    \pgfarrowssetlineend{\pgfarrowinset}
    \ifpgfarrowreversed
      \pgfarrowssetlineend{\pgfarrowlength}
    \fi
    \pgfarrowssettipend{\pgfarrowlength}
    \pgfarrowshullpoint{\pgfarrowlength}{0pt}
    \pgfarrowsupperhullpoint{0pt}{.5\pgfarrowwidth}
    \pgfarrowshullpoint{\pgfarrowinset}{0pt}
    \pgfarrowssavethe\pgfarrowinset
    \pgfarrowssavethe\pgfarrowlinewidth
    \pgfarrowssavethe\pgfarrowlength
    \pgfarrowssavethe\pgfarrowwidth
  },
  drawing code = {
    \pgfsetdash{}{+0pt}
    \ifdim\pgfarrowlinewidth=\pgflinewidth\else\pgfsetlinewidth{+\pgfarrowlinewidth}\fi
    \pgfpathmoveto{\pgfqpoint{\pgfarrowlength}{0pt}}
    \pgfpathlineto{\pgfqpoint{0pt}{.5\pgfarrowwidth}}
    \pgfpathlineto{\pgfqpoint{\pgfarrowinset}{0pt}}
    \pgfpathlineto{\pgfqpoint{0pt}{-.5\pgfarrowwidth}}
    \pgfpathclose
    \ifpgfarrowopen
      \pgfusepathqstroke
    \else
      \ifdim\pgfarrowlinewidth>0pt\pgfusepathqfillstroke\else\pgfusepathqfill\fi
    \fi
  },
  parameters = {
    \the\pgfarrowlinewidth,%
    \the\pgfarrowlength,%
    \the\pgfarrowwidth,%
    \the\pgfarrowinset,%
    \ifpgfarrowopen o\fi%
  },
}

\newdimen\ipeminipagewidth

\tikzstyle{ipe import} = [
  x=1bp, y=1bp,
  ipe node stretch/.store in=\ipenodestretch,
  ipe stretch normal/.style={ipe node stretch=1},
  ipe stretch normal,
  ipe node/.style={
    anchor=base west, inner sep=0, outer sep=0, scale=\ipenodestretch
  },
  ipe mark scale/.store in=\ipemarkscale,
  ipe mark tiny/.style={ipe mark scale=1.1},
  ipe mark small/.style={ipe mark scale=2},
  ipe mark normal/.style={ipe mark scale=3},
  ipe mark large/.style={ipe mark scale=5},
  ipe mark normal, % Set default
  ipe circle/.pic={
    \draw[line width=0.2*\ipemarkscale]
      (0,0) circle[radius=0.5*\ipemarkscale];
    \coordinate () at (0,0);
  },
  ipe disk/.pic={
    \fill (0,0) circle[radius=0.6*\ipemarkscale];
    \coordinate () at (0,0);
  },
  ipe fdisk/.pic={
    \filldraw[line width=0.2*\ipemarkscale]
      (0,0) circle[radius=0.5*\ipemarkscale];
    \coordinate () at (0,0);
  },
  ipe box/.pic={
    \draw[line width=0.2*\ipemarkscale, line join=miter]
      (-.5*\ipemarkscale,-.5*\ipemarkscale) rectangle
      ( .5*\ipemarkscale, .5*\ipemarkscale);
    \coordinate () at (0,0);
  },
  ipe square/.pic={
    \fill
      (-.6*\ipemarkscale,-.6*\ipemarkscale) rectangle
      ( .6*\ipemarkscale, .6*\ipemarkscale);
    \coordinate () at (0,0);
  },
  ipe fsquare/.pic={
    \filldraw[line width=0.2*\ipemarkscale, line join=miter]
      (-.5*\ipemarkscale,-.5*\ipemarkscale) rectangle
      ( .5*\ipemarkscale, .5*\ipemarkscale);
    \coordinate () at (0,0);
  },
  ipe cross/.pic={
    \draw[line width=0.2*\ipemarkscale, line cap=butt]
      (-.5*\ipemarkscale,-.5*\ipemarkscale) --
      ( .5*\ipemarkscale, .5*\ipemarkscale)
      (-.5*\ipemarkscale, .5*\ipemarkscale) --
      ( .5*\ipemarkscale,-.5*\ipemarkscale);
    \coordinate () at (0,0);
  },
  /pgf/arrow keys/.cd,
  ipe arrow normal/.style={scale=1},
  ipe arrow tiny/.style={scale=.4},
  ipe arrow small/.style={scale=.7},
  ipe arrow large/.style={scale=1.4},
  ipe arrow normal,
  /tikz/.cd,
  ipe arrows/.style={
    ipe normal/.tip={
      ipe _pointed[length=1bp, width=.666bp, inset=0bp,
                   quick, ipe arrow normal]},
    ipe pointed/.tip={
      ipe _pointed[length=1bp, width=.666bp, inset=0.2bp,
                   quick, ipe arrow normal]},
    ipe linear/.tip={
      ipe _linear[length = 1bp, width=.666bp,
                  ipe arrow normal, quick]},
    ipe fnormal/.tip={ipe normal[fill=white]},
    ipe fpointed/.tip={ipe pointed[fill=white]},
    ipe double/.tip={ipe normal[] ipe normal},
    ipe fdouble/.tip={ipe fnormal[] ipe fnormal},
    ipe arc/.tip={ipe normal},
    ipe farc/.tip={ipe fnormal},
    ipe ptarc/.tip={ipe pointed},
    ipe fptarc/.tip={ipe fpointed},
  },
  ipe arrows, % Set default sizes
]

\tikzset{
  rgb color/.code args={#1=#2}{%
    \definecolor{tempcolor-#1}{rgb}{#2}%
    \tikzset{#1=tempcolor-#1}%
  },
}

\makeatother